\documentclass[11pt]{amsart}

\usepackage{amsmath}
\usepackage{amssymb}
\usepackage{bbm}
\usepackage{pdfsync}
\usepackage{esint} %integrals
\usepackage[breaklinks=true]{hyperref}
\usepackage[utf8]{inputenc}

\usepackage[T1]{fontenc}
\usepackage{mathabx}
\usepackage[textsize=tiny,disable]{todonotes}
\newcommand{\R}{{\mathbb R}}       % Field of real numbers
\newcommand{\N}{{\mathbb N}}

\newcommand{\Z}{{\mathbb Z}}       % Ring of integer numbers
\newcommand{\DD}{{\mathcal D}}

\newcommand{\HH}{{\mathcal H}}
\newcommand{\LL}{{\mathcal L}}
\newcommand{\PP}{{\mathcal P}}
\newcommand{\QQ}{{\mathcal Q}}

\newcommand{\cS}{{\mathcal S}}

\newcommand{\TT}{{\mathcal T}}

\newcommand{\RR}{{\mathcal R}}
\newcommand{\cM}{{\mathcal M}}
\newcommand{\cA}{{\mathcal A}}
\newcommand{\cF}{{\mathcal F}}
\newcommand{\Ch}{{\mathcal Ch}}
\newcommand{\EE}{{\mathcal E}}
\newcommand{\UU}{{\mathcal U}}

\newcommand{\cI}{{\mathcal I}}
\newcommand{\cJ}{{\mathcal J}}

\newcommand{\diam}{{\rm diam}}
\newcommand{\dist}{{\rm dist}}

\newcommand{\rf}[1]{{(\ref{#1})}}

\newcommand{\supp}{\operatorname{supp}}

\newcommand{\vphi}{{\varphi}}
\newcommand{\ve}{{\varepsilon}}
\newcommand{\vv}{{\vspace{2mm}}}
\newcommand{\vvv}{{\vspace{3mm}}}
\newcommand{\wt}[1]{{\widetilde{#1}}}
\newcommand{\wh}[1]{{\widehat{#1}}}
\newcommand{\wk}[1]{{\widecheck{#1}}}

\newcommand{\noi}{\noindent}
\newcommand{\rest}{{\lfloor}}

\newcommand{\sss}{{\mathsf {Stop}}}
\newcommand{\ttt}{{\mathsf {Top}}}
\newcommand{\treg}{{\mathsf{Treg}}}
\newcommand{\DB}{{\mathsf {DB}}}
\newcommand{\Gen}{{\mathsf {Gen}}}
\newcommand{\tree}{{\rm Tree}}

\newcommand{\pv}{\operatorname{pv}}

\newcommand{\GG}{{\mathsf G}}

\newcommand{\bad}{{\mathsf{Bad}}}

\newcommand{\HD}{{\mathsf{HD}}}
\newcommand{\hd}{{\mathsf{hd}}}
\newcommand{\GH}{{\mathsf{GH}}}
\newcommand{\Ot}{{\mathsf{Ot}}}

\newcommand{\LD}{{\mathsf{LD}}}
\newcommand{\BR}{{\mathsf {BR}}}

\newcommand{\MDW}{{\mathsf{MDW}}}
\newcommand{\Stop}{{\mathsf{Stop}}}

\newcommand{\Trc}{{\mathsf{Trc}}}
\newcommand{\sL}{{\mathsf{L}}}
\newcommand{\sD}{{\mathsf{D}}}

\newcommand{\sF}{{\mathsf{F}}}
\newcommand{\sG}{{\mathsf{G}}}
\newcommand{\sH}{{\mathsf{H}}}
\newcommand{\sM}{{\mathsf{M}}}

\newcommand{\Neg}{{\mathsf{Neg}}}

\newcommand{\OP}{{\mathsf{OP}}}
\newcommand{\End}{{\mathsf{End}}}
\newcommand{\Reg}{{\mathsf{Reg}}}
\newcommand{\NDB}{{\mathsf{NDB}}}

\newcommand{\Ty}{{\mathsf{Ty}}}

\newcommand{\GDF}{{\mathsf{GDF}}}

\def\Xint#1{\mathchoice
{\XXint\displaystyle\textstyle{#1}}%
{\XXint\textstyle\scriptstyle{#1}}%
{\XXint\scriptstyle\scriptscriptstyle{#1}}%
{\XXint\scriptscriptstyle\scriptscriptstyle{#1}}%
\!\int}
\def\XXint#1#2#3{{\setbox0=\hbox{$#1{#2#3}{\int}$ }
\vcenter{\hbox{$#2#3$ }}\kern-.58\wd0}}

\def\avint{\;\Xint-}

\usepackage{pgf,tikz} %tikz: dibuixos i esquemes, com el cercle

\definecolor{ffffff}{rgb}{1.0,1.0,1.0}
\definecolor{qqqqff}{rgb}{0.0,0.0,1.0}
\definecolor{ffqqqq}{rgb}{1.0,0.0,0.0}
\definecolor{zzzzqq}{rgb}{0.6,0.6,0.0}
\definecolor{marronet}{rgb}{0.6,0.2,0}
\definecolor{negre}{rgb}{0,0,0}
\definecolor{vermell}{rgb}{0.8,0.05,0.05}
\definecolor{blau}{rgb}{0.3,0.2,1.}
\definecolor{blauclar}{rgb}{0.,0.,1.}
\definecolor{grisfosc}{rgb}{0.25098039215686274,0.25098039215686274,0.25098039215686274}
\definecolor{verd}{rgb}{0.1,0.6,0.1}
\definecolor{taronja}{rgb}{0.9,0.6,0.05}
\definecolor{vermellclar}{rgb}{1.,0.,0.}
\definecolor{verdet}{rgb}{0,0.8,0.1}
\definecolor{blauverd}{rgb}{0,0.4,0.2}
\definecolor{grisclar}{rgb}{0.6274509803921569,0.6274509803921569,0.6274509803921569}

\newtheorem{theorem}{Theorem}[section]
\newtheorem{lemma}[theorem]{Lemma}
\newtheorem{mlemma}[theorem]{Main Lemma}

\newtheorem{coro}[theorem]{Corollary}

\newtheorem{mpropo}[theorem]{Main Proposition}
\newtheorem{claim}[theorem]{Claim}
\newtheorem*{claim*}{Claim}

\newtheorem*{theorem*}{Theorem}

\theoremstyle{definition}

\theoremstyle{remark}
\newtheorem{rem}[theorem]{\bf Remark}

\numberwithin{equation}{section}

\newcommand{\brem}{\begin{rem}}
\newcommand{\erem}{\end{rem}}

\begin{document}

\title[The Riesz transforms and the Painlev\'e problem]{The measures with $L^2$-bounded Riesz transform and the Painlev\'e problem}

\author[D. D\k{a}browski]{Damian D\k{a}browski}
\address{Damian D\k{a}browski\\ Departament de Matem\`atiques, Universitat Aut\`onoma de Barcelona, and Barcelona Graduate School of Mathematics (BGSMath)\newline\indent Edifici C Facultat de Ci\`encies, 08193 Bellaterra, Barcelona, Catalonia, Spain\newline\indent
Current address: P.O. Box 35 (MaD), 40014 University of Jyväskylä, Finland}

\email{damian.m.dabrowski ``at'' jyu.fi}

\author[X. Tolsa]{Xavier Tolsa}
\address{Xavier Tolsa \\
ICREA, Passeig Llu\'{\i}s Companys 23 08010 Barcelona, Catalonia;  
Universitat Aut\`onoma de Barcelona and Centre de Recerca Matem\`atica,
08193 Bellaterra, Catalonia.
}

\email{xtolsa ``at'' mat.uab.cat}
%\keywords{Harmonic function, unique continuation, frequency function}

\thanks{Both authors were supported by 2017-SGR-0395 (Catalonia) and MTM-2016-77635-P (MINECO, Spain). D.D. was supported by the Academy of Finland via the project \textit{Incidences on Fractals}, grant No. 321896. X.T. was supported by
 the European Research Council (ERC) under the European Union's Horizon 2020 research and innovation programme (grant agreement 101018680).}


\begin{abstract}
In this work we provide a geometric characterization of the measures $\mu$ in $\R^{n+1}$ with polynomial upper growth of degree $n$ such that the $n$-dimensional Riesz transform $\RR\mu (x) = \int \frac{x-y}{|x-y|^{n+1}}\,d\mu(y)$ belongs to $L^2(\mu)$. More precisely, it is shown that
$$\|\RR\mu\|_{L^2(\mu)}^2 + \|\mu\|\approx \int\!\!\int_0^\infty \beta_{2,\mu}(x,r)^2\,\frac{\mu(B(x,r))}{r^n}\,\frac{dr}r\,d\mu(x) +
\|\mu\|,$$
where $\beta_{\mu,2}(x,r)^2 = \inf_L \frac1{r^n}\int_{B(x,r)} \left(\frac{\dist(y,L)}r\right)^2\,d\mu(y),$
with the infimum taken over all affine $n$-planes $L\subset\R^{n+1}$. 
As a corollary, we obtain a characterization of the removable sets for Lipschitz
harmonic functions in terms of a metric-geometric potential and we deduce that the class of removable sets for Lipschitz harmonic functions is invariant by bilipschitz mappings.

\end{abstract}

\maketitle

\tableofcontents

\maketitle

\tableofcontents

\section{Introduction}

Given a Radon measure $\mu$ in $\R^{n+1}$, its 
($n$-dimensional) Riesz transform at $x\in\R^{n+1}$ is defined by
$$\RR\mu (x) = \int \frac{x-y}{|x-y|^{n+1}}\,d\mu(y),$$
whenever the integral makes sense. For $f\in L^1_{loc}(\mu)$,
one writes $\RR_\mu f(x) = \RR(f\mu)(x)$.
Given $\ve>0$, the $\ve$-truncated Riesz transform of $\mu$ equals
$$\RR_\ve\mu (x) = \int_{|x-y|>\ve} \frac{x-y}{|x-y|^{n+1}}\,d\mu(y),$$
and the operator $\RR_{\mu,\ve}$ is defined by
 $\RR_{\mu,\ve}f(x) = \RR_\ve(f\mu)(x)$. 
 
We say that $\RR_\mu$ is bounded in $L^2(\mu)$ if
the operators $\RR_{\mu,\ve}$ are bounded in $L^2(\mu)$ uniformly on $\ve$, and then we denote
$$\|\RR_\mu\|_{L^2(\mu)\to L^2(\mu)} = \sup_{\ve>0} \|\RR_{\mu,\ve}\|_{L^2(\mu)\to L^2(\mu)}.$$
We also write
$$\RR_*\mu(x) = \sup_{\ve>0} |\RR_{\ve}\mu(x)|, \qquad  \pv\RR\mu(x) = \lim_{\ve>0} \RR_{\ve}\mu(x),$$
in case that the latter limit exists. Sometimes we will abuse notation by writing $\RR\mu$ instead of $\pv\RR\mu$.

This paper provides a full geometric description of the measures $\mu$ with no point masses such that $\RR_\mu$ is bounded in
$L^2(\mu)$. 
In the case $n=1$, such description has already been obtained (see \cite{MV}, \cite{Leger}, \cite{Tolsa-duke}), relying on the connection between Menger curvature and the Cauchy kernel
found by Melnikov \cite{Melnikov}. In higher dimensions, a similar connection is missing,
and thus obtaining analogous results presents major difficulties. %The obtention of such geometric description solves a long-standing open problem in this area. 
In the case when the measure $\mu$ is AD-regular (i.e., Ahlfors-David regular) that geometric description is equivalent to the 
codimension $1$ David-Semmes problem, solved by Nazarov, Tolsa, and Volberg
in \cite{NToV1}. Recall that a measure $\mu$ is AD-regular (or $n$-AD-regular) if there exists a constant $C>0$ such that
$$C^{-1}\,r^n\leq \mu(B(x,r))\leq C\,r^n\quad \mbox{ for all $x\in\supp\mu$ and $0<r\leq \diam(\supp\mu)$.}$$
One of the main motivations for the description of the measures $\mu$ such that $\RR_\mu$ is bounded in $L^2(\mu)$
is the characterization of the removable singularities for Lipschitz harmonic functions. Also, one may expect other
applications regarding the study of harmonic and elliptic measures. Indeed, in some of the recent advances on this topic, the connection between harmonic measure, the Riesz transform, and rectifiability has played an essential role (see \cite{AHM3TV}, \cite{AMT}, and \cite{AMTV}, for example).

Next we need to introduce additional notation. For a ball $B\subset \R^{n+1}$, we consider its $n$-dimensional density (with respect to $\mu$):
$$\theta_\mu(B)= \frac{\mu(B)}{r(B)^n},$$
and its $\beta_{2,\mu}$ coefficient:
$$\beta_{2,\mu}(B) = \inf_L \left(\frac1{r(B)^n}\int_B \left(\frac{\dist(x,L)}{r(B)}\right)^2\,d\mu(x)\right)^{1/2},$$
where the infimum is taken over all $n$-planes $L\subset\R^{n+1}$ and $r(B)$ stands for the radius of $B$. For
$B=B(x,r)$ we may also write $\theta_\mu(x,r)$ and $\beta_{2,\mu}(x,r)$ instead of $\theta_\mu(B)$ and 
$\beta_{2,\mu}(B)$. The coefficients $\beta_{2,\mu}$ were introduced by David and Semmes in their fundamental works \cite{DS1}, \cite{DS2} on uniform rectifiability. They can be considered as $L^2$ variants of 
some analogous coefficients considered previously by Peter Jones in his celebrated travelling salesman theorem \cite{Jones}.

The first main result of this paper is the following.

\begin{theorem}\label{teomain1}
Let $\mu$ be a Radon measure in $\R^{n+1}$ %with compact support 
satisfying the polynomial growth condition
\begin{equation}\label{eqgrow00}
\mu(B(x,r))\leq \theta_0\,r^n\quad \mbox{ for all $x\in\supp\mu$ and all $r>0$}
\end{equation}
and such that $\|\RR_*\mu\|_{L^1(\mu)}<\infty$.
Then
\begin{equation}\label{eqbetawolff}
\int\!\!\int_0^\infty \beta_{2,\mu}(x,r)^2\,\theta_\mu(x,r)\,\frac{dr}r\,d\mu(x)\leq C\,(\big\|\pv\RR\mu\|_{L^2(\mu)}^2
+\theta_0^2\,\|\mu\|\big),
\end{equation}
where $C$ is an absolute constant.
\end{theorem}

Let us remark that the growth condition \rf{eqgrow00} and the assumption that $\|\RR_*\mu\|_{L^1(\mu)}<\infty$ ensures that $\RR_*\mu(x)<\infty$ $\mu$-a.e., and in turn this implies the existence of principal values $\pv\RR\mu(x)$ $\mu$-a.e. This can be easily deduced from \cite{NToV2} and \cite{Mattila-Verdera}, using arguments analogous to the ones for the Cauchy transform in \cite[Theorem 8.1 and Corollary 8.17]{Tolsa-llibre}, say.
So $\big\|\pv\RR\mu\|_{L^2(\mu)}$ is well defined. Nevertheless, we remark that the issue of the existence of principal values for the Riesz transform is not essential in Theorem \ref{teomain1}, and this
could be avoided by replacing $\pv\RR\mu$ in \rf{eqgrow00} by a suitable representative of $\RR\mu$ defined $\mu$-a.e.\ (see for example 
\cite{JN1} for some related techniques).

 A converse to the estimate \rf{eqbetawolff} also holds: if $\mu$ satisfies the growth condition
\rf{eqgrow00}, then
\begin{equation}\label{eqbetawolff'}
\|\pv\RR\mu\|_{L^2(\mu)}^2\leq C\,\int\!\!\int_0^\infty \beta_{2,\mu}(x,r)^2\,\theta_\mu(x,r)\,\frac{dr}r\,d\mu(x) 
+C\,\theta_0^2\,\|\mu\|,
\end{equation}
where $C$ is an absolute constant. This was shown in \cite{Azzam-Tolsa} in the case $n=1$, and in \cite{Girela} in full generality.

From \rf{eqbetawolff'}, Theorem \ref{teomain1}, and a direct application of the $T1$ theorem and Cotlar's inequality for non-doubling measures (\cite{NTrV1}, \cite{NTrV2}) one deduces the following.

\begin{theorem}\label{teomain2}
Let $\mu$ be a Radon measure in $\R^{n+1}$ with no point masses. Then $\RR_\mu$ is bounded in $L^2(\mu)$ if and
only if it satisfies the polynomial growth condition
\begin{equation}\label{eqgrow01}
\mu(B(x,r))\leq C\,r^n\quad \mbox{ for all $x\in\supp\mu$ and all $r>0$}
\end{equation}
and
\begin{equation}\label{eqbetawolff2}
\int_B\int_0^{r(B)} \beta_{2,\mu}(x,r)^2\,\theta_\mu(x,r)\,\frac{dr}r\,d\mu(x)\leq C^2\,\mu(B)\quad\mbox{ for any ball
$B\subset\R^{n+1}$.}
\end{equation}
Further, the optimal constant $C$ is comparable to $\|\RR_\mu\|_{L^2(\mu)\to L^2(\mu)}$.
\end{theorem}

In the case $n=1$ the preceding results are already known. They were proven in \cite{Azzam-Tolsa}, relying on the
corona decomposition involving the curvature of $\mu$ from \cite{Tolsa-bilip}. 
Further, for arbitrary $n>1$, it is also known that if one assumes that there exists a sufficiently large class of $L^2(\mu)$ bounded singular integral operators with an odd Calder\'on-Zygmund kernel, then the condition \rf{eqbetawolff2}
 holds, in any codimension (i.e., assuming that, instead of $\R^{n+1}$, the ambient space is $\R^d$, with $d\geq n$). This was proved in \cite{JNT}.

Observe that when $\mu$
is $n$-AD-regular, $\theta_\mu(x,r)\approx1$ for all $x\in\supp\mu$ and $0<r\leq \diam(\supp\mu)$, and so
the condition \rf{eqbetawolff2} is equivalent to the uniform $n$-rectifiability of $\mu$, by \cite{DS1}.
 So one deduces that the $L^2(\mu)$ boundedness
of $\RR_\mu$ implies the uniform $n$-rectifiability of $\mu$ and then one recovers the solution of the David-Semmes
problem from \cite{NToV1}. Recall that a measure $\mu$ in $\R^d$ is called {\em uniformly $n$-rectifiable} (UR) if it is $n$-AD-regular and
there exist constants $\kappa, M >0$ such that for all $x \in \supp\mu$ and all $0<r\leq \diam(\supp\mu)$ 
there is a Lipschitz mapping $g$ from the ball $B_n(0,r)\subset\R^{n}$ to $\R^d$ with $\text{Lip}(g) \leq M$ such that
$$
\mu(B(x,r)\cap g(B_{n}(0,r)))\geq \kappa\, r^{n}.$$

It is worth comparing Theorem \ref{teomain2} with a related result obtained in \cite{JNRT} in connection
with the fractional Riesz transform $\RR^s$ associated with the kernel
$x/|x|^{s+1}$ for $s\in (n,n+1)$. The precise result, which involves the $s$-dimensional density $\theta_\mu^s(x,r) = \mu(B(x,r))r^{-s}$, is the following.

\begin{theorem*}[\cite{JNRT}]
Let $\mu$ be a Radon measure in $\R^{n+1}$ with no point masses and let $s\in(n,n+1)$. Then $\RR_\mu^s$ is bounded in $L^2(\mu)$ if and
only if \begin{equation}\label{eqbetawolff3}
\int_B\int_0^{r(B)} \theta^s_\mu(x,r)^2\,\frac{dr}r\,d\mu(x)\leq C^2\,\mu(B)\quad\mbox{ for any ball
$B\subset\R^{n+1}$.}
\end{equation}
Further, the optimal constant $C$ is comparable to $\|\RR_\mu^s\|_{L^2(\mu)\to L^2(\mu)}$.
\end{theorem*}

In the preceding theorem one does not need to ask any growth condition analogous to \rf{eqgrow01}
(with $n$ interchanged with $s$) because this condition is already implied by \rf{eqbetawolff3}. Observe that in
\rf{eqbetawolff3} the density $\theta^s_\mu(x,r)$ replaces $\beta_{2,\mu}(x,r)^2$ in \rf{eqbetawolff2}, which scales similarly to
$\theta^s(x,r)$ when $n=s$. On the other hand, the proof of the last theorem in \cite{JNRT} makes an
extensive use of blowup techniques, which essentially rely on the fact that any measure $\mu$ satisfying the growth condition $\mu(B(x,r))\leq r^s$ for all $x\in\R^{n+1}$, $r>0$, and such that 
$\RR^s\mu=0$ in a suitable $BMO(\mu)$ sense, must be the zero measure (see \cite{JN1}, \cite{JN2}). In 
the codimension $1$ case, one might expect that if $\mu$ satisfies \rf{eqgrow01}
 and $\RR\mu=0$ in the $BMO(\mu)$ sense, then $\mu=c\HH^n|_L$ for some $n$-plane $L$. However, this is 
 still an open problem. If this were known to be true, probably in the present paper we
 could use blowup arguments analogous to the ones in \cite{JNRT}.

As shown in \cite{Azzam-Tolsa}, the finiteness of the double integral on the left hand side of \rf{eqbetawolff}
is equivalent to the existence of a suitable corona decomposition for $\mu$ satisfying an appropriate packing
condition. This condition is stable by bilipschitz maps (see also \cite{Girela} for more details). So we get the
following corollary.

\begin{coro}\label{coro1}
Let $\mu$ be a Radon measure in $\R^{n+1}$ with no point masses. Let $\vphi:\R^{n+1}\to\R^{n+1}$ be a bilipschitz
map. Let $\sigma=\vphi\#\mu$ be the image measure of $\mu$ by $\vphi$.
If $\RR_\mu$ is bounded in $L^2(\mu)$, then $\RR_\sigma$ is bounded in $L^2(\sigma)$. Further,
$$\|\RR_\sigma\|_{L^2(\sigma)\to L^2(\sigma)}\leq C\,\|\RR_\mu\|_{L^2(\mu)\to L^2(\mu)},$$
where $C$ depends only on the bilipschitz constant of $\vphi$.
\end{coro}

Remark that, up to now, the preceding result was not  known even for the case of invertible affine maps such as the one defined by
$\vphi(x_1,\ldots,x_{n+1}) = (2x_1,x_2,\ldots,x_{n+1}).$

As shown in \cite{Girela}, the conditions \rf{eqgrow01} and \rf{eqbetawolff2} imply the $L^2(\mu)$ boundedness of any singular integral operator of the form
$$T_\mu f(x) =\int K(x-y)\,f(y)\,d\mu(y),$$
where $K$ is an odd kernel such that
\begin{equation}
\label{eqKernel}
|\nabla^j K(x)|\lesssim \frac1{|x|^{n+j}}\quad \mbox{ for all $x\neq0$ and $0\leq j\leq 2$}
\end{equation}
(remark that $T_\mu$ is said to be bounded in $L^2(\mu)$ if the truncated operators $T_{\mu,\ve}$, defined analogously to $\RR_{\mu,\ve}$, are bounded in $L^2(\mu)$ uniformly on $\ve>0$).
Then we deduce the following.

\begin{coro}\label{coro1.5}
Let $\mu$ be a Radon measure in $\R^{n+1}$ with no point masses. Let $T_\mu$ be a singular integral 
operator associated with an odd kernel $K$ satisfying \rf{eqKernel}.
If $\RR_\mu$ is bounded in $L^2(\mu)$, then $T_\mu$ is also bounded in $L^2(\mu)$. Further,
$$\|T_\mu\|_{L^2(\mu)\to L^2(\mu)}\leq C\,\|\RR_\mu\|_{L^2(\mu)\to L^2(\mu)},$$
where $C$ depends just on $n$ and the implicit constants in \rf{eqKernel}.
\end{coro}

We turn now to the applications of the results above to Lipschitz harmonic functions and Lipschitz harmonic
capacity. Given a compact set $E\subset \R^{n+1}$, one says that $E$ is removable for Lipschitz harmonic functions 
if for any open set $\Omega\supset E$, any function $f:\Omega\to\R$ which is Lipschitz in $\Omega$ and harmonic in
$\Omega\setminus E$
can be extended in a harmonic way to the whole $\Omega$. To study this problem and some related questions on approximation by Lipschitz harmonic functions it is useful to introduce the Lipschitz harmonic capacity $\kappa$
(see \cite{Paramonov} and \cite{Mattila-Paramonov}). This is defined by
$$\kappa(E) = \sup |\langle \Delta f,1\rangle|,$$
where the supremum is taken over all Lipschitz functions $f:\R^{n+1}\to\R$ which are harmonic in $\R^{n+1}\setminus E$ and satisfy $\|\nabla f\|_\infty\leq1$, with $\Delta f$ understood in the sense of distributions. It turns
out that $E$ is removable for Lipschitz harmonic functions if and only if $\kappa(E)=0$.

Extending previous results for analytic capacity from \cite{Tolsa-sem},  Volberg showed in \cite{Volberg} that
$$\kappa(E) \approx \sup \mu(E),$$
where the supremum is taken over all measures $\mu$ satisfying the polynomial growth condition 
\rf{eqgrow01} with constant $C=1$ and such that $\|\RR_\mu\|_{L^2(\mu)\to L^2(\mu)}\leq 1$.
Combining this result with Theorem \ref{teomain2}, we obtain:

\begin{coro}\label{coro2}
Let $E\subset\R^{n+1}$ be compact. Then
$$\kappa(E)\approx \mu(E),$$
where the supremum is taken over all Radon measures $\mu$ such that
$$
\mu(B(x,r))\leq r^n\quad \mbox{ for all $x\in\supp\mu$ and all $r>0$}
$$
and
$$
\int\!\!\int_0^\infty \beta_{2,\mu}(x,r)^2\,\theta_\mu(x,r)\,\frac{dr}r\,d\mu(x)\leq \mu(E).
$$
\end{coro}

To derive this corollary, remark that if $\mu$ satisfies the conditions above, by Chebyshev one deduces that there is a big piece
$F\subset E\cap\supp\mu$, with $\mu(F)\approx\mu(E)$, such that the measure $\wt\mu = \mu|_F$ satisfies \rf{eqbetawolff2}, and so
$\RR_{\wt \mu}$ is bounded in $L^2(\wt \mu)$. Hence $\kappa(E)\gtrsim \mu(F)\approx\mu(E)$. The converse
direction of the corollary is a straightforward consequence of the aforementioned theorem of Volberg and 
Theorem \ref{teomain2}.

As explained above, the conditions on the measure $\mu$ in Corollary \ref{coro2} are stable by bilipschitz maps. So we deduce that if $\vphi:\R^{n+1}
\to\R^{n+1}$ is bilipschitz, then
$$\kappa(E)\approx \kappa(\vphi(E))\quad\mbox{ for any compact set $E\subset\R^{n+1}$,}$$
with the implicit constant just depending on the bilipschitz constant of $\vphi$ and the ambient dimension.

Another suggestive characterization of the capacity $\kappa(E)$ can be given in terms of the following potential, 
which we call the Jones-Wolff potential of $\mu$:
$$U_\mu(x) = \sup_{r>0} \theta_\mu(x,r) + \left(\int_0^\infty \beta_{2,\mu}(x,r)^2\,\theta_\mu(x,r)\,\frac{dr}r\right)^{1/2}.$$

\begin{coro}\label{coro3}
Let $E\subset\R^{n+1}$ be compact. Then
$$\kappa(E) \approx \sup\{\mu(E):\,U_\mu(x)\leq 1\,\forall x\in E\}.$$
\end{coro}

An immediate consequence of this result is that $E$ is non-removable for Lipschitz harmonic functions if and only if
it supports a non-zero measure such that $U_\mu(x)\leq1$ for all $x\in E$. 

The characterization of the capacity $\kappa$ and of removable sets for Lipschitz harmonic functions 
in terms of a metric-geometric potential such as $U_\mu$ should be considered as an analogue of the characterization
of analytic capacity and of removable sets for bounded analytic functions
in terms of curvature of measures \cite{Tolsa-sem}. So one can think of the results stated in Corollaries
\ref{coro2} and \ref{coro3} as possible solutions of the Painlevé problem for Lipschitz harmonic functions.

\subsection{Sketch of the proof}
Next we will describe the main ideas involved in the proof of Theorem \ref{teomain1}, as well as the 
main difficulties and innovations.

At the center of our proof is the notion of \emph{cubes dominated from below}, which can be described as follows. For a Radon measure $\sigma$ we consider the Wolff type
energy
$$\mathbb E(\sigma) = \int\!\! \int_0^\infty \left(\frac{\sigma(B(x,r))}{r^{n-\frac38}}\right)^2\,\frac{dr}r\,d\sigma(x)=\int\!\! \int_0^\infty r^{\frac34}\,\theta_\sigma(B(x,r))^2\,\frac{dr}{r}\,d\sigma(x).$$
Given the Radon measure $\mu$, we consider a suitably modified version of the David-Mattila lattice $\DD_\mu$ associated 
with $\mu$. Then for a given $\PP$-doubling cube $Q\in\DD_\mu$, we let $\EE_\infty(9Q)$ be a modified, discrete version of $\mathbb E(\mu|_{9Q})\,\ell(Q)^{-3/4}$.
We say that a $\PP$-doubling cube $Q$ is $M$-dominated from below, and we write $Q\in\DB(M)$, if
$$\EE_\infty(9Q) \geq M^2\,\Theta(Q)^2\,\mu(Q),$$
where $M\gg1$ is some fixed constant, $\Theta(Q)$ is another discrete version of $\theta_\mu(2B_Q)$, and $B_Q$ is a ball concentric with $Q$, containing $Q$, with radius
comparable to $\ell(Q)$ (for the precise definitions of $\PP$-doubling cubes, $\Theta(Q)$, and $\EE_\infty(9Q)$ see Section \ref{sec:Pdoubling}). Essentially, $Q\in\DB(M)$ means that there are many subcubes of $Q$ which are not too small, and whose $\mu$-density is much larger than that of $Q$.

The proof of Theorem \ref{teomain1} is divided into two main propositions. In Main Proposition \ref{propomain} we prove the desired estimate up to an error term involving cubes dominated from below:
\begin{equation}\label{eq:1}
\int\!\!\int_0^\infty \beta_{2,\mu}(x,r)^2\,\theta_\mu(x,r)\,\frac{dr}r\,d\mu(x)\leq C(M)\,\Big(\|\RR\mu\|_{L^2(\mu)}^2
+\theta_0^2\,\|\mu\| + \sum_{Q\in\DB(M)} \EE_\infty(9Q)\Big).
\end{equation}
In Main Proposition \ref{propomain2} we prove that if the $\DB$-parameter $M$ is larger than some dimensional constant $M_0$, then we may estimate the error term:
\begin{equation}\label{eqwwwpp}
\sum_{Q\in\DB(M)} \EE_\infty(9Q)\leq C\,\big(\|\RR\mu\|_{L^2(\mu)}^2 +\theta_0^2\,\|\mu\| \big).
\end{equation}
Clearly, these two estimates give Theorem \ref{teomain1}.

Originally, Main Proposition \ref{propomain} was proved in \cite{DT}, while Main Proposition \ref{propomain2} was proved in \cite{Tolsa-riesz}. In this article we merge the two proofs, making the whole argument more streamlined and significantly shorter.\footnote{Roughly speaking, Sections \ref{sec:DMlatt} -- \ref{sec6} and \ref{sec3.3} -- \ref{sec9} come from \cite{DT}, while Sections \ref{sectrans} and \ref{sec5**} -- \ref{sec100} come from \cite{Tolsa-riesz}.} Our proof is divided into 4 parts. 

\underline{Part 1}, which consists of Sections \ref{sec:DMlatt} and \ref{sec:Pdoubling}, contains various preliminaries, including an enhanced version of the dyadic lattice of David and Mattila.\vv

In \underline{Part 2}, which includes Sections \ref{sec4}, \ref{sec6}, and \ref{sectrans}, we prove results which are eventually used in the proofs of both \eqref{eq:1} and \rf{eqwwwpp}. Roughly speaking, showing either \eqref{eq:1} or \rf{eqwwwpp} involves proving the following implication: ``if the density of $\mu$ oscillates at certain scales and locations, then the corresponding Haar coefficients of $\RR\mu$ are large''. In Part 2 we prove this implication.

First, in Section \ref{sec4} we conduct a stopping time argument involving $\HD$ (high density), $\LD$ (low density), and $\OP$ (``optional'') stopping time conditions (the optional condition is assumed to be void in the proof of \eqref{eq:1}, whereas in the proof of \eqref{eqwwwpp} the optional condition involves the $\DB$-family). We say that a cube $R\in\DD_\mu$ has \emph{moderate decrement of Wolff energy}, denoted by $R\in\MDW$, if $R$ contains quite many $\HD$ stopping cubes. To each $R\in\MDW$ we associate certain extended tree $\TT$ (obtained by running the stopping time argument again), and we say that the extended tree is \emph{tractable}, denoted by $R\in\Trc$, if the density of many cubes in some intermediate generations increases (this is ensured by the $\MDW$ condition), and later in many stopping cubes the density decreases. In other words, for $R\in\Trc$ the $\mu$-density oscillates a lot at many scales and locations of the extended tree associated to $R$. As it turns out, to each cube in $\MDW$ one may associate a family of tractable trees, and so we can concentrate on the tractable trees.

Given a tractable tree $\TT$, in Section \ref{sec6} we construct a measure $\eta$ that approximates the measure $\mu$ at the level of some regularized stopping cubes of $\TT$. Then, we use a variational argument to obtain good lower bounds for $\|\RR\eta\|_{L^2(\eta)}$. These bounds are transferred to $\RR\mu$ (i.e., to the Haar coefficients of $\RR\mu$ for the cubes in the tree) in Section \ref{sectrans}. The idea of applying a variational argument like this originates from the work \cite{ENVo} by Eiderman, Nazarov, and Volberg and the reduction to the tractable trees comes from the work \cite{Reguera-Tolsa} by Reguera and the second author of this paper. The article \cite{JNRT} includes an improved version of that variational argument. Unlike the present paper, \cite{JNRT} makes an extensive use of compactness arguments, which do not work so well in our situation, where the geometry plays a more important role.

The implementation of the variational argument and the transference of the estimates from the approximating measure $\eta$ to $\mu$ is more difficult in the present
paper than in \cite{Reguera-Tolsa} or in other related works such as \cite{JNRT}. Some of the
difficulties arise from the fact that, for technical reasons (essentially, we need that many cubes of 
the intermediate generations with high density are located far from the boundary of the root of the tree), we have to consider trees of ``enlarged cubes''. This causes an overlapping between different trees that
has to be quantified carefully (this is done in Sections \ref{sec-layers} and \ref{sec7}). 
On the other hand, the transference of the lower estimate for $\|\RR\eta\|_{L^2(\eta)}$ 
to the Haar coefficients of $\RR\mu$ for the cubes in the tree originates many error terms. Roughly speaking, in order to be able to transfer that lower bound for $\|\RR\eta\|_{L^2(\eta)}$ to $\mu$ we need
the error terms to be smaller than the lower bound of $\|\RR\eta\|_{L^2(\eta)}$. Some of these
error terms are difficult to handle and we only can bound them in terms of the 
energies $\EE_\infty(9Q)$, which are problematic in case $Q\in\DB$. For this to work, we need 
an enhanced version of the dyadic lattice of David and Mattila that is obtained in Section
\ref{sec5}. This is an essential tool for our arguments. Furthermore, we can only quantify the presence of cubes from $\DB$ in most of the trees $\TT$ by a bootstrapping argument which gives rather weak bounds. \vv

In \underline{Part 3}, spanning Sections \ref{sec3.3} -- \ref{sec9}, we prove \eqref{eq:1}. In Section \ref{sec3.3} we perform a corona decomposition of the dyadic lattice $\DD_\mu$ into trees of cubes where the density of the cubes does not oscillate too much. We call the family of roots of these trees $\ttt$. Then, roughly speaking, in each tree the measure
$\mu$ behaves as if it were AD-regular, and from the $L^2(\mu)$ boundedness of $\RR_\mu$ and
\cite{NToV1}, one should expect that $\mu$ is close to some uniformly rectifiable measure at the locations and scales of the cubes in the tree, so that one can obtain a good packing condition
for the $\beta_{\mu,2}$ coefficients of the cubes in the tree.
For this strategy to work, we need to show that the family $\ttt$ satisfies a suitable packing condition. This is done in Sections \ref{sec-layers} and \ref{sec8}, and the proof relies on the results from Part 2.
The last stage of the proof of \eqref{eq:1} consists of estimating the 
$\beta_{\mu,2}$ coefficients in each tree where the density does not oscillate too much.
This step, which requires a delicate approximation by an AD-regular measure which has its own interest,
is performed in Section~\ref{sec9}.	\vv
	
In \underline{Part 4}, containing Sections \ref{sec5**} -- \ref{sec100}, we prove \rf{eqwwwpp}. The first step is the construction of a \emph{good dominating family}
$\GDF$, closely related to $\DB$. It consists of cubes $Q$ which, in a sense, contain many stopping cubes whose density is much larger than
the density $\Theta(Q)$ (in particular, $\GDF\subset\MDW$). The selection of this family, in Section \ref{sec5**}, is one of the key steps for the proof of \rf{eqwwwpp}. In Section \ref{sec7} we quantify the overlaps between the tractable trees associated to different
$R\in\GDF$. At this point the optional stopping time condition $\OP$ from Section \ref{sec4} is used. Finally, in Section \ref{sec100} the proof of \rf{eqwwwpp} is concluded, relying on the estimates obtained in Part 2, which are applied to the tractable trees associated to the cubes from $\GDF$.

\vv

In the whole paper we denote by $C$ or $c$ some constants that may depend on the dimension and perhaps other fixed parameters. Their values may change at different occurrences. On the contrary, constants with subscripts, like $C_0$, retain their values.
For $a,b\geq 0$, we write $a\lesssim b$ if there is $C>0$ such that $a\leq Cb$. We write $a\approx b$ to mean $a\lesssim b\lesssim a$.

\vv

% ***************************************************************************
%\vv

%\part{Preliminaries}

\newcounter{parte}

\bigskip
\begin{center} 
	\Large Part \refstepcounter{parte}\theparte\label{part-1}: Preliminaries
\end{center}
\smallskip

\addcontentsline{toc}{section}{\bf Part 1: Preliminaries}

\section{The modified dyadic lattice of David and Mattila}\label{sec:DMlatt}\label{sec5}

In this section we introduce the dyadic lattice of cubes
with small boundaries of David-Mattila \cite{David-Mattila} associated with a Radon measure $\mu$. The properties of the lattice are summarized in the next lemma. Later on we will show how its construction
can be modified in order to obtain additional properties relevant for our arguments.

\begin{lemma}[David, Mattila]
	\label{lemcubs}
	Let $\mu$ be a compactly supported Radon measure in $\R^{d}$.
	Consider two constants $C_0>1$ and $A_0>5000\,C_0$ and denote $E=\supp\mu$. 
	Then there exists a sequence of partitions of $E$ into
	Borel subsets $Q$, $Q\in \DD_{\mu,k}$, with the following properties:
	\begin{itemize}
		\item For each integer $k\geq0$, $E$ is the disjoint union of the ``cubes'' $Q$, $Q\in\DD_{\mu,k}$, and
		if $k<l$, $Q\in\DD_{\mu,l}$, and $R\in\DD_{\mu,k}$, then either $Q\cap R=\varnothing$ or else $Q\subset R$.
		\vv
		
		\item The general position of the cubes $Q$ can be described as follows. For each $k\geq0$ and each cube $Q\in\DD_{\mu,k}$, there is a ball $B(Q)=B(x_Q,r(Q))$ such that
		$$x_Q\in E, \qquad A_0^{-k}\leq r(Q)\leq C_0\,A_0^{-k},$$
		$$E\cap B(Q)\subset Q\subset E\cap 28\,B(Q)=E \cap B(x_Q,28r(Q)),$$
		and
		$$\mbox{the balls\, $5B(Q)$, $Q\in\DD_{\mu,k}$, are disjoint.}$$
		
		\vv
		\item The cubes $Q\in\DD_{\mu,k}$ have small boundaries. That is, for each $Q\in\DD_{\mu,k}$ and each
		integer $l\geq0$, set
		$$N_l^{ext}(Q)= \{x\in E\setminus Q:\,\dist(x,Q)< A_0^{-k-l}\},$$
		$$N_l^{int}(Q)= \{x\in Q:\,\dist(x,E\setminus Q)< A_0^{-k-l}\},$$
		and
		$$N_l(Q)= N_l^{ext}(Q) \cup N_l^{int}(Q).$$
		Then
		\begin{equation}\label{eqsmb2}
			\mu(N_l(Q))\leq (C^{-1}C_0^{-3d-1}A_0)^{-l}\,\mu(90B(Q)).
		\end{equation}
		\vv
		
		\item Denote by $\DD_{\mu,k}^{db}$ the family of cubes $Q\in\DD_{\mu,k}$ for which
		\begin{equation}\label{eqdob22}
			\mu(100B(Q))\leq C_0\,\mu(B(Q)).
		\end{equation}
		We have that $r(Q)=A_0^{-k}$ when $Q\in\DD_{\mu,k}\setminus \DD_{\mu,k}^{db}$
		and
		\begin{equation}\label{eqdob23}
			\mu(100B(Q))\leq C_0^{-l}\,\mu(100^{l+1}B(Q))\quad
			\mbox{for all $l\geq1$ with $100^l\leq C_0$ and $Q\in\DD_{\mu,k}\setminus \DD_{\mu,k}^{db}$.}
		\end{equation}
	\end{itemize}
\end{lemma}

\vv

\begin{rem}\label{rema00}
	We choose the constants $C_0$ and $A_0$ so that
	$$A_0 = C_0^{C(d)},$$
	where $C(d)$ depends
	%\footnote{Probably $C(d)=10d^{20}$ works, but this should be checked.}
	just on $d$ and $C_0$ is big enough.
\end{rem}

We use the notation $\DD_\mu=\bigcup_{k\geq0}\DD_{\mu,k}$. Observe that the families $\DD_{\mu,k}$ are only defined for $k\geq0$. So the diameters of the cubes from $\DD_\mu$ are uniformly
bounded from above.
%For $Q\in\DD$, we set $\DD(Q) =\{P\in\DD:P\subset Q\}$.
%Given $Q\in\DD_k$, we denote $J(Q)=k$. 
We set
$\ell(Q)= 56\,C_0\,A_0^{-k}$ and we call it the side length of $Q$. Notice that 
$$C_0^{-1}\ell(Q)\leq \diam(28B(Q))\leq\ell(Q).$$
Observe that $r(Q)\approx\diam(Q)\approx\ell(Q)$.
Also we call $x_Q$ the center of $Q$, and the cube $Q'\in \DD_{\mu,k-1}$ such that $Q'\supset Q$ the parent of $Q$.
We denote the family of cubes from $\DD_{\mu,k+1}$ which are contained in $Q$ by $\Ch(Q)$, and we call their elements children or sons of $Q$.
We set
$B_Q=28 B(Q)=B(x_Q,28\,r(Q))$, so that 
$$E\cap \tfrac1{28}B_Q\subset Q\subset B_Q\subset B(x_Q,\ell(Q)/2).$$

For a given $\gamma\in(0,1)$, let $A_0$ be big enough so that the constant $C^{-1}C_0^{-3d-1}A_0$ in 
\rf{eqsmb2} satisfies 
$$C^{-1}C_0^{-3d-1}A_0>A_0^{\gamma}>10.$$
Then we deduce that, for all $0<\lambda\leq1$,
\begin{align}\label{eqfk490}
	\mu\bigl(\{x\in Q:\dist(x,E\setminus Q)\leq \lambda\,\ell(Q)\}\bigr) + 
	\mu\bigl(\bigl\{x\in 3.5B_Q\setminus Q:\dist&(x,Q)\leq \lambda\,\ell(Q)\}\bigr)\\
	&\leq_\gamma
	c\,\lambda^{\gamma}\,\mu(3.5B_Q).\nonumber
\end{align}

We denote
$\DD_\mu^{db}=\bigcup_{k\geq0}\DD_{\mu,k}^{db}$.
Note that, in particular, from \rf{eqdob22} it follows that
\begin{equation}\label{eqdob*}
	\mu(3B_{Q})\leq \mu(100B(Q))\leq C_0\,\mu(Q)\qquad\mbox{if $Q\in\DD_\mu^{db}.$}
\end{equation}
For this reason we will call the cubes from $\DD_\mu^{db}$ doubling. 
Given $Q\in\DD_\mu$, we denote by $\DD_\mu(Q)$
the family of cubes from $\DD_\mu$ which are contained in $Q$. Analogously,
we write $\DD_\mu^{db}(Q) = \DD^{db}_\mu\cap\DD_\mu(Q)$.

As shown in \cite[Lemma 5.28]{David-Mattila}, every cube $R\in\DD_\mu$ can be covered $\mu$-a.e.\
by a family of doubling cubes:
\vv

\begin{lemma}\label{lemcobdob}
	Let $R\in\DD_\mu$. Suppose that the constants $A_0$ and $C_0$ in Lemma \ref{lemcubs} are
	chosen as in Remark \ref{rema00}. Then there exists a family of
	doubling cubes $\{Q_i\}_{i\in I}\subset \DD_\mu^{db}$, with
	$Q_i\subset R$ for all $i$, such that their union covers $\mu$-almost all $R$.
\end{lemma}

The following result is proved in \cite[Lemma 5.31]{David-Mattila}.
\vv

\begin{lemma}\label{lemcad22}
	Let $k\le j$, $R\in\DD_{\mu,k}$ and $Q\in\DD_{\mu,j}\cap\DD_{\mu}(R)$ be a cube such that all the intermediate cubes $S$,
	$Q\subsetneq S\subsetneq R$ are non-doubling (i.e.\ belong to $\DD_\mu\setminus \DD_\mu^{db}$).
	Suppose that the constants $A_0$ and $C_0$ in Lemma \ref{lemcubs} are
	chosen as in Remark \ref{rema00}. 
	Then
	\begin{equation}\label{eqdk88}
		\mu(100B(Q))\leq A_0^{-10d(j-k-1)}\mu(100B(R)).
	\end{equation}
\end{lemma}
\vv
%Let us remark that the constant $10$ in \rf{eqdk88} can be replaced by any other positive 
%constant if $A_0$ and $C_0$ are chosen suitably in Lemma \ref{lemcubs}, as shown in (5.30) of
%\cite{David-Mattila}.

%Recall that, given a ball (or an arbitrary set) $B\subset \R^{n+1}$, we consider its $n$-dimensional density:
%$$\theta_\mu(B)= \frac{\mu(B)}{\diam(B)^n}.$$
%We will also write $\theta_\mu(x,r)$ instead of $\theta_\mu(B(x,r))$.

From this lemma we deduce:

\vv
\begin{lemma}\label{lemcad23}
	Let $Q,R\in\DD_\mu$ be as in Lemma \ref{lemcad22}.
	Then
	$$\theta_\mu(100B(Q))\leq (C_0A_0)^{n+1}\,A_0^{-9d(j-k-1)}\,\theta_\mu(100B(R))$$
	and
	$$\sum_{S\in\DD_\mu:Q\subset S\subset R}\theta_\mu(100B(S))\leq c\,\theta_\mu(100B(R)),$$
	with $c$ depending on $C_0$ and $A_0$.
\end{lemma}

For the easy proof, see
\cite[Lemma 4.4]{Tolsa-memo}, for example.

For $f\in L^2(\mu)$ and $Q\in\DD_\mu$ we define
\begin{equation}\label{eqdq1}
	\Delta_Q f=\sum_{S\in\Ch(Q)}m_{\mu,S}(f)\chi_S-m_{\mu,Q}(f)\chi_Q,
\end{equation}
where $m_{\mu,S}(f)$ stands for the average of $f$ on $S$ with respect to $\mu$.
Then we have the orthogonal expansion, for any cube $R\in\DD_\mu$,
$$\chi_{R} \bigl(f - m_{\mu,R}(f)\bigr) = \sum_{Q\in\DD_\mu(R)}\Delta_Q f,$$
in the $L^2(\mu)$-sense, so that
$$\|\chi_{R} \bigl(f - m_{\mu,R}(f)\|_{L^2(\mu)}^2 = \sum_{Q\in\DD_\mu(R)}\|\Delta_Q f\|_{L^2(\mu)}^2.$$

In this paper we will have to estimate terms such as $\|\RR(\chi_{Q}\mu)\|_{L^2(\mu\rest_{2B_Q\setminus Q})}$, 
which leads to dealing with integrals of the form
$$\int_{2B_Q\setminus Q}\left(\int_Q \frac1{|x-y|^n}\,d\mu(y)\right)^2 d\mu(x).$$
Our next objective is to show that integrals such as this one can be estimated in terms of the Wolff type energy $\EE(2Q)$, to be defined soon.

We need some additional notation.\todo{the definition of $\lambda Q$ was above Lemma 2.8, but $2Q$ appears already in Lemma 2.6, so I moved the definition}
Given $Q\in\DD_\mu$ and $\lambda>1$, we denote by $\lambda Q$ the union of cubes $P$ from the same
generation as $Q$ such that $\dist(x_Q,P)\leq \lambda \,\ell(Q)$. Notice that
\begin{equation}\label{eqlambq12}
	\lambda Q\subset B(x_Q,(\lambda+\tfrac12)\ell(Q)).
\end{equation}
Also, we let
$$\DD_\mu(\lambda Q)=\{P\in\DD_\mu:P\subset \lambda Q,\,\ell(P)\leq \ell(Q)\},$$
and, for $k\geq0$,
$$\DD_{\mu,k}(\lambda Q) =\{P\in\DD_\mu:P\subset \lambda Q,\,\ell(P)=A_0^{-k} \ell(Q)\},\qquad
\DD_\mu^k(\lambda Q) = \bigcup_{j\geq k} \DD_{\mu,j}(\lambda Q).
$$
\vv

%For a given $Q\in\DD_\mu$ and $\lambda\geq1$, we denote $\ell(\lambda Q) = \lambda\,\ell(Q)$ and
%we consider the ``density''
%$$\Theta_\mu(Q)=\frac{\mu(2 B_Q)}{\ell(Q)^n}.$$
%Clearly, $\Theta_\mu(Q)\approx_{A_0,C_0}\theta_\mu(2B_Q)$.

\begin{lemma}\label{lemDMimproved}
	Let $\mu$ be a compactly supported Radon measure in $\R^{d}$.
	Assume that $\mu$ has polynomial growth of degree $n$ and let $\gamma\in(0,1)$. The lattice $\DD_\mu$ from Lemma
	\ref{lemcubs} can be constructed so that the following holds for all
	all $Q\in\DD_{\mu}$:
	\begin{align*}
		\int_{2B_Q\setminus Q}\left(\int_Q \frac1{|x-y|^n}\,d\mu(y)\right)^2 d\mu(x) 
		+ &\int_{Q}\left(\int_{2B_Q\setminus Q} \frac1{|x-y|^n}\,d\mu(y)\right)^2 d\mu(x)\\
		&\leq C(\gamma)\sum_{P\in\DD_\mu: P\subset 2Q} \left(\frac{\ell(P)}{\ell(Q)}\right)^\gamma\theta_\mu(2B_P)^2\mu(P).
	\end{align*}
\end{lemma}

Remark that the polynomial growth assumption is just necessary to ensure that some of the integrals above are finite. In fact, the constant $C(\gamma)$ does not depend on the polynomial growth constant. 

To prove the lemma, we denote
\begin{equation}\label{eqdmuint}
	\wt \DD_\mu^{int}(Q) = \big\{P\in\DD_\mu(Q):2B_P\cap (\supp\mu\setminus Q)\neq \varnothing\big\}
\end{equation}
and
\begin{equation}\label{eqdmuext}
	\wt \DD_\mu^{ext}(Q) = \big\{P\in\DD_\mu:\ell(P)\leq \ell(Q),P\subset \R^{n+1}\setminus Q,\,2B_P\cap Q\neq \varnothing\big\}.
\end{equation}
Also,
\begin{equation}\label{eqdmutot}
	\wt \DD_\mu(Q) = \wt \DD_\mu^{int}(Q) \cup \wt \DD_\mu^{ext}(Q),
\end{equation}
and, for $k\geq0$,\todo{I removed "$P\subset \lambda Q$" from the definition of $\wt \DD_{\mu,k}$, I think this was an artefact}
$$\wt \DD_{\mu,k}(Q) = 
\{P\in\wt \DD_\mu:\ell(P)=  A_0^{-k}\ell(Q)\}.$$

We need some auxiliary results. The first one is the following.

\begin{lemma}\label{lemDMimproved2}
	Let $\mu$ be a compactly supported Radon measure in $\R^{d}$ and $Q\in\DD_\mu$. For any  $\alpha\in(0,1)$, we have
	\begin{align}\label{eqdosint}
		\int_{2B_Q\setminus Q}\left(\int_Q \frac1{|x-y|^n}\,d\mu(y)\right)^2d\mu(x) \,
		+ &\int_{Q}\left(\int_{2B_Q\setminus Q} \frac1{|x-y|^n}\,d\mu(y)\right)^2d\mu(x) \\
		& \lesssim_{\alpha,A_0}
		\sum_{P\in\wt \DD_\mu(Q)} \left(\frac{\ell(Q)}{\ell(P)}\right)^\alpha\theta_\mu(2B_P)^2
		\,\mu(P).\nonumber
	\end{align}
\end{lemma}

\begin{proof}
	Observe that, for $x\in 2B_Q\setminus Q$,
	\begin{align*}
		\int_Q\frac1{|x-y|^n}\,d\mu(y) & = \left(\int_{y\in Q:|x-y|\geq r(B_Q)/2} + \int_{y\in Q:|x-y|< r(B_Q)/2}\right)
		\frac1{|x-y|^n}\,d\mu(y)\\
		&\lesssim_{A_0} \frac{\mu(Q)}{r(B_Q)^n} + \sum_{P\in\wt \DD_\mu^{ext}(Q):x\in P} \theta_\mu(2B_P).
	\end{align*}
	Thus,
	\begin{multline*}
		\int_{2B_Q\setminus Q}\left(\int_Q \frac1{|x-y|^n}\,d\mu(y)\right)^2d\mu(x) \\
		\lesssim_{A_0} \left(\frac{\mu(Q)}{r(B_Q)^n}\right)^2\,\mu(2B_Q) +
		\int_{2B_Q\setminus Q}\bigg(\sum_{P\in\wt \DD_\mu^{ext}(Q):x\in P} \theta_\mu(2B_P)\bigg)^2d\mu(x).
	\end{multline*}
	By H\"older's inequality, for any $\alpha>0$,
	\begin{align*}
		\bigg(\sum_{P\in\wt \DD_\mu^{ext}(Q):x\in P} &\theta_\mu(2B_P)\bigg)^2\\
		& \leq \bigg(\sum_{P\in\wt \DD_\mu^{ext}(Q):x\in P} \left(\frac{\ell(Q)}{\ell(P)}\right)^\alpha\theta_\mu(2B_P)^2\bigg) \cdot \bigg(\sum_{P\in\wt \DD_\mu^{ext}(Q):x\in P} \left(\frac{\ell(P)}{\ell(Q)}\right)^{\alpha}\bigg).
	\end{align*}
	The last sum above is bounded above by
	$$\sum_{P\in\DD_\mu:x\in P,P\subset 2Q} \left(\frac{\ell(P)}{\ell(Q)}\right)^{\alpha}\leq C(\alpha).
	$$
	Therefore,
	\begin{align*}
		\int_{2B_Q\setminus Q}& \left(\int_Q \frac1{|x-y|^n}\,d\mu(y)\right)^2d\mu(x) \\
		& \lesssim_{\alpha,A_0} \frac{\mu(Q)^2\,\mu(2B_Q)}{r(B_Q)^{2n}}+
		\int_{2B_Q\setminus Q} \sum_{P\in\wt \DD_\mu^{ext}(Q):x\in P} \left(\frac{\ell(Q)}{\ell(P)}\right)^\alpha\theta_\mu(2B_P)^2\,d\mu(x)\\
		& \lesssim  \theta_\mu(2B_Q)^2\,\mu(Q)  +
		\sum_{P\in\wt \DD_\mu^{ext}(Q)} \left(\frac{\ell(Q)}{\ell(P)}\right)^\alpha\theta_\mu(2B_P)^2
		\,\mu(P).
	\end{align*}
	
	The estimate of the second integral on the left hand side of \rf{eqdosint} is analogous.
\end{proof}
\vv

\begin{lemma}\label{lemrecur5}
	Let $\mu$ be a compactly supported Radon measure in $\R^{d}$ and let $\gamma\in (0,1)$.
	Assume that $\mu$ has polynomial growth of degree $n$ and let $\gamma\in(0,1)$. The lattice $\DD_\mu$ from Lemma
	\ref{lemcubs} can be constructed so that the following holds for all
	all $Q\in\DD_{\mu}$:
	\begin{equation}\label{eqfir5}
		\sum_{S\in\wt \DD_{\mu,1}(Q)} \sum_{P\in\DD_\mu(2S)} \left(\frac{\ell(P)}{\ell(Q)}\right)^\gamma\theta_\mu(2B_P)^2\,\mu(P)\lesssim
		C_0^{6d+1}A_0^{-1}\!\!
		\sum_{P\in \DD_\mu^1(2Q)}\left(\frac{\ell(P)}{\ell(Q)}\right)^\gamma \theta_\mu(2B_P)^2\,\mu(P).
	\end{equation} 
\end{lemma}

\begin{proof}
	We will describe the relevant changes required in the arguments in \cite[Theorem 3.2]{David-Mattila} in order to get the estimate \rf{eqfir5}. We will
	use the same notation as in that theorem, with the exception of the constant $A$ in \cite[Theorem 3.2]{David-Mattila}, which here we denote by $A_0$. 
	
	Denote $E=\supp\mu$.
	For each generation $k\geq 0$, the starting point to construct $\DD_{\mu,k}$
	consists of choosing, for each $x\in E$, a suitable radius $r^k(x)$ such that
	\begin{equation}\label{eqrk1}
		A_0^{-k}\leq r^k(x)\leq C_0 A_0^{-k}
	\end{equation}
	depending on the doubling properties of the ball $B(x,r^k(x))$ (see \cite[(3.17)-(3.20)]{David-Mattila}).
	Next, one chooses two auxiliary radii $r_1^k(x)$ and $r_2^k(x)$ such that
	$$\frac{11}{10}\,r^k(x) < r_1^k(x)<\frac{12}{10}\,r^k(x),$$
	$$25\,r^k(x) < r_2^k(x)<26\,r^k(x),$$
	and such that the following small boundary conditions hold:
	\begin{equation}\label{eqthin1*}
		\mu\big(\big\{y\in\R^{d}\!: \dist(y,\partial B(x,r_1^k(x)))\leq \tau\,r^k(x)\big\}\big)
		\leq C\tau\,\mu\big(B(x,\tfrac{13}{10}r^k(x))\big)\quad
		\mbox{for $0<\tau < \tfrac1{10}$,}
	\end{equation}
	and
	\begin{equation}\label{eqthin2*}
		\mu\big(\big\{y\in\R^{d}\!: \dist(y,\partial B(x,r_2^k(x)))\leq \tau\,r^k(x)\big\}\big)
		\leq C\tau\,\mu\big(B(x,27r^k(x))\big)\quad
		\mbox{for $0<\tau < 1$.}
	\end{equation}
	
	At this point we will require the auxiliary radii $r_1^k(x)$ and $r_2^k(x)$ to be chosen so that an additional condition holds. Set $A(x,r,R) = B(x,R)\setminus B(x,r).$ 
Observe first that
\begin{align}\label{eqcalr45}
\int_{\tfrac{11}{10}r^k(x)}^{\tfrac{12}{10}r^k(x)}  &\sum_{j\geq0}  A_0^{-\gamma(j+1)} \int_{A(x,\,t-300C_0A_0^{-k-1},\,t+300C_0A_0^{-k-1})}
 \theta_\mu(y,112C_0A_0^{-k-j-1})^2\,d\mu(y)\,dt\\
& \leq\sum_{j\geq0}  A_0^{-\gamma(j+1)} \int_{B(x,\tfrac{12}{10}r^k(x)+300C_0A_0^{-k-1})}
 \theta_\mu(y,112C_0A_0^{-k-j-1})^2\,\LL^1(I_{x,y,k})\,d\mu(y),\nonumber
\end{align} 
where we applied Fubini and we denoted by $I_{x,y,k}$ the interval
\begin{align*}
I_{x,y,k} & = \{t\in\R:t-300C_0A_0^{-k-1}\leq |x-y|\leq t+300C_0A_0^{-k-1}\}\\
&= \big[|x-y|-300C_0A_0^{-k-1},|x-y|+300C_0A_0^{-k-1}\big].
\end{align*}
Obviously, its Lebesgue measure is $\LL^1(I_{x,y,k})= 600C_0A_0^{-k-1}$, and so the left hand side of \rf{eqcalr45}
is bounded above by
$$600 \,C_0\,A_0^{-k-1}\sum_{j\geq0} A_0^{-\gamma(j+1)} \!\! \int_{B(x,\tfrac{13}{10}r^k(x))}
 \theta_\mu(y,112C_0A_0^{-k-j-1})^2 \,d\mu(y).
$$  
	Thus, by Chebyshev, the set $U_{1}^k\subset\R$ of those $t\in [\tfrac{11}{10}r^k(x),\tfrac{12}{10}r^k(x)]$ such that
	\begin{align*}
		\sum_{j\geq0}  A_0^{-\gamma(j+1)} &\int_{A(x,t-300C_0A_0^{-k-1},t+300C_0A_0^{-k-1})}
		\theta_\mu(y,112C_0A_0^{-k-j-1})^2\,d\mu(y)\\
		&> \frac{10^5\,C_0\,A_0^{-k-1}}{r^k(x)}
		\sum_{j\geq0} A_0^{-\gamma(j+1)} 
		\int_{B(x,\tfrac{13}{10}r^k(x))}
		\theta_\mu(y,112C_0A_0^{-k-j-1})^2 \,d\mu(y)
	\end{align*}
	satisfies 
	$$|U_{1}^k|\leq \frac1{100}\,r^k(x).$$
	By a standard argument involving the  boundedness of the maximal Hardy-Littlewood operator 
	from $L^1(\R)$ to $L^{1,\infty}(\R)$, one can deduce that there exists some 
	$$r_1^k(x)\in [\tfrac{11}{10}r^k(x),\tfrac{12}{10}r^k(x)]\setminus
	U_{1}^k
	$$ 
	such that \rf{eqthin1*} holds. The fact that $r_1^k(x)\not\in\ U_{1}^k$
	ensures that
	\begin{align}\label{eqal848}
		\sum_{j\geq0}  A_0^{-\gamma(j+1)} &\int_{A(x,r_1^k(x)-300C_0A_0^{-k-1},r_1^k(x)+300C_0A_0^{-k-1})}
		\theta_\mu(y,112C_0A_0^{-k-j-1})^2\,d\mu(y)\\
		&\leq \frac{10^5\,C_0\,A_0^{-k-1}}{r^k(x)}
		\sum_{j\geq0} A_0^{-\gamma(j+1)} 
		\int_{B(x,\tfrac{13}{10}r^k(x))}
		\theta_\mu(y,112C_0A_0^{-k-j-1})^2 \,d\mu(y).
		\nonumber
	\end{align}
	An analogous argument shows that 
	$r^k_2(x)$ can be taken such that, besides \rf{eqthin2*}, the 
	preceding estimate also holds with $r_1^k(x)$ replaced
	by $r_2^k(x)$ and $B(x,\tfrac{13}{10}r^k(x))$ replaced by $B(x,27r^k(x))$.
	
	As in \cite[Theorem 3.2]{David-Mattila}, we denote $B_1^k(x) = B(x,r_1^k(x))$ and $B_2^k(x) = B(x,r_2^k(x))$, and by a Vitali type covering lemma we select a family of points $x\in I^k$
	such that the balls $\{B(x,5r^k(x))\}_{x\in I^k}$ are disjoint, while the balls $\{B(x,25r^k(x))\}_{x\in I^k}$ cover $E$. We also denote
	$$B_3^k(x) =B_2^k(x)\setminus \bigg(\bigcup_{y\in I^k\setminus\{x\}} B_1^k(y)\bigg).$$
	For $x\in I_k$, let $J(x)$ be the family of those
	$y\in I^k\setminus\{x\}$ such that $B_1^k(y)\cap B_2^k(x)\neq
	\varnothing$. As explained in \cite[Theorem 3.2]{David-Mattila}, using \rf{eqrk1} it is easy to check
	$\# J(x) \leq C C_0^d$.
	
	Next we consider an order in $I^k$ such that
	$$\mbox{$y<x$ in $I^k$ whenever $\mu(B(x,90r^k(x)))<\mu(B(y,90r^k(y)))$}$$
	and we define
	$$B_4^k(x) = B_3^k(x) \setminus \bigg(\bigcup_{y\in I^k:y<x} B_3^k(y)\bigg).$$
	Again, as explained in \cite[Theorem 3.2]{David-Mattila}, using \rf{eqrk1} it is easy to check
	that, for each $x\in I^k$, there are at most $C C_0^{n+1}$ sets $B_3^k(y)$ that intersect $B_3^k(x)$, with 
	$y\in I^k$.
	
	The family $\{B_4^k(x)\}_{x\in I^k}$ is a first approximation to $\{Q\}_{Q\in\DD_\mu^k}$.
	Indeed, by the arguments in \cite[Theorem 3.2]{David-Mattila}, for each $x\in I^k$ one constructs 
	a set $Q^k(x)\subset E$
	such that, denoting
	$\DD_{\mu,k} = \{Q^k(x)\}_{x\in I^k},$
	the properties stated in Lemma \ref{lemcubs} hold, with $r(Q^k(x))=r^k(x)$ and $B(Q^k(x))=B(x,r^k(x))$.
	In particular, 
	$$B(x,r^k(x))\cap E\subset Q^k(x) \subset B(x,28r^k(x)).$$
	Also, as shown in \cite[(3.61)]{David-Mattila}, it holds 
	\begin{equation}\label{eqdisb4}
		\dist(y,\partial B_4^k(x)) \leq 51C_0A_0^{-k-1}\quad \mbox{ for all $y\in N_1(Q(x))$.}
	\end{equation}
	
	For a cube $Q=Q^k(x)\in\DD_{\mu,k}$, we write $r(Q)=r^k(x)$, $B(Q)=B(x,r^k(x))$ and
	$B_i(Q) = B_i^k(x)$ for $i=1,\ldots,4$. 
	By an argument quite similar to the one used in \cite[Theorem 3.2]{David-Mattila} to prove \rf{eqdisb4}, we will show now that
	\begin{equation}\label{eqdisb4'}
		2S\subset \UU_{5A_0^{-1}\ell(Q)}(\partial B_4(Q))\quad \mbox{ for any $S\in\wt \DD_{\mu,1}(Q)$,}
	\end{equation}
	where $\UU_\ell(A)$ stands for the $\ell$-neighborhood of $A$.
	This will be needed below to prove 
	\rf{eqfir5}. The condition $S\in\wt \DD_{\mu,1}(Q)$ tells us that either $S\subset Q$ and 
	$2B_S\cap (E\setminus Q)\neq \varnothing$, or $S\subset E\setminus Q$ and 
	$2B_S\cap  Q\neq \varnothing$. Assume the first option (the arguments for the second one are analogous). So there exists some point $z\in E\setminus Q$ such that $|x_S-z|\leq 2r(B_S)=56r(S)$.
	Let $x\in I^k$ be such that $z\in Q^k(x)$. Then we have $\dist(z,B_4^k(x))\leq 50C_0A_0^{-k-1}$
	and also $\dist(x_S,B_4^k(x_Q))\leq 50C_0A_0^{-k-1}$, by 
	\cite[(3.50)]{David-Mattila} (see also the first paragraph after \cite[(3.61)]{David-Mattila}). 
	Since the sets $B_4^k(x_Q)$, $B_4^k(x)$ are disjoint, we deduce that
	$$\dist(x_S,\partial B_4^k(x))\leq 50C_0A_0^{-k-1} + 56\,r(S) \leq 106C_0A_0^{-k-1}< 2A_0^{-1}\ell(Q).
	$$
	Together with the fact that $2S\subset B\big(x_S,\tfrac52\ell(S)\big)$, this gives \rf{eqdisb4'}.
	
	Notice that, for each $j\geq0$,
	$$
	\sum_{P\in\DD_{\mu,j}(2S)} \left(\frac{\ell(P)}{\ell(Q)}\right)^\gamma\theta_\mu(2B_P)^2\,\mu(P)
	\lesssim C_0^{2n}A_0^{-\gamma(j+1)} \int_{2S} \theta_\mu(x,2A_0^{-j}\ell(S))^2\,d\mu(x).$$
	Then we obtain
	\begin{align}\label{eqhfk2}
		\sum_{S\in\wt \DD_{\mu,1}(Q)} \sum_{P\in\DD_\mu(2S)}& \left(\frac{\ell(P)}{\ell(Q)}\right)^\gamma\theta_\mu(2B_P)^2\,\mu(P) \\
		&\lesssim C_0^{2n}
		\sum_{S\in\wt \DD_{\mu,1}(Q)} \sum_{j\geq0} A_0^{-\gamma(j+1)} \int_{2S} \theta_\mu(x,2A_0^{-j}\ell(S))^2\,d\mu(x)\nonumber\\
		& \lesssim C_0^{2n+d}
		\sum_{j\geq0} A_0^{-\gamma(j+1)} \int_{\UU_{5A_0^{-1}\ell(Q)}(\partial B_4(Q))}
		\theta_\mu(x,2A_0^{-j-1}\ell(Q))^2\,d\mu(x).\nonumber
	\end{align}
	Denote by $\wt J(Q)$ the family of cubes $R\in\DD_{\mu,k}$ such that $B_2(R)\cap B_4(Q)\neq\varnothing$, so that, by the above construction we have
	$$\partial B_4(Q)\subset \bigcup_{R\in \wt J(Q)} (\partial B_1(R) 
	\cup \partial B_2(R)).$$
	Also, notice that $\#\wt J(Q)\leq C\,C_0^d$.
	From \rf{eqal848} we deduce that, for each $R\in\wt J(Q)$ and $i=1,2$,
	\begin{align}\label{eqrig56}
		\sum_{j\geq0} A_0^{-\gamma(j+1)} &\int_{\UU_{5A_0^{-1}\ell(R)}(\partial B_i(R))}
		\theta_\mu(x,2A_0^{-j-1}\ell(R))^2\,d\mu(x)\\
		& \leq C\,C_0A_0^{-1}
		\sum_{j\geq0} A_0^{-\gamma(j+1)} \int_{27B(R)}
		\theta_\mu(x,2A_0^{-j-1}\ell(R))^2\,d\mu(x)\nonumber \\
		& \leq C\,C_0A_0^{-1}
		\sum_{P\in \DD_\mu^{k+1}:P\cap 27B(R)\neq\varnothing}\left(\frac{\ell(P)}{\ell(Q)}\right)^\gamma\, \theta_\mu(x_P,3\ell(P))^2\,\mu(P),\nonumber
	\end{align}
	Notice that for $Q,R\in\DD_{\mu,k}$ as above, the condition $B_2(R)\cap B_4(Q)\neq\varnothing$ implies that $26B(Q) \cap 26B(R)\neq\varnothing$. Then, if
	$P\in \DD_\mu^{k+1}$ is such that $P\cap 27B(R)\neq\varnothing$, we derive
	%$|x_Q-x_R|\leq 26(r(Q) + r(R))$ and so
	$$\dist(x_Q,P)\leq |x_Q-x_R| + 27 r(R)
	\leq 26(r(Q) + r(R)) + 27 r(R) \leq 
	79C_0A_0^{-k} = \frac{79}{56}\,\ell(Q).$$
	Then, since $\ell(P)\leq A_0^{-1}\ell(Q)$, we infer that
	\begin{equation}\label{eqinc732}
		B(x_P,3\ell(P))\cap\supp\mu\subset 2Q.
	\end{equation}
	Also, we can write
	\begin{align*}
		\theta_\mu(x_P,3\ell(P))^2\,\mu(P) &\lesssim \frac{\mu(B(x_P,3\ell(P)))^3}{\ell(P)^{2n}}
		\lesssim \frac1{\ell(P)^{2n}}
		\bigg(\sum_{\substack{P'\in\DD_\mu:\ell(P')=\ell(P),\\
				P'\cap  B(x_P,3\ell(P))\neq\varnothing}}
		\mu(P')\bigg)^3\\
		& \lesssim C_0^{2d} \!\!\!\sum_{\substack{P'\in\DD_\mu:\ell(P')=\ell(P),\\
				P'\cap  B(x_P,3\ell(P))\neq\varnothing}}\!\frac{\mu(P')^3}{\ell(P')^{2n}}\lesssim
		C_0^{2d} \!\!\!\sum_{\substack{P'\in\DD_\mu:\ell(P')=\ell(P),\\
				P'\cap  B(x_P,3\ell(P))\neq\varnothing}}\!\!\theta_\mu(2B_{P'})^2\,\mu(P'),
	\end{align*}
	where we used the fact that the sums above are only over $CC_0^d$ terms at most.
	Together with \rf{eqinc732}, this implies that
	the right hand side of \rf{eqrig56} does not exceed
	\begin{multline*}
		C\,C_0^{2d+1}A_0^{-1} \sum_{P\in \DD_\mu^{k+1}:P\cap 27B(R)\neq\varnothing}\left(\frac{\ell(P)}{\ell(Q)}\right)^\gamma  \!\!\!\!\!\sum_{\substack{P'\in\DD_\mu(2Q):\ell(P')=\ell(P),\\
				P'\cap  B(x_P,3\ell(P))\neq\varnothing}}\!\!\!\theta_\mu(2B_{P'})^2\,\mu(P')
		\\ \leq 
		C\,C_0^{3d+1}A_0^{-1}
		\sum_{P'\in \DD_\mu^{1}(2Q)}\left(\frac{\ell(P')}{\ell(Q)}\right)^\gamma\, \theta_\mu(2B_{P'})^2\,\mu(P').
	\end{multline*}
	By this estimate and \rf{eqhfk2}, summing over all $R\in\wt J(Q)$, we get
	$$\sum_{S\in\wt \DD_{\mu,1}(Q)} \sum_{P\in\DD_\mu(2S)} \left(\frac{\ell(P)}{\ell(Q)}\right)^\gamma\theta_\mu(2B_P)^2\,\mu(P)\lesssim
	C_0^{6d+1}A_0^{-1}\!\!
	\sum_{P\in \DD_\mu^1(2Q)}\left(\frac{\ell(P)}{\ell(Q)}\right)^\gamma\, \theta_\mu(2B_P)^2\,\mu(P),
	$$%\end{align*}
	as wished.
\end{proof}

\vv

By Lemma \ref{lemDMimproved2}, it is clear that to complete the proof of Lemma \ref{lemDMimproved} it suffices to show the following result.

\begin{lemma}\label{lemdmutot}
	Let $\mu$ be a compactly supported Radon measure in $\R^{d}$.
	Assume that $\mu$ has polynomial growth of degree $n$ and let $\gamma\in(0,1)$. The lattice $\DD_\mu$ from Lemma \ref{lemcubs} can be constructed so that the following holds for all
	all $Q\in\DD_{\mu}$:
	\begin{equation}\label{eqfhq29}
		\sum_{P\in\wt \DD_\mu(Q)} \left(\frac{\ell(Q)}{\ell(P)}\right)^{\frac{1-\gamma}2} \theta_\mu(2B_P)^2
		\,\mu(P) \lesssim_{A_0,\gamma}\sum_{P\in\DD_\mu: P\subset 2Q} \left(\frac{\ell(P)}{\ell(Q)}\right)^\gamma\theta_\mu(2B_P)^2\mu(P).
	\end{equation}
\end{lemma}

\begin{proof}
	To prove \rf{eqfhq29} notice that, for each $k\geq1$, by Lemma \ref{lemrecur5} we have
	\begin{align*}
		\sum_{R\in\wt \DD_{\mu,k}(Q)}
		\sum_{P\in \DD_{\mu}^1(2R)} &\left(\frac{\ell(P)}{\ell(R)}\right)^\gamma \theta_\mu(2B_P)^2 \,\mu(P)\\
		&\leq \sum_{R\in\wt \DD_{\mu,k-1}(Q)}
		\sum_{S\in\wt \DD_{\mu,1}(R)} 
		\sum_{P\in \DD_{\mu}(2S)}  \left(\frac{\ell(P)}{\ell(S)}\right)^\gamma\theta_\mu(2B_P)^2 \,\mu(P)\\
		& = A_0^{\gamma} \sum_{R\in\wt \DD_{\mu,k-1}(Q)}
		\sum_{S\in\wt \DD_{\mu,1}(R)} 
		\sum_{P\in \DD_{\mu}(2S)}  \left(\frac{\ell(P)}{\ell(R)}\right)^\gamma\theta_\mu(2B_P)^2 \,\mu(P)\\
		& \leq C C_0^{6d+1} A_0^{\gamma-1} \sum_{R\in\wt \DD_{\mu,k-1}(Q)}
		\sum_{P\in \DD_\mu^1(2R)}\left(\frac{\ell(P)}{\ell(R)}\right)^\gamma\, \theta_\mu(2B_P)^2\,\mu(P).
	\end{align*}
	Iterating the preceding estimate, we obtain
	\begin{align*}
		\sum_{P\in\wt \DD_{\mu,k+1}(Q)} 
		\theta_\mu(2B_P)^2 \,\mu(P) & \leq 
		\sum_{R\in\wt \DD_{\mu,k}(Q)}
		\sum_{P\in \DD_{\mu,1}(2R)} \theta_\mu(2B_P)^2 \,\mu(P)\\
		& \leq A_0^\gamma\sum_{R\in\wt \DD_{\mu,k}(Q)}
		\sum_{P\in \DD_{\mu}^1(2R)} \left(\frac{\ell(P)}{\ell(R)}\right)^\gamma \theta_\mu(2B_P)^2 \,\mu(P)\\
		& \leq A_0^\gamma\,(C C_0^{6d+1} A_0^{\gamma-1})^k \sum_{P\in \DD_\mu^1(2Q)}\left(\frac{\ell(P)}{\ell(Q)}\right)^\gamma\, \theta_\mu(2B_P)^2\,\mu(P).
	\end{align*}
	Therefore,
	\begin{align*}
		&\sum_{P\in\wt \DD_\mu(Q)} \left(\frac{\ell(Q)}{\ell(P)}\right)^{(1-\gamma)/2}\theta_\mu(2B_P)^2
		\,\mu(P)\\
		& \quad\leq \theta_\mu(2B_Q)^2\,\mu(Q) + \sum_{k\geq1} A_0^{(1-\gamma) k/2}\sum_{P\in\wt \DD_{\mu,k}(Q)} \theta_\mu(2B_P)^2
		\,\mu(P)\\
		& \quad\leq \theta_\mu(2B_Q)^2\,\mu(Q) + A_0^\gamma\sum_{k\geq1} A_0^{(1-\gamma) k/2}(C C_0^{6d+1} A_0^{\gamma-1})^{k-1}\!\!\!\!
		\sum_{P\in \DD_\mu^1(2Q)}\!\left(\frac{\ell(P)}{\ell(Q)}\right)^\gamma\, \theta_\mu(2B_P)^2\,\mu(P)\\
		& \quad\lesssim_{A_0,\gamma} \sum_{P\in \DD_\mu(2Q)}\left(\frac{\ell(P)}{\ell(Q)}\right)^\gamma\, \theta_\mu(2B_P)^2\,\mu(P),
	\end{align*}
	taking $A_0$ big enough for the last estimate. This yields \rf{eqfhq29}.
\end{proof}

\vv

% ***********************************************************************************************************************************

\section{\texorpdfstring{$\PP$}{P}-doubling cubes, cubes dominated from below, and the Main Propositions}\label{sec:Pdoubling}

In the rest of the paper we assume that $\mu$ is a compactly supported Radon measure in $\R^{n+1}$ with polynomial 
growth of degree $n$ and
that $\DD_\mu$ is a David-Mattila dyadic lattice satisfying the properties 
described in the preceding section, in particular, the ones in Lemmas \ref{lemcubs}
and \ref{lemDMimproved}, with $\gamma=9/10$.
By rescaling, we assume that $\DD_{\mu,k}$ is defined for all $k\geq k_0$, with $A_0^{-k_0}\approx
\diam(\supp\mu)$, and we also assume that there is a unique cube in $\DD_{\mu,k_0}$ which coincides
with the whole $\supp\mu$.
Further, from now on, we allow all the constants $C$ and all implicit constants in the relations
``$\lesssim$'', ``$\approx$", to depend on the parameters $C_0, A_0$ of the dyadic lattice of David-Mattila.

\subsection{\texorpdfstring{$\PP$}{P}-doubling cubes and the family \texorpdfstring{$\hd^k(Q)$}{hdk(Q)}}\label{subsec:Pdoubling}

We denote
$$\PP(Q) = \sum_{R\in\DD_\mu:R\supset Q} \frac{\ell(Q)}{\ell(R)^{n+1}} \,\mu(2B_R).$$
We say that a cube $Q$ is $\PP$-doubling if
$$\PP(Q) \leq C_d\,\frac{\mu(2B_Q)}{\ell(Q)^n},$$
for  $C_d =4A_0^n$. We write
$$Q\in\DD_\mu^\PP.$$
Notice that
$$\PP(Q) \approx_{C_0} \sum_{R\in\DD_\mu:R\supset Q} \frac{\ell(Q)}{\ell(R)} \,\theta_\mu(2B_R).$$
and thus saying that $Q$ is $\PP$-doubling implies that 
$$\sum_{R\in\DD_\mu:R\supset Q} \frac{\ell(Q)}{\ell(R)} \,\theta_\mu(2B_R)\leq C_d'\,\theta_\mu(2B_Q)$$
for some $C_d'$ depending on $C_d$. Conversely, the latter condition implies that $Q$ is $\PP$-doubling with another constant $C_d$ depending on $C_d'$.

From the properties of the David-Mattila lattice, we deduce the following.

\begin{lemma}\label{lempois00}
	Suppose that $C_0$ and $A_0$ are chosen suitably. If $Q$ is $\PP$-doubling, then $Q\in\DD_\mu^{db}$.
	Also, any cube $R\in\DD_\mu$ such that $R\cap 2Q\neq\varnothing$ and $\ell(R)=A_0\ell(Q)$ belongs to $\DD_\mu^{db}$.
\end{lemma}

\begin{proof}
	Let $Q\in\DD_\mu^\PP$. Regarding the fist statement of the lemma, if $Q\not\in\DD_\mu^{db}$, by \rf{eqdob23} we have
	$$\mu(2B_Q)\leq \mu(100B(Q))\leq C_0^{-l}\,\mu(100^{l+1}B(Q))\quad
	\mbox{for all $l\geq1$ with $100^l\leq C_0$}.$$
	In particular, if $Q'$ denotes the parent of $Q$,
	$$\mu(2B_Q)\leq C_0^{-l}\,\mu(2B_{Q'})\quad
	\mbox{for all $l\geq1$ with $100^l\leq C_0$}.$$
	So,
	\begin{equation}\label{eqigu8208}
		\mu(2B_Q)\leq C_0^{-c\log C_0}\,\mu(2B_{Q'})
	\end{equation}
	for some $c>0$. Using now that $Q$ is $\PP$-doubling, we get
	$$\mu(2B_Q)\leq C_0^{-c\log C_0}\,C_d\,\frac{\ell(Q')^{n+1}}{\ell(Q)^{n+1}} \mu(2B_Q)
	= 4C_0^{-c\log C_0}\,A_0^{2n+1} \mu(2B_Q)
	.$$
	Recall now that, as explained in Remark \ref{rema00}, we assume that $A_0=C_0^{C(n)}$, for some constant
	$C(n)$ depending just on $n$. Then, clearly, the preceding inequality fails if $C_0$ is big enough, which gives the desired contradiction.
	\vv
	
	To prove the second statement of the lemma, suppose that $Q\in\DD_{\mu,k}^\PP$ and let $R\in\DD_\mu$ be such that $R\cap 2Q\neq\varnothing$ and $\ell(R)=A_0\ell(Q)$.
	By the definition and the fact that $R\subset B_R$ we get
	$$|x_Q-x_R|\leq 3\ell(Q) + r(B_R).$$
	Since $$r(B_R)=28\,r(R) \geq 28A_0^{-k+1} = \frac12\,C_0^{-1}A_0\,\ell(Q)\geq 2500\,\ell(Q),$$
	we deduce that 
	$$2B_Q\subset 2B_R.$$
	If $R'$ denotes the parent of $R$ and $Q'''$ the great-grandparent of $Q$ (so that $\ell(Q''')=A_0^3\ell(Q)=A_0\ell(R')$), by an analogous argument we infer that
	$$2B_{R'}\subset 2B_{Q'''}.$$
	Then, using also that $Q$ is $\PP$-doubling, we obtain
	$$\mu(2B_{R'})\leq \mu(2B_{Q'''}) \leq C_d\,\left(\frac{\ell(Q''')}{\ell(Q)}\right)^{n+1}\!\mu(2B_Q)
	\leq C_d\,A_0^{3(n+1)}\,\mu(2B_R) = 4A_0^{4n+3}\,\mu(2B_R).$$
	If $R\not\in\DD_\mu^{db}$, arguing as in \rf{eqigu8208}, we infer that
	$$\mu(2B_R)\leq C_0^{-c\log C_0}\,\mu(2B_{R'}),$$
	which contradicts the previous statement if $C_0$ is big enough (recalling that $A_0=C_0^{C(n)}$).
\end{proof}

\vv
Notice that, by the preceding lemma, if $Q$ is $\PP$-doubling, then 
$$\sum_{R\in\DD_\mu:R\supset Q} \frac{\ell(Q)^{n+1}}{\ell(R)^{n+1}} \,\mu(2B_R) \lesssim_{C_d} \mu(Q).$$

For technical reasons that will be more evident below, it is appropriate to consider a discrete version of the density $\theta_\mu$. Given  $Q\in\DD_\mu$, we let
$$\Theta(Q) = A_0^{kn} \quad \mbox{ if\, $\dfrac{\mu(2B_Q)}{\ell(Q)^n}\in [A_0^{kn},A_0^{(k+1)n})$}.$$
Clearly, $\Theta(Q)\approx \theta_\mu(2B_Q)$.
Notice also that if $\Theta(Q) = A_0^k$ and $P$ is a son of $Q$, then
$$
\frac{\mu(2 B_P)}{\ell(P)^n} \leq \frac{\mu(2 B_Q)}{\ell(P)^n} = A_0^n\,\frac{\mu(2 B_Q)}{\ell(Q)^n}.
$$
Thus,
\begin{equation}\label{eqson1}
	\Theta(P)\leq A_0^n\,\Theta(Q)\quad \mbox{  for every son $P$ of $Q$.}
\end{equation}

Given $Q\in\DD_\mu$ and $k\geq1$, we denote by $\hd^k(Q)$ the family of maximal cubes $P\in\DD_\mu$ satisfying
\begin{equation}\label{a0tilde}
	\ell(P)<\ell(Q), \qquad \Theta(P)\geq  A_0^{kn}\Theta(Q).
\end{equation}
%From \rf{eqson1}, it follows that if $P\in\hd^k(Q)$, then $\Theta(P)= A_0^{kn}\Theta(Q)$.

\vv
\begin{lemma}\label{lempdoubling}
	Let $Q\in\DD_\mu$ be $\PP$-doubling. Then, for $k\geq4$, every $P\in\hd^k(Q)\cap\DD_\mu(9Q)$ is also $\PP$-doubling
	and moreover $\Theta(P)=A_0^{kn}\Theta(Q)$.
\end{lemma}

Remark that this lemma implies that, under the assumptions in the lemma,
\begin{equation}\label{eqforakdf33}
\Theta(P)\approx A_0^{nk}\Theta(Q)\quad\mbox{ for all $P\in\hd^k(Q)\cap\DD_\mu(9Q)$
and all
 $k\geq1$.}
\end{equation}

\begin{proof}
	First we show that $\Theta(P)=A_0^{kn}\Theta(Q)$. 
	The fact that $\Theta(P)\geq A_0^{kn}\,\Theta(Q)$ is clear. To see the converse inequality,
	denote by $\wh Q$ the parent of $Q$. Notice that any cube $S\subset 9Q$ with $\ell(S)=\ell(Q)$
	satisfies
	$$\frac{\mu(2B_S)}{\ell(S)^n}\leq \frac{\mu\big(2B_{\wh Q}\big)}{\ell(Q)^n} = A_0^{n}\frac{\mu(2B_{\wh Q})}{\ell(\wh Q)^n} \leq A_0^{n+1}\,\PP(Q) \leq C_d\,A_0^{n+1}\,\frac{\mu(2B_Q)}{\ell(Q)^n} < A_0^{3n}\,\frac{\mu(2B_Q)}{\ell(Q)^n}.$$
	Therefore,
	$$\Theta(S)\leq A_0^{3n}\,\Theta(Q).$$
	As a consequence, if $P\in\hd^k(Q)\cap\DD_\mu(9Q)$ with $k\geq4$, then its parent $\wh P$ satisfies
	$\Theta(\wh P)<A_0^{kn}\,\Theta(Q)$, which implies that $\Theta(P)\leq A_0^{kn}\,\Theta(Q)$.
	
	\vv
	To see that $P$ is $\PP$-doubling,
	we split
	\begin{equation}\label{eqvdx1}
		\PP(P) =\sum_{\substack{R\in\DD_\mu:R\supset P\\ \ell(R)\leq \ell(Q)}} \frac{\ell(P)}{\ell(R)^{n+1}} \,\mu(2B_R)+
		\sum_{\substack{R\in\DD_\mu:R\supset P\\ \ell(R)> \ell(Q)}}\frac{\ell(P)}{\ell(R)^{n+1}} \,\mu(2B_R).
	\end{equation}
	The cubes $R$ in the first sum on the right hand side satisfy
	$\Theta(R)\leq \Theta(P)$, by the definition of $\hd^k(Q)$. Thus,
	$$\sum_{\substack{R\in\DD_\mu:R\supset P\\ \ell(R)\leq \ell(Q)}} \frac{\ell(P)}{\ell(R)^{n+1}} \,\mu(2B_R)\leq 
	A_0^n 
	\sum_{\substack{R\in\DD_\mu:R\supset P\\ \ell(R)\leq \ell(Q)}}  \frac{\ell(P)}{\ell(R)} \,\Theta(R)\leq 
	2A_0^n\,\Theta(P)\leq 
	2A_0^n\,\frac{\mu(2B_P)}{\ell(P)^n}.$$
	Concerning the last sum in \rf{eqvdx1}, notice that the cubes $R$ in that sum satisfy $\ell(R)
	>\ell(Q)$.
	Using that $A_0\gg1$, it follows easily that $2B_R\subset 2B_{R'}$, where $R'$ is the cube containing $Q$ such that $\ell(R')=A_0\,\ell(R)$.
	Consequently, denoting by $\wh Q$ the parent of $Q$,
	\begin{align*}
		\sum_{\substack{R\in\DD_\mu:R\supset P\\ \ell(R)> \ell(Q)}}\frac{\ell(P)}{\ell(R)^{n+1}} \,\mu(2B_R)
		& \leq \sum_{R'\in\DD_\mu:R'\supset \wh Q} \frac{\ell(P)}{A_0^{-1}\ell(R')} \,\frac{\mu(2B_{R'})}{(A_0^{-1} \ell(R'))^n}\\
		& = A_0^{n+1} \frac{\ell(P)}{\ell(Q)} 
		\sum_{R'\in\DD_\mu:R'\supset \wh Q} \frac{\ell(Q)}{\ell(R')^{n+1}} \,\mu(2B_{R'})\\
		& \leq A_0^{n}\,\PP(Q) \leq A_0^{n}C_d\,\frac{\mu(2Q)}{\ell(Q)^n} \leq  A_0^{2n}C_d\,\Theta(Q)\\
		& 
		\leq \frac{A_0^{2n}C_d}{A_0^{4n}}\,\Theta(P)\leq \frac{C_d}{A_0^{2n}}\,\frac{\mu(2B_P)}{\ell(P)^n},
	\end{align*}
	where in the last two lines we took into account that $\ell(P)\leq A_0^{-1}\ell(Q)$ (because $P\in\hd^k(Q)$ for some $k\geq4$), that 
	$Q$ is $\PP$-doubling, and again that $P\in\hd^k(Q)$ for some $k\geq4$.
	
	From the estimates above, we infer that
	$$\PP(P) \leq \bigg(2A_0^n + \frac{C_d}{A_0^{2n}}\bigg) \,\frac{\mu(2B_P)}{\ell(P)^n} \leq C_d\,\frac{\mu(2B_P)}{\ell(P)^n},$$
	since $C_d= 4A_0^n$.
\end{proof}
\vv

\begin{lemma}\label{lemdobpp}
	Let $Q_0,Q_1,\ldots,Q_m$ be a family of cubes from $\DD_\mu$ such that $Q_j$ is a child of $Q_{j-1}$ for $1\leq j\leq 
	m$. Suppose that $Q_j$ is not $\PP$-doubling for $1\leq j\leq m$.
	Then
	\begin{equation}\label{eqcad35}
		\frac{\mu(2B_{Q_m})}{\ell(Q_m)^n}\leq A_0^{-m/2}\,\PP(Q_0).
	\end{equation}
	and
	\begin{equation}\label{eqcad35'}
		\PP(Q_m)\leq 2A_0^{-m/2}\,\PP(Q_0).
	\end{equation}
	
\end{lemma}

\begin{proof}
	Let us denote
	$\wt\Theta(R) = \frac{\mu(2B_R)}{\ell(R)^n},$ so that
	$$\PP(Q) = \sum_{R\in\DD_\mu:R\supset Q} \frac{\ell(Q)}{\ell(R)} \,\wt\Theta(R).$$
	For $1\leq j \leq m$,
	the fact that $Q_j$ is not $\PP$-doubling implies that
	\begin{equation}\label{eqsak33}
		\wt \Theta(Q_j) \leq \frac1{C_d}\,\PP(Q_j) = \frac1{C_d}\Biggl (\sum_{k=0}^{j-1} \frac{\ell(Q_j)}{\ell(Q_{j-k})}\,
		\wt\Theta(Q_{j-k})+ \frac{\ell(Q_j)}{\ell(Q_0)}\,\PP(Q_0)\Biggr).
	\end{equation}
	We will prove \rf{eqcad35} by induction on $j$. For $j=0$ this is in an immediate consequence of the
	definition of $\PP(Q_0)$. Suppose that \rf{eqcad35} holds for $0\leq h\leq j$, with $j\leq m-1$, and let us 
	consider the case $j+1$. From \rf{eqsak33} and the induction hypothesis we get
	\begin{align*}
		\wt\Theta(Q_{j+1}) & \leq  \frac1{C_d}\Biggl (\wt\Theta(Q_{j+1}) + \sum_{k=1}^j \frac{\ell(Q_{j+1})}{\ell(Q_{j+1-k})}\,
		\wt\Theta(Q_{j+1-k})+ \frac{\ell(Q_{j+1})}{\ell(Q_0)}\,\PP(Q_0)\Biggr)\\
		&= \frac1{C_d}\Biggl (\wt\Theta(Q_{j+1}) + \sum_{k=1}^j A_0^{-k}\,
		\wt\Theta(Q_{j+1-k})+ A_0^{-j-1}\,\PP(Q_0)\Biggr)\\
		&\leq \frac1{C_d}\Biggl (\wt\Theta(Q_{j+1}) + \sum_{k=1}^j A_0^{-k}\,A_0^{(-j-1+k)/2}\PP(Q_0)
		+ A_0^{-j-1}\,\PP(Q_0)\Biggr)
	\end{align*}
	Since 
	$$\sum_{k=1}^j A_0^{-k}\,A_0^{(-j-1+k)/2} = A_0^{(-j-1)/2}\sum_{k=1}^j A_0^{-k/2}\leq A_0^{-j/2},$$
	we obtain
	\begin{align*}
		\wt\Theta(Q_{j+1})  &\leq  \frac1{C_d}\bigl (\wt\Theta(Q_{j+1}) + A_0^{-j/2}\,\PP(Q_0)+ A_0^{-j-1}\,\PP(Q_0)\bigr)\\
		&\leq \frac1{C_d}\bigl (\wt\Theta(Q_{j+1}) + 2\,A_0^{-j/2}\,\PP(Q_0)\bigr) \\
	\end{align*}
	It is straightforward to check that this yields $\wt\Theta(Q_{j+1})\leq A_0^{-(j+1)/2}\,\PP(Q_0)$.
	
	The estimate \rf{eqcad35'} follows easily from \rf{eqcad35}:
	\begin{align*}
		\PP(Q_m) &= \sum_{k=0}^{m-1} \frac{\ell(Q_m)}{\ell(Q_{m-k})}\,
		\wt\Theta(Q_{m-k})+ \frac{\ell(Q_m)}{\ell(Q_0)}\,\PP(Q_0)\\
		&\leq \sum_{k=0}^{m-1} A_0^{-k}\,A_0^{-(m-k)/2}\,\PP(Q_0)+ A_0^{-m}\,\PP(Q_0)\\
		&\leq A_0^{-m/2}\sum_{k=0}^{m-1} A_0^{-k/2}\,\PP(Q_0)+ A_0^{-m}\,\PP(Q_0)\leq 2\,A_0^{-m/2}\,\PP(Q_0).
	\end{align*}
\end{proof}

\vv

% ********************************************************************************************

\subsection{The energies $\EE$, $\EE^H$, $\EE_\infty$, and the cubes dominated from below}\label{subsec:DB}

For given $\lambda\geq1$ and $Q\in\DD_\mu$, we consider the energy
$$\EE(\lambda Q) = \sum_{P\in\DD_\mu(\lambda Q)} \left(\frac{\ell(P)}{\ell(Q)}\right)^{3/4}\Theta(P)^2\,\mu(P).$$
%where $\DD_\mu(\lambda Q)$ is the subfamily of the cubes from $\DD_\mu$ contained in $\lambda Q$.
%Observe that the  summation above is only over the cubes $P$ such that $\ell(P)\geq\ell_0$.
We also denote
$$\EE^H(\lambda Q) = \sum_{k\geq 0} 
\sum_{P\in\hd^k(Q)\cap\DD_\mu(\lambda Q)} \left(\frac{\ell(P)}{\ell(Q)}\right)^{3/4}\Theta(P)^2\,\mu(P)$$
and
$$\EE_\infty(\lambda Q) = \sup_{k\geq1}
\sum_{P\in\hd^k(Q)\cap\DD_\mu(\lambda Q)} \left(\frac{\ell(P)}{\ell(Q)}\right)^{\!1/2}\Theta(P)^2\,\mu(P).$$

\vv
\begin{lemma}\label{lemenergias}
For every $Q\in\DD_\mu$
we have
$$\EE(9Q) \lesssim \EE^H(9Q)\lesssim\EE_\infty(9Q).$$
\end{lemma}

\begin{proof}
For a given $R\in\hd^k(Q)$, we denote by $\tree_H(R)$ the family of cubes from $\DD_\mu$ that 
are contained in $R$ and are not contained in any cube from $\hd^{k+1}(Q)$.
Using that $\Theta(P)\lesssim \Theta(R)$ for all $P\in\tree_H(R)$ (remember that we do not keep track of the implicit constants depending on $A_0$), we get
\begin{align*}
\EE(9Q) & =
\sum_{k\geq0} 
\sum_{R\in\hd^k(Q)\cap\DD_\mu(9Q)}\, \sum_{P\in\tree_H(R)}
\left(\frac{\ell(P)}{\ell(Q)}\right)^{3/4}\Theta(P)^2\,\mu(P)\\
& \lesssim 
\sum_{k\geq0} 
\sum_{R\in\hd^k(Q)\cap\DD_\mu(9Q)} \Theta(R)^2\sum_{P\in\tree_H(R)}
\left(\frac{\ell(P)}{\ell(Q)}\right)^{3/4}\,\mu(P)\\
&\lesssim 
\sum_{k\geq0} 
\sum_{R\in\hd^k(Q)\cap\DD_\mu(9Q)} \left(\frac{\ell(R)}{\ell(Q)}\right)^{3/4}\Theta(R)^2\,\mu(R) = \EE^H(9Q).
\end{align*}

To show $\EE^H(9Q)\lesssim\EE_\infty(9Q)$, 
denote
$$m_k(Q) = \frac1{\ell(Q)}\max\{\ell(P):P\in\hd^k(Q)\cap\DD_\mu(9Q)\}.$$
Then we have
\begin{align*}
\EE^H(9Q)
& = 
\sum_{k\geq 0} 
\sum_{P\in\hd^k(Q)\cap\DD_\mu(9Q)} \left(\frac{\ell(P)}{\ell(Q)}\right)^{3/4}\Theta(P)^2\,\mu(P)\\
& \leq \sum_{k\geq 0} 
 m_k(Q)^{1/4}
\sum_{P\in\hd^k(Q)\cap\DD_\mu(9Q)} \left(\frac{\ell(P)}{\ell(Q)}\right)^{\!1/2}\Theta(P)^2\,\mu(P)
\end{align*}
To estimate $m_k(Q)$, observe that if $P\in\hd^k(Q)\cap\DD_\mu(9Q)$ and $k\geq4$, then
$$A_0^{kn}\,\theta_\mu(2B_Q)\approx\theta_\mu(2B_P) \lesssim \frac{\ell(Q)^n}{\ell(P)^n} \,\theta_\mu(2B_Q).$$
Hence,
$\ell(P)\lesssim A_0^{-k}\,\ell(Q)$, and thus, since this also holds in the case $1\leq k\leq 3$,
\begin{equation}\label{eqmkpet3}
m_k(Q) \lesssim  A_0^{-k}\qquad \mbox{ for all $k\geq1$.}
\end{equation}
Consequently, 
\begin{align*}
\EE^H(9Q)&\lesssim  \sum_{k\geq 0}A_0^{-k/4}
\sum_{P\in\hd^k(Q)\cap\DD_\mu(9Q)} \left(\frac{\ell(P)}{\ell(Q)}\right)^{\!1/2}\theta_\mu(2B_P)^2\,\mu(P)\\
& \lesssim  \sum_{k\geq 0}A_0^{-k/4} \EE_\infty(9Q)
\approx \EE_\infty(9Q).
\end{align*}
\end{proof}

\vv

\begin{rem}\label{remmk}
For the record, notice that, given $Q\in\DD_\mu^\PP$ and
\begin{equation}\label{eqmkpet2}
m_k(Q) = \frac1{\ell(Q)}\max\{\ell(P):P\in\hd^k(Q)\cap\DD_\mu(9Q)\},
\end{equation}
as shown in \rf{eqmkpet3}, it turns out that 
\begin{equation}\label{eqmkpet4}
m_k(Q) \leq  C_1 A_0^{-k}.
\end{equation}
\end{rem}

\vv

% ********************************************************************************************
% ********************************************************************************************

Given $M\gg1$ (we will choose $M>  A_0^{2n}\gg 1$), we say that $Q\in\DD_\mu$ is $M$-dominated from below
if
there exists some $k\geq1$ such that
\begin{equation}\label{eqDB}
\sum_{P\in\hd^k(Q)\cap\DD_\mu(9Q)} \left(\frac{\ell(P)}{\ell(Q)}\right)^{\!1/2}\Theta(P)^2\,\mu(P) > M^2\,\Theta(Q)^2\,\mu(9Q),
\end{equation}
or in other words,
\begin{equation}\label{eqDB'}
\EE_\infty(9Q)> M^2\,\Theta(Q)^2\,\mu(9Q),
\end{equation}
We denote by $\DB(M)$ the family of cubes from $\DD_\mu^\PP$ that are $M$-dominated from below. 
Notice that the cubes from $\DB(M)$ are assumed to be $\PP$-doubling.

\vv

% ********************************************************************************************

\subsection{The Main Propositions and the proof of Theorem \ref{teomain1}}\label{subsec:main thm}

The proof of Theorem \ref{teomain1} consists of two main propositions. The first one is the following.

\begin{mpropo}[First Main Proposition]\label{propomain}
	Let $\mu$ be a Radon measure in $\R^{n+1}$ 
	with compact support which has polynomial growth of degree $n$ with constant $\theta_0$, that is, 
	$$
	\mu(B(x,r))\leq \theta_0\,r^n\quad \mbox{ for all $x\in\supp\mu$ and all $r>0$}.
	$$
	Suppose also that $\|\RR_*\mu\|_{L^1(\mu)}<\infty$. Then, for any choice of $M>1$,
	\begin{equation}\label{eqpropo*}
		\sum_{Q\in\DD_\mu} \beta_{\mu,2}(2B_Q)^2\,\Theta(Q)\,\mu(Q)\leq C\,\big(\|\RR\mu\|_{L^2(\mu)}^2 + \theta_0^2\,\|\mu\|
		+ \sum_{Q\in\DB(M)}\EE_\infty(9Q)\big),
	\end{equation}
	with $C$ depending on $M$.
\end{mpropo}

\vv

The second main ingredient of the proof of Theorem \ref{teomain1} is the following.

\begin{mpropo}[Second Main Proposition]\label{propomain2}
Let $\mu$ be a Radon measure in $\R^{n+1}$ 
	with compact support which has polynomial growth of degree $n$ with constant $\theta_0$ and such that
that $\|\RR_*\mu\|_{L^1(\mu)}<\infty$. Let $M_0=A_0^{k_0 n}$, where $k_0$ is some big enough absolute constant depending just on $n$.
Then
$$\sum_{Q\in\DB(M_0)} \EE_\infty(9Q)\leq C\, \big(\|\RR\mu\|_{L^2(\mu)}^2 + 
\theta_0^2\,\|\mu\|\big),$$
where $C$ depends just on $n$ and the parameters of the dyadic lattice $\DD_\mu$.
\end{mpropo}
\vv

By combining the two Main Propositions and choosing $M$ and $M_0$ appropriately, we deduce that
$$\sum_{Q\in\DD_\mu} \beta_{2,\mu}(2B_Q)^2\,\Theta(Q)\,\mu(Q)\leq C\,\big(\|\RR\mu\|_{L^2(\mu)}^2 + \theta_0^2\,\|\mu\|
\big),$$
which implies Theorem \ref{teomain1} in the case when $\mu$ is compactly supported. The general case for $\mu$ not compactly supported follows from a standard reduction argument, by considering $\mu|_{B(0,r)}$, where $B(0,r)$ has small boundary (see Lemmas 9.43 and 9.44 from \cite{Tolsa-llibre}, for example), and letting $r\to\infty$.

The rest of the paper is devoted to the proof of the two Main Propositions.

\vv
% ********************************************************************************************
% ********************************************************************************************
% ********************************************************************************************

%\part{Tractable trees}
%\vvv

\bigskip
\begin{center} 
	\Large Part \refstepcounter{parte}\theparte\label{part-2}: Tractable trees
\end{center}
\smallskip

\addcontentsline{toc}{section}{\bf Part 2: Tractable trees}

\section{The cubes with moderate decrement of Wolff energy and the associated tractable trees} \label{sec4}

\subsection{The family \texorpdfstring{$\MDW$}{MDW} and the enlarged cubes}\label{secMDW}\label{subsec:enlar}

We let
$$\Lambda = A_0^{k_\Lambda n},$$
where $k_\Lambda>4$ is some positive large number that will be fixed below, depending just on $n$ and the $\DB$-parameter $M$. Its dependence on $M$ is of the form
\begin{equation}\label{eq:LambdadepM}
	\Lambda\ge\max(CM, M^{\frac{8n-1}{8n-2}}),
\end{equation}
for some big dimensional constant $C$.

Given $R\in\DD_\mu^\PP$, we denote 
$$\HD(R) = \hd^{k_\Lambda}(R).$$
Also, we take $\delta_0\in (0,\Lambda^{-4n^2})$. In fact, the precise value of $\delta_0$ will be fixed at the beginning of Section 
\ref{sec3.3}.
%$$\delta_0 = \Lambda^{-N_0 - \frac1{2N}} = A_0^{-n\,k_\Lambda(N_0 + \frac1{2N})}$$
%where both $N$ and $N_0$ are big natural numbers depending on $n$ that will be fixed below. 
%Further we assume that $k_{\Lambda}$ is a multiple of $2N$, so that $k_{\Lambda}(N_0 + \frac1{2N})$ is integer.
We let $\LD(R)$ be the family of cubes $Q\in\DD_{\mu}$
which
%\footnote{Notice we ask cubes from $\LD(R)$ to be $\PP$-doubling. It has to be checked this is OK.} 
are maximal and satisfy
$$\ell(Q)< \ell(R)\quad \mbox{ and }\quad\PP(Q) \leq \delta_0\,\Theta(R).$$

Next we introduce an ``optional'' family. For the proof of the First Main Proposition, we will take $\OP(R)= \varnothing$, while for the proof
of the Second Main Proposition we will take $\OP(R)=\NDB(R)$, where
 $\NDB(R)$ (which stands for ``near $\DB$'') is the family of cubes $Q$ which do not belong to $\LD(R)\cup\HD(R)$ and satisfy the following:
\begin{itemize}
\item $\ell(Q)<\lambda\,\ell(R)$, with $\lambda= c_3\,\Lambda^{-4}$ as in \rf{eqlambda0}, and
\item there exists another cube $Q'\in\DB(M)$ of the same generation as $Q$ such that $Q'\subset 20Q$.
\end{itemize}
We let $\bad(R)$ be the family of maximal cubes from $\LD(R)\cup\HD(R)\cup\OP(R)$ (not necessarily contained in $R$) and we denote
$$\sss(R)= \bad(R)\cap\DD_\mu(R).$$
Remark that the arguments in Part 2 (Sections \ref{sec4} - \ref{sec9}) are valid with both choices of the optional family $\OP(R)$. In Parts 2 and 3 we consider the $\DB$-parameter $M$ to be fixed, and we will usually write $\DB$ instead of $\DB(M)$.

For a family $I\subset\DD_\mu$, we denote
$$\sigma(I) = \sum_{P\in I}\Theta(P)^2\,\mu(P).$$ 
We say that a cube $R\in\DD_{\mu}$ has moderate decrement of Wolff energy, and we write
$R\in\MDW$, if $R$ is $\PP$-doubling and
\begin{equation}\label{eq:MDWdef}
\sigma(\HD(R)\cap\sss(R))\geq B^{-1}\,\sigma(R),
\end{equation}
where
$$B= \Lambda^{\frac1{100n}}.$$  

\vv

%In this section we show how to associate to each $R\in\MDW$ a suitable family of tractable trees (to be defined later).

For $R\in\MDW$, the fact that the cubes from family $\HD(R)\cap\sss(R)$ may be located close to $\supp\mu\setminus R$ 
may cause problems when trying to obtain estimates involving the Riesz transform. For this reason we need to
introduce some ``enlarged cubes''.
Given $j\geq0$ and $R\in\DD_{\mu,k}$, 
we let 
$$e_j(R) = R \cup \bigcup Q,$$
where the last union runs over the cubes $Q\in\DD_{\mu,k+1}$ such that
\begin{equation}\label{eqxrq83}
\dist(x_R,Q)< \frac{\ell(R)}2 + 2j\ell(Q).
\end{equation}
We say that $e_j(R)$ is an enlarged cube.
Notice that, since $\diam(Q)\leq \ell(Q)$,
\begin{equation}\label{eqxrq84}
\supp\mu\cap B\big(x_R,\tfrac12\ell(R) + 2j\ell(Q)\big)\subset e_j(R) \subset B\big(x_R,\tfrac12\ell(R) + (2j+1)\ell(Q)\big).
\end{equation}
Also,  we have
\begin{equation}\label{eqqj8d}
e_j(R)\subset 2R\quad \mbox{ for $0\leq j\leq \frac34 A_0$,}
\end{equation}
since, for any $Q\in\DD_{\mu,k+1}$ satisfying \rf{eqxrq83}, its parent satisfies $\wh Q$
$$\dist(x_R,\wh Q)< \frac{\ell(R)}2 + 2j A_0^{-1}\ell(\wh Q) \leq 2\ell(R).$$

%We denote by $\NDB$ the family of cubes $Q$ for which there exists another cube $Q'\in\DB$ of the same generation as $Q$ such that $Q'\subset 9Q$. 
For $R\in\MDW$, we let
$$\sss(e_j(R)) = \bad(R) \cap\DD_\mu(e_j(R)),$$
where $\DD_\mu(e_j(R))$ stands for the family of cubes from $\DD_\mu$ which are contained in
$e_j(R)$ and have side length at most $\ell(R)$. % Notice that we are not assuming $R\in\ttt$.

\vv
\begin{lemma}\label{lem:43}
For any $R\in\MDW$ there exists some $j$, with $10\leq j\leq A_0/4$ such that
\begin{equation}\label{eqsigmaj}
\sigma(\HD(R)\cap\sss(e_{j}(R))) \leq B^{1/4} \sigma(\HD(R)\cap\sss(e_{j-10}(R))),
\end{equation}
 assuming $A_0$ big enough, depending just on $n$. 
\end{lemma}

\begin{proof}
Given $R\in\MDW$, suppose that such $j$ does not exist. Let $j_0$ be the largest integer which is multiple of $10$ and smaller that $A_0/4$.
Then we get
\begin{align*}
\sigma(\HD(R)\cap\sss(e_{j_0}(R))) & \geq B^{\frac14}\sigma(\HD(R)\cap\sss(e_{j_0-10}(R)))\\
&\geq
\ldots \geq \big(B^{\frac14}\big)^{\frac{j_0}{10}-1}\sigma(\HD(R)\cap\sss(R)) \overset{\eqref{eq:MDWdef}}{\geq} B^{\frac{j_0}{40}-\frac54}\sigma(R).
\end{align*}
By \rf{eqqj8d}, we have $e_{j_0}(R)\subset 2R$ and thus
$$\sigma(\HD(R)\cap\sss(e_{j_0}(R)))= \sum_{Q\in\HD(R)\cap \sss(e_{j_0}(R))} \Lambda^2\Theta(R)^2\mu(Q)\leq 
\Lambda^2\Theta(R)^2\mu(2R).$$
Since $R$ is $\PP$-doubling (and in particular $R\in\DD_\mu^{db}$), denoting by $\wh R$ the parent of $R$,  we derive 
\begin{equation}\label{eqdoub*11}
\mu(2R)\leq \mu(2B_{\wh R}) \leq \frac{\ell(\wh R)^{n+1}}{\ell(R)}\,\PP(R) \leq C_d\,A_0^{n+1}
\mu(2B_R)\leq C_0\,C_d\,A_0^{n+1}
\mu(R).
\end{equation}
So we deduce that
$$B^{\frac{j_0}{40}-\frac54}\sigma(R)\leq C_0\,C_d\,A_0^{n+1}\,\Lambda^2\sigma(R),$$
or equivalently, recalling the choice of $B$ and $C_d$,
$$\Lambda^{\frac{1}{100n}\left(\frac{j_0}{40}-\frac54\right) -2} \leq 4 C_0\,A_0^{2n+1}.$$
Since $\Lambda\geq A_0^n$ and $j_0\approx A_0$, it is clear that this inequality is violated if $A_0$ is big enough, depending just on $n$.
\end{proof}
\vv

Given $R\in\MDW$,  let $j\geq 10$ be minimal such that \rf{eqsigmaj} holds. We denote
$h(R)=j-10$ and we write
$$e(R) = e_{h(R)}(R),\qquad e'(R) = e_{h(R)+1}(R), \quad e''(R)=e_{h(R)+2}(R), \quad e^{(k)}(R) = e_{h(R)+k}(R),$$
for $k\geq 1$. 
We let
\begin{align*}
B(e(R)) &= B\big(x_R,(\tfrac12 + 2A_0^{-1}h(R))\ell(R)\big),\\
B(e'(R)) & = B\big(x_R,(\tfrac12 + 2A_0^{-1}(h(R)+1))\ell(R)\big),\\
B(e''(R)) & = B\big(x_R,(\tfrac12 + 2A_0^{-1}(h(R)+2))\ell(R)\big),\\
B(e^{(k)}(R)) & = B\big(x_R,(\tfrac12 + 2A_0^{-1}(h(R)+k))\ell(R)\big).
\end{align*}
By construction (see \rf{eqxrq84}) we have 
$$B(e'(R))\cap\supp\mu\subset e'(R),$$
and analogously replacing $e'(R)$ by $e(R)$ or $e''(R)$.
Remark also that
$$e(R)\subset B(e'(R))\quad \text{ and }\quad \dist(e(R),\partial B(e'(R))) \geq A_0^{-1}\ell(R),$$
and, analogously,
$$e'(R)\subset B(e''(R))\quad \text{ and }\quad \dist(e'(R),\partial B(e''(R))) \geq A_0^{-1}\ell(R).$$

\vv

\begin{lemma}\label{lem-calcf}
For each $R\in\MDW$  we have
$$B(e''(R)) \subset (1+8A_0^{-1})\,\,B(e(R)) \subset B(e^{(6)}(R)),$$
and more generally, for $k\geq 2$ such that $h(R)+k-2\leq A_0/2$,
$$B(e^{(k)}(R)) \subset (1+8A_0^{-1})\,\,B(e^{(k-2)}(R)) \subset B(e^{(k+4)}(R)).$$
Also,
$$B(e^{(10)}(R))\subset B\big(x_R,\tfrac32 \ell(R)\big).$$
\end{lemma}

\begin{proof}
This follows from straightforward calculations. 
Indeed,
\begin{multline*}
r(B(e^{(k)}(R)))  = \frac{(\tfrac12 + 2A_0^{-1}(h(R)+k))\ell(R)}{(\tfrac12 + 2A_0^{-1}(h(R)+k-2))\ell(R)}
\,r(B(e^{(k-2)}(R)))\\
 = 1 + \frac{8A_0^{-1}}{1 + 4A_0^{-1}(h(R)+k-2)}\,r(B(e^{(k-2)}(R)))\leq (1+8A_0^{-1})\,r(B(e^{(k-2)}(R))).
\end{multline*}
Also, using that $h(R)+k-2\leq A_0/2$,
\begin{align*}
(1+8A_0^{-1})\,r(B(e^{(k-2)}(R))) & = (1+8A_0^{-1})\,\big(\tfrac12 + 2A_0^{-1}(h(R)+k-2)\big)\,\ell(R)\\
& \leq \big(\tfrac12 + 2A_0^{-1}(h(R)+k-2) + 4A_0^{-1} + 8A_0^{-1}
\big)\,\ell(R)\\
& =
r(B(e^{(k+4)}(R))).
\end{align*}

The last statement of the lemma follows from the fact that $h(R)+10\leq A_0/4<A_0/2$:
$$
B(e^{(10)}(R)) = B\big(x_R,(\tfrac12 + 2A_0^{-1}(h(R)+10))\ell(R)\big) \subset
B\big(x_R,(\tfrac12 + 2)\ell(R)\big) = B\big(x_R,\tfrac32 \ell(R)\big).$$
\end{proof}
\vv

\subsection{Generalized trees and negligible cubes}\label{subsec:generalized}

Next we need to define some families that can be considered as ``generalized trees''. First, we introduce some additional notation regarding the stopping cubes. For $R\in\DD_{\mu}^\PP$ we set
$$\HD_{1}(R) = \sss(R)\cap \HD(R).$$ 
Assume additionally that $R\in\MDW$. We write $\sss(e(R))=\sss(e_{h(R)}(R))$ and $\sss(e'(R))=\sss(e_{h(R)+1}(R))$. Furthermore,
$$\HD_1(e(R)) = \sss(e(R))\cap \HD(R),$$
and
$$\HD_{1}(e'(R)) = \sss(e'(R))\cap \HD(R).$$
We define $\HD_{1}(e^{(k)}(R))$ for $2\le k\le 10$ analogously. Also, we set 
$$\HD_2(e'(R)) = \bigcup_{Q\in \HD_1(e'(R))} (\sss(Q)\cap \HD(Q))$$
and
\begin{equation}\label{eqstop2}
\sss_2(e'(R)) = \big(\sss(e'(R)) \setminus \HD_1(e'(R))\big) \cup \bigcup_{Q\in \HD_1(e'(R))} \sss(Q).
\end{equation}
We let $\TT_\sss(e'(R))$ be the family of cubes made up of $R$ and all the cubes of the next generations which are contained in $e'(R)$ but are not 
strictly contained in any cube from $\sss_2(e'(R))$.

Observe that the defining property of $\MDW$ \eqref{eq:MDWdef} can now be rewritten as 
\begin{equation}\label{eq:MDWdef2}
	\sigma(R)\le B\, \sigma(\HD_1(R)).
\end{equation}
Moreover, by \eqref{eqsigmaj} and the definition of $e(R)$ we have
\begin{equation}\label{eq:sigmae'lesigmae}
\sigma(\HD_1(e^{(10)}(R)))\le B^{1/4}\sigma(\HD_1(e(R))).
\end{equation}

We define now the family of negligible cubes. We say that a cube $Q\in\TT_\sss(e'(R))$ is negligible for $\TT_\sss(e'(R))$, and we write $Q\in\Neg(e'(R))$ if
there does not exist any cube from $\TT_\sss(e'(R))$ that contains $Q$ and is $\PP$-doubling.  

\vv
\begin{lemma}\label{lemnegs}
Let $R\in\MDW$. If $Q\in\Neg(e'(R))$, then $Q\subset e'(R)\setminus R$, $Q$ is not contained in any cube from $\HD_1(e'(R))$, and
\begin{equation}\label{eqcostat}
\ell(Q) \gtrsim \delta_0^{2}\,\ell(R).
\end{equation}
\end{lemma}

\begin{proof}
Let $Q\in\Neg(e'(R))$. We have $Q\subset e'(R)\setminus R$ due to the fact that $R$ is $\PP$-doubling.
For the same reason, $Q$ is not contained in any cube from $\HD_1(e'(R))$.

To prove \rf{eqcostat}, assume that $\ell(Q)\leq A_0^{-2}\ell(R)$. Otherwise the inequality is immediate.
By Lemma \ref{lemdobpp}, since all the ancestors $Q_1,\ldots,Q_m$ of $Q$ that are contained in $e'(R)$ are not $\PP$-doubling, it follows that $Q_1$ (the parent of $Q$) satisfies
$$\PP(Q_1)\lesssim A_0^{-m/2-1}\,\PP(Q_m).$$
Because $Q_m\subset e'(R)\subset 2R$ and $\ell(Q_m)=A_0^{-1}\ell(R)$, it is easy to see that $\PP(Q_m)\lesssim \PP(R)\lesssim C_{d}\,\Theta(R)$, and so
\begin{equation*}
\PP(Q_1)\lesssim A_0^{-m/2}\,\Theta(R)\approx\left(\frac{\ell(Q)}{\ell(R)}\right)^{1/2}\,\Theta(R).
\end{equation*}
By the definition of $\LD(R)$, we know that $\PP(Q_1)\geq\delta_0\,\Theta(R)$, which together with 
the previous estimate yields \rf{eqcostat}.
\end{proof}
\vv

The cubes from $\sss_2(e'(R))$ need not be $\PP$-doubling, which is problematic for some of the estimates involving the Riesz transform localized around the
trees $\TT_\sss(e'(R))$ that will be required later. For this reason, we need to consider enlarged versions of these trees. For $R\in\MDW$, we let $\End(e'(R))$ be the family made up of the following cubes:
\begin{itemize}
\item the cubes from $\sss_2(e'(R))\cap \Neg(e'(R))$,
\item the cubes that are contained in any cube from $\sss_2(e'(R))\setminus \Neg(e'(R))$ which are $\PP$-doubling and, moreover, are maximal.
\end{itemize}
Notice that all the cubes from $\End(e'(R))$ are $\PP$-doubling, with the exception of the ones from $\Neg(e'(R))$.
We let $\TT(e'(R))$ be the family of cubes that are contained in $e'(R)$ and are not 
strictly contained in any cube from $\End(e'(R))$.

 \vv
 
\subsection{Tractable trees}\label{subsec:trc} 
Given $R\in\MDW$, we say that $\TT(e'(R))$ is tractable (or that $R$ is tractable) if 
$$\sigma(\HD_2(e'(R)))\leq B\,\sigma(\HD_1(e(R))).$$
In this case we write $R\in\Trc$.

 Our next objective consists in showing how we can associate a family of tractable trees to any $R\in\MDW$.
% , so that we can reduce the estimate of $\sigma(\ttt)$ (with $\ttt$ to be defined below) and $\sigma(\DB)$ to estimating the Haar coefficients of $\RR\mu$ from below on such family of tractable trees.
First we need the following lemma.

\begin{lemma}\label{lemalg1}
Let $R\in\MDW$ be such that $\TT(e'(R))$ is not tractable. Then there exists a family $\GH(R)\subset
\HD_1(e'(R))\cap\MDW$ satisfying:
\begin{itemize}
\item[(a)] The balls $B(e''(Q))$, with $Q\in\GH(R)$ are pairwise disjoint.
\item[(b)] For every $Q\in\GH(R)$, $\sigma(\HD_1(e(Q)))\geq \sigma(\HD_1(Q))\geq B^{1/2}\sigma(Q)$.
\item[(c)] $$B^{1/4} \sum_{Q\in\GH(R)} \sigma(\HD_1(e(Q))) \gtrsim \sigma(\HD_2(e'(R))).$$
\end{itemize}
\end{lemma}

The name ``$\GH$'' stands for ``good high (density)''. Remark that the property (c) and the fact that $R\not\in\Trc$ yield
\begin{equation}\label{eqiter582}
\sum_{Q\in\GH(R)} \sigma(\HD_1(e(Q))) \gtrsim B^{3/4}\,\sigma(\HD_1(e(R))),
\end{equation}
which is suitable for iteration.

\begin{proof}[Proof of Lemma \ref{lemalg1}]
Let $R\in\MDW$ be such that $\TT(e'(R))$ is not tractable. Notice first that
$$\sigma(\HD_1(e'(R)))\overset{\eqref{eq:sigmae'lesigmae}}{\leq} B^{1/4}
\sigma(\HD_1(e(R)))\leq B^{-3/4}\sigma(\HD_2(e'(R))) .$$
Let $I\subset \HD_1(e'(R))$ be the subfamily of the cubes $Q$ such that
$$\sigma(\HD_1(Q))< B^{1/2}\sigma(Q).$$
Then we have 
\begin{align*}
\sum_{Q\in I} \sigma(\HD_2(e'(R))\cap\DD_\mu(Q)) & \leq B^{1/2} \sum_{Q\in I} \sigma(Q)
\leq B^{1/2} \sigma(\HD_1(e'(R))) \\
&\leq \frac{B^{1/2}}{B^{3/4}}\,\sigma(\HD_2(e'(R)))\leq
 \frac12\,\sigma(\HD_2(e'(R))).
\end{align*}
Therefore,
\begin{align}\label{eqdj723}
\sum_{Q\in \HD_1(e'(R))\setminus I} \sigma(\HD_2(e'(R))\cap\DD_\mu(Q)) & = \sigma(\HD_2(e'(R))) -
\sum_{Q\in I} \sigma(\HD_2(e'(R))\cap\DD_\mu(Q))\\
& \geq \frac12\,\sigma(\HD_2(e'(R))).\nonumber
\end{align}

Next we will choose a family $J\subset \HD_1(e'(R))\setminus I$ satisfying
\begin{itemize}
\item[(i)] The balls $B(e''(Q))$, with $Q\in J$, are pairwise disjoint.
\item[(ii)] $$B^{1/4}\sum_{Q\in J}\sigma(\HD_1(e(Q)))
\gtrsim \sum_{Q\in \HD_1(e'(R))\setminus I} \sigma(\HD_2(e'(R))\cap\DD_\mu(Q)).$$
\end{itemize}
Then, choosing $\GH(R) = J$ we will be done. Indeed, the  property (a) in the statement of the lemma is the same as (i), and the property (b) is a consequence of the fact that
$J\subset I^c$ and the definition of $I$. This also implies that $\GH(R)\subset\MDW$. Finally, the property (c) follows from \rf{eqdj723} and (ii).

Let us see how $J$ can be constructed. By the covering Theorem 9.31 from \cite{Tolsa-llibre}, there
is a family $J_0\subset \HD_1(e'(R))\setminus I$ such that
\begin{itemize}
\item[1)] The balls $B(e''(Q))$, with $Q\in J_0$, have finite superposition, that is, 
$$\sum_{Q\in J_0}\chi_{B(e''(Q))}\leq C,$$
and
\item[2)] 
$$\bigcup_{Q\in \HD_1(e'(R))\setminus I} B(e''(Q)) \subset \bigcup_{Q\in J_0} (1+8A_0^{-1})\,B(e''(Q)),$$
\end{itemize}
Actually, in Theorem 9.31 from \cite{Tolsa-llibre} the result above is stated for a finite family of
balls. However, it is easy to check that the same arguments work as soon as the family $\HD_1(e'(R))\setminus I$ is countable and can be ordered so that $\HD_1(e'(R))\setminus I=\{Q_1,Q_2,\ldots\}$,
with $\ell(Q_1)\geq \ell(Q_2)\geq\ldots$. Further, one can check that the constant $C$ in 1)
does not exceed some absolute constant times $A_0^{n+1}$.

From the finite superposition property 1), by rather standard arguments which are analogous to the
ones in the proof of Besicovitch's covering theorem in \cite[Theorem 2.7]{Mattila-llibre}, say, 
one deduces that $J_0$ can be split into $m_0$ subfamilies $J_1,\ldots, J_{m_0}$ so that, for each $k$,  the balls $\{B(e''(Q)): Q\in J_k\}$  are pairwise disjoint, with $m_0\leq C(A_0)$.

Notice that the condition 2) and Lemma \ref{lem-calcf} applied to $Q$ ensure that
\begin{equation}\label{equni98-1}
\bigcup_{Q\in \HD_1(e'(R))\setminus I} Q\subset \bigcup_{Q\in \HD_1(e'(R))\setminus I} B(e''(Q) )\subset \bigcup_{Q\in J_0} (1+8A_0^{-1})\,B(e''(Q)) \subset \bigcup_{Q\in J_0} B(e^{(8)}(Q)).
\end{equation}
Next we choose $J:=J_k$ to be the family such that
$$\sum_{Q\in J_k}\sigma(\HD_1(e(Q)))$$
is maximal among $J_1,\ldots,J_{m_0}$, so that
\begin{align*}
\sum_{Q\in J}\sigma(\HD_1(e(Q))) & \geq \frac1{m_0}\,
\sum_{Q\in J_0}\sigma(\HD_1(e(Q)))\\
& \overset{\eqref{eq:sigmae'lesigmae}}{\geq} \frac{1}{m_0\,B^{1/4}} \sum_{Q\in J_0}\sigma(\HD_1(e^{(8)}(Q)))\\
& \overset{\rf{equni98-1}}{\ge} \frac{1}{m_0\,B^{1/4}} \sum_{Q\in \HD_1(e'(R))\setminus I} \sigma(\HD_1(Q))\\
& = \frac{1}{m_0\,B^{1/4}}\sum_{Q\in \HD_1(e'(R))\setminus I} \sigma(\HD_2(e'(R))\cap\DD_\mu(Q)).
\end{align*}
This proves (ii).
\end{proof}
\vv

Given $R\in\MDW$, we will construct now a subfamily of cubes from $\MDW$ generated by $R$,
which we will denote $\Gen(R)$, by iterating the construction of Lemma \ref{lemalg1}.
The algorithm goes as follows.
Given $R\in\MDW$, we denote 
$$\Gen_0(R) = \{R\}.$$
If $R\in\Trc$, we set $\Gen_1(R)=\varnothing$, and otherwise
$$\Gen_1(R) = \GH(R),$$
where $\GH(R)$ is defined in Lemma \ref{lemalg1}.
For $j\geq 2$, we set
$$\Gen_{j}(R) = \bigcup_{Q\in\Gen_{j-1}(R)\setminus \Trc} \GH(Q).$$
For $j\geq0$, we also set
$$\Trc_j(R) = \Gen_j(R)\cap\Trc,$$
and
$$\Gen(R) = \bigcup_{j\geq0}\Gen_j(R),\qquad\Trc(R) = \bigcup_{j\geq0}\Trc_j(R).$$
\vv

\begin{lemma}\label{eqtec74}
For $R\in\MDW$, we have
\begin{equation}\label{eqtec741}
\bigcup_{Q\in\Trc(R)}Q\subset\bigcup_{Q\in\Gen(R)}Q \subset B(e''(R)).
\end{equation}
Also,
\begin{equation}\label{eqiter*44}
\sigma(\HD_1(e(R)))\leq \sum_{j\geq0} B^{-j/2}\sum_{Q\in\Trc_j(R)}\sigma(\HD_1(e(Q))).
\end{equation}
\end{lemma}

\begin{proof}
%Obviously, in the case $R\in\Trc$, \rf{eqiter*44} holds.

The first inclusion in \eqref{eqtec741} holds because $\Trc(R)\subset\Gen(R)$. So we only have to show the second inclusion.

By construction, for any $R'\in\MDW$, $\GH(R')\subset \HD_1(e'(R'))$, and thus any $Q\in \GH(R')$ is contained in $e'(R')$. This implies that 
$$|x_{R'}-x_Q|\leq r(B(e'(R'))) + \frac12\,\ell(Q)\leq  \Big(1+ 2A_0^{-1}+ \frac12\,A_0^{-1}\Big)\ell(R') \leq 1.1\,\ell(R').$$
Then, given $Q\in\Gen_j(R)$, $x\in Q$, and $0\leq k\leq j$, if we denote by $R_k$ the cube from $\Gen_k(R)$ such that $Q\in\Gen_{j-k}(R_k)$, we have
\begin{align*}
|x_R-x|& \leq |x_R- x_{R_1}|+ \sum_{k=1}^{j-1} |x_{R_k}- x_{R_{k+1}}| + |x_Q-x|\\
& \leq r(B(e'(R)))+ \frac12\,A_0^{-1}\,\ell(R) + \sum_{k=1}^{j-1}1.1\,A_0^{-k}\ell(R) + \frac12\,A_0^{-1}\,\ell(R)\\
& \leq r(B(e'(R)))+ 2\,A_0^{-1}\,\ell(R),
\end{align*}
which shows that $Q\subset B(e''(R))$.

To prove the second statement in the lemma, observe that, for $Q\in\Gen_{j-1}(R)\setminus \Trc$,
by \rf{eqiter582} applied to $Q$ we have
$$\sum_{P\in\GH(Q)} \sigma(\HD_1(e(P))) \geq c\,B^{3/4}\,\sigma(\HD_1(e(Q)))\geq 
B^{1/2}\,\sigma(\HD_1(e(Q))),
$$
assuming $\Lambda$, and thus $B$, big enough. Therefore,
\begin{align*}
\sum_{P\in\Gen_j(R)}\sigma(\HD_1(e(P))
& = 
\sum_{Q\in\Gen_{j-1}(R)\setminus \Trc} \,\sum_{P\in\GH(Q)}\sigma(\HD_1(e(P)))\\
& \geq 
B^{1/2}\sum_{Q\in\Gen_{j-1}(R)\setminus \Trc}\sigma(\HD_1(e(Q)))
\end{align*}
So,
$$\sum_{Q\in\Gen_{j-1}(R)}\sigma(\HD_1(e(Q)))\leq 
\sum_{Q\in\Trc_{j-1}(R)}\sigma(\HD_1(e(Q))) +
B^{-1/2}\sum_{P\in\Gen_j(R)}\sigma(\HD_1(e(P))).
$$
Iterating this estimate, and taking into account that, by the polynomial growth of $\mu$,
$\Gen_{j-1}(R)=\varnothing$ for some large $j$, we get \rf{eqiter*44}.
\end{proof}

\vv

% **************************************************************************************************************************

% ********************************************************************************************

\section{The Riesz transform on the tractable trees: the approximating measures \texorpdfstring{$\eta$}{eta}, \texorpdfstring{$\nu$}{nu}, and the variational argument}\label{sec6}

In this section,
for a given $R\in\MDW$ such that $\TT(e'(R))$ is tractable (i.e., $R\in\Trc$),
we will define a suitable measure $\eta$ that approximates $\mu$ at the level of the cubes from
$\TT(e'(R))$ and we will estimate $\|\RR\eta\|_{L^p(\eta)}$ from below. To this end,
we will apply a variational argument in $L^p$ by techniques inspired by 
\cite{Reguera-Tolsa} and \cite{JNRT}.  In the next section we will transfer these
estimates to $\RR\mu$.

%  *************************************************************************************************

\subsection{The suppressed Riesz transform and a Cotlar type inequality}\label{sec6.1}

%Millor posar aixo cap al prinicipi.

Let $\Phi:\R^{n+1}\to[0,\infty)$ be a $1$-Lipschitz function.
Below we will need to work with the suppressed Riesz kernel
\begin{equation}\label{eqsuppressed}
K_\Phi(x,y) = \frac{x-y}{\bigl (|x-y|^2+\Phi(x)\Phi(y)\bigr)^{(n+1)/2}}
\end{equation}
and the associated operator 
$$\RR_\Phi\alpha(x) =\int K_\Phi(x,y)\,d\alpha(y),$$
where $\alpha$ is a signed measure in $\R^{n+1}$.
For a positive measure $\omega$ and $f\in L^1_{loc}(\omega)$, we write $\RR_{\Phi,\omega} f = \RR_\Phi (f\,\omega)$.
The kernel $K_\Phi$ (or a variant of this) appeared for the first
time in the work of Nazarov, Treil and Volberg in connection with Vitushkin's conjecture (see
\cite{Volberg}). 
This is a Calder\'on-Zygmund kernel which satisfies the properties:
\begin{equation}\label{eqkafi1}
|K_\Phi(x,y)|\lesssim \frac1{\big(|x-y| + \Phi(x) + \Phi(y)\big)^n}
\end{equation}
and
\begin{equation}\label{eqkafi2}
|\nabla_x K_\Phi(x,y)|+ |\nabla_y K_\Phi(x,y)|
\lesssim \frac1{\big(|x-y| + \Phi(x) + \Phi(y)\big)^{n+1}}
\end{equation}
for all $x,y\in\R^{n+1}$.

Also, if $\ve\approx\Phi(x)$, then we have
\begin{equation}
\label{e.compsup''}
\bigl|\RR_{\ve}\alpha(x) - \RR_{\Phi}\alpha(x)\bigr|\lesssim  \sup_{r> \Phi(x)}\frac{|\alpha|(B(x,r))}{r^n},
\end{equation}
with the implicit constant in the inequality depending on the implicit constant in the comparability $\ve\approx\Phi(x)$.
See Lemmas 5.4 and 5.5 in \cite{Tolsa-llibre}. 

The following result is an easy consequence of a $Tb$ theorem of Nazarov, Treil and Volberg.
See Chapter 5 of \cite{Tolsa-llibre}, for example. We will use this to prove \rf{eqacpsi}.

\begin{theorem}\label{teontv}
Let $\omega$ be a Radon measure in $\R^{n+1}$ and let $\Phi:\R^{n+1}\to[0,\infty)$ be a $1$-Lipschitz function. Suppose that
\begin{itemize}
\item[(a)] $\omega(B(x,r))\leq c\,r^n$ for all $r\geq \Phi(x)$, and
\item[(b)] $\sup_{\ve>\Phi(x)}|\RR_\ve\omega(x)|\leq C$.
\end{itemize}
Then $\RR_{\Phi,\omega}$ is bounded in $L^p(\omega)$, for $1<p<\infty$, with a bound on its norm depending only on $p$, $c$ and
$C$. In particular, $\RR_\omega$ is bounded in $L^p(\omega)$ on the set $\{x:\Phi(x)=0\}$.
\end{theorem}
\vv

%\subsection{A Cotlar type inequality}

We define the energy $W_\omega$ (with respect to $\omega$) of a set $F\subset \R^{n+1}$ as
$$W_\omega(F) = \iint_{F\times F} \frac1{\diam(F)\,|x-y|^{n-1}}\,d\omega(x)d\omega(y).$$
%For $\omega=\mu$ and a ball $B$, $W_\mu(B)$ can be estimated in terms of $\EE(B)$. 
%This fact will used in Section \ref{sec99} but not in the current section. See Lemma \ref{lemw-e} for more details.
We say that a ball $B\subset \R^{n+1}$ is $(a,b)$-doubling $$\omega(aB)\leq b\,\omega(B).$$
We denote by $\cM_\omega f$ the usual centered maximal Hardy-Littlewood operator applied to $f$:
$$\cM_\omega f(x) = \sup_{r>0}\,\frac1{\omega(B(x,r))}\int_{B(x,r)}|f|\,d\omega,$$
and by $\cM_\omega^{(r_0,a,b)} f$ the version
$$\cM_\omega^{(r_0,a,b)} f(x) = \sup\frac1{\omega(B(x,r))}\int_{B(x,r)}|f|\,d\omega,$$
where the $\sup$ is taken over all radii $r>r_0$ such that the ball $B(x,r)$ is $(a,b)$-doubling.

\vv

\begin{lemma}\label{lemcotlar1}
Let $x\in R$, $r_0>0$, and $\theta_1>0$.
Suppose that for all $r\geq r_0$ 
$$\theta_\omega(x,r)\leq \theta_1$$
and 
$$W_\omega(B(x,r))\leq \theta_1\,\omega(B(x,r))\qquad \mbox{if $B(x,r)$ is $(16,128^{n+2})$-doubling.}$$
Then
\begin{equation}\label{eqcot99}
\sup_{\ve\geq r_0} |\RR_\ve\omega(x)|\lesssim \cM_\omega^{(r_0,16,128^{n+2})}(\RR\omega)(x) + \theta_1.
\end{equation}
\end{lemma}

\begin{proof}
For a given $\ve\geq r_0$, let $k\geq0$ be minimal so that, for $r=128^k\,\ve$, the ball $B(x,r)$ is $(128,128^{n+2})$-doubling (in particular, this implies that $B(x,8r)$ is $(16,128^{n+2})$-doubling). It is easy to see that such $k$ exists using the assumption $\theta_\omega(x,r)\leq \theta_1$.
By a standard estimate (see Lemma 2.20 from \cite{Tolsa-llibre}), it follows that
$$|\RR_\ve \omega(x)|\leq |\RR_{8r} \omega(x)| + C\,\frac{\omega(B(x,8r))}{(8r)^n} \leq |\RR_{8r} \omega(x)| + C\,\theta_1.$$
It is immediate to check that for any $x'\in B(x,2r)$,
\begin{equation}\label{eqss1}
|\RR_{8r}\omega(x) - \RR\chi_{B(x,4r)^c}\omega(x')|\lesssim \theta_1.
\end{equation}
Consider radial $C^1$ functions $\psi_1$ and $\psi_2$ such that
$$\chi_{B(x,4r)}\leq \psi_1\leq \chi_{B(x,8r)}\qquad \text{and} \qquad
\chi_{B(x,r)}\leq \psi_2\leq \chi_{B(x,2r)},$$
and $\text{Lip}(\psi_1)\le r,\ \text{Lip}(\psi_2)\le r.$ Given a function $f\in L^1_{loc}(\omega)$, denote by $m_{\psi_2\omega}f$ the $(\psi_2\,\omega)$-mean of $f$, i.e.,
$$m_{\psi_2\omega}f = \frac1{\|\psi_2\|_{L^1(\omega)}} \int f\,\psi_2\,d\omega.$$
Notice that
$$\bigl|m_{\psi_2\omega}(\RR\omega)\bigr|\leq \frac1{\omega(B(x,r))} \int_{B(x,2r)}
|\RR\omega|\,d\omega \approx\avint_{B(x,2r)}
|\RR\omega|\,d\omega 
\leq \cM_\omega^{(r_0,16,128^{n+2})}(\RR\omega)(x).$$
From \rf{eqss1} we deduce that
$$\big|\RR_{8r}\omega(x) - m_{\psi_2\omega}\bigl(\RR(\chi_{B(x,4r)^c}\omega)\bigr)\big| \leq m_{\psi_2\omega}\bigl(|\RR_{8r}\omega(x) - \RR(\chi_{B(x,4r)^c}\omega)|\bigr)
\lesssim\theta_1.$$
Then we have
\begin{align*}
|\RR_{8r}\omega(x)|&\lesssim\theta_1 + |m_{\psi_2\omega}(\RR(\chi_{B(x,4r)^c}\omega)|\\
& \lesssim\theta_1 + \bigl|m_{\psi_2\omega}\bigl(\RR(\chi_{B(x,4r)^c}\omega) - \RR((1-\psi_1)\omega)\bigr)\bigr| +\bigl|m_{\psi_2\omega}\bigl(\RR(\psi_1\omega)\bigr)\bigr|+ \bigl|m_{\psi_2\omega}(\RR\omega)\bigr|\\
& \lesssim\theta_1 +  \bigl|m_{\psi_2\omega}\bigl(\RR(\psi_1\omega)\bigr)\bigr|
+ \cM_\omega^{(r_0,16,128^{n+2})} (\RR\omega)(x).
\end{align*}
To estimate the middle term on the right hand side we use the antisymmetry of $\RR$:
\begin{align*}
\bigl|m_{\psi_2}\bigl(\RR(\psi_1\omega)\bigr)\bigr|
& =\frac1{\|\psi_2\|_{L^1(\omega)}} \left|\iint K(y-z)\,\psi_1(y)\,\psi_2(z)\,d\omega(y)\,d\omega(z)\right|\\
& =\frac1{2\|\psi_2\|_{L^1(\omega)}} \left|\iint K(y-z)\,\bigl(\psi_1(y)\,\psi_2(z) - \psi_1(z)\,\psi_2(y)\bigr)\,d\omega(y)\,d\omega(z)\right|\\
& \lesssim \frac1{\omega(B(x,r))} \iint_{B(x,8r)\times B(x,8r)} \!\frac1{r\,|y-z|^{n-1}}\,d\omega(y)\,d\omega(z)
 \approx\frac{W_\omega(B(x,8r))}{\omega(B(x,8r))} \lesssim \theta_1.
\end{align*}
\end{proof}

\vv
%\begin{rem}
%For every $a\geq2$ the lemma is also valid replacing $(128,1128^{n+2})$-doubling balls by $(a,b)$-doubling balls, with $b\geq(8a)^{n+2}$, say, and
%5$\cM_\omega^{(r_0,2,16^{d+1})}$ by $\cM_\omega^{(r_0,a,b)}$.
%\end{rem}

\vv
\begin{rem}\label{rem**}
In fact, the proof of the preceding lemma shows that, given any measure $\omega$, $x\in\R^{n+1}$ and
$\ve>0$,
\begin{equation}\label{eqcotlar*99}
|\RR_\ve\omega(x)|\lesssim \avint_{B(x,2\ve')} |\RR\omega|\,d\omega + \sup_{r\geq \ve} \frac{\omega(B(x,r))}{r^n} + \frac{W_\omega(B(x,8\ve'))}{\omega(B(x,8\ve'))},
\end{equation}
where $\ve'=2^{7k}\ve$, with $k\geq0$ minimal such that the ball $B(x,\ve')$ is $(128,128^{n+2})$-doubling. 
%Of course, a similar estimate holds if we replace the $(16,16^{n+2})$-doubling condition
%by $(16,C)$-doubling with any constant $C>1$, and then the implicit constant in \rf{eqcotlar*99}
%depends on $C$.
%Further, a quick inspection of the proof shows that in \rf{eqcotlar*99} one can replace the term 
%$\ds \sup_{r\geq \ve} \frac{\omega(B(x,r))}{r^n}$ by\footnote{Perhaps this requires more details.}
%$$P_\omega(x,\ve') := \int \frac{\ve'}{\big(|x-y|+\ve'\big)^{n+1}}\,d\omega(y).$$
\end{rem}

\vv

% ********************************************************************************************

%\subsection{The families $\End(e'(R))$ and $\Reg(e'(R))$ and the approximating measure $\eta$}
\subsection{The family \texorpdfstring{$\Reg(e'(R))$}{Reg(e'(R))} and the approximating measure \texorpdfstring{$\eta$}{eta}}\label{sec6.2*}

In the remaining of this section we fix a cube $R\in\MDW$ such that $\TT(e'(R))$ is tractable.
%Recall that the family $\TT(e'(R))$ was constructed by stopping time conditions involving the families $\LD(\,\cdot\,)$ and $\HD(\,\cdot\,)$.
% , and that $\sss(e'(R))$ is the family of maximal $\PP$-doubling cubes which contained both in $e'(R)$ and in some cube from $\LD(R)\cup\HD(R)\cup\NDB$.

We need to define some regularized family of ending cubes for $\TT(e'(R))$. 
First, let 
$$d_R(x) = \inf_{Q\in\TT(e'(R))}\big(\dist(x,Q) + \ell(Q)\big).$$
Notice that $d_R$ is a $1$-Lipschitz function. 
%We denote by $Z(e'(R))$ the subset of the points $x\in e'(R)$ such that $d_R(x)=0$.
Given $0<\ell_0\ll\ell(R)$, we denote
\begin{equation}\label{eql00*23}
d_{R,\ell_0}(x) = \max\big(\ell_0,d_R(x)\big),
\end{equation}
which is also $1$-Lipschitz.
For each $x\in e'(R)$ we take the largest cube $Q_x\in\DD_\mu$ 
such that $x\in Q_x$ with
\begin{equation}\label{eqdefqx}
\ell(Q_x) \leq \frac1{60}\,\inf_{y\in Q_x} d_{R,\ell_0}(y).
\end{equation}
We consider the collection of the different cubes $Q_x$, $x\in e'(R)$, and we denote it by $\Reg(e'(R))$ (this stands for ``regularized cubes''). 
%Also, we let $\qgood$ (this stands for ``quite good'') be
%the family of cubes $Q\in \DD$ such that $Q$ is contained in $B(x_0,2 K r_0)$ and $Q$ is not strictly contained in any
%cube of the family $\Reg$. 
%Observe that $\reg\subset \qgood$.

The constant $\ell_0$ is just an auxiliary parameter that prevents $\ell(Q_x)$ from vanishing.  Eventually $\ell_0$ will be taken extremely small. In particular, we assume $\ell_0$ small enough
so that 
\begin{equation}\label{eql00}
\mu\bigg(\bigcup_{Q\in\HD_1(e(R)):\ell(Q)\geq \ell_0} Q \bigg)\geq \frac12\,\mu\bigg(\bigcup_{Q\in\HD_1(e(R))} Q\bigg).
\end{equation}

We let $\TT_\Reg(e'(R))$ be the family of cubes made up of $R$ and all the cubes of the next generations which are contained in $e'(R)$ but are not 
strictly contained in any cube from $\Reg(e'(R))$.

\vv

\begin{lemma}\label{lem74}
The cubes from $\Reg(e'(R))$ are pairwise disjoint and satisfy the following properties:
\begin{itemize}
\item[(a)] If $P\in\Reg(e'(R))$ and $x\in B(x_{P},50\ell(P))$, then $10\,\ell(P)\leq d_{R,\ell_0}(x) \leq c\,\ell(P)$,
where $c$ is some constant depending only on $n$. 
%In particular, $B(z_{P},50\ell(P))\cap W_0=\varnothing$.

\item[(b)] There exists some absolute constant $c>0$ such that if $P,\,P'\in\Reg(e'(R))$ satisfy $B(x_{P},50\ell(P))\cap B(x_{P'},50\ell(P'))
\neq\varnothing$, then
$$c^{-1}\ell(P)\leq \ell(P')\leq c\,\ell(P).$$
\item[(c)] For each $P\in \Reg(e'(R))$, there are at most $C$ cubes $P'\in\Reg(e'(R))$ such that
$$B(x_{P},50\ell(P))\cap B(x_{P'},50\ell(P'))
\neq\varnothing,$$
 where $C$ is some absolute constant.
 
 %\item[(d)] If $x\not\in B(x_0,\frac1{8} K r_0)$, then $d(x)\approx |x-x_0|$. Thus, if $P\in\reg$ and $B(z_{P},50\ell(P))\not\subset  B(x_0,\frac1{8} K r_0)$, then $\ell(P)\gtrsim  K r_0$.
\end{itemize}
\end{lemma}

The proof of this lemma is standard. See for example \cite[Lemma 6.6]{Tolsa-memo}.

\vvv

Next we define a measure $\eta$ which, in a sense, approximates $\mu\rest_{e'(R)}$ at the level of the family $\Reg(e'(R))$. 
We let
$$\eta = \sum_{P\in\Reg(e'(R))} \mu(P) \frac{\LL^{n+1}\rest_{\tfrac12B(P)}}{\LL^{n+1}(\tfrac12B(P))},
$$
where $\LL^{n+1}$ stands for the Lebesgue measure in $\R^{n+1}$.
With each $Q\in\TT_{\Reg}(e'(R))$ we associate another ``cube'' $Q^{(\eta)}$ defined as follows:
$$Q^{(\eta)}= \bigcup_{P\in\Reg(e'(R)):P\subset Q} \tfrac12B(P).$$
Further, we consider a lattice $\DD_\eta$ associated with the measure $\eta$ which is made up of the cubes
$Q^{(\eta)}$ with $Q\in \TT_{\Reg}(e'(R))$ and other cubes which are descendants of the cubes from $\Reg(e'(R))$.
We assume that $\DD_\eta$ satisfies the first two properties of Lemma \ref{lemcubs} with the same parameters $A_0$ and $C_0$ as $\DD_\mu$. It is straightforward to check that $\DD_\eta$ can be constructed in this way.
For $S\in\DD_\eta$ such that $S=Q^{(\eta)}$ with $Q\in\TT_{\Reg}(e'(R))$, we let $Q=S^{(\mu)}$. Further, we write
$\ell(S):=\ell(Q)$, $B_S:=B_Q$, and $\Theta(S):=\Theta(Q)$.

\vv

%We let\footnote{Perhaps it is appropriate to define the smaller cubes contained in cubes from $\Reg$.}
%$$\DD_\eta = \{s(Q)\}_{Q\in\DD_\mu, Q\subset e'(R)}.$$

% ********************************************************************************************

\subsection{The auxiliary family \texorpdfstring{$\sH$}{H}}\label{subsec:H}

Given $p\geq1$ and a family $I\subset\DD_\mu$, we denote
$$\sigma_p(I) = \sum_{P\in I}\Theta(P)^p\,\mu(P),$$
so that $\sigma(I)=\sigma_2(I)$. 
Recall that  
$$\HD(R) = \hd^{k_{\Lambda}}(R),$$
and that
$$\HD_1(e(R)) = \sss(e(R))\cap \HD(R),\qquad  \HD_1(e'(R)) = \sss(e'(R))\cap \HD(R).
$$
For $R\in\MDW$ and $j\geq0$, denote 
$$\sH_j(e'(R))= \TT_\Reg(e'(R)) \cap \hd^{k_{\Lambda} + j}(R),$$
so that $\sH_0(e'(R))=\HD_1(e'(R))\cap\TT_\Reg(e'(R))$. 
\brem\label{rem:Hjempty}
Note that for $j> k_{\Lambda}+2$ we have $\sH_j(e'(R))=\varnothing$. Indeed, by the definition of $\Reg(e'(R))$ and $\TT_\Reg(e'(R))$, for each $Q\in \TT_\Reg(e'(R))$ there exists $P\in \TT(e'(R))$ such that $2B_Q\subset 2B_P$ and $\ell(Q)\le \ell(P)\le A_0^2\ell(Q)$. Thus,
\begin{equation*}
\frac{\mu(2B_Q)}{\ell(Q)^n}\le \frac{\mu(2B_P)}{\ell(Q)^n}\le A_0^{2n}\frac{\mu(2B_P)}{\ell(P)^n}.
\end{equation*}
Since for each $P\in \TT(e'(R))$ we have $\Theta(P)\le \Lambda^2\Theta(R)$, it follows that $\Theta(Q)\le A_0^{2n}\Lambda^2\Theta(R)$.
\erem

%In the case $R\in\Red$,  we set 
%$$\sH_j(e'(R))= \TT_0(e'(R)) \cap \hd^{k_\Lambda-k_1 + j}(R).$$
%Notice that it would have been more appropriate to use a notation different from $\sH_j(e'(R))$, such as $\sH_{0,j}(e'(R))$. The reason for our abuse of notation is the wish simplify the exposition below.

The fact that $\max_{j\geq0} \sigma_p(\sH_j(e'(R)))$ may be much larger than $\sigma_p(\HD_1(e'(R))$ may cause problems in some estimates. For this reason,
we need to introduce an auxiliary family $\sH$. We deal with this issue in this section.

Recall that, for $R\in\DD_{\mu,k}$,
$$e(R) = e_{h(R)}(R) \quad \mbox{ and }\quad e'(R) = e_{h(R)+1},$$
where
$$e_i(R) = R \cup \bigcup Q,$$
with the union running over the cubes $Q\in\DD_{\mu,k+1}$ such that
$$
\dist(x_R,Q)< \frac{\ell(R)}2 + 2i\ell(Q).
$$
For $j\geq0$, we set 
$$e_{i,j}(R) = \bigcup_{Q\in\DD_{\mu,k+1}:Q\subset e_i(R)} e_j(Q),$$
and we let $\sH_k(e_{i,j}(R))$ be the subfamily of the cubes from $\sH_k(e'(R))$ which are contained
in $e_{i,j}(R)$.

From now on, we let $\ve_n$ be some positive constant depending just on $n$. 
The precise value of $\ve_n$ is chosen at the end of Section \ref{sectrans} (see \rf{eqchooseen}).
%In the present paper, later on, we will simply
%take $\ve_n=1/15$. However, for another application of the results from this section in \cite{Tolsa-riesz} it
%is convenient to allow $\ve_n$ to depend on $n$.

\vv
\begin{lemma}\label{lem:66}
Let $p\in (1,2]$.
For any $R\in\MDW$ there exist some $j,k$, with $10\leq j\leq A_0/4$ and  $0\leq k\leq k_{\Lambda}+2$ such that
\begin{equation}\label{eqsigmah}
\sigma_p(\sH_{m}(e_{h(R),j+1}(R))) \leq \Lambda^{\ve_n} \,\sigma_p(\sH_{k}(e_{h(R),j}(R)))
\quad \mbox{for all $m\geq0$,}
\end{equation}
 assuming $A_0$ big enough (possibly depending on $n$). 
\end{lemma}

\begin{proof}
For each $j$, we denote by $0\le k_j\le k_{\Lambda}+2$ the integer such that
$$\sigma_p(\sH_{k_j}(e_{h(R),j}(R))) = \max_{0\leq k\leq k_{\Lambda}+2} \sigma_p(\sH_{k}(e_{h(R),j}(R))).$$
The lemma can be rephrased in the following way: there exists $10\leq j\leq A_0/4$ such that
\begin{equation*}
\sigma_p(\sH_{k_{j+1}}(e_{h(R),j+1}(R))) \leq \Lambda^{\ve_n} \,\sigma_p(\sH_{k_j}(e_{h(R),j}(R))).
\end{equation*}
We prove this by contradiction. Suppose the estimate above fails for all $10\leq j\leq A_0/4$, and let $j_0$ be the largest integer smaller than $A_0/4$.
Then we have
\begin{multline}\label{eq:651}
\sigma_p(\sH_{k_{j_0}}(e_{h(R),j_0}(R)))
  \geq \Lambda^{\ve_n}\,\sigma_p(\sH_{k_{j_0-1}}(e_{h(R),j_0-1}(R)))\\
\geq
\ldots\geq \Lambda^{\ve_n(j_0-10)}\sigma_p(\sH_{k_{10}}(e_{h(R),10}(R))) \geq \Lambda^{\ve_n(j_0-10)}\sigma_p(\sH_{0}(e_{h(R),10}(R)))\\
\geq  \Lambda^{\ve_n(j_0-10)}\sigma_p(\sH_{0}(e_{h(R)}(R))) \overset{\rf{eql00}}{\geq} \frac12\,\Lambda^{\ve_n(j_0-10)}
\sigma_p(\HD_1(e(R))),
\end{multline}

Concerning the left hand side of the inequality above, since $e_{h(R),j_0}(R)\subset 2R$ and $k_{j_0}\le k_{\Lambda}+2$, we have
$$
\sigma_p(\sH_{k_{j_0}}(e_{h(R),j_0}(R)))\leq A_0^{2np}
\Lambda^{2p}\,\Theta(R)^p\mu(2R).$$
Due to the fact that $R$ is $\PP$-doubling, as in \rf{eqdoub*11} we have
$\mu(2R)\leq C_0\,C_d\,A_0^{n+1}
\mu(R).$ Thus,
\begin{equation}\label{eq:652}
\sigma_p(\sH_{k_{j_0}}(e_{h(R),j_0}(R)))\leq C_0\,C_d\,A_0^{2np+n+1}
\Lambda^{2p}\sigma_p(R).
\end{equation}

Concerning the right hand side of \eqref{eq:651}, observe that, denoting $\Theta(\HD_1)=\Lambda\Theta(R)=\Theta(Q)$ for any $Q\in\HD_1(e(R))$, we have
$$\sigma_p(\HD_1(e(R))) = \Theta(\HD_1)^{p-2}\sigma(\HD_1(e(R)))\overset{\eqref{eq:MDWdef2}}{\ge} B^{-1}\Theta(\HD_1)^{p-2}\sigma(R) \geq \Lambda^{-1}\Lambda^{p-2}\sigma_p(R).
$$
We deduce from \eqref{eq:651}, \eqref{eq:652}, and the above, that
$$\frac12\,\Lambda^{\ve_n(j_0-10)}\Lambda^{p-3}
\sigma_p(R)\leq C_0\,C_d\,A_0^{2np+n+1}
\Lambda^{2p}\sigma_p(R).$$
Since $\Lambda\geq A_0^n$ and $j_0\approx A_0$, it is clear that this inequality is violated if $A_0$ is big 
enough, depending just on $n$.
\end{proof}
\vv

% ********************************************************************************************

Let $j(R),\,k(R)$ be such that $10\leq j(R)\leq A_0/4$,  $0\leq k(R)\leq k_{\Lambda}+2$ and
$$
\sigma_p(\sH_{m}(e_{h(R),j(R)+1}(R))) \leq \Lambda^{\ve_n} \,\sigma_p(\sH_{k(R)}(e_{h(R),j(R)}(R)))
\quad \mbox{for all $m\geq0$,}
$$
We denote
$$\cS_\mu = \bigcup_{Q\in\Reg:Q\subset e_{h(R),j(R)}(R)} Q,\qquad
\cS_\eta = \bigcup_{Q\in\Reg:Q\subset e_{h(R),j(R)}(R)} \tfrac12B(Q)$$
and
$$\cS_\mu' = \bigcup_{Q\in\Reg:Q\subset e_{h(R),j(R)+1}(R)} Q,\qquad
\cS_\eta' = \bigcup_{Q\in\Reg:Q\subset e_{h(R),j(R)+1}(R)} \tfrac12B(Q).$$
Notice that, by construction,
\begin{equation}\label{eqtec732}
\dist(\supp\mu \setminus e'(R),\cS_\mu')\geq cA_0^{-1} \ell(R),
\end{equation}
where $c>0$ is an absolute constant.

For $m=1,2,3,4$, denote by $V_m$ the $m A_0^{-3}\ell(R)$-neighborhood of $\cS_\eta$.
Let $\vphi_R$ be a $C^1$ function which equals $1$ in $V_3$, vanishes out of $V_4$, and such that $\|\vphi_R\|_\infty\leq 1$ and
$\|\nabla\vphi_R\|_\infty\leq 2A_0^3 \ell(R)^{-1}$.
Observe that, for $x\in \supp\mu\setminus e'(R)$, $\dist(x,V_4)\gtrsim \ell(R)$.
In fact, from \rf{eqtec732} one can derive that
\begin{equation}\label{eqtec733}
\dist(Q,\supp\mu\setminus e'(R)) \gtrsim \ell(R)\quad\mbox{ for all $Q\in\Reg(e'(R))$ such that
$B_Q\cap V_4\neq\varnothing$,}
\end{equation}
taking into account that $\ell(Q)\leq \frac{A_0^{-1}}{60}\,\ell(R)$ for every $Q\in\Reg(e'(R))$.

 We consider the measure
$$\nu = \vphi_R\,\eta$$
and the function
$$G(x) = 2A_0^3\int_{\cS_\eta'\setminus V_2} 
 \frac1{\ell(R)\,|x-y|^{n-1}}\,d\eta(y).$$
 Notice that
$$G(x)\lesssim \Theta(R)\quad\mbox{ for all $x\in V_1$.}$$

To shorten notation, we write
$$\sH= \sH_{k(R)}(e_{h(R),j(R)}(R)),\qquad \sH'= \sH_{k(R)}(e_{h(R),j(R)+1}(R))$$
and
$$H= \bigcup_{Q\in\sH} Q^{(\eta)},\qquad  H'= \bigcup_{Q\in\sH'}Q^{(\eta)}.$$
We also set
$$\Theta(\sH) = A_0^{(k_{\Lambda}+k(R))n}\,\Theta(R),$$
so that for any $Q\in\sH$ we have
$\Theta(Q) = \Theta(\sH).$

\subsection{Some technical lemmas}\label{subsec9.5}

\vv
\begin{lemma}\label{lem*921}
Let $A\subset \R^{n+1}$ be the set of those $x\in\R^{n+1}$ which belong to some 
$(16,128^{n+2})$-doubling (with respect to $\nu$) ball $B\subset\R^{n+1}$ such that
$$W_\nu(B)\geq M\,\Theta(\sH)\,\nu(B)\quad \mbox{ and } \quad
\theta_\eta( \gamma B)\leq c_1\,\Theta(\sH) \quad\mbox{for all $\gamma\geq1$.}$$
Then, for $M\geq1$ big enough, 
$$\nu(A) \lesssim \frac{\Lambda^{\ve_n}}{M} \,\nu(H).$$
\end{lemma}

\begin{proof}
Observe first that, for any ball $B\subset\R^{n+1}$ with $r(B)\in[A_0^{-k-1},A_0^{-k}]$,
\begin{align*}
W_\nu(B) &\lesssim
%\theta_\nu(B)\,\nu(B) + \sum_{j\geq k+1} \iint_{\begin{subarray}\,(x,y)\in B\times B\\
%A_0^{-j-2}<|x-y|\leq A_0^{-j-1} \end{subarray}}
%\frac1{r(B)\,|x-y|^{n-1}}\, d\nu(x)\,d\nu(y)\\
%& \lesssim
\theta_\nu(B)\,\nu(B) \\
&\quad+ \sum_{j\geq k+1} \sum_{Q\in\DD_{\eta,j}:Q\subset 2B}\int_{x\in Q}\int_{y:
A_0^{-j-2}<|x-y|\leq A_0^{-j-1}}
\frac1{r(B)\,|x-y|^{n-1}}\, d\nu(x)\,d\nu(y)\\
& \lesssim \sum_{Q\in\DD_{\eta}:Q\subset 2B} \frac{\ell(Q)}{r(B)}\,\theta_\nu(2B_Q)\,\nu(Q).
\end{align*}

To prove the lemma, we apply Vitali's $5r$-covering lemma to get a family of $(16,128^{n+2})$-doubling balls $B_i$, $i\in I$,
which satisfy the following:
\begin{itemize}
\item the balls $2B_i$, $i\in I$, are pairwise disjoint,
\item $A\subset \bigcup_{i\in I} 10B_i$,
\item 
$W_\nu(B_i)\geq M\,\Theta(\sH)\,\nu(B_i)$ and $\theta_\eta(\gamma B_i)\leq c_1\,\Theta(\sH)$ and  for all $i\in I$ and $\gamma\geq1$.
\end{itemize}
Then we deduce
\begin{align}\label{eqplug6}
\nu(A) &\leq \sum_{i\in I} \nu(10B_i) \lesssim \sum_{i\in I} \nu(B_i)
\leq \frac1{M\,\Theta(\sH)} \sum_{i\in I} W_\nu(B_i)\\
& \lesssim \frac1{M\,\Theta(\sH)} \sum_{i\in I}\sum_{Q\in\DD_{\eta}:Q\subset 2B_i} \frac{\ell(Q)}{r(B_i)}\,\theta_\nu(2B_Q)\,\nu(Q).\nonumber
\end{align}
Now we take into account that all the cubes $Q$ which are not contained in any cube $P$ with $P^{(\mu)}\in\sH'$
satisfy $\theta_\nu(2B_Q)\leq \theta_\eta(2B_Q)\lesssim \Theta(\sH)$. Then, for each $i\in I$,
\begin{align*}
\sum_{Q\in\DD_{\eta}:Q\subset 2B_i} \frac{\ell(Q)}{r(B_i)}\,\theta_\nu(2B_Q)\,\nu(Q) & \leq
\sum_{P^{(\mu)}\in \sH'}\,
\sum_{Q\in\DD_{\eta}:Q\subset 2B_i\cap P} \frac{\ell(Q)}{r(B_i)}\,\theta_\eta(2B_Q)\,\eta(Q)\\
&\quad + \Theta(\sH)\sum_{Q\in\DD_{\eta}:Q\subset 2B_i} \frac{\ell(Q)}{r(B_i)}\, \nu(Q).\\
\end{align*}
Observe that the last term on the right hand side is bounded above by $\Theta(\sH)\nu(2B_i)\approx
\Theta(\sH)\nu(B_i)$. 
So plugging the previous estimate into \rf{eqplug6}, we get
\begin{align*}
\nu(A) \lesssim \frac{1}{M\,\Theta(\sH)} \sum_{i\in I}
\sum_{P^{(\mu)}\in \sH'}\,
\sum_{Q\in\DD_{\eta}:Q\subset 2B_i\cap P} \frac{\ell(Q)}{r(B_i)}\,\theta_\eta(2B_Q)\,\eta(Q)
+ \frac{1}{M} \sum_{i\in I} \nu(B_i).
\end{align*}
Since $B_i\subset A$ for each $i$, the last term is at most $\frac12\nu(A)$ for $M$ big enough. Thus,
\begin{equation}\label{eqnua12}
\nu(A) \lesssim  \frac{1}{M\,\Theta(\sH)} \sum_{i\in I}
\sum_{P^{(\mu)}\in \sH'}\,
\sum_{Q\in\DD_{\eta}:Q\subset 2B_i\cap P} \frac{\ell(Q)}{r(B_i)}\,\theta_\eta(2B_Q)\,\eta(Q).
\end{equation}

To estimate the term on the right hand side above, we denote by $\sF_k$ the family of cubes from
$\DD_\eta$ which are contained in some cube $Q^{(\eta)}$ with $Q\in\sH_{k}':=\sH_{k}(e_{h(R),j(R)+1}(R))$ and are not contained
in any cube $P^{(\eta)}$ with $P\in\sH_{k+1}':=\sH_{k+1}(e_{h(R),j(R)+1}(R))$. Notice that
$$\theta_\eta(2B_Q)\lesssim \Theta(\sH_{k+1})\approx \Theta(\sH_{k}),$$
where $\Theta(\sH_{k})=\Theta(Q')$ for $Q'\in\sH_{k}'$ (this does not depend on the specific cube $Q'$). Then, for each $i\in I$, we have
\begin{align}\label{eqalj4}
\sum_{P^{(\mu)}\in \sH'}
\sum_{Q\in\DD_{\eta}:Q\subset 2B_i\cap P} \frac{\ell(Q)}{r(B_i)}\,\theta_\eta(2B_Q)\eta(Q) &=\!\!
\sum_{k\geq k(R)}\sum_{P^{(\mu)}\in \sH'_k}
\sum_{Q\in\sF_k:Q\subset 2B_i\cap P} \frac{\ell(Q)}{r(B_i)}\,\theta_\eta(2B_Q)\eta(Q)\\
& \lesssim 
\sum_{k\geq k(R)}\sum_{P^{(\mu)}\in \sH'_k} \Theta(\sH_k)
\sum_{Q\in\sF_k:Q\subset 2B_i\cap P} \frac{\ell(Q)}{r(B_i)}\,\eta(Q).\nonumber
\end{align}
We claim now that for $Q$ in the last sum, we have
\begin{equation}\label{eqcond*492}
\ell(Q)\lesssim A_0^{-(k-k(R))}\,r(B_i).
\end{equation}
To check this, let $P(Q,k)\in \sH_k'$ be such that $Q\subset P(Q,k)$. From the condition 
\begin{equation}\label{eqcond492}
\theta_\eta( \gamma B_i)\leq c_1\,\Theta(\sH) \quad\mbox{for all $\gamma\geq1$}
\end{equation}
and the fact that $P(Q,k)\cap 2B_i\neq\varnothing$ (because $Q\subset 2B_i$) we infer that 
$r(B_i)\geq \ell(P(Q,k))$ for $k$ big enough. Otherwise we would find a ball $\gamma B_i$ containing $P(Q,k)$ with radius comparable to $\ell(P(Q,k))$, so that
$$\theta_\eta(\gamma B_i) \geq c\,\Theta(P(Q,k)) >c_1\,\Theta(\sH)$$
for $k$ big enough (depending on $c_1$), contradicting \rf{eqcond492}. So we have $P(Q,k)\subset 6B_i$ and thus
$$c_1\Theta(\sH)\geq \theta_\eta(6B_i)\gtrsim \frac{\ell(P(Q,k))^n}{r(B_i)^n}\,\Theta(P(Q,k))= 
\frac{\ell(P(Q,k))^n}{r(B_i)^n}\,A_0^{n(k-k(R))}\,\Theta(\sH).$$
This gives 
$$\ell(Q)\leq \ell(P(Q,k))\lesssim A_0^{-(k-k(R))}\,r(B_i)$$
and proves \rf{eqcond*492} for $k$ big enough, and thus for all $k$ if we choose the implicit constant in \rf{eqcond*492} suitably.

From \rf{eqcond*492} we deduce that the right hand side of \rf{eqalj4}
does not exceed
\begin{multline*}
\sum_{k\geq k(R)}A_0^{-(k-k(R))/2} \Theta(\sH_k)\sum_{P^{(\mu)}\in \sH'_k} 
\sum_{Q\subset 2B_i\cap P} \left(\frac{\ell(Q)}{r(B_i)}\right)^{1/2}\,\eta(Q)\\
\lesssim \sum_{k\geq k(R)}A_0^{-(k-k(R))/2} \,\Theta(\sH_k)\sum_{P^{(\mu)}\in \sH'_k} 
\eta(2B_i\cap P).
\end{multline*}
Plugging this estimate into \rf{eqnua12} and recalling that the balls $2B_i$ are disjoint, we get
\begin{align*}
\nu(A) & \lesssim \frac{1}{M\,\Theta(\sH)} \sum_{i\in I}
\sum_{k\geq k(R)}A_0^{-(k-k(R))/2} \sum_{P^{(\mu)}\in \sH'_k} \Theta(\sH_k)\, \eta(2B_i\cap P)\\
& \leq
\frac{1}{M\,\Theta(\sH)}
\sum_{k\geq k(R)}A_0^{-(k-k(R))/2} \sum_{P^{(\mu)}\in \sH'_k} \Theta(\sH_k) \,\eta(P).
\end{align*}
By H\"older's inequality, for each $k\geq k(R)$ we have
$$\sum_{P^{(\mu)}\in \sH'_k} \Theta(\sH_k) \eta(P) \leq \sigma_p(\sH'_k)^{1/p}\,\eta(H')^{1/p'}.
$$
We can estimate the right hand side using \rf{eqsigmah}:
\begin{align*}
\sigma_p(\sH'_k)&\le \Lambda^{\ve_n}\,\sigma_p(\sH),\\ \eta(H')&=\frac{\sigma_p(\sH')}{\Theta(\sH)^p} \le \Lambda^{\ve_n}\,\frac{\sigma_p(\sH)}{\Theta(\sH)^p} = \Lambda^{\ve_n}\,\eta(H).
\end{align*}
Thus,
$$\nu(A) \lesssim \frac{\Lambda^{\ve_n}}{M\,\Theta(\sH)} \sum_{k\geq k(R)}A_0^{-(k-k(R))/2} \,\sigma_p(\sH)^{1/p}\,\eta(H)^{1/p'}
\approx
\frac{\Lambda^{\ve_n}}{M}\,\eta(H)\approx \frac{\Lambda^{\ve_n}}{M}\,\nu(H).$$
\end{proof}
\vv

Let $\TT_{\sH'}$ denote the family of all cubes from $\DD_\mu$ 
 made up of $R$ and all the cubes of the next generations which are contained in $e_{h(R),j(R)+1}(R)$ but are not 
strictly contained in any cube from from $\sH'$. We consider the function
\begin{equation}\label{eqdefPsi1}
\Psi(x) = \inf_{Q\in \TT_{\sH'}} \big(\ell(Q) +\dist(x,Q)\big).
\end{equation}
Notice that $\Psi$ is a $1$-Lipschitz function. We also have the following result, which will be proven
by quite standard arguments.

\begin{lemma}\label{lem6.77}
For all $x\in\R^{n+1}$, we have
\begin{equation}\label{ecreixnupsi}
\sup_{r\geq \Psi(x)} \frac{\nu(B(x,r))}{r^n} \leq \sup_{r\geq \Psi(x)} \frac{\eta(B(x,r))}{r^n} \lesssim \Theta(\sH)\quad\mbox{ for all $x\in \R^{n+1}$.}
\end{equation}
\end{lemma}

\begin{proof}
The first inequality in \rf{ecreixnupsi} is trivial. Concerning the second one,
in the case $r>\ell(R)/10$ we just use that
$$\eta(B(x,r))\leq \mu(e'(R))\lesssim\mu(R)\lesssim \Theta(R)\,\ell(R)^n\lesssim \Theta(\sH)\,r^n.$$
So we may assume that $\Psi(x)<r\leq \ell(R)/10$.
By the definition of $\Psi(x)$, there exists $Q\in\TT_{\sH'}$ such that
$$\ell(Q) + \dist(x,Q)\leq r.$$
Therefore, $B_Q\subset B(x,4r)$ and so there exists an ancestor $Q'\supset Q$ which belongs to $\TT_{\sH'}$ such that $B(x,r)\subset  2B_{Q'}$, with $\ell(Q')\approx r$. 
Then, $\Theta(Q')\leq \Theta(\sH)$ and
$$\eta(B(x,r))\leq \sum_{P\in\Reg(e'(R)):P^{(\eta)}\cap 2B_{Q'}\neq\varnothing} \eta(P^{(\eta)}) = \sum_{P^{(\eta)}\in\Reg(e'(R)):P\cap 2B_{Q'}\neq\varnothing} \mu(P).
$$
Observe now that if $P\in\Reg(e'(R))$ satisfies $P^{(\eta)}\cap 2B_{Q'}\neq\varnothing$, then $\ell(P)\lesssim \ell(Q')$ (this is an easy consequence of Lemma \ref{lem74}(b) and the fact that $Q'$ contains some cube from $\Reg(e'(R))$). So we deduce that
$P\subset CB_{Q'}$. Hence,
$$\eta(B(x,r))\leq \mu(CB_{Q'}) \lesssim \Theta(\sH)\,\ell(Q')^n \approx \Theta(\sH)\,r^n.$$
\end{proof}

\vv

In the next subsection we are going to use a variational argument to show that for some constant $c_0$, depending at most\footnote{The constant $c_0$
	will be chosen of the form $c_0=A_0^{C(n)}$, and it will not depend on $\Lambda$, $\ve_n$, or other
	parameters of the construction.} on $n$ and $A_0$, we have
$$\int \big|(|\RR\nu(x)| - G(x)- c_0\,\Theta(\sH))_+\big|^p\,d\nu(x) \gtrsim \Lambda^{-p'\ve_n}\sigma_p(\sH).$$
See Lemma \ref{lemvar} for details. We prove this estimate by contradiction, so that in particular we will assume that
\begin{equation}\label{eqassu8}
\int \big|(|\RR\nu(x)| - G(x)- c_0\,\Theta(\sH))_+\big|^p\,d\nu(x) \leq \sigma_p(\sH).
\end{equation}
Below we prove a few consequences of \eqref{eqassu8} that will be useful in the proof of Lemma \ref{lemvar}.

We denote
$$\RR_{*,\Psi}\nu(x) = \sup_{\ve>\Psi(x)} \big|\RR_\ve\nu(x)\big|.$$

\begin{lemma}\label{lemrieszpetit}
Suppose that \eqref{eqassu8} holds. Then,
\begin{equation}\label{eqassu9}
\nu\big(\big\{x\in \cS_\eta:\,\RR_{*,\Psi}\nu(x) > M\,\Theta(\sH)\big\}\big) \leq C_2\, 
\frac{\Lambda^{\ve_n}}{M} \,\nu(H).
\end{equation}
\end{lemma}

\begin{proof}
Recall that we denote by $V_1$ the $A_0^{-3}\ell(R)$-neighborhood of $\cS_\eta$
and that
$$G(x)\lesssim \Theta(R)\quad\mbox{ for all $x\in V_1$.}$$ 
%In fact, the last estimate also holds in the $\frac34c_4\ell(R)$-neighborhood of 
%$e_{h(R),j(R)}(R)$, which we will denote by $\wt V_1$.
Then the assumption in the lemma implies that
\begin{equation}\label{eqsig84}
\int_{V_1} \big|(\RR\nu(x)-2c_0\,\Theta(\sH))_+|^p\,d\nu(x) \lesssim \sigma_p(\sH).
\end{equation}

By Remark \ref{rem**}, for any $x\in \cS_\eta$ and $\ve>\Psi(x)$,
$$|\RR_\ve\nu(x)|\lesssim \avint_{B(x,2\ve')} |\RR\nu|\,d\nu + \sup_{r\geq \ve} \frac{\nu(B(x,r))}{r^n} + \frac{W_\nu(B(x,8\ve'))}{\nu(B(x,8\ve'))},$$
where $\ve'=2^{7k}\ve$, with $k\geq0$ minimal such that the ball $B(x,\ve')$ is $(128,128^{n+2})$-doubling. 
In the case $\ve'\geq \frac12 A_0^{-3}\,\ell(R)$, by standard arguments,
$$
|\RR_\ve \nu(x)|\leq |\RR_{\ve'} \nu(x)| + C\,\frac{\nu(B(x,\ve'))}{(\ve')^n} \leq C\,
\Theta(R)< M\,\Theta(\sH),$$
for $\Lambda$ (or $M$) big enough.

In the case $\ve'< \frac12 A_0^{-3}\ell(R)$, we have $B(x,2\ve')\subset V_1$ and thus
$$|\RR_\ve\nu(x)|\lesssim \cM_\nu^{(\Psi(x),16,128^{n+2})}(\chi_{V_1}\RR\nu)(x) 
+ \sup_{r\geq \Psi(x)} \frac{\nu(B(x,r))}{r^n} + \sup_{r\in D(x):r\geq \Psi(x)} \\
 \frac{W_\nu(B(x,r))}{\nu(B(x,r))},
$$
where $D(x)$ denotes the set of radii $r>0$ such that $B(x,r)$ is $(16,128^{n+2})$-doubling.
Also, as shown in \rf{ecreixnupsi},
$$\sup_{r\geq \Psi(x)} \frac{\nu(B(x,r))}{r^n} \lesssim \Theta(\sH).$$

We deduce that in any case (i.e., for any $\ve'$), assuming $M$ larger than some absolute constant,
\begin{align*}
\nu\big(\big\{x\in \cS_\eta:\,\RR_{*,\Psi}\nu(x) &> M\,\Theta(\sH)\big\}\big) \\& \leq
\nu\Big(\Big\{x\in \cS_\eta:\,\cM_\nu^{(\Psi(x),16,128^{n+2})}(\chi_{V_1}\RR\nu)(x) > \frac{M\,\Theta(\sH)}3\Big\}\Big)\\
& \!\!+
\nu\Big(\Big\{x\in \cS_\eta:\!\!\!\sup_{r\in D(x):r\geq \Psi(x)} \!
 \frac{W_\nu(B(x,r))}{\nu(B(x,r))} > \frac{M\,\Theta(\sH)}3\Big\}\Big)\\
 & =: T_1 + T_2.
\end{align*}

To deal with the term $T_1$, notice that if $\cM_\nu^{(\Psi(x),16,128^{n+2})}(\chi_{V_1}\RR\nu)(x) > \frac{M}{3}\Theta(\sH)$, then
\begin{equation*}
\cM_\nu^{(\Psi(x),16,128^{n+2})}((|\chi_{V_1}\RR\nu|-\frac{M}{6}\Theta(\sH))_+)(x) > \frac{M}6 \Theta(\sH),
\end{equation*}
and thus using the weak $L^p(\nu)$ boundedness of $\cM_\nu^{(\Psi(x),16,128^{n+2})}$ and \rf{eqsig84} we obtain, assuming $M$ big enough,
\begin{multline*}
T_1 \le \nu\Big(\Big\{x\in \cS_\eta:\,\cM_\nu^{(\Psi(x),16,128^{n+2})}((|\chi_{V_1}\RR\nu|-\frac{M}{6}\Theta(\sH))_+)(x) > \frac{M}6 \Theta(\sH)\Big\}\Big) \\
\lesssim \frac1{(M\Theta(\sH))^p} \int_{V_1}\big|(\RR\nu(x)-\frac{M}{6}\,\Theta(\sH))_+|^p\,d\nu
\lesssim \frac1{M^p}\,\nu(H).
\end{multline*}

Concerning $T_2$, by Lemma \ref{lem*921},
$$T_2 \lesssim \frac{\Lambda^{\ve_n}}{M} \,\nu(H).$$
So we have
$$\nu\big(\big\{x\in \cS_\eta:\,\RR_{*,\Psi}\nu(x) > M\,\Theta(\sH)\big\}\big)\lesssim 
\frac1{M^p}\,\nu(H) + \frac{\Lambda^{\ve_n}}{M} \,\nu(H)\lesssim \frac{\Lambda^{\ve_n}}{M} \,\nu(H).$$
\end{proof}
\vv

For $\gamma>1$, we let
$$Z(\gamma) = \big\{x\in \cS_\eta': \RR_{*,\Psi}\nu(x) > \gamma \Lambda^{\ve_n}\,\Theta(\sH)\big\}.$$
Notice that, under the assumption  \rf{eqassu8}, by Lemma \ref{lemrieszpetit}
 (with $M= \gamma \Lambda^{\ve_n}$), we have
\begin{equation}\label{eqzgam1}
\nu(Z(\gamma)\cap \cS_\eta) \leq C_2\gamma^{-1}\,\nu(H).
\end{equation}
For $x\in Z(\gamma)$, let
\begin{equation}\label{eqe3194}
e(x,\gamma) = \sup\{\ve: \ve>\Psi(x),\,|\RR_\ve\nu(x)|>\gamma \Lambda^{\ve_n}\,\Theta(\sH)\}
\end{equation}
%For $x\in S_\eta\setminus Z_0$, we set $e(x)=0$.
We define the exceptional set $Z'(\gamma)$ as
$$Z'(\gamma) = \bigcup_{x\in Z(\gamma)} B(x,e(x,\gamma)).$$

The next lemma shows that $\nu(Z'(\gamma)\cap \cS_\eta)$ is small if $\gamma$ is taken big enough.

\vv
\begin{lemma} \label{tamt}
If $y\in Z'(\gamma)$, then $\RR_{*,\Psi}\nu(y) > (\gamma \Lambda^{\ve_n}- c_2)\Theta(\sH)$,
for some $c_2>0$.
Thus, under the assumption  \rf{eqassu8},
if $\gamma\geq 2c_2$, then
\begin{equation} \label{wert1}
\nu(Z'(\gamma)\cap \cS_\eta) \leq \frac{2C_2}{\gamma}\,\nu(H).
\end{equation}
\end{lemma}

\begin{proof}
The arguments are similar to the ones in Lemma 5.2 from
\cite{Tolsa-llibre}. However, for the reader's convenience we will explain here the basic details.

By standard arguments (which just use the fact that the Riesz kernel is a Calder\'on-Zygmund kernel),
for all $y\in B(x,e(x,\gamma))$, with $x\in Z(\gamma)$,
 we have
$$|\RR_{e(x,\gamma)}\nu(x) -\RR_{2e(x,\gamma)}\nu(y)|\lesssim \sup_{r\geq e(x,\gamma)} \frac{\nu(B(x,r))}{r^n} \lesssim \Theta(\sH),$$
taking into account that $e(x,\gamma)\geq\Psi(x)$ and recalling \rf{ecreixnupsi} for the last estimate.
So we have
$$|\RR_{2e(x,\gamma)}\nu(y)| \geq |\RR_{e(x,\gamma)}\nu(x)|- c_2\,\Theta(\sH) \geq  \gamma\,\Lambda^{\ve_n}\,\Theta(\sH) -  c_2\,\Theta(\sH).$$
Observe now that
$$\Psi(y)\leq \Psi(x) + |x-y|< 2e(x,\gamma).$$ 
So 
$$\RR_{*,\Psi}\nu(y) \geq |\RR_{2e(x,\gamma)}\nu(y)|> (\gamma \Lambda^{\ve_n}- c_2)\Theta(\sH),$$
which proves the first statement of the lemma.

In particular, taking $\gamma\geq 2c_2$, from the last estimate we derive 
$$\RR_{*,\Psi}\nu(y) \geq \frac\gamma2\,\Lambda^{\ve_n}\,\Theta(\sH),$$
and so $y\in Z(\gamma/2)$, which shows that $Z'(\gamma)\subset Z(\gamma/2)$. Thus,
by \rf{eqzgam1},
$$\nu(Z'(\gamma)\cap \cS_\eta) \leq \nu(Z(\gamma/2)\cap \cS_\eta) 
\leq 2C_2\gamma^{-1}\,\nu(H).$$
\end{proof}
\vv

We choose $\gamma=\max(1,2c_2,4C_2)$, so that, under the assumption  \rf{eqassu8},
\begin{equation}\label{eqnugam*}
\nu(Z'(\gamma)\cap \cS_\eta) \leq \frac12\,\nu(H).
\end{equation}
Also we define
$$\Phi(x) = \max(\Psi(x),\dist(x,\R^{n+1}\setminus Z'(\gamma)).$$
Notice that $\Phi$ is a $1$-Lipschitz function which coincides with $\Psi(x)$ away from $Z'(\gamma)$ and
that 
\begin{equation}\label{eqphigam}
\Phi(x)\geq e(x,\gamma)\quad \mbox{ for all $x\in Z(\gamma)$.}
\end{equation}

\vv

\begin{lemma}\label{lemrieszphi4}
The suppressed operator $\RR_\Phi$ is bounded in $L^p(\nu)$, with
$\|\RR_\Phi\|_{L^p(\nu)\to L^p(\nu)}\lesssim \Lambda^{\ve_n}\Theta(\sH).$
\end{lemma}

\begin{proof}
This is an immediate consequence of Theorem \ref{teontv} and the construction of $\Phi$ above.
Indeed, by \rf{ecreixnupsi},
$$\nu(B(x,r)) \lesssim \Theta(\sH)\,r^n \quad\mbox{ for all $r\geq\Phi(x)$.}$$
Also, by \rf{eqphigam},
$$\sup_{\ve>\Phi(x)}|\RR_\ve\nu(x)|\leq \sup_{\ve>e(x,\gamma)}|\RR_\ve\nu(x)|
\leq \gamma \Lambda^{\ve_n}\,\Theta(\sH)
\quad \mbox{ for all $x\in Z(\gamma)$.}$$
On the other hand, by the definition of $Z(\gamma)$ we also have
$$\sup_{\ve>\Phi(x)}|\RR_\ve\nu(x)|\leq \sup_{\ve>\Psi(x)}|\RR_\ve\nu(x)|\leq
\gamma \Lambda^{\ve_n}\,\Theta(\sH) \quad \mbox{ for $x\in \cS_\eta'\setminus Z(\gamma)$.} 
$$
So we can apply Theorem \ref{teontv}, taking 
$\omega= (C\Lambda^{\ve_n}\Theta(\sH))^{-1}\,\nu$ with an appropriate absolute constant $C$, and
then the lemma follows.
\end{proof}

\vv

\begin{lemma}\label{lemH0}
Under the assumption \rf{eqassu8},
there exists a subset $H_0\subset H$ such that:
\begin{itemize}
\item[(a)] $\eta(H_0)\geq \frac12\,\eta(H)$,
\item[(b)] $\RR_{\Psi,\eta}$ is bounded from $L^p(\eta\rest_{H_0})$ to $L^p(\eta\rest_{\cS'_\eta})$, with 
$\|\RR_\Psi\|_{L^p(\eta\rest_{H_0})\to L^p(\eta\rest_{\cS'_\eta})}\lesssim \Lambda^{\ve_n}\Theta(\sH).$
\end{itemize}
\end{lemma}

\begin{proof}
We let 
$$H_0= H\setminus Z'(\gamma),$$
so that, by \rf{eqnugam*},
$$\eta(H_0) = \nu(H_0) \geq \nu(H) - \nu(Z'(\gamma)\cap \cS_\eta) \geq \frac12 \,\nu(H) = \frac12 \,\eta(H).$$

To prove (b), take $f\in L^p(\eta\rest_{H_0})$ and observe that for $x\in \cS'_\eta$, by \rf{e.compsup''},
$$|\RR_{\Psi} (f\eta)(x)| = |\RR_{\Psi} (f\nu)(x)|\leq |\RR_{\Psi(x)} (f\nu)(x)| + 
\cM_{\Psi,n}(f\nu)(x),$$
where $\RR_{\Psi(x)}$ is the $\Psi(x)$-truncated Riesz transform and
\begin{equation}\label{eq:maximaloppsi}
\cM_{\Psi,n} \alpha(x) = \sup_{r>\Psi(x)}\frac{|\alpha|(B(x,r))}{r^n}
\end{equation}
for any signed Radon measure $\alpha$. Taking into account that $\Phi(x)\geq\Psi(x)$, we can split
$$\RR_{\Psi(x)} (f\nu)(x) = \RR_{\Phi(x)} (f\nu)(x) + \int_{y\in H_0:\Psi(x)<|x-y|\leq \Phi(x)} 
\frac{x-y}{|x-y|^{n+1}}\,f(y)\,d\nu(y).$$
To estimate the last integral, notice that for $y\in H_0$, $\Psi(y) = \Phi(y)$, and then the condition $\Psi(x)<|x-y|$ implies that
$$\Phi(x) \leq \Phi(y) + |x-y| = \Psi(y) + |x-y| \leq \Psi(x) + 2|x-y|<3|x-y|.$$
So the last integral above is bounded by
$$\int_{\Phi(x)/3<|x-y|\leq \Phi(x)} \frac1{|x-y|^n}\,|f(y)|\,d\nu(y)\lesssim \cM_{\Psi,n}(f\nu)(x).$$
From the preceding estimate and \rf{e.compsup''} we derive that
$$|\RR_{\Psi} (f\eta)(x)| \leq |\RR_{\Phi(x)} (f\nu)(x)| + C\,\cM_{\Psi,n}(f\nu)(x) \leq
|\RR_{\Phi} (f\nu)(x)| + C\,\cM_{\Psi,n}(f\nu)(x).$$

From the last inequality and Lemma \ref{lemrieszphi4}, taking into account that $\eta$ coincides with $\nu$ on $V_1$, we deduce that 
$$\|\RR_\Psi\|_{L^p(\eta\rest_{H_0})\to L^p(\eta\rest_{V_1})}\lesssim \Lambda^{\ve_n}\Theta(\sH) 
+ \|\cM_{\Psi,n}\|_{L^p(\eta)\to L^p(\eta)}.$$
From the growth condition \rf{ecreixnupsi} and standard covering lemmas, it follows easily that
$$\|\cM_{\Psi,n}\|_{L^p(\eta)\to L^p(\eta)}\lesssim\Theta(\sH).$$
On the other hand, using the fact that $\dist(H_0,\R^{n+1}\setminus V_1)\gtrsim \ell(R)$, it is immediate (by Schur's criterion, for example) to check that also
$$\|\RR_\Psi\|_{L^p(\eta\rest_{H_0})\to L^p(\eta\rest_{\cS_\eta'\setminus V_1})}\lesssim 
\Theta(R)\leq \Theta(\sH).$$
\end{proof}

\vv

% ********************************************************************************************

\subsection{The variational argument}

\begin{lemma}\label{lemvar}
There is a constant $c_0>0$, depending at most on $n$ and $A_0$, such that for any $p\in (1,2]$, if $\Lambda$ is big enough, we have
$$\int \big|(|\RR\nu(x)| - G(x)- c_0\,\Theta(\sH))_+\big|^p\,d\nu(x) \gtrsim \Lambda^{-p'\ve_n}\sigma_p(\sH),$$
where $p'=p/(p-1)$ and the implicit constant depends on $p$.
\end{lemma}

\begin{proof}
Suppose that, for a very small $0<\lambda<1$ to be fixed below,
\begin{equation}\label{eqsupp1}
\int \bigl|\bigl(|\RR\nu| - G - c_0\,\Theta(\sH)\bigr)_+\bigr|^p\,d\nu\leq \lambda\,\sigma_p(\sH).
\end{equation}
%We consider the subfamily $\sH_0\subset \sH$ of the cubes $Q$ from $\sH$ that are contained in $S_\eta$ and satisfy  
%$$\frac1{\eta(Q)} \int_Q |\RR\nu|\leq \Theta(\sH).$$
%We also denote
%$$H_0=\bigcup_{Q\in\sH_0}Q.$$ 

%Millor dir aixo cert per maximal de HL doblant. Cubs de H que contenen algun punt en que el maximal 

Let $H_0\subset H$ be the set found in Lemma \ref{lemH0}.
 Consider the measures of the form
$\tau=a\,\nu$, with $a\in L^\infty(\nu)$, $a\geq 0$, and let $F$ be the functional
\begin{align*}\label{e.functional}
F(\tau) & = 
\int \bigl|\bigl(|\RR\tau| - G- c_0\,\Theta(\sH)\bigr)_+\bigr|^p\,d\tau
 + \lambda\,\|a\|_{L^\infty(\nu)} \,\sigma_p(\sH) + \lambda \,\frac{\nu(H_0)}{\tau(H_0)}\,\sigma_p(\sH).\nonumber
\end{align*}
Let
$$m= \inf F(\tau),$$ where the infimum is taken over all the measures $\tau=a\,\nu$, with $a\in L^\infty(\nu)$. We call such measures admissible. 
Since $\nu$ is admissible we infer that
\begin{equation}
\label{e.admis}
m \leq F(\nu) \leq 3\lambda\,\sigma_p(\sH).
\end{equation}
So the infimum $m$ is attained over the collection of
measures $\tau=a\nu$ such that $\|a\|_{L^\infty(\nu)}\leq 3$ and $
\tau(H_0)\geq\frac13\,\nu(H_0)$. In particular, by taking a weak-$*$ limit in $L^\infty(\nu)$, this guaranties
the existence of a minimizer.

Let $\tau$ be an admissible measure such that 
\begin{equation}\label{e.admis2}
	m=F(\tau)\leq 3\lambda\,\sigma_p(\sH).
\end{equation}
To simplify notation, we denote
$$f(x) = G + c_0\,\Theta(\sH).$$
We claim that
\begin{equation}\label{eqclaim*}
\bigl|\bigl(|\RR\tau(x)| - f(x)\bigr)_+|^p + 
p \,\RR_\tau^*\Bigl[\bigl|\bigl(|\RR\tau| - f\bigr)_+|^{p-1} |\RR\tau|^{-1} \RR\tau \Bigr](x)
 \lesssim \lambda\,\Theta(\sH)^p \quad \text{\!\!in $\supp\tau$,}
\end{equation}
where for a vector field $U$, we wrote
$$\RR_\tau^*U =\RR^*(U\tau) = -\sum_{i=1}^{n+1}\RR_i(U_i\,\tau).$$
Assume \rf{eqclaim*} for the moment. Since the function on the left hand side is subharmonic
in $\R^{n+1}\setminus\supp\tau$ 
and continuous in $\R^{n+1}$ (recall that $\tau$ and $\eta$ are absolutely continuous with respect to Lebesgue measure, with bounded densities), we deduce that the same estimate holds in the whole $\R^{n+1}$.

Next we need to construct an auxiliary function $\vphi$. First we claim that 
there exists a subfamily of cubes $\sH^0\subset \sH$ such that
\begin{itemize}
\item[(i)] the balls $3B_Q\equiv3B_{Q^{(\mu)}}$, with $Q^{(\mu)}\in\sH^0$, are disjoint,
\item[(ii)] $\frac1{12}\nu(Q)\leq \tau(Q\cap H_0)\leq 3 \nu(Q)$ for all $Q^{(\mu)}\in\sH^0$, and
\item[(iii)] $\sum_{Q^{(\mu)}\in\sH^0}\tau(Q\cap H_0)\approx \nu(H)$.
\end{itemize}
The existence of the family $\sH^0$ follows easily from the fact that
 $\tau(H_0)\geq\frac13\nu(H_0)\geq \frac16\,\nu(H)$ and $\tau=a\,\nu$ with $\|a\|_{L^\infty(\nu)}\leq 3$. Indeed,
 notice that the family $I$ of cubes $Q^{(\mu)}\in \sH$ such that 
$\tau(Q\cap H_0)\leq \frac1{12}\nu(Q)$ satisfies
$$\sum_{Q^{(\mu)}\in I}\tau(Q\cap H_0)\leq \frac1{12}\sum_{Q^{(\mu)}\in I}\nu(Q)\leq \frac1{12}\,\nu(H)
\leq \frac12\,\tau(H_0).$$
So
$$\sum_{Q^{(\mu)}\in \sH\setminus I} \tau(Q)\geq \sum_{Q^{(\mu)}\in \sH\setminus I} \tau(Q\cap H_0)\geq \frac12\,\tau(H_0).$$
By Vitali's 5r-covering lemma, there exists a subfamily $ \sH^0\subset \sH\setminus I$ 
such that the balls $\{3B_{Q}\}_{Q^{(\mu)} \in\sH^0}$ are disjoint and the balls $\{15B_{Q}\}_{Q^{(\mu)}\in\sH^0}$ cover
$\bigcup_{Q^{(\mu)}\in \sH\setminus I}Q$. From the fact that the cubes from $\sH^0$ are $\PP$-doubling and the properties of the family $\Reg(e'(R))$ in Lemma \ref{lem74}, we have
$\nu(Q) \approx \nu(15B_Q)$ if $Q^{(\mu)}\in\sH^0$, and thus
\begin{align*}
\sum_{Q^{(\mu)}\in\sH^0} \tau(Q\cap H_0) & \geq \frac1{12}\sum_{Q^{(\mu)}\in\sH^0} \nu(Q) \approx
\sum_{Q^{(\mu)}\in\sH^0} \nu(15B_Q) \\
& \geq \sum_{Q^{(\mu)}\in \sH\setminus I} \nu(Q)\geq \frac13\,
\sum_{Q^{(\mu)}\in \sH\setminus I} \tau(Q)\geq \frac16\,\tau(H_0)\geq \frac1{36}\,\nu(H).
\end{align*}
The converse estimate $\sum_{Q^{(\mu)}\in\sH^0}\tau(Q\cap H_0)\lesssim \nu(H)$ follows trivially from  $\|a\|_{L^\infty(\nu)}\leq3$.

We consider the function
\begin{equation}\label{eqvphi99}
\vphi = \sum_{Q^{(\mu)}\in\sH^0} \Theta(Q)\,\vphi_Q,
\end{equation}
where $\chi_{B_Q}\leq \vphi_Q\leq \chi_{1.1B_Q}$, $\|\nabla\vphi_Q\|_\infty\lesssim\ell(Q)^{-1}$. 
Remark that
$$\RR^*(\nabla\vphi\,\LL^{n+1}) = \tilde c\,\vphi,$$
where $\tilde c$ is some non-zero absolute constant.
Notice also that
\begin{equation}\label{eqnormpsi}
\|\nabla\vphi\|_1\lesssim \sum_{Q^{(\mu)}\in\sH^0} \Theta(Q)\,\ell(Q)^n\approx \tau(H_0).
\end{equation}
Multiplying the left hand side of \rf{eqclaim*} by $|\nabla\vphi|$, integrating with respect to the Lebesgue measure $\LL^{n+1}$, taking into account that the estimate in
\rf{eqclaim*} holds on the whole $\R^{n+1}$, and using \rf{eqnormpsi}, we get
\begin{multline}\label{eqpss3}
\int \bigl|\bigl(|\RR\tau| - f\bigr)_+|^p\,|\nabla\vphi|\,d\LL^{n+1} + p\! \int \RR_\tau^*\Bigl[\bigl|\bigl(|\RR\tau| - f\bigr)_+|^{p-1} |\RR\tau|^{-1} \RR\tau \Bigr]\,
|\nabla\vphi|\,d\LL^{n+1}\\
\lesssim \lambda\,\sigma_p(\sH).
\end{multline}

Next we estimate the second integral on the left hand side of the preceding inequality, which we denote by $I$.  For that purpose we will also need the estimate 
\begin{equation}\label{eqacpsi}
\int |\RR(|\nabla \vphi|\,d\LL^{n+1})|^p\,d\tau\lesssim \Lambda^{p\ve_n}\sigma_p(\sH),
\end{equation}
which we defer to Lemma \ref{lemtech79}.
Using this inequality we get
\begin{multline*}
|I| = \left|\int 
\bigl|\bigl(|\RR\tau| - f\bigr)_+|^{p-1} |\RR\tau|^{-1} \RR\tau \cdot
\RR(|\nabla \vphi|\,\LL^{n+1})\,d\tau\right| \\
\leq \left(\int \bigl|\bigl(|\RR\tau| - f\bigr)_+|^{p} \,d\tau\right)^{1/p'}
\left(\int |\RR(|\nabla \vphi|\,\LL^{n+1})|^p\,d\tau\right)^{1/p}\\
\le (F(\tau))^{1/p'}\left(\int |\RR(|\nabla \vphi|\,\LL^{n+1})|^p\,d\tau\right)^{1/p}\\
 \overset{\rf{e.admis2},\eqref{eqacpsi}}{\lesssim} \bigl (\lambda\,\sigma_p(\sH)\bigr)^{1/p'} \bigl (\Lambda^{p\ve_n}\sigma_p(\sH)\bigr)^{1/p} = \lambda^{1/p'}\,\Lambda^{\ve_n}\sigma_p(\sH).
\end{multline*}
From \rf{eqpss3} and the preceding estimate we deduce
that
$$\int \bigl|\bigl(|\RR\tau| - f\bigr)_+|^p\,|\nabla\vphi|\,d\LL^{n+1}\lesssim  \lambda^{1/p'}\,\Lambda^{\ve_n}
\sigma_p(\sH).
$$

Notice now that, for all $x\in \supp\vphi$,
$$G(x) = 2A_0^3\int_{\cS_\eta'\setminus V_2} 
 \frac1{\ell(R)\,|x-y|^{n-1}}\,d\eta(y) \lesssim \frac{\eta(\cS_\eta')}{\ell(R)\,\dist(\cS_\eta'\setminus V_2, \cS_\eta)^{n-1}} \lesssim \Theta(R),$$
and so
$$|f(x)|\leq C\, \Theta(R) + c_0\,\Theta(\sH)\leq 2c_0\,\Theta(\sH) \qquad\mbox{for all $x\in \supp\vphi$,}$$ 
taking into account that $\Lambda\gg1$.
So we have
\begin{align}\label{eqcont4}
\int |\RR\tau|^p\,|\nabla\vphi|\,d\LL^{n+1} &\lesssim \int \bigl|\bigl(|\RR\tau| - f\bigr)_+|^p\,|\nabla\vphi|\,d\LL^{n+1}
+ \int |f|^p\,|\nabla\vphi|\,d\LL^{n+1}\\
& \lesssim 
  \bigl(\lambda^{1/p'}\Lambda^{\ve_n} + c_0^p\bigr)\,\sigma_p(\sH).\nonumber
\end{align}

To get a contradiction, note that by the construction of $\vphi$ and by the properties of $\tau$, we have
\begin{multline}\label{eqprev}
\left|\int \RR\tau\cdot\nabla\vphi\,d\LL^{n+1}\right| = \left|\int \RR^*(\nabla\vphi\,\LL^{n+1})\,d\tau \right|=  
|\tilde c| \int \vphi\,d\tau\\
\ge |\tilde c|\,
\Theta(\sH) \sum_{Q^{(\mu)}\in\sH^0} \tau(B_Q) \gtrsim |\tilde c|\,
\Theta(\sH) \tau(H_0)\geq \frac{|\tilde c|}6\,
 \Theta(\sH) \nu(H).
 \end{multline}
However, by H\"older's inequality, \eqref{eqnormpsi}, and \rf{eqcont4}, we have
\begin{align*}%\label{eqa1}
\left|\int \RR\tau\cdot\nabla\vphi\,d\LL^{n+1}\right| & \leq \left(\int |\RR\tau|^p \,|\nabla\vphi|\,d\LL^{n+1} \right)^{1/p}
\left(\int |\nabla\vphi|\,d\LL^{n+1}\right)^{1/p'} \\
&\lesssim \bigl(\lambda^{1/(p\,p')} \Lambda^{\ve_n/p}+ c_0\bigr)
\, \Theta(\sH)\nu(H),
\end{align*}
which contradicts \rf{eqprev} if $c_0$ is chosen small enough and 
$$\lambda\leq c\,\Lambda^{-p'\ve_n},$$
with $c$ small enough.
\end{proof}

\vv

\begin{proof}[\bf Proof of \rf{eqclaim*}]
Recall that $\tau=a\,\nu$ is a minimizer for $F(\cdot)$ that in particular satisfies
$F(\tau)\leq 3\lambda\,\sigma_p(\sH),$ by \rf{e.admis2}.
We have to show that, for all $x\in\supp\tau$,
$$
\bigl|\bigl(|\RR\tau(x)| - f(x)\bigr)_+|^p + 
p \,\RR_\tau^*\Bigl[\bigl|\bigl(|\RR\tau| - f\bigr)_+|^{p-1} |\RR\tau|^{-1} \RR\tau \Bigr](x)
 \lesssim \lambda\,\Theta(\sH)^p.
$$

Let $x_0\in\supp\tau$ 
and let $B=B(x_0,\rho)$, with $\rho>0$ small. For $0<t<1$ we consider the competing measure $\tau_t =
a_t\,\tau$, where $a_t$ is defined as follows:
$$a_t = (1-t\,\chi_B)a.$$
Notice that, for each $0<t<1$, $a_t$ is a non-negative function such that $$\|a_t\|_{L^\infty(\nu)}\leq \|a\|_{L^\infty(\nu)}.$$
Taking the above into account we deduce that
\begin{align*}
F(\tau_t) & = 
\int \bigl|\bigl(|\RR\tau_t| - f\bigr)_+\bigr|^p\,d\tau_t
 + \lambda\,\|a_t\|_{L^\infty(\nu)} \,\sigma_p(\sH) + \lambda \,\frac{\nu(H_0)}{\tau_t(H_0)}\,\sigma_p(\sH)\\
& \leq
\int \bigl|\bigl(|\RR\tau_t| - f\bigr)_+\bigr|^p\,d\tau_t
 + \lambda\,\|a\|_{L^\infty(\nu)} \,\sigma_p(\sH) + \lambda \,\frac{\nu(H_0)}{\tau_t(H_0)}\,\sigma_p(\sH)
\\& =:\wt F(\tau_t).
\end{align*}
Since $\wt F(\tau) = F(\tau) \leq F(\tau_t)\leq \wt F(\tau_t)$ for $t\geq 0$, we infer that 
\begin{equation}\label{eqvar49}
\frac1{\tau(B)}\,\frac{d}{dt}\,\wt F(\tau_t)\biggr|_{t=0+} \geq 0.
\end{equation}
Using that
$$
\frac d{dt}|\RR\tau_t(x)|\biggr|_{t=0} = \frac1{|\RR\tau(x)|} \RR\tau(x)\cdot \RR (-\chi_B\tau )(x)
,$$
an easy computation gives 
\begin{align*}
\frac{d}{dt}\,\wt F(\tau_t)\biggr|_{t=0} & = -\int_B \bigl|\bigl(|\RR\tau| - f\bigr)_+|^p \,d\tau\\
& \quad +
p \int \bigl|\bigl(|\RR\tau| - f\bigr)_+|^{p-1}\, \frac1{|\RR\tau|} \RR\tau\cdot \RR (-\chi_B\tau)\,
d\tau \\
&\quad +  
\lambda \,\frac{\nu(H_0)}{\tau(H_0)^2}\,\sigma_p(\sH)\,\tau(B\cap H_0).
\end{align*}
Recalling that $\tau(H_0)\geq\frac13\,\nu(H_0)\ge \frac16\,\nu(H)$, from \rf{eqvar49} and the preceding calculation we derive
\begin{multline}\label{eqmult82}
 \frac1{\tau(B)} 
\int_B \bigl|\bigl(|\RR\tau| - f\bigr)_+|^p \,d\tau 
 +
 \frac p{\tau(B)} \int \bigl|\bigl(|\RR\tau| - f\bigr)_+|^{p-1}\, \frac1{|\RR\tau|} \RR\tau\cdot \RR (\chi_B\tau )
\,
d\tau
 \\ 
\leq 
3\lambda \,\frac{\sigma_p(\sH)}{\tau(H_0)} =  3\lambda\, \Theta(\sH)^p \,\frac{\nu(H)}{\tau(H_0)} \le 18 \lambda \, \Theta(\sH)^p.
\end{multline}

 We rewrite the left hand side of \rf{eqmult82} as
$$ 
 \frac 1{\tau(B)} \int_B \bigl|\bigl(|\RR\tau| - f\bigr)_+|^p \,d\tau 
 +
 \frac p{\tau(B)} \int_B \RR_\tau^*\Bigl[\bigl|\bigl(|\RR\tau| - f\bigr)_+|^{p-1}\, |\RR\tau|^{-1} \RR\tau \Bigr]
\,
d\tau
.$$
Taking into account that the functions in the integrands are continuous on $\supp(\tau)$, 
letting the radius $\rho$ of $B$ tend to $0$, it turns out that the above expression converges to
$$\bigl|\bigl(|\RR\tau(x_0)| - f(x_0)\bigr)_+|^p 
 +
p\, \RR_\tau^*\Bigl[\bigl|\bigl(|\RR\tau| - f\bigr)_+|^{p-1}\, |\RR\tau|^{-1} \RR\tau \Bigr](x_0).$$
The desired estimate \rf{eqclaim*} follows from the above and \rf{eqmult82}.
\end{proof}
\vv

In order to complete the proof of Lemma \ref{lemvar} we need the following technical result.

\begin{lemma}\label{lemtech79}
Suppose that \rf{eqsupp1} holds with $\lambda\leq1$ and let $\vphi$ be as in \rf{eqvphi99}.
Then,
$$
\int |\RR(|\nabla \vphi|\,d\LL^{n+1})|^p\,d\nu\lesssim \Lambda^{p\ve_n}\sigma_p(\sH),
$$
\end{lemma}

\begin{proof}
Recall that 
$$\vphi = \sum_{Q^{(\mu)}\in\sH^0} \Theta(Q)\,\vphi_Q,$$
with $\chi_{B_Q}\leq \vphi_Q\leq \chi_{1.1B_Q}$, $\|\nabla\vphi_Q\|_\infty\lesssim\ell(Q)^{-1}$.
We consider the function
$$g= \sum_{Q^{(\mu)}\in\sH^0} g_Q,$$
where $g_Q = c_Q\,\chi_{Q\cap H_0}$, with $c_Q=\Theta(Q)\int|\nabla\vphi_Q|\,d\LL^{n+1}\,\eta(Q\cap H_0)^{-1}$.
Observe that
\begin{equation}\label{eq:gQintegral}
\int g_Q\,d\eta = \Theta(Q) \int|\nabla\vphi_Q|\,d\LL^{n+1}
\end{equation}
and that
$$0\leq c_Q \lesssim \frac{\Theta(Q) \,\ell(Q)^n}{\eta(Q\cap H_0)}\approx \frac{\mu(Q^{(\mu)})}{\eta(Q)} =1.$$

The first step of our arguments consists in comparing $\RR(|\nabla\vphi|\,d\LL^{n+1})(x)$ to 
$\RR_{\Psi(x)}(g\,\eta)(x)$, with $\Psi$ given by \rf{eqdefPsi1}.
 We will prove that, for each $Q^{(\mu)}\in \sH^0$,
 \begin{align}\label{eqtech4}
|\RR(\Theta(Q)|\nabla\vphi_Q|&\,\LL^{n+1})(x) - \RR_{\Psi(x)}(g_Q\,\eta)(x)|\\
& \lesssim \frac{\Theta(Q)\,\ell(Q)^{n+1}}{\dist(x,Q)^{n+1}+
\ell(Q)^{n+1}} + \chi_{2B_Q}(x)\, |\RR_{\Psi(x)}(g_Q\,\eta)(x)|\nonumber \\
&\quad+ \int_{c\Psi(x)\leq |x-y|\leq\Psi(x)} \frac{g_Q(y)}{|x-y|^{n}}\,d\eta(y),\nonumber
\end{align}
for some fixed $c>0$.
 The arguments to prove this estimate are quite standard.
 
 Suppose first that $x\in 2B_Q$. In this case, we have
 \begin{align*}
|\RR(\Theta(Q)|\nabla\vphi_Q|\,\LL^{n+1})(x)| &\lesssim \Theta(Q)\int_{1.1B_Q} \frac{1}{\ell(Q)\,|x-y|^n}\,d\LL^{n+1}(y) \lesssim \Theta(Q),
\end{align*}
which yields
$$|\RR(\Theta(Q)|\nabla\vphi_Q|\,\LL^{n+1})(x) - \RR_{\Psi(x)}(g_Q\,\eta)(x)|
 \lesssim \Theta(Q) + |\RR_{\Psi(x)}(g_Q\,\eta)(x)|,$$
 and shows that \rf{eqtech4} holds in this situation.
 
In the case $x\not\in 2B_Q$ we write
\begin{align*}
\RR(\Theta(Q)|&\nabla\vphi_Q|\,\LL^{n+1})(x) - \RR_{\Psi(x)}(g_Q\,\eta)(x)\\  &=
\RR\big(\Theta(Q)|\nabla\vphi_Q|\,\LL^{n+1} - g_Q\,\eta\big)(x) +
\int_{|x-y|\leq\Psi(x)} \frac{x-y}{|x-y|^{n+1}} g_Q(y)\,d\eta(y)\\
& \overset{\eqref{eq:gQintegral}}{=} 
\int \left(\frac{x-y}{|x-y|^{n+1}} - \frac{x-x_Q}{|x-x_Q|^{n+1}}\right) \bigl[\Theta(Q)|\nabla\vphi_Q(y)|\,d\LL^{n+1}(y) - g_Q(y)\,d\eta(y)\bigr]\\
& \quad\quad\quad + \int_{|x-y|\leq\Psi(x)} \frac{x-y}{|x-y|^{n+1}} g_Q(y)\,d\eta(y)\\
& =: I_1(x)+ I_2(x).
\end{align*}
Concerning the term $I_1(x)$, recalling that $\supp\vphi_Q \cup\supp g_Q\subset 1.1\overline{ B_Q}$,
we obtain
\begin{align*}
|I_1(x)|
&\lesssim \int \frac{\ell(Q)}{|x-x_Q|^{n+1}} \bigl[\Theta(Q)|\nabla\vphi_Q(y)|\,d\LL^{n+1}(y) + g_Q(y)\,d\eta(y)\bigr]
\lesssim \frac{\ell(Q)}{|x-x_Q|^{n+1}}\,\eta(Q).
\end{align*}
Regarding $I_2(x)$, we write
$$
I_2(x)\leq \int_{y\in Q:|x-y|\leq\Psi(x)} \frac1{|x-y|^{n}} \,g_Q(y)\,d\eta(y).$$
Notice that for $y\in Q$, since $x\notin 2B_Q$, we have
$$C|x-y|\geq \ell(Q)\overset{\eqref{eqdefPsi1}}{\ge}\Psi(y) \geq \Psi(x) - |x-y|.$$
Thus, $|x-y|\geq c\,\Psi(x)$, and so
$$I_2(x)\leq \int_{c\Psi(x)\leq |x-y|\leq\Psi(x)} \frac1{|x-y|^{n}}\, g_Q(y)\,d\eta(y).
$$
Gathering the estimates for $I_1(x)$ and $I_2(x)$ we see that \rf{eqtech4} also holds in this case.

From \rf{eqtech4} we infer that
\begin{align*}
|\RR(|\nabla\vphi|\,\LL^{n+1})&(x) - \RR_{\Psi(x)}(g\,\eta)(x)|\\
& \lesssim \sum_{Q^{(\mu)}\in \sH^0}\frac{\Theta(Q)\,\ell(Q)^{n+1}}{\dist(x,Q)^{n+1}+
\ell(Q)^{n+1}} + 
\sum_{Q^{(\mu)}\in \sH^0} \chi_{2B_Q}(x)\, |\RR_{\Psi(x)}(g_Q\,\eta)(x)|\\
&\quad + \int_{c\Psi(x)\leq |x-y|\leq\Psi(x)} \frac1{|x-y|^{n}}\, g(y)\,d\eta(y)\\
& = S_1(x) + S_2(x) + S_3(x).
\end{align*}
We split
\begin{equation}\label{eqs123}
\int |\RR(|\nabla\vphi|\,\LL^{n+1})(x) - \RR_{\Psi(x)}(g\,\eta)(x)|^p\,d\eta(x)
\lesssim \sum_{i=1}^3 \int |S_i(x)|^p\,d\eta(x).
\end{equation}

We estimate the first summand by duality. Consider a function $h\in L^{p'}(\eta)$. Then, 
\begin{equation}\label{eqmul429**}
\int S_1(x)\,h(x)\,d\eta(x)= 
  \sum_{Q^{(\mu)}\in \sH^0} \eta(Q) \int\frac{\ell(Q)}{\dist(x,Q)^{n+1}+
\ell(Q)^{n+1}}\,h(x)\,d\eta(x).
\end{equation}
For each $Q^{(\mu)}\in\sH^0$,
 using the fact that $\eta(\lambda Q)\leq C\,\Theta(\sH)\,\ell(\lambda Q)^n$ for every $\lambda\geq 1$,
 we obtain 
$$\int  \frac{\ell(Q)}{\dist(x,Q)^{n+1}+
\ell(Q)^{n+1}}\,h(x)\,d\eta(x)\lesssim C\,\Theta(\sH)\,\inf_{y\in Q} \cM_\eta h(y).$$
Therefore, the right side of \rf{eqmul429**} does not exceed
\begin{align*}
C\Theta(\sH)\sum_{Q^{(\mu)}\in \sH^0} \eta(Q)\,\inf_{y\in Q} \cM_\eta h(y) & \lesssim
\Theta(\sH) \int_{H}\cM_\eta h(y)\,d\eta(y) \\ 
& \le \Theta(\sH)\,\eta(H)^{1/p}\,\|\cM_\eta h\|_{L^{p'}(\eta)}
\lesssim \Theta(\sH)\,\eta(H)^{1/p}\,\|h\|_{L^{p'}(\eta)}.
\end{align*}
So we deduce that
$$\int |S_1(x)|^p\,d\eta(x)\lesssim \sigma_p(\sH).
$$

Regarding the summand in \rf{eqs123} involving $S_2$, since the balls $2B_Q$ are disjoint, we have
$$\int |S_2(x)|^p\,d\eta(x) = 
\sum_{Q^{(\mu)}\in \sH^0} \int_{2B_Q} |\RR_{\Psi(x)}(g_Q\,\eta)(x)|^p\,d\eta(x) 
\leq \sum_{Q^{(\mu)}\in \sH^0}\| \RR_{\Psi(\cdot)}(g_Q\,\eta)\|_{L^p(\eta)}^p,$$
where $\RR_{\Psi(\cdot)}(g_Q\,\eta)(x) = \RR_{\Psi(x)}(g_Q\,\eta)(x)$.
Finally, to estimate the last summand in \rf{eqs123}, we take into account that 
$|S_3(x)|\lesssim \cM_{\Psi,n}(g\,\eta)(x),$ where $\cM_{\Psi,n}(g\,\eta)$ is the maximal operator from \eqref{eq:maximaloppsi}.
Hence,
$$\int |S_3(x)|^p\,d\eta(x)  \lesssim \int | \cM_{\Psi,n}(g\,\eta)|^p\,d\eta \lesssim \Theta(\sH)^p\,
\|g\|_{L^p(\eta)}^p \lesssim \sigma_p(\sH).$$

Gathering the estimates obtained for $S_1$, $S_2$, $S_3$, by \rf{eqs123} we get
$$\|\RR(|\nabla\vphi|\,\LL^{n+1})\|_{L^p(\eta)}^p \lesssim \sigma_p(\sH) +\sum_{Q^{(\mu)}\in \sH^0}
\|\RR_{\Psi(\cdot)}(g_Q\,\eta)\|_{L^p(\eta)}^p + \|\RR_{\Psi(\cdot)}(g\eta)\|_{L^p(\eta)}^p.$$
From \rf{e.compsup''}, we deduce that
\begin{align*}
\|\RR(|\nabla\vphi|\,\LL^{n+1})\|_{L^p(\eta)}^p & \lesssim \sigma_p(\sH) +\sum_{Q^{(\mu)}\in \sH^0}
\|\RR_{\Psi}(g_Q\,\eta)\|_{L^p(\eta)}^p + \|\RR_{\Psi}(g\eta)\|_{L^p(\eta)}^p\\
&\quad + \sum_{Q^{(\mu)}\in \sH^0}
\|\cM_{\Psi,n}(g_Q\,\eta)\|_{L^p(\eta)}^p  + \|\cM_{\Psi,n}(g\,\eta)\|_{L^p(\eta)}^p.
\end{align*}
Using now that, by Lemma \ref{lemH0}, 
 $\RR_{\Psi,\eta}$ is bounded from $L^p(\eta\rest_{H_0})$ to $L^p(\eta\rest_{\cS'_\eta})$ with 
$$\|\RR_\Psi\|_{L^p(\eta\rest_{H_0})\to L^p(\eta\rest_{\cS'_\eta})}\lesssim \Lambda^{\ve_n}\Theta(\sH),$$
and that the same happens for $\cM_{\Psi,n}$, we get
$$\|\RR(|\nabla\vphi|\,\LL^{n+1})\|_{L^p(\eta)}^p \lesssim \sigma_p(\sH) + \Lambda^{p\ve_n}\Theta(\sH)^p\,\|g\|_{L^p(\eta)}^p \lesssim \Lambda^{p\ve_n}\sigma_p(\sH).$$
\end{proof}

\vv

\subsection{Lower estimates for \texorpdfstring{$\RR\eta$}{R\_eta}}

\begin{lemma}\label{lemrieszeta}
Let $R\in\MDW$ be such that $R\in\Trc$ and let $V_4$ and $\eta$ be as in Section \ref{subsec9.5}.
Assume $\Lambda>0$ big enough and let $c_0$ be as in Lemma \ref{lemvar}.
Then we have
$$\int_{V_4} \big|(|\RR\eta(x)| - \frac{c_0}2\,\Theta(\HD_1))_+\big|^p\,d\eta(x) \gtrsim \Lambda^{-(p'+1)\ve_n}\sigma_p(\HD_1(e(R))),$$
for any $p\in (1,\infty)$, with the implicit constant depending on $p$.
\end{lemma}

\begin{proof}
Recall that $\nu = \vphi_R\,\eta$. For all $x\in \supp\eta$, using that $\|\nabla\vphi_R\|_\infty\leq 2A_0^{-3}\ell(R)^{-1}$, we obtain
\begin{align*}
\big|\vphi_R(x)\,\RR\eta(x) - \RR\nu(x)\big| &= \big|\vphi_R(x)\,\RR\eta(x) - \RR(\vphi_R\eta)(x)\big|\\
& = \left|\int \frac{(\vphi_R(x) - \vphi_R(y))\,(x-y)}{|x-y|^{n+1}}\,d\eta(y)\right|\\
& \leq 2A_0^3\int_{y\in \cS_\eta\setminus V_2} \frac1{\ell(R)\,|x-y|^{n-1}}\,d\eta(y)\\
&\quad +
2A_0^3\int_{y\in V_2:\vphi_R(y)\neq \vphi_R(x)} \frac1{\ell(R)\,|x-y|^{n-1}}\,d\eta(y).
\end{align*}
Recall that $\vphi_R$ equals $1$ on $V_3$ and vanishes out of $V_4$.
Then it is clear that the last integral on the right hand side above vanishes if $x\in V_3$ and
it does not exceed $C\,\Theta(R)$ if $x\in V_3^c$.
So in any case
$$\big|\vphi_R(x)\,\RR\eta(x) - \RR\nu(x)\big|\leq G(x) + C_3\,\Theta(R).$$

From the preceding estimate we infer that
\begin{align*}
|\vphi_R(x)\,\RR\eta(x)| - \frac{c_0}2\,\Theta(\HD_1) & \geq 
|\RR\nu(x)| - G(x) - C_3\,\Theta(R)
- \frac{c_0}2\,\Theta(\HD_1)\\
& \geq |\RR\nu(x)| - G(x) 
- c_0\,\Theta(\HD_1).
\end{align*}
Therefore, since $|(\;\cdot\;)_+|^p$ is non-decreasing,
\begin{align*}
\int_{V_4} \big|(|\RR\eta| - \frac{c_0}2\,\Theta(\HD_1))_+\big|^p\,d\eta & 
\geq \int \big|(|\vphi_R\,\RR\eta| - \frac{c_0}2\,\Theta(\HD_1))_+\big|^p\,d\eta\\
& \geq \int  \big|(|\RR\nu| - G 
- c_0\,\Theta(\HD_1))_+\big|^p\,d\eta\\
& \geq \int  \big|(|\RR\nu| - G 
- c_0\,\Theta(\sH))_+\big|^p\,d\nu.
\end{align*}
By Lemmas \ref{lemvar} and \ref{lem:66}, 
$$\int  \big|(|\RR\nu| - G 
- c_0\,\Theta(\sH))_+\big|^p\,d\nu\geq c \Lambda^{-p'\ve_n}\sigma_p(\sH)\geq
 	 c\Lambda^{-(p'+1)\ve_n}\sigma_p(\HD_1(e(R))),$$
and so the lemma follows.
\end{proof}
\vv

% ********************************************************************************************

\section{Lower estimates for the Haar coefficients of $\RR\mu$ for cubes near a typical tractable tree}\label{sectrans}

For $R\in\MDW$ and $P\in\DD_{\mu}$, we write $P\sim \TT(e'(R))$ if there exists some
$P'\in\TT(e'(R))$ such that 
\begin{equation}\label{defsim0}
A_0^{-2}\ell(P)\leq \ell(P')\leq A_0^2\,\ell(P)\quad \text{ and }\quad 20P'\cap20P\neq\varnothing.
\end{equation}
We say that $\TT(e'(R))$ is a typical tractable tree, and we write $R\in \Ty$ if
\begin{equation}\label{defty}
\sum_{P\in\DB:P\sim\TT(e'(R))} \EE_\infty(9P)\leq \Lambda^{\frac{-1}{3n}}\,\sigma(\HD_1(e(R))).
\end{equation}
The reason for the terminology ``typical tree'' is due to the fact that such trees are prevalent, in a sense, because of Lemma
\ref{lemsuper**} below.

In all this section we assume that $R\in\MDW$ is such that $R\in\Trc\cap\Ty$, i.e., $\TT(e'(R))$ is tractable and typical, and we consider the measure $\eta$ constructed in Section \ref{sec6.2*}.
Roughly speaking, our objective is to transfer the lower estimate we obtained for $\RR\eta$ in Lemma \ref{lemrieszeta} to 
the Haar coefficients of
$\RR\mu$ for cubes close to $\TT(e'(R))$. 
Recall that $\RR\mu(x)$ exists $\mu$-a.e.\ as a principal value under the assumptions
of Proposition \ref{propomain}.
%Remark that the assumption that $\mu$ has polynomial growth and the
%fact that $\|\RR\mu\|_{L^2(\mu)}<\infty$ imply that $\RR\mu(x)$ exists $\mu$-a.e.\ in the principal value sense.
%This follows from the results in \cite{NToV2} and arguments analogous to the ones for the Cauchy transform in
%\cite[Chapter 8]{Tolsa-llibre}. So, from now on the function $\RR\mu$ should be understood in the principal value sense.  

\vv

\subsection{The operators $\RR_{\TT_\Reg}$, $\RR_{\wt \TT}$,  and $\Delta_{\wt \TT}\RR$}\label{subsec:opera}

To simplify notation, in this section we will write
$$\End=\End(e'(R)), \quad \;\Reg = \Reg(e'(R)),\; \quad \;\TT = \TT(e'(R)),\quad\,\text{and}\quad\; \TT_\Reg = \TT_\Reg(e'(R)).$$
We need to consider an enlarged version of the generalized tree $\TT$, due to some technical difficulties that arise because the cubes from $\Neg:= \Neg(e'(R))\cap\End$ are not $\PP$-doubling.
To this end, denote by $\Reg_\Neg$ the family of the cubes from $\Reg$ which are contained in some cube from $\Neg$.
Let $\sD_\Neg$ be the subfamily of the cubes $P\in\Reg_\Neg$ for which there exists some
$\PP$-doubling cube $S\in\TT_\Reg$ that contains $P$. By the definition of $\Neg$, such cube $S$ should be contained in the cube from $Q\in\Neg$ such that $P\subset Q$. 
We also denote by $\sM_\Neg$ the family of maximal $\PP$-doubling cubes which belong to $\TT_\Reg$ and are contained in some
cube from $\Neg$, so that, in particular, any cube from $\sD_\Neg$ is contained in another from $\sM_\Neg$.

We define
$$\wt\End = (\End \setminus \Neg) \cup \sM_\Neg,$$
and we let $\wt \TT=\wt \TT(e'(R))$ be the family of cubes that belong to $\TT_\Reg$ but are not strictly contained in any cube from $\wt\End$.
Further, we write
\begin{equation}\label{eqdef*f}
Z = Z(e'(R)) = e'(R)\setminus \bigcup_{Q\in\End} Q\quad \text{ and }\quad\wt Z = \wt Z(e'(R)) = e'(R)\setminus \bigcup_{Q\in\wt\End} Q
\end{equation}
Remark that $\wt Z\supset Z$.
Then we denote
$$\RR_{\TT_\Reg}\mu(x) = \sum_{Q\in\Reg} \chi_Q(x)\,\RR(\chi_{2R\setminus 2Q}\mu)(x),$$
$$\RR_{\wt \TT}\mu(x) = \sum_{Q\in\wt\End} \chi_Q(x)\,\RR(\chi_{2R\setminus 2Q}\mu)(x),$$
and
$$\Delta_{\wt\TT}\RR\mu(x) = \sum_{Q\in\wt\End} \chi_Q(x)\,\big(m_{\mu,Q}(\RR\mu) - m_{\mu,2R}(\RR\mu)\big)
+ \chi_{Z}(x) \big(\RR\mu(x) -  m_{\mu,2R}(\RR\mu)\big)
.$$
%Recall that $\RR\mu(x)$ is understood in the principal value sense (which exists $\mu$-a.e.\ because we assume that $\mu$ has polynomial growth and $\|\RR\mu\|_{L^2(\mu)}<\infty$.
Observe that $\Delta_{\wt\TT}\RR\mu$ vanishes in $\wt Z\setminus Z$.

The cubes from $\Reg$ have the advantage over the cubes from $\wt \End$ that their sizes change smoothly so that, for example, neighboring cubes have comparable side lengths. However, they need not be $\PP$-doubling or doubling, unlike the cubes from $\wt \End$. In particular, it is not clear
how to estimate their energy $\EE_\infty$ in terms of the condition $\DB$ (recall that the cubes from $\DB$ are asked to be $\PP$-doubling).

For $Q\in\DD_\mu$, we define
$$\QQ_\Reg(Q) = \sum_{P\in\Reg} \frac{\ell(P)}{D(P,Q)^{n+1}}\,\mu(P),$$
where 
$$D(P,Q) = \ell(P) + \dist(P,Q) + \ell(Q).$$
The coefficient $\QQ_\Reg(Q)$ will be used to bound some ``error terms'' in our transference 
arguments. We will see later how they can be estimated in terms of the coefficients $\PP(Q)$. Notice that, unlike $\PP(Q)$, the coefficients $\QQ_\Reg(Q)$ depend on the family
$\Reg$.
\vv

\begin{lemma}\label{lemaprox1}
For any $Q\in\Reg$ such that $(Q\cup \frac12 B(Q))\cap V_4\neq\varnothing$ and $x\in Q$, $y\in \frac12B(Q)$,
$$\big|\RR_{\TT_\Reg}\mu(x) - \RR\eta(y)\big| \lesssim \Theta(R) + \PP(Q) + \QQ_\Reg(Q).$$
\end{lemma}

\begin{proof}
By the triangle inequality, for $x$, $y$ and $Q$ as in the lemma, we have
\begin{align*}
\big|\RR_{\TT_\Reg}\mu(x) - \RR\eta(y)\big| & \leq \big|\RR_\mu\chi_{2R\setminus e'(R)}(x)\big| 
+ \big|\RR_\mu\chi_{e'(R)\setminus 2Q}(x) - \RR_\mu\chi_{e'(R)\setminus 2Q}(x_Q)\big|\\
& \quad 
+ \big|\RR_\mu\chi_{2Q\setminus Q}(x_Q)\big| 
+\big|\RR_\mu\chi_{e'(R)\setminus Q}(x_Q) - \RR_\eta \chi_{(\frac12B(Q))^c}(x_Q)\big|\\
&\quad + \big|\RR_\eta \chi_{(\frac12B(Q))^c}(x_Q) - \RR_\eta \chi_{(\frac12B(Q))^c}(y)\big| + \big|\RR_\eta \chi_{\frac12B(Q)}(y)\big|\\
& = I_1+ \ldots + I_6.
\end{align*}
To estimate $I_1$, notice that, by \eqref{eqtec733}, $\dist(Q,\supp\mu\setminus e'(R))\gtrsim\ell(R)$.
Then it follows that
$$I_1\lesssim \frac{\mu(2R)}{\ell(R)^n} \lesssim \Theta(R).$$
Arguing similarly one also gets 
\begin{equation*}
	I_3\lesssim \frac{\mu(2Q)}{\ell(Q)^n}\lesssim \PP(Q).
\end{equation*}
By standard arguments involving continuity of the Riesz kernel, we also deduce that
$$I_2\lesssim\PP(Q),\quad I_5\lesssim\PP(Q).$$
Concerning the term $I_6$, using that $\eta$ is a constant multiple of $\LL^{n+1}$ on 
$\frac12B(Q)$, we get
$$I_6\leq \int_{\frac12B(Q)}\frac1{|x-y|^n}\,d\eta(y)\lesssim \frac{\eta(\tfrac12B(Q))}{\ell(Q)^n}
\lesssim \PP(Q).$$

Finally we deal with the term $I_4$. To this end, recall that $\mu(P) = \eta(\frac12B(P))$ for all 
$P\in \Reg$, and so
\begin{align*}
I_4 & \leq \sum_{P\in\Reg:P\neq Q} \left| \int K(x_Q-z)\,d\big(\eta\rest_{\frac12B(P)} - \mu\rest_P\big)(z)\right|\\
& \leq \sum_{P\in\Reg:P\neq Q}  \int |K(x_Q-z)-K(x_Q-x_P)|\,d\big(\eta\rest_{\frac12B(P)} + \mu\rest_P\big)(z)\\
& \lesssim \sum_{P\in\Reg:P\neq Q} \frac{\ell(P)}{|x_Q-x_P|^{n+1}}\,\mu(P)\approx \sum_{P\in\Reg:P\neq Q} \frac{\ell(P)}{D(Q,P)^{n+1}}\,\mu(P) \leq \QQ_\Reg(Q). 
\end{align*}
For the estimates in the last line we took into account the properties of the family $\Reg$ described 
in Lemma \ref{lem74}.
Gathering the estimates above, the lemma follows.
\end{proof}

\vv
\begin{lemma}\label{lemaprox2}
For any $Q\in\wt\End$ and $x,y\in Q$,
$$\big|\RR_{\wt\TT}\mu(x) - \Delta_{\wt\TT}\RR\mu(y)\big| \lesssim \left(\frac{\EE(4R)}{\mu(R)}\right)^{1/2} +  \left(\frac{\EE(2Q)}{\mu(Q)}\right)^{1/2}.$$
\end{lemma}

\begin{proof}
For $Q\in\wt\End$ and $x,y\in Q$, we have
\begin{align}\label{eqal842}
\RR_{\wt\TT}\mu(x) - \Delta_{\wt\TT}\RR\mu(y) & = \RR(\chi_{2R\setminus 2Q}\mu)(x) - 
\big(m_{\mu,Q}(\RR\mu) - m_{\mu,2R}(\RR\mu)\big)\\
& = \big(\RR(\chi_{2R\setminus 2Q}\mu)(x) - m_{\mu,Q}(\RR\chi_{2R\setminus 2Q}\mu) \big) - m_{\mu,Q}(\RR(\chi_{2Q}\mu))\nonumber\\
&\quad -
\big(m_{\mu,Q}(\RR(\chi_{2R^c}\mu))- m_{\mu,2R}(\RR\mu)\big).\nonumber
\end{align}
To estimate the first term on the right hand side, notice that for all $x,z\in Q$, by standard arguments we have
$$\big|\RR(\chi_{2R\setminus 2Q}\mu)(x) - \RR(\chi_{2R\setminus 2Q}\mu)(z)\big|\lesssim \PP(Q).$$
Averaging over $z\in Q$, we get
$$\big|\RR(\chi_{2R\setminus 2Q}\mu)(x) - m_{\mu,Q}(\RR\chi_{2R\setminus 2Q}\mu) \big|
\lesssim \PP(Q).$$
To estimate $m_{\mu,Q}(\RR(\chi_{2Q}\mu))$ we use the fact that, by the antisymmetry of the Riesz kernel, $m_{\mu,Q}(\RR(\chi_{Q}\mu))=0$ and\footnote{Here we use the assumption that $\RR_*\mu\in L^1(\mu)$.
Indeed, this implies that $\RR_*(\chi_Q\mu)\in L^1(\mu)$ (using Lemma \ref{lemDMimproved}, for example). Also, by antisymmetry we have
$\int_Q \RR_\ve(\chi_Q \mu)\,d\mu=0$ for any $\ve>0$. Then the $\mu$-a.e.\  convergence of $\RR_\ve(\chi_Q \mu)(x)$
to the principal value $\RR(\chi_Q \mu)(x)$ as $\ve\to0$ and the dominated convergence theorem ensure that $\int_Q \RR(\chi_Q \mu)\,d\mu=0$.}
thus
$$\big|m_{\mu,Q}(\RR(\chi_{2Q}\mu))\big| = \big|m_{\mu,Q}(\RR(\chi_{2Q\setminus Q}\mu))\big| 
\leq \avint_{z\in Q}\int_{\xi\in 2Q\setminus Q}\frac1{|z-\xi|^n}\,d\mu(\xi)d\mu(z).$$
By Cauchy-Schwarz and Lemma \ref{lemDMimproved}, we have
\begin{multline}\label{eqboundar3}
\avint_{z\in Q}\int_{\xi\in 2Q\setminus Q}\frac1{|z-\xi|^n}\,d\mu(\xi)d\mu(z)\\
\le \avint_{z\in Q}\int_{\xi\in 2Q\setminus 2B_Q}\frac1{|z-\xi|^n}\,d\mu(\xi)d\mu(z) + \avint_{z\in Q}\int_{\xi\in 2B_Q\setminus Q}\frac1{|z-\xi|^n}\,d\mu(\xi)d\mu(z)\\
\leq
  \PP(Q)+\bigg(\avint_{z\in Q}\bigg(\int_{\xi\in 2B_Q\setminus Q}\frac1{|z-\xi|^n}\,d\mu(\xi)\bigg)^2d\mu(z)\bigg)^{1/2}
   \lesssim \PP(Q) + \left(\frac{\EE(2Q)}{\mu(Q)}\right)^{1/2}.
\end{multline}

Finally we deal with the last term on the right hand side of \rf{eqal842}. We split it:
\begin{align*}
\big|m_{\mu,Q}(\RR(\chi_{2R^c}\mu))- m_{\mu,2R}(\RR\mu)\big|
& \leq \big|m_{\mu,Q}(\RR(\chi_{3R\setminus 2R}\mu))| + \big|m_{\mu,2R}(\RR(\chi_{3R}\mu))\big|\\
&\quad +
\big|m_{\mu,Q}(\RR(\chi_{3R^c}\mu))- m_{\mu,2R}(\RR(\chi_{3R^c}\mu))\big| \\
& = J_1 + J_2 + J_3.
\end{align*}
Clearly,
$$J_1\lesssim \frac{\mu(3R)}{\ell(R)^n}\lesssim \PP(R).$$
Also, by the antisymmetry of the Riesz kernel and arguing as in \rf{eqboundar3},
\begin{align*}
J_2=& \big|m_{\mu,2R}(\RR(\chi_{3R}\mu))\big| = \big|m_{\mu,2R}(\RR(\chi_{3R\setminus 2R}\mu))\big| 
\\
& \leq  \avint_{z\in 2R}\int_{\xi\in 3R\setminus 2R}\frac1{|z-\xi|^n}\,d\mu(\xi)d\mu(z)
\lesssim \PP(R)+\left(\frac{\EE(4R)}{\mu(R)}\right)^{1/2}.
\end{align*}
Regarding $J_3$, by standard methods, for all $z,z'\in 2Q$, 
$$\big|\RR(\chi_{3R^c}\mu)(z) - \RR(\chi_{3R^c}\mu)(z')\big| \lesssim \PP(R).$$
Hence averaging for $z\in Q$, $z'\in 2R$, we obtain
$$J_3\lesssim \PP(R).$$

We showed that 
$$\big|\RR_{\wt\TT}\mu(x) - \Delta_{\wt\TT}\RR\mu(y)\big| \lesssim \left(\frac{\EE(4R)}{\mu(R)}\right)^{1/2} +  \left(\frac{\EE(2Q)}{\mu(Q)}\right)^{1/2} + \PP(Q) + \PP(R).$$
To get rid of the last two terms, recall that the cubes in $\wt\End$ are $\PP$-doubling, and so is $R$. Hence, by the definition of the energy $\EE(2Q)$, we have
\begin{equation*}
\PP(Q)\lesssim \left(\frac{\EE(2Q)}{\mu(Q)}\right)^{1/2},
\end{equation*}
and an analogous estimate for $\PP(R)$. This finishes the proof.
\end{proof}

\vv

\begin{lemma}\label{lemaprox3}
Let $1<p\leq2$.
For any $Q\in\wt\End$ such that $\ell_0\leq \ell(Q)$,
$$\int_Q \big|\RR_{\wt\TT}\mu - \RR_{\TT_\Reg}\mu\big|^p\,d\mu \lesssim \EE(2Q)^{\frac p2} \,\mu(Q)^{1-\frac p2}.$$
\end{lemma}

\begin{proof}
The lemma follows from H\"older's inequality and the estimate
\begin{equation}\label{eqf932}
\int_Q \big|\RR_{\wt\TT}\mu - \RR_{\TT_\Reg}\mu\big|^2\,d\mu \lesssim \EE(2Q),
\end{equation}
which we prove below.

Notice first that the cube $Q$ as above is covered by the cubes $P\in\DD_{\mu}(Q)\cap\Reg$. Indeed, if $Q\in \sM_\Neg$, then $Q\in\TT_{\Reg}$ and this is trivially true. On the other hand, if $Q\in\wt\End\setminus\sM_\Neg\subset\End$, then we have $Q\in\TT$, and the condition $\ell_0\leq \ell(Q)$ implies that $d_{R,\ell_0}(x)\leq\ell(Q)$. Hence, 
$Q$ is covered by the cubes $P\in\DD_\mu(Q)\cap\Reg$.

Observe now that, for $Q\in\wt\End$, $P\in\Reg$ such that $P\subset Q$, and $x\in P$,
$$\RR_{\TT_\Reg}\mu(x) - \RR_{\wt\TT}\mu(x) = \RR(\chi_{2R\setminus 2P}\mu)(x) -
\RR(\chi_{2R\setminus 2Q}\mu)(x) = \RR(\chi_{2Q\setminus 2P}\mu)(x).$$
Thus,
$$\big|\RR_{\TT_\Reg}\mu(x) - \RR_{\wt\TT}\mu(x)\big| \lesssim \sum_{S\in\DD_\mu:P\subset S\subset Q}
\theta_\mu(2B_S).$$
Notice now that, if $\wh Q$ denotes the cube from $\End$ that contains $Q$ (which coincides with
$Q$ when $P\not\in\Reg_\Neg$), we have
$$d_R(x) \gtrsim \dist(P,\supp\mu\setminus \wh Q) \geq \dist(P,\supp\mu\setminus Q) 
\quad \mbox{ for all $x\in P$.}$$
So by Lemma \ref{lem74} we also have
\begin{equation}\label{eqreg71}
\ell(P) \gtrsim \dist(P,\supp\mu\setminus Q)
\end{equation}
Thus,
$$\sum_{S\in\DD_\mu:P\subset S\subset Q}
\theta_\mu(2B_S) \lesssim \sum_{S\in\wt\DD_\mu^{int}(Q):S\supset P}
\theta_\mu(2B_S),$$
with $\wt\DD_\mu^{int}(Q)$ defined in \rf{eqdmuint}.

So we have
\begin{align*}
\int_Q \big|\RR_{\wt\TT}\mu - \RR_{\TT_\Reg}\mu\big|^2\,d\mu &=\sum_{P\in\Reg:P\subset Q}
\int_P \big|\RR_{\wt\TT}\mu - \RR_{\TT_\Reg}\mu\big|^2\,d\mu\\
& \lesssim \sum_{P\in\Reg:P\subset Q}
\bigg(\sum_{S\in\wt\DD_\mu^{int}(Q):S\supset P}
\theta_\mu(2B_S)\bigg)^2\,\mu(P).
\end{align*}
By H\"older's inequality, for any $\alpha>0$,
\begin{align*}
\bigg(\sum_{S\in\wt\DD_\mu^{int}(Q):S\supset P} &
\theta_\mu(2B_S)\bigg)^2\\
&\leq \bigg(\sum_{S\in\wt\DD_\mu^{int}(Q):S\supset P}\left(\frac{\ell(Q)}{\ell(S)}\right)^\alpha
\theta_\mu(2B_S)^2\bigg) \bigg(\sum_{S\in\wt\DD_\mu^{int}(Q):S\supset P}\left(\frac{\ell(S)}{\ell(Q)}\right)^{\alpha}\bigg)\\
& \lesssim_\alpha \sum_{S\in\wt\DD_\mu^{int}(Q):S\supset P}\left(\frac{\ell(Q)}{\ell(S)}\right)^\alpha
\theta_\mu(2B_S)^2.
\end{align*}
Therefore,
\begin{align*}
\int_Q \big|\RR_{\wt\TT}\mu - \RR_{\TT_\Reg}\mu\big|^2\,d\mu & \lesssim_\alpha 
\sum_{P\in\Reg:P\subset Q}\,\sum_{S\in\wt\DD_\mu^{int}(Q):S\supset P}\left(\frac{\ell(Q)}{\ell(S)}\right)^\alpha
\theta_\mu(2B_S)^2\mu(P)\\
& =
\sum_{S\in\wt\DD_\mu^{int}(Q)}\left(\frac{\ell(Q)}{\ell(S)}\right)^\alpha
\theta_\mu(2B_S)^2 \sum_{P\in\Reg:P\subset S}\mu(P)\\
& \le \sum_{S\in\wt\DD_\mu^{int}(Q)}\left(\frac{\ell(Q)}{\ell(S)}\right)^\alpha
\theta_\mu(2B_S)^2\,\mu(S).
\end{align*}
By Lemma \ref{lemdmutot}, for $\alpha$ small enough, the right hand side is bounded above by 
$C\EE(2Q)$, so that \rf{eqf932} holds.
\end{proof}

\vv

\vv

% ********************************************************************************************

\subsection{Estimates for the $\PP$ and $\QQ_\Reg$ coefficients of some cubes from $\wt \End$ and $\Reg$}

We will transfer the lower estimate obtained for the $L^p(\eta)$ norm of $\RR\eta$ in Lemma
\ref{lemrieszeta} to $\RR_{\TT}\mu$, $\RR_{\TT_\Reg}\mu$, and $\Delta_{\TT}\RR\mu$ by means of
Lemmas \ref{lemaprox1}, \ref{lemaprox2}, and \ref{lemaprox3}. To this end, we will need careful 
estimates for the $\PP$ and $\QQ_\Reg$ coefficients of cubes from $\wt\End$ and  $\Reg$. This is the
task we will perform in this section.
\vv

%\subsection{Notation}

Given $R\in\MDW$, recall that
$\HD_1(e'(R))=\HD(R)\cap\sss(e'(R))$. 
To shorten notation, we will write $\HD_1=\HD_1(e'(R))$ in this section.
Recall that $\wt\End = \End\setminus\Neg \cup \sM_{\Neg}$. By \rf{eqstop2}, we have
\begin{equation}\label{eqsplit71}
\wt\End =  \LD_1 \cup \OP_1 \cup \LD_2  \cup \OP_2 \cup \HD_2\cup \sM_\Neg,
\end{equation}
where we introduced the following notations:
\begin{itemize}
%\item $\Neg=\Neg(e'(R))$.
\item $\LD_1$ is the subfamily of $\wt\End$ of those maximal $\PP$-doubling cubes which are contained both in $e'(R)$ and in some cube from $\LD(R)\cap\sss_1(e'(R))\setminus \Neg$.
\item $\OP_1$ is the  subfamily of $\wt\End$ of those maximal $\PP$-doubling cubes which are contained both in $e'(R)$ and in some cube from $\OP(R)\cap\sss_1(e'(R))\setminus \Neg$.
\item $\LD_2$ is the subfamily of $\wt\End$ of those maximal $\PP$-doubling cubes which are contained in some cube
$Q\in \LD(Q')\cap \sss(Q')\setminus \Neg$ for some $Q'\in \HD_1$.
\item $\OP_2$ is the subfamily of $\wt\End$ of those maximal $\PP$-doubling cubes which are contained in some cube
$Q\in \OP(Q')\cap\sss(Q')$ for some $Q'\in \HD_1$.
\item $\HD_2= \bigcup_{Q'\in\HD_1}(\HD(Q')\cap\sss(Q'))$.
\end{itemize}
Remark that the splitting in \rf{eqsplit71} is disjoint. Indeed, notice that, by the definition
of $\sM_\Neg$, the cubes from $\OP_2\cup\HD_2$ do not belong to $\sM_\Neg$, since they are strictly contained in some cube from $\HD_1$, which is $\PP$-doubling, in particular.

For $i=1,2$, we also denote by $\Reg_{\LD_i}$ the subfamily of the cubes from $\Reg$ which are contained in some cube from $\LD_i$, and we define $\Reg_{\OP_i}$, $\Reg_{\HD_2}$, $\Reg_\Neg$, and
$\Reg_{\sM_\Neg}$ analogously.\footnote{Notice that $\sD_\Neg=\Reg_{\sM_\Neg}$.}
We let $\Reg_\Ot$ be the ``other'' cubes from $\Reg$: the ones which are not contained in any cube from $\End$ (which, in particular, have side length comparable to $\ell_0$).
Also, we let $\Reg_{\DB}$ be the subfamily of the cubes from $\Reg$ which are contained in some cube from $\wt\End\cap\DB$. Notice that we have the splitting
\begin{equation}\label{eqsplitreg0}
\Reg = \Reg_{\LD_1} \cup \Reg_{\OP_1}\cup \Reg_{\LD_2}\cup \Reg_{\OP_2}\cup \Reg_{\HD_2}\cup \Reg_{\Neg}\cup \Reg_{\Ot}.
\end{equation}
The families above may intersect the family $\Reg_{\DB}$.

Given a family $\cI\subset\DD_\mu$ and $1<p\leq2$, we denote
$$\Sigma_p^\PP(\cI) = \sum_{Q\in \cI} \PP(Q)^p\,\mu(Q),\qquad \Sigma_p^\QQ(\cI) = \sum_{Q\in \cI} \QQ_\Reg(Q)^p\,\mu(Q).$$
We also write $\Sigma^\PP(\cI)= \Sigma_2^\PP(\cI)$, $\Sigma^\QQ(\cI)= \Sigma_2^\QQ(\cI)$.

\vv

\begin{lemma}\label{lemenereg}
For any $Q\in\wt\End$, 
$$\Sigma^\PP(\Reg\cap\DD_\mu(Q)) \lesssim \EE(2Q).$$
\end{lemma}

\begin{proof}
	For all $S\in\Reg\cap \DD_\mu(Q)$, by H\"older's inequality, we have
	\begin{align*}
	\PP(S)^2 &\lesssim \bigg(\sum_{P:S\subset P\subset Q}\frac{\ell(S)}{\ell(P)}\,\Theta(P)\bigg)^2
	+ \bigg(\frac{\ell(S)}{\ell(Q)}\,\PP(Q)\bigg)^2\\
	& \leq \bigg(\sum_{P:S\subset P\subset Q}\frac{\ell(S)}{\ell(P)}\,\Theta(P)^2\bigg)
	\bigg(\sum_{P:S\subset P\subset Q}\frac{\ell(S)}{\ell(P)}\bigg) +
	\PP(Q)^2\\
	& \lesssim \sum_{P:S\subset P\subset Q}\frac{\ell(S)}{\ell(P)}\,\Theta(P)^2+ 
	\PP(Q)^2.
	\end{align*}
	We deduce that
	\begin{align*}
	\Sigma^\PP(\Reg\cap\DD_\mu(Q)) &=
	\sum_{S\in\Reg\cap \DD_\mu(Q)} \PP(S)^2\,\mu(S)\\
	& \lesssim
	\sum_{S\in\Reg\cap \DD_\mu(Q)} 
	\sum_{P:S\subset P\subset Q}\frac{\ell(S)}{\ell(P)}\,\Theta(P)^2\,\mu(S) +
	\sum_{S\in\Reg\cap \DD_\mu(Q)} \PP(Q)^2\,\mu(S).
	\end{align*}
	Clearly, the last sum is bounded by $\PP(Q)^2\,\mu(Q),$ which in turn is bounded by $\EE(2Q)$ since $Q$ is $\PP$-doubling. Concerning the first term,
	arguing as in \rf{eqreg71}, we deduce that the cubes $P$ in the sum belong to 
	$\wt\DD_\mu^{int}(Q)$. Thus, by Fubini,
	\begin{align*}
	\Sigma^\PP(\Reg\cap\DD_\mu(Q)) &\lesssim 
	\sum_{P\in\wt \DD_\mu^{int}(Q)} \Theta(P)^2 \sum_{S\in\Reg:S\subset P}\frac{\ell(S)}{\ell(P)}\,\mu(S)
	+ \EE(2Q)\\
	& \leq \sum_{P\in\wt \DD_\mu^{int}(Q)}\! \Theta(P)^2 \,\mu(P) + \EE(2Q)\lesssim \EE(2Q).
	\end{align*}
\end{proof}

%This is proven in the same way as Lemma 7.4 from \cite{DT}.

%\vv

%\begin{rem}\label{remtambe44}
%Let $Q\in\wt\End$, $P\in\Reg$, and $S\in\DD_\mu$ such that $P\subset S\subset Q$. By the same arguments
%as in Lemma \ref{lemenereg}, it follows that
%$$\Sigma^\PP(\Reg\cap\DD_\mu(S)) \lesssim \PP(S)^2\,\mu(S) + \EE(2S).$$
%\end{rem}

\vv
\begin{lemma}\label{lempoiss} 
We have:
\begin{itemize}
\item[(i)] If $Q\in\LD_1$, then $\PP(Q)\lesssim \delta_0\,\Theta(R)$.
%\vspace{1mm}
\item[(ii)] If $Q\in\LD_2$, then $\PP(Q)\lesssim \delta_0\,\Lambda\,\Theta(R)$.
\vspace{1mm}

\item[(iii)] If $Q\in\OP_1$, then $\PP(Q)\lesssim \Lambda\,\Theta(R)$.
\vspace{1mm}

\item[(iv)] If $Q\in\OP_2\cup \HD_2$, then $\PP(Q)\lesssim \Lambda^2\,\Theta(R)$.

\vspace{1mm}

\item[(v)] If $Q\in\Neg\cup\sM_\Neg$, then $\PP(Q)\lesssim \left(\frac{\ell(Q)}{\ell(R)}\right)^{1/3}\,\Theta(R)$.
\end{itemize}
\end{lemma}

\begin{proof}
The statements (i) and (ii) follow directly from the definitions of $\LD_1$ and $\LD_2$ and Lemma \ref{lemdobpp}.
The statement (iii) is due to the fact that, by the stopping conditions, the cubes  $Q'\in\sss_1(e'(R))$ satisfy $\PP(Q')\lesssim \Lambda\,\PP(R)\approx \Lambda\,\Theta(R)$. Then 
it just remains to notice that if
$Q$ is a maximal doubling cube contained in $Q'$, from Lemma \ref{lemdobpp} it follows that
$\PP(Q)\lesssim\PP(Q')$.

The property (iv) for the cubes in $\OP_2$ follows by arguments analogous to the ones for (iii), taking into account that such cubes are maximal doubling cubes contained in cubes 
from $\sss(R')$, for some $R'\in\HD_1$ with $\PP(R')\approx \Theta(R')=\Lambda\,\Theta(R)$.\footnote{This argument shows that, in fact, $\PP(Q)\lesssim\Lambda^2\Theta(R)$ for all $Q\in\TT(e'(R))$.}
On the other hand the cubes $Q\in\HD_2$ satisfy $\Theta(Q)=\Lambda^2\,\Theta(R)$ and are $\PP$-doubling by construction.

Finally, we turn our attention to (v). By the definitions of $\Neg$ and $\sM_\Neg$ and Lemma \ref{lemdobpp}, for all
$S\in\DD_\mu$ such that $Q\subset S\subset R$, we have
$$\Theta(S)\lesssim \left(\frac{\ell(S)}{\ell(R)}\right)^{1/2}\,\PP(R).$$
Thus,
\begin{align*}
\PP(Q) &\approx \frac{\ell(Q)}{\ell(R)}\,\PP(R) + \sum_{S:Q\subset S\subset R}\frac{\ell(Q)}{\ell(S)}\,\Theta(S)\\
& \lesssim \frac{\ell(Q)}{\ell(R)}\,\Theta(R) + \sum_{S:Q\subset S\subset R}\frac{\ell(Q)}{\ell(S)}
\,\left(\frac{\ell(S)}{\ell(R)}\right)^{1/2}\,\Theta(R)\\
& \lesssim \frac{\ell(Q)}{\ell(R)}\,\Theta(R) + \sum_{S:Q\subset S\subset R}\,\left(\frac{\ell(Q)}{\ell(R)}\right)^{1/2}\,\Theta(R)\\
& \lesssim \frac{\ell(Q)}{\ell(R)}\,\Theta(R) + \log\left(\frac{\ell(R)}{\ell(Q)}\right)\,\left(\frac{\ell(Q)}{\ell(R)}\right)^{1/2}\,\Theta(R) \lesssim \left(\frac{\ell(Q)}{\ell(R)}\right)^{1/3}\,\Theta(R).
\end{align*}

\end{proof}
\vv

\begin{rem}\label{rem9.7}
We will assume that 
$$\delta_0\leq \Lambda^{-2}.$$
With this choice, it follows that, by the preceding lemma,
 $$\PP(Q)\lesssim \delta_0^{1/2}\,\Theta(R)\quad \mbox{ for all $Q\in\LD_1\cup\LD_2$.}$$
\end{rem}

\vv
\begin{lemma}\label{lemregmolt}
Suppose that $R\in\Trc\cap\Ty$. Then

\begin{equation}\label{eqlem*1}
\Sigma^\PP(\wt\End)\approx \sigma(\wt\End) \lesssim B\,\sigma(\HD_1),
\end{equation}

\begin{equation}\label{eqlem*2}
\Sigma^\PP(\Reg_\DB) \lesssim \sum_{Q\in\wt \End\cap \DB}\EE_\infty(9Q)\leq
\sum_{Q\in\DB:Q\sim \TT}\EE_\infty(9Q)
\lesssim \Lambda^{\frac{-1}{3n}}\,\sigma(\HD_1),
\end{equation}

\begin{equation}\label{eqlem*3}
\Sigma^\PP(\wt\End\cap \DB)\approx\sigma(\wt\End\cap \DB)\leq M^{-2}\,\Lambda^{\frac{-1}{3n}}\,\sigma(\HD_1),
\end{equation}

\begin{equation}\label{eqlem*4}
\Sigma^\PP(\Reg_{\OP_1}) + \Sigma^\PP(\Reg_{\OP_2})\lesssim 
\sum_{Q\in\OP_1\cup\OP_2}\EE_\infty(9Q)
\lesssim \Lambda^{\frac{-1}{3n}}\,\sigma(\HD_1),
\end{equation}

\begin{equation}\label{eqlem*4.5}
\Sigma^\PP(\Reg_{\LD_1}\cup \Reg_{\LD_2}) \lesssim \sum_{Q\in
\LD_1\cup \LD_2} \EE_\infty(9Q)\lesssim 
\big(B\,M^2\,\delta_0 + \Lambda^{\frac{-1}{3n}}\big)\,\sigma(\HD_1).
\end{equation}

\noi Also, for $1<p\leq2$,

\begin{equation}\label{eqlem*5}
\Sigma_p^\PP(\Reg_{\LD_1}\setminus\Reg_{\DB}) + \Sigma_p^\PP(\Reg_{\LD_2})\lesssim
\Big(B\,\Lambda^2\,\delta_0^{\frac p{2}} + \Lambda^{\frac{-p}{6n}}\Big)\,
\sigma_p(\HD_1),
\end{equation}

\begin{equation}\label{eqlem*6}
\Sigma_p^\PP(\Reg_{\HD_2})\lesssim \Sigma_p^\PP(\HD_2) \approx \sigma_p(\HD_2) \lesssim B\,\Lambda^{p-2}\,\sigma_p(\HD_1),
\end{equation}

\begin{equation}\label{eqlem*7}
\Sigma_p^\PP(\Reg_{\OP_2})\lesssim 
\Lambda^{\frac{-p}{6n}}\,\sigma_p(\HD_1).
\end{equation}
\end{lemma}
\vv

\begin{proof}
To prove \rf{eqlem*1}, notice first that $\Sigma^\PP(\wt\End)\approx \sigma(\wt\End)$ because
the cubes from $\wt\End$ are $\PP$-doubling. Also, by \rf{eqsplit71} we have
$$\Sigma^\PP(\wt\End) =  \Sigma^\PP(\LD_1) + \Sigma^\PP(\LD_2) +\Sigma^\PP(\OP_1)+ \Sigma^\PP(\OP_2)
+ \Sigma^\PP(\HD_2) + \Sigma^\PP(\sM_\Neg).$$
By Lemma \ref{lempoiss} and the subsequent remark, we have
\begin{equation}\label{eqld12}
\Sigma^\PP(\LD_1)+ \Sigma^\PP(\LD_2) \lesssim \delta_0\,\Theta(R)^2\,\mu(R) = \delta_0\,\sigma(R).
\end{equation}
Again by Lemma \ref{lempoiss}, we have $\PP(Q)\lesssim\Theta(R)$ for all $Q\in\sM_\Neg$ and thus\footnote{We will obtain better estimates for $\Sigma^\PP(\sM_\Neg)$ below.}
$$\Sigma^\PP(\sM_\Neg)\lesssim \sigma(R).$$
Also, since $\TT$ is a tractable tree, we have
$$\sigma(R)\leq B\,\sigma(\HD_1)\quad \text{ and }\quad
\Sigma^\PP(\HD_2)\approx \sigma(\HD_2) \lesssim B\,\sigma(\HD_1).$$

Regarding the family $\OP_1\cup\OP_2$, denote 
$$\sss_\OP = \sss_1(e'(R))\cap\OP(R) \cup \bigcup_{S\in\HD_1} \big(\sss(S)\cap\OP(S)\big),$$
so that the cubes from $P\in\OP_1\cup\OP_2$ are maximal $\PP$-doubling cubes contained in some cube from
$\sss_\OP$. By Lemma \ref{lemdobpp}, if $P\in\OP_1\cup\OP_2$ is contained in $Q\in\sss_\OP$,
then $\PP(P)\lesssim \PP(Q)$. In this case, $\OP(\,\cdot\,) = \NDB(\,\cdot\,)$, and so there exists
$Q'=Q'(Q)\in\DB$ such that $\ell(Q')=\ell(Q)$ and $Q'\subset 20Q$. 
In particular, this implies that $Q'\sim\TT(e'(R))$ and that $\PP(Q)\approx\PP(Q')\approx\Theta(Q')$, and by the definition of $\DB$,
\begin{equation}\label{eqDB99}
\EE_\infty(9Q')\gtrsim M^2\,\sigma(Q').
\end{equation}
Thus,
\begin{align}\label{eqndb945}
\Sigma^\PP(\OP_1\cup\OP_2) & = \sum_{Q\in\sss_\OP} \sum_{P\in\wt\End\cap\DD_\mu(Q)} \PP(P)^2\,\mu(P)\\
& \lesssim \sum_{Q\in\sss_\OP}\PP(Q)^2\,\mu(Q) \approx \sum_{Q\in\sss_\OP}\sigma(Q'(Q))\nonumber\\
& \lesssim \frac1{M^2}\sum_{Q\in\sss_\OP}\EE_\infty(9Q'(Q))\lesssim \frac1{M^2}\sum_{Q'\in\DB:Q'\sim\TT} \EE_\infty(9Q'),\nonumber
\end{align}
where in the last estimate we took into account that for each $Q'\in\DB$ there is at most a bounded number of cubes $Q\in\sss_\OP$ such that $Q'=Q'(Q)$ (possibly depending on $n$ and $A_0$).
Since $\TT$ is a typical tree, by the definition in \rf{defty}, we have
\begin{equation}\label{eqDB09}
\sum_{Q'\in\DB:Q'\sim\TT} \EE_\infty(9Q')\leq \Lambda^{\frac{-1}{3n}}\,\sigma(\HD_1).
\end{equation}
So we have
\begin{equation}\label{eqalighs33}
\Sigma^\PP(\OP_1\cup\OP_2)\lesssim M^{-2}\Lambda^{\frac{-1}{3n}}\,\sigma(\HD_1).
\end{equation}
Gathering the estimates above, and recalling that $B=\Lambda^{\frac{1}{100n}}$, the estimate \rf{eqlem*1} follows.

\vv

To prove \rf{eqlem*2}, we apply Lemma \ref{lemenereg} and the fact that $\TT$ is a typical tree again:
\begin{align*}
\Sigma^\PP(\Reg_\DB) & =\sum_{Q\in\wt\End\cap\DB} \Sigma^\PP(\Reg_\DB\cap \DD_\mu(Q)) \\
& \lesssim
\sum_{Q\in\wt\End\cap\DB} \EE(9Q)\lesssim \sum_{Q\in\wt\End\cap\DB} \EE_\infty(9Q).
\end{align*}
Recall that $\wt\End = \End\setminus\Neg \cup \sM_\Neg$. For $Q\in \End$ we have  have $Q\in \TT$, and so in particular $Q\sim\TT$ (see \rf{defsim0}). On the other hand, if $Q\in\sM_\Neg$, then by the construction of the family
$\Reg$ there exists some cube $Q'\in\TT$ such that
\begin{equation}\label{eqnegsim*}
\ell(Q)\leq \ell(Q')\leq A_0^2\,\ell(Q)\quad \text{ and }\quad 20Q'\cap20Q\neq\varnothing.
\end{equation}
Hence, $Q\sim\TT$, and so
\begin{align*}
\sum_{Q\in\wt\End\cap\DB}  \EE_\infty(9Q) & = \sum_{Q\in\End\cap\DB}  \EE_\infty(9Q)
+ \sum_{Q\in\sM_\Neg\cap\DB} \EE_\infty(9Q)\\
& \lesssim \sum_{S\in\DB:S\sim\TT} \EE_\infty(9S)
\leq \Lambda^{\frac{-1}{3n}}\,\sigma(\HD_1),
\end{align*}
taking into account that $\TT$ is a typical tree. This completes the proof of \rf{eqlem*2}.
Notice also that \rf{eqlem*3} follows from \rf{eqlem*2} and the fact that the cubes from $\DB$ satisfy
\rf{eqDB99}.

\vv
Next we turn our attention to \rf{eqlem*4}. By Lemma \ref{lemenereg} we have
\begin{align*}
\Sigma^\PP(\Reg_{\OP_1}\cup\Reg_{\OP_2})& \lesssim \sum_{Q\in \OP_1\cup\OP_2} \EE_\infty(9Q)\\
&\leq \sum_{Q\in \wt\End\cap\DB} \EE_\infty(9Q) + \sum_{Q\in (\OP_1\cup\OP_2)\setminus \DB} \EE_\infty(9Q).
\end{align*}
By \rf{eqlem*2}, the first term on the right hand side above does not exceed $C\Lambda^{\frac{-1}{3n}}\,\sigma(\HD_1)$. Concerning the second term, we use the fact for the cubes $Q\not\in\DB$,
we have $\EE_\infty(9Q)\lesssim M^2\,\sigma(Q)$ together with \rf{eqalighs33}. Then we get
$$\sum_{Q\in (\OP_1\cup\OP_2)\setminus \DB} \EE_\infty(9Q)
\lesssim M^2 \,\sigma(\OP_1\cup\OP_2)\lesssim \Lambda^{\frac{-1}{3n}}\,\sigma(\HD_1).$$

\vv

Regarding the estimate \rf{eqlem*4.5}, by \rf{eqld12} and Lemma
\ref{lemenereg}, we obtain
\begin{align*}
\Sigma^\PP((\Reg_{\LD_1}\cup \Reg_{\LD_2})\setminus\Reg_{\DB}) & \lesssim \sum_{Q\in
(\LD_1\cup \LD_2)\setminus\DB} \EE_\infty(9Q)\lesssim M^2\,\Sigma^\PP(\LD_1\cup \LD_2) \\
&\lesssim M^2\,\delta_0\,\sigma(R)\leq 
B\,M^2\,\delta_0\,\sigma(\HD_1).
\end{align*}
Together with \rf{eqlem*2}, this yields \rf{eqlem*4.5}.

To get \rf{eqlem*5}, we apply H\"older's inequality in the preceding estimate:
\begin{align*}
\Sigma_p^\PP((\Reg_{\LD_1}\cup \Reg_{\LD_2})\setminus\Reg_{\DB}) & \leq 
(\Sigma^\PP((\Reg_{\LD_1}\cup \Reg_{\LD_2})\setminus\Reg_{\DB}))^{\frac p2}\,\mu(e'(R))^{1-\frac p2}\\
& \lesssim (B\,M^2\,\delta_0\,\sigma(\HD_1))^{\frac p2}\,\mu(R)^{1-\frac p2}.
\end{align*}
Observe now that, writing $HD_i=\bigcup_{Q\in\HD_i}Q$,
\begin{equation}\label{eqmuhd1}
\mu(HD_1) = \frac1{\Lambda^2\,\Theta(R)^2}\,\sigma(\HD_1) \geq \frac1{B\,\Lambda^2\,\Theta(R)^2}\,\sigma(R) = \frac1{B\,\Lambda^2}\,\mu(R).
\end{equation}
Thus, using also that $M\leq \Lambda$,
\begin{align}\label{eqld7345}
\Sigma_p^\PP((\Reg_{\LD_1}\cup \Reg_{\LD_2})\setminus\Reg_{\DB}) &
\lesssim (B\,M^2\,\delta_0^{\frac2{n+2}}\sigma(\HD_1))^{\frac p2}\,(B\,\Lambda^2\,\mu(HD_1))^{1-\frac p2} \\
& \leq B\,\Lambda^2\,\delta_0^{\frac p{n+2}}\sigma_p(\HD_1).\nonumber
\end{align}
On the other hand, if we let $\Reg_{\DB_2}$ be the subfamily of the cubes from $\Reg_{\DB}$ which
are contained in some cube from $\HD_1$, by H\"older's inequality, Lemma
\ref{lemenereg}, and \rf{eqDB09}, we get
\begin{align*}
\Sigma_p^\PP(\Reg_{\DB_2}) &\lesssim \Sigma^\PP(\Reg_{\DB_2})^{\frac p2}\,\mu(HD_1)^{1-\frac p2}\\
&\lesssim \bigg(\sum_{Q'\in\DB:Q'\sim\TT} \EE_\infty(9Q')\bigg)^{\frac p2}\,\mu(HD_1)^{1-\frac p2}\\
& \lesssim \big(\Lambda^{\frac{-1}{3n}}\,\sigma(\HD_1)\big)^{\frac p2}\,\mu(HD_1)^{1-\frac p2}= \Lambda^{\frac{-p}{6n}}\,\sigma_p(\HD_1).
\end{align*}
Gathering \rf{eqld7345} and the last estimate, we get \rf{eqlem*5}.

To prove \rf{eqlem*6}, recall that $\Theta(Q) \leq\Lambda^2\,\Theta(R)$ for all $Q\in\TT$.
This implies that also $\PP(Q) \lesssim\Lambda^2\,\Theta(R)$ for all $Q\in\Reg$, by Lemma \ref{lemdobpp}. Consequently,
$$\Sigma_p^\PP(\Reg_{\HD_2})\lesssim \Lambda^{2p}\,\Theta(R)^p\,\mu(HD_2) = \sigma_p(\HD_2) \approx
\Sigma_p^\PP(\HD_2).$$
On the other hand, since $R\in\Trc$,
\begin{align*}
\sigma_p(\HD_2) & = \Theta(\HD_2)^{p-2}\,\sigma(\HD_2)\leq B\,\Theta(\HD_2)^{p-2}\,\sigma(\HD_1)
\\ &= B\,\frac{\Theta(\HD_2)^{p-2}}{\Theta(\HD_1)^{p-2}}\,\sigma_p(\HD_1) = B\,\Lambda^{p-2}\,\sigma_p(\HD_1),
 \end{align*}
 which completes the proof of \rf{eqlem*6}.
 
Finally, observe that \rf{eqlem*7} follows from \rf{eqlem*4} using H\"older's inequality and the fact that
the cubes from $\Reg_{\OP_2}$ are contained in cubes from $\HD_1$:
$$\Sigma_p^\PP(\Reg_{\OP_2}) \leq \Sigma^\PP(\Reg_{\OP_2})^{\frac p2}\mu(HD_1)^{1-\frac p2} \lesssim \big(\Lambda^{\frac{-1}{3n}}\sigma(\HD_1)\big)^{\frac p2}\mu(HD_1)^{1-\frac p2}
=  \Lambda^{\frac{-p}{6n}}\,\sigma_p(\HD_1).$$
\end{proof}

\vv
For the record, notice that \rf{eqndb945} also shows that the family $\sss_\OP$ defined just above
\rf{eqDB99} satisfies
$$\Sigma^\PP(\sss_\OP)= \sum_{Q\in\sss_\OP}\PP(Q)^2\,\mu(Q) \lesssim \frac1{M^2}\sum_{Q'\in\DB:Q'\sim\TT} \EE_\infty(9Q'),$$
which together with \rf{eqDB09} yields
\begin{equation}\label{eqsssndb6}
\Sigma^\PP(\sss_\OP)\lesssim M^{-2}\Lambda^{\frac{-1}{3n}}\,\sigma(\HD_1).
\end{equation}

\vv

\begin{lemma}\label{lemneg3}
Suppose that $R\in\Trc\cap\Ty$. Then
\begin{equation}\label{eqlemneg01}
\Sigma^\PP(\sM_\Neg)\lesssim\Sigma^\PP(\Neg) \lesssim 
\big(\delta_0^2\,B 
+M^{-2}\Lambda^{\frac{-1}{3n}}\big)\,\sigma(\HD_1),
\end{equation}
and
\begin{equation}\label{eqlemneg02}
\Sigma^\PP(\Reg_\Neg) \lesssim \Sigma^\PP(\Neg) + \sum_{Q\in\sM_\Neg} \EE_\infty(9Q)\lesssim \big(\delta_0^2\,B\,\Lambda^2 
+\Lambda^{\frac{-1}{3n}}\big)\,\sigma(\HD_1).
\end{equation}
\end{lemma}

\begin{proof}
Recall that $\Neg= \End(e'(R))\cap \Neg(e'(R)) = \Stop_2(e'(R))\cap \Neg(e'(R)),$ by the definition of $\End(e'(R))$. By Lemma \ref{lemnegs}, the cubes from $\Neg$ are not contained in any cube from $\HD_1$. Consequently, each cube from $\Neg$ belongs either to $\LD(R)$ or to $\OP(R)$, and in the latter case $\OP(R)=\NDB(R)$. Thus, 
if we denote
$$\Neg_\LD = \Neg \cap\LD(R)\cap\sss(e'(R)),\quad\;\Neg_\OP = \Neg \cap\OP(R)\cap\sss(e'(R)),$$
it is clear that
$$\Neg = \Neg_\LD \cup \Neg_\OP.$$
We can split $\Reg_\Neg$ in an analogous way:
$$\Reg_\Neg = \Reg_{\Neg_\LD} \cup \Reg_{\Neg_\OP}.$$
Recall that the cubes from $\Neg$ are not $\PP$-doubling and that $\PP(Q)\lesssim \left(\frac{\ell(Q)}{\ell(R)}\right)^{1/3}\,\Theta(R)$ for all $Q\in\Neg$.

Regarding $\Sigma^\PP(\Neg_\LD)$, from the definition of $\LD(R)$ we deduce that
$$\Sigma^\PP(\Neg_\LD)\leq \Sigma^\PP(\LD(R)\cap \sss(e'(R)) \leq \delta_0^2\,\Theta(R)^2\,\mu(e'(R)) \approx\delta_0^2\,\sigma(R)\le \delta_0^2 B\,\sigma(\HD_1),$$
where in the last estimate we used the $\MDW$ condition.

To estimate $\Sigma^\PP(\Neg_\OP)$ we just take into account that $\Neg_\OP\subset \sss_\OP$,
 and then by \rf{eqsssndb6} we have
$$\Sigma^\PP(\Neg_\OP)\leq
\Sigma^\PP(\sss_\OP)\lesssim M^{-2}\Lambda^{\frac{-1}{3n}}\,\sigma(\HD_1).
$$
Gathering the estimates obtained for $\Sigma^\PP(\Neg_\LD)$ and $\Sigma^\PP(\Neg_\OP)$ we get
$$\Sigma^\PP(\Neg) \lesssim \big(\delta_0^2\,B 
+M^{-2}\Lambda^{\frac{-1}{3n}}\big)\,\sigma(\HD_1),$$
as wished.
The fact that $\Sigma^\PP(\sM_\Neg)\lesssim \Sigma^\PP(\Neg)$ follows from the fact that if $P\in\sM_\Neg$ satisfies $P\subset Q\in \Neg$, then $\PP(P)\lesssim \PP(Q)$ because $P$ is a maximal
$\PP$-doubling cube contained in $Q$.

\vv
Next we deal with \rf{eqlemneg02}.
 We split
$$\Sigma^\PP(\Reg_\Neg) = \Sigma^\PP(\Reg_{\sM_\Neg}) + \Sigma^\PP(\Reg_\Neg\setminus \Reg_{\sM_\Neg}).$$
Notice that, for all $P\in\Reg_\Neg\setminus \Reg_{\sM_\Neg}$ such that $P\subset Q\in\Neg$, by Lemma \ref{lemdobpp}, 
$\PP(P)\lesssim\PP(Q)$. Then it follows that
$$\Sigma^\PP(\Reg_\Neg\setminus \Reg_{\sM_\Neg})\lesssim \Sigma^\PP(\Neg).$$

To estimate $\Sigma^\PP(\Reg_{\sM_\Neg})$, we apply Lemma \ref{lemenereg}:
\begin{align}\label{eq83e}
\Sigma^\PP(\Reg_{\sM_\Neg})  & =\sum_{S\in\sM_\Neg} \Sigma^\PP(\Reg_{\sM_\Neg}\cap\DD_\mu(S))
 \lesssim \sum_{S\in\sM_\Neg} \EE_\infty(9S) \\
 & \leq  \sum_{S\in\sM_\Neg\setminus\DB} \EE_\infty(9S) + \sum_{S\in\DB\cap\wt \End} \EE_\infty(9S).\nonumber
\end{align}
By the definition of $\DB$ and the fact that for $S\in\sM_{\Neg}$ with $S\subset Q\in\Neg$ we have $\Theta(S)\lesssim\PP(Q)$,
we derive
$$\sum_{S\in\sM_\Neg\setminus\DB} \EE_\infty(9S)\lesssim M^2\sum_{S\in\sM_\Neg\setminus\DB} \sigma(S)\lesssim
M^2\,\Sigma^\PP(\Neg).$$
On the other hand, by \rf{eqlem*2},
$$\sum_{S\in\DB\cap\wt \End} \EE_\infty(9S) \leq\Lambda^{\frac{-1}{3n}}\,\sigma(\HD_1).
$$
 Therefore,
$$\Sigma^\PP(\Reg_{\sM_\Neg})\lesssim M^2\,\Sigma^\PP(\Neg) + \Lambda^{\frac{-1}{3n}}\,\sigma(\HD_1).$$

Gathering the estimates above, and taking into account that $M<\Lambda$ (by \rf{eq:LambdadepM}), we deduce
\begin{align*}
\Sigma^\PP(\Reg_\Neg)&\lesssim M^2\,\Sigma^\PP(\Neg) + \Lambda^{\frac{-1}{3n}}\,\sigma(\HD_1)\\
& \lesssim
\big(\delta_0^2\,B\,M^2 
+\Lambda^{\frac{-1}{3n}}\big)\,\sigma(\HD_1) + \Lambda^{\frac{-1}{3n}}\,\sigma(\HD_1)\leq \big(\delta_0^2\,B\,\Lambda^2 
+\Lambda^{\frac{-1}{3n}}\big)\,\sigma(\HD_1).
\end{align*}
\end{proof}
\vv

Next lemma deals with $\Sigma_p^\PP(\Reg_\Ot)$. Below we write $\Reg_\Ot(\ell_0)$ to recall the dependence of the family $\Reg_\Ot$ on the parameter
$\ell_0$ in \rf{eql00*23}. Recall that the set $Z$ was defined in \eqref{eqdef*f}.

\begin{lemma}\label{lemregot}
For $1\leq p\leq2$, we have
\begin{equation}\label{eqgaga34}
\limsup_{\ell_0\to0} \Sigma_p^\PP(\Reg_\Ot(\ell_0)) \lesssim \Lambda^{2p}\,\Theta(R)^p\,\mu(Z).
\end{equation}
Consequently, if $\mu(Z)\leq \ve_Z\,\mu(R)$, then
\begin{equation}\label{eqgaga35}
\limsup_{\ell_0\to0} \Sigma_p^\PP(\Reg_\Ot(\ell_0)) \lesssim B\,\Lambda^4\,\ve_Z\,\sigma_p(\HD_1).
\end{equation}
\end{lemma}

\begin{proof}
If $x\in Q\in\End$ with $\ell(Q)\geq \ell_0$, then 
$$d_{R,\ell_0}(x)\leq \max(\ell_0,\ell(Q)) = \ell(Q)$$
(recall that $d_{R,\ell_0}$ is defined in \rf{eql00*23}), and thus $x$ is contained in some cube $Q'\in\Reg$ with $\ell(Q')\leq \ell(Q)$, by the definition of the family $\Reg$. So $Q'\subset Q$ and then $Q'\in\Reg\setminus \Reg_\Ot$. Therefore,
\begin{equation}\label{eqregotin}
\bigcup_{P\in\Reg_\Ot} P \subset e'(R) \setminus \bigcup_{Q\in\End:\ell(Q)>\ell_0} Q,
\end{equation}
and so
\begin{equation}\label{eqlimot62}
\limsup_{\ell_0\to0} \mu\bigg(\bigcup_{P\in\Reg_\Ot(\ell_0)} P\bigg) \leq \mu\bigg(e'(R) \setminus \bigcup_{Q\in\End} Q\bigg) = \mu(Z).
\end{equation}
To complete the proof of \rf{eqgaga34} it just remains to notice that, using the fact that $\PP(P)\lesssim \Lambda^2\,\Theta(R)$ for all $P\in\Reg$, we have
$$\Sigma_p^\PP(\Reg_\Ot(\ell_0)) \lesssim \Lambda^{2p}\,\Theta(R)^p \,\mu\bigg(\bigcup_{P\in\Reg_\Ot(\ell_0)} P\bigg).$$

Regarding the second statement of the lemma recall that, as in \rf{eqmuhd1},
$\mu(HD_1) \geq  \frac1{B\,\Lambda^2}\,\mu(R)$, which implies that
$$\mu(Z) \leq B\,\Lambda^2\,\ve_Z\,\mu(HD_1).$$
Plugging this estimate into \rf{eqgaga34}, we get 
$$\limsup_{\ell_0\to0} \Sigma_p^\PP(\Reg_\Ot(\ell_0)) \lesssim B\,\Lambda^{2+p}\,\ve_Z\,\Theta(\HD_1)^p\,\mu(HD_1) \leq B\,\Lambda^4\,\ve_Z\,\sigma_p(\HD_1).$$
\end{proof}

\vv

\begin{rem}
By Lemma \ref{lemregmolt}, for $p\in (1,2)$, if $\Lambda\gg B\gg1$ and
$\delta_0\ll\Lambda^{-1}$, we have
$$\Sigma_p^\PP(\Reg_{\LD_1}\setminus\Reg_{\DB}) + \Sigma_p^\PP(\Reg_{\LD_2}) +
\Sigma_p^\PP(\Reg_{\HD_2}) +
\Sigma_p^\PP(\Reg_{\OP_2})\ll \sigma_p(\HD_1).$$
Analogously, by Lemma \ref{lemregot}, if $\ve_Z\ll (B\Lambda^4)^{-1}$ and $\ell_0$ is small enough,
$$\Sigma_p^\PP(\Reg_{\Ot})\ll \sigma_p(\HD_1).$$
To transfer the lower estimates for the $L^p(\eta)$ norm of $\RR\eta$ in Lemma
\ref{lemrieszeta} to $\RR_{\TT}\mu$, $\RR_{\TT_\Reg}\mu$, and $\Delta_{\TT}\RR\mu$, it would be useful to have estimates for $\Sigma_p^\PP(\Reg_{\OP_1})$, $\Sigma_p^\PP(\Reg_{\DB})$, and $\Sigma_p^\PP(\Neg)$ analogous to the ones above. However, we have not been able to get them. 
This fact originates some technical difficulties, which will be solved in the next lemmas.
\end{rem}
\vv

\begin{rem}\label{rem9.12}
By Lemma \ref{lemregmolt}, for $1<p\leq 2$, and $\ell_0$ small enough we have
\begin{multline*}
\Sigma_p^\PP(\Reg_{\LD_1}\setminus\Reg_{\DB}) + \Sigma_p^\PP(\Reg_{\LD_2}) +
\Sigma_p^\PP(\Reg_{\OP_2}) 
\lesssim 
\Big(B\,\Lambda^2\,\delta_0^{\frac 1{2}} + \Lambda^{\frac{-1}{6n}}\Big)\,
\sigma_p(\HD_1).
\end{multline*}
Also, by Lemma \ref{lemregot}, for $\ell_0$ small enough,  if
$\mu(Z)\leq \ve_Z\,\mu(R)$, 
$$\Sigma_p^\PP(\Reg_\Ot(\ell_0))\lesssim 
B\,\Lambda^4\,\ve_Z\,\sigma_p(\HD_1).$$
From now on we will assume that $\ell_0$ small enough so that this holds. Remember also that we chose
$$B=\Lambda^{\frac1{100n}}.$$
We assume also $\delta_0$ and $\ve_Z$ small enough so that
\begin{equation}\label{eqassu78}
B\,\Lambda^2\,\delta_0^{\frac 1{2}} + B\,\Lambda^4\,\ve_Z + B\,\Lambda^2\,\delta_0\leq \Lambda^{\frac{-1}{3n}}.
\end{equation}
In this way, we have
\begin{align}\label{eqtot999}
\Sigma_p^\PP(\Reg_{\LD_1}\setminus\Reg_{\DB}) + \Sigma_p^\PP(\Reg_{\LD_2}) +
\Sigma_p^\PP(\Reg_{\OP_2})  + \Sigma_p^\PP(\Reg_\Ot)\lesssim 
 \Lambda^{\frac{-1}{6n}} \,
\sigma_p(\HD_1),
\end{align}
under the assumption that $\mu(Z)\leq \ve_Z\,\mu(R)$. Also, from \rf{eqlemneg02}, it follows that
\begin{equation}\label{eqassu78699}
\Sigma^\PP(\Reg_\Neg) \lesssim \big(\delta_0\,B\,\Lambda^2 
+\Lambda^{\frac{-1}{3n}}\big)\,\sigma(\HD_1)\lesssim \Lambda^{\frac{-1}{3n}} \,
\sigma(\HD_1).
\end{equation}
\end{rem}

\vv

\begin{lemma}\label{lemreg73}
Suppose that $R\in\Trc\cap\Ty$ and that $\mu(Z)\leq \ve_Z\,\mu(R)$.
We have
$$\Sigma^\PP(\Reg)\lesssim B\,\sigma(\HD_1).$$
For $p_0 = 2- \frac1{18n}$, we have
$$\Sigma_{p_0}^\PP(\Reg)\lesssim \Lambda^{\frac{-1}{25n}}\,\sigma_{p_0}(\HD_1).$$
\end{lemma}

\begin{proof}
By \rf{eqtot999} and \rf{eqlem*6}, for $1<p\leq 2$,
\begin{multline*}
\Sigma_p^\PP(\Reg_{\LD_1}\setminus\Reg_{\DB}) + \Sigma_p^\PP(\Reg_{\LD_2}) +
\Sigma_p^\PP(\Reg_{\OP_2}) + \Sigma_p^\PP(\Reg_{\HD_2})+ \Sigma_p^\PP(\Reg_\Ot)\\
\lesssim 
\Big(\Lambda^{\frac{-1}{6n}} + B\,\Lambda^{p-2}\Big)\,
\sigma_p(\HD_1).
\end{multline*}
Also, by \rf{eqlem*2}, \rf{eqlem*4}, and \rf{eqassu78699},
$$\Sigma^\PP(\Reg_\DB) + \Sigma^\PP(\Reg_{\OP_1}) + \Sigma^\PP(\Reg_{\Neg}) \lesssim 
\big(\delta_0\,B\,\Lambda^2
+\Lambda^{\frac{-1}{3n}}\big)\,\sigma(\HD_1)\lesssim
\Lambda^{\frac{-1}{3n}}\,\sigma(\HD_1),
$$
by the assumptions on $\delta_0$, $\Lambda$, and $B$ in Remarks \ref{rem9.7} and \ref{rem9.12}.
Choosing $p=2$ above and adding the preceding estimates, we get the first statement in the lemma.

To get the second inequality in the lemma, we apply H\"older's inequality and \rf{eqmuhd1} in the last estimate, and we get
\begin{align}\label{eqsam438}
\Sigma_p^\PP(\Reg_\DB) + \Sigma_p^\PP(\Reg_{\OP_1}) + \Sigma_p^\PP(\Reg_{\Neg}) &\lesssim \big(\Lambda^{\frac{-1}{3n}}\,\sigma(\HD_1)\big)^{\frac p2}\,\mu(R)^{1-\frac p2}\\
& \lesssim
\big(\Lambda^{\frac{-1}{3n}}\,\sigma(\HD_1)\big)^{\frac p2}\,\big( B\,\Lambda^2\mu(HD_1)\big)^{1-\frac p2}\nonumber\\
& \leq \Lambda^{\frac{-1}{6n}}\,\big(\Lambda^3\big)^{1-\frac p2}\,\sigma_p(\HD_1).\nonumber
\end{align}
For $p=2-\frac1{18n}$, we have 
$$\big(\Lambda^3\big)^{1-\frac p2} = \Lambda^{\frac1{12n}},$$
and thus
\begin{equation}\label{eqsam439}
\Sigma_p^\PP(\Reg_\DB) + \Sigma_p^\PP(\Reg_{\OP_1})+\Sigma_p^\PP(\Reg_{\Neg}) \lesssim \Lambda^{\frac{-1}{12n}}\,\sigma_p(\HD_1).
\end{equation}
Also,
$$B\,\Lambda^{p-2} = \Lambda^{\frac1{100n}}\,\Lambda^{\frac{-1}{18n}} = \Lambda^{\frac{-41 }{900n}}
< \Lambda^{\frac{-1 }{25n}},$$
and thus
\begin{multline*}
\Sigma_p^\PP(\Reg_{\LD_1}\setminus\Reg_{\DB}) + \Sigma_p^\PP(\Reg_{\LD_2}) +
\Sigma_p^\PP(\Reg_{\OP_2}) + \Sigma_p^\PP(\Reg_{\HD_2})+ \Sigma_p^\PP(\Reg_\Ot)\\
\lesssim \Lambda^{\frac{-1 }{25n}}\,
\sigma_p(\HD_1).
\end{multline*}
Gathering the estimates above, the lemma follows.
\end{proof}

\vv
Remark that in order to transfer the lower estimates for $\RR\eta$ to $\RR\mu$ we will need
to take $p$ very close to $1$ below, in particular with $p<p_0$.

Next lemma shows how one can estimate the $\QQ_\Reg$ coefficients in terms of the $\PP$ coefficients.
\vv

\begin{lemma}\label{lemregpq}
For all $p\in(1,\infty)$,
$$\Sigma_p^\QQ(\Reg)\lesssim \Sigma_p^\PP(\Reg).$$
\end{lemma}

%Remark that the lemma does not assert that $\Sigma_p^\QQ(I)\lesssim\Sigma_p^\PP(I)$ for any family $I\subset \Reg$. 
%As far as we know this is not correct.

\begin{proof}
By duality,
\begin{equation}\label{eqdual912}
\Sigma_p^\QQ(\Reg)^{1/p} = \bigg(\sum_{Q\in\Reg}\QQ_\Reg(Q)^p\,\mu(Q)\bigg)^{1/p}
= \sup \sum_{Q\in\Reg} \QQ_\Reg(Q)\,g_Q\,\mu(Q),
\end{equation}
where the supremum is taken over all sequences $g=\{g_Q\}_{Q\in\Reg}$ such that 
\begin{equation}\label{eqdual912b}
\sum_{Q\in\Reg} |g_Q|^{p'}\,\mu(Q)\leq1.
\end{equation}
We will identify the sequence $g$ with the function 
$$\wt g= \sum_{Q\in\Reg}g_Q\,\chi_Q,$$
so that the sum in \rf{eqdual912b} equals $\|\wt g \|_{L^{p'}(\mu)}^{p'}$.
 By the definition of $\QQ_\Reg$ and Fubini we have
\begin{align}\label{eqdjkq44}
\sum_{Q\in\Reg} \QQ_\Reg(Q)\,g_Q\,\mu(Q) &= \sum_{Q\in\Reg} \sum_{P\in\Reg} \frac{\ell(P)}{D(P,Q)^{n+1}}\,\mu(P)\,g_Q\,\mu(Q)\\
& = \sum_{P\in\Reg} \bigg( \sum_{Q\in\Reg} \frac{\ell(P)}{D(P,Q)^{n+1}}\,g_Q\,\mu(Q)\bigg)\,\mu(P).
\nonumber
\end{align}
For each $P\in\Reg$, we have
\begin{align}\label{eqalg539}
\sum_{Q\in\Reg} \frac{\ell(P)}{D(P,Q)^{n+1}}\,|g_Q|\,\mu(Q) & = \sum_{j\geq0}\;
\sum_{Q\in\Reg:2^j\ell(P)\leq D(P,Q)\leq 2^{j+1}\ell(P)} \frac{\ell(P)}{D(P,Q)^{n+1}}\,|g_Q|\,\mu(Q)\\
& \leq \sum_{j\geq0}\;
\sum_{Q\in\Reg:D(P,Q)\leq 2^{j+1}\ell(P)} \frac{2^{-j}}{(2^j\ell(P))^{n}}\,|g_Q|\,\mu(Q).\nonumber
\end{align}

Observe now that the condition
$\ell(Q) + \dist(P,Q)\leq D(P,Q)\leq 2^{j+1}\ell(P),$
implies that 
$Q\subset B(x_Q,\ell(Q))\subset B(x_P,2^{j+3}\ell(P)).$
%the cubes $Q\in\Reg$ such that $D(P,Q)\leq 2^{j+1}\ell(P)$ are contained in 
%$B(x_P,C2^j\ell(P))$ for some absolute constant $C$. Indeed, the condition
%$$\ell(Q) + \dist(P,Q)\leq D(P,Q)\leq 2^{j+1}\ell(P),$$
%implies that 
%$$Q\subset B(x_Q,\ell(Q))\subset B(x_P,2^{j+3}\ell(P)).$$
%In case that 
%$$B(x_P,3\ell(P))\cap B(x_Q,3\ell(Q))\neq\varnothing,$$
%then $\ell(P)\approx \ell(Q)$ by Lemma \ref{lem74} (b), and so we also have
%$$Q\subset B(x_Q,\ell(Q))\subset B(x_P,C\ell(P)),$$
%which concludes the proof of the claim.
From \rf{eqalg539} and this fact, we infer that
\begin{align*}
\sum_{Q\in\Reg} \frac{\ell(P)}{D(P,Q)^{n+1}}\,|g_Q|\,\mu(Q) &
\leq  \sum_{j\geq0}\;
\sum_{Q\in\Reg:Q\subset B(x_P,C2^j\ell(P))} \frac{2^{-j}}{(2^j\ell(P))^{n}}\,|g_Q|\,\mu(Q)\\
& \leq\sum_{j\geq0}
 \frac{2^{-j}}{(2^j\ell(P))^{n}}\,\int_{B(x_P,C2^j\ell(P))}|\wt g|\,d\mu
\end{align*}
Notice now that, for all $x\in P$,
\begin{align*}
\int_{B(x_P,C2^j\ell(P))}|\wt g|\,d\mu & \leq \int_{B(x,C'2^j\ell(P))}|\wt g|\,d\mu\\
&\leq \mu(B(x,C'2^j\ell(P)))\,\cM_\mu \wt g(x)\leq \mu(B(x_P,C''2^j\ell(P)))\,\cM_\mu \wt g(x),
\end{align*}
where $\cM_\mu$ is the centered Hardy-Littlewood maximal operator. Thus,
\begin{align*}
\sum_{Q\in\Reg} \frac{\ell(P)}{D(P,Q)^{n+1}}\,|g_Q|\,\mu(Q)
& \leq\sum_{j\geq0}
 \frac{2^{-j}\mu(B(x_P,C''2^j\ell(P)))}{(2^j\ell(P))^{n}}\,\cM_\mu \wt g(x)\\
 &\approx
 \sum_{k\geq0}
 2^{-k}\theta_\mu(2^kB_P)\,\cM_\mu \wt g(x) \approx \PP(P)\,\,\cM_\mu \wt g(x).
 \end{align*}
Plugging this estimate into \rf{eqdjkq44}, and taking the infimum for $x\in P$, we get
\begin{align*}
\sum_{Q\in\Reg} \QQ_\Reg(Q)\,|g_Q|\,\mu(Q) & \lesssim \sum_{P\in\Reg} \PP(P)\,\inf_{x\in P}\cM_\mu \wt g(x)\,\mu(P)\\
& \leq \bigg(\sum_{P\in\Reg} \PP(P)^p\,\mu(P)\bigg)^{1/p}
\bigg(\sum_{P\in\Reg} \int_P |\cM_\mu \wt g|^{p'}\,d\mu\bigg)^{1/p'}\\
& = \Sigma_p^\PP(\Reg)^{1/p}\,\|\cM_\mu\wt g\|_{L^{p'}(\mu\rest_{e'(R)})}\\
&\lesssim
\Sigma_p^\PP(\Reg)^{1/p}\,\|\wt g\|_{L^{p'}(\mu)}\leq \Sigma_p^\PP(\Reg)^{1/p},
\end{align*}
which concludes the proof of the lemma, by \rf{eqdual912}.
\end{proof}

\vv

% ********************************************************************************************

\subsection{Transference of the lower estimates for $\RR\eta$ to $\Delta_{\wt \TT}\RR\mu$}
Denote
$$F = e'(R) \setminus \bigcup_{Q\in\Reg_{\HD_2}} Q.$$
By the splitting \rf{eqsplitreg0}, it is clear that $F$ coincides with the union of the cubes in the family
$$\Reg_F:=
 \Reg_{\LD_1} \cup \Reg_{\LD_2}\cup \Reg_{\OP_1}\cup \Reg_{\OP_2}\cup \Reg_{\Neg}\cup \Reg_{\Ot}.
$$
We also write
$$F_\eta= \bigcup_{Q\in\Reg_F} \frac12 B(Q),$$
so that the measure $\mu|_F$ is well approximated by $\eta|_{F_\eta}$, in a sense. 

Recall that in Lemma \ref{lemrieszeta} we showed that  
$$\int_{V_4} \big|(|\RR\eta(x)| - \frac{c_0}2\,\Theta(\HD_1))_+\big|^p\,d\eta(x)
 \gtrsim \Lambda^{-(p'+1)\ve_n}\sigma_p(\HD_1(e(R))),$$
for any $p\in (1,2]$, with $c_0$ as in Lemma \ref{lemvar}. 
We will show below that  a similar lower estimate holds if we restrict the integral on the 
left side to $HD_2$. The first step is the next lemma.

We let $\wt V_4$ be the union of the balls $\frac12B(Q)$, with $Q\in\Reg$, that intersect $V_4$.

\vv
\begin{lemma}\label{leminteta*}
Suppose that $R\in\Trc\cap\Ty$ and also that $\mu(Z)\leq \ve_Z\,\mu(R)$ and
$$\|\Delta_\wt\TT \RR\mu\|_{L^2(\mu)}^2\leq \Lambda^{-1}\,\sigma(\HD_1).$$
Then, for $p_0 = 2- \frac1{18n}$, assuming $\ve_Z\leq 
 \Lambda^{-300n}$ and $\ell_0$ small enough,
$$\int_{F_\eta\cap\wt V_4}\big|(|\RR\eta| - \frac{c_0}4 \Theta(\HD_1))_+\big|^{p_0}
\,d\eta\lesssim \Lambda^{\frac{-1}{25n}}\,\sigma_{p_0}(\HD_1).$$
\end{lemma}

Remark that the proof of this lemma takes advantage of the good estimates we have
obtained for $\Sigma^\PP_{p_0}(\Reg)$ in  Lemma \ref{lemreg73}. Later on we will show that an analogous estimate holds for all $p\in (1,p_0)$ (see Lemma \ref{lemNZ} below).

\begin{proof}
Denote
$$F_a= \bigcup_{Q\in\wt \End\setminus \HD_2} Q,\qquad F_b= \bigcup_{Q\in\Reg_\Neg\setminus \Reg_{\sM_\Neg}} Q,\qquad F_\Ot= \bigcup_{Q\in\Reg_\Ot} Q,$$
so that 
$F=F_a\cup F_b\cup F_\Ot$.
Write also 
$$\End_a = \LD_1 \cup \LD_2 \cup \OP_1  \cup \OP_2  \cup\sM_\Neg$$
and
$$\Reg_a = \Reg_{\LD_1}\cup\Reg_{\LD_2}\cup\Reg_{\OP_1}\cup\Reg_{\OP_2}\cup\Reg_{\sM_\Neg},$$
so that
$$F_a= \bigcup_{Q\in \End_a} Q = \bigcup_{Q\in \Reg_a} Q.$$
We also consider
$$F_{a,\eta} = \bigcup_{Q\in \Reg_a} \frac12B(Q),\qquad F_{b,\eta}= \bigcup_{Q\in\Reg_\Neg\setminus \Reg_{\sM_\Neg}} \frac12B(Q),  \qquad F_{\Ot,\eta}= \bigcup_{Q\in\Reg_\Ot} \frac12B(Q),$$
so that these sets approximate $F_a,F_b,F_\Ot$ at the level of the family $\Reg$, in a sense. Moreover, we have $F_\eta=F_{a,\eta}\cup F_{b,\eta}\cup F_{\Ot,\eta}$.

We split
$$\int_{F_\eta\cap\wt V_4}\big|(|\RR\eta| - \frac{c_0}4 \Theta(\HD_1))_+\big|^{p_0}
\,d\eta = \int_{F_{a,\eta}\cap\wt V_4} \ldots +  \int_{F_{b,\eta}\cap\wt V_4} \ldots + \int_{F_{\Ot,\eta}\cap\wt V_4}\ldots =:I_a+ I_b + I_\Ot,$$
where ``$\ldots$'' stands for 
$\big|(|\RR\eta| - \frac{c_0}4 \Theta(\HD_1))_+\big|^{p_0}
\,d\eta$. 

\vv
\noi {\bf Estimates for $I_b$.}
We claim that $I_b=0$. Indeed, given $Q\in\Reg_\Neg\setminus \Reg_{\sM_\Neg}$, notice that all the cubes $S$ such that $Q\subset S\subset R$ satisfy
$$\Theta(S)\lesssim \left(\frac{\ell(S)}{\ell(R)}\right)^{1/2}\,\Theta(R),$$
and so, for any $x\in Q$,
\begin{align*}
|\RR\eta(x)| &\lesssim \sum_{S:Q\subset S\subset R} \theta_\eta(4B_S) \lesssim 
\sum_{S:Q\subset S\subset R} \theta_\mu(CB_S) \\
&\lesssim \sum_{S:Q\subset S\subset R}\left(\frac{\ell(S)}{\ell(R)}\right)^{1/2}\,\Theta(R) + \Theta(R)\lesssim \Theta(R).
\end{align*}
Hence, for $\Lambda$ big enough, $(|\RR\eta(x)| - \frac{c_0}4 \Theta(\HD_1))_+=0$, which proves our claim.

\vv
\noi {\bf Estimates for $I_a$.} 
By Lemma \ref{lemaprox1}, for all $Q\in\Reg_a$ such that $\frac12B(Q)\subset\wt V_4$, all $x\in \frac12B(Q)$, and all $y\in Q$,
$$|\RR\eta(x)| \leq |\RR_{\TT_\Reg}\mu(y)| + C\Theta(R) + C\PP(Q) + C\QQ_\Reg(Q).$$
Thus, for $\Lambda$ big enough, since $C\Theta(R)=C\Lambda^{-1}\Theta(\HD_1)<\frac{c_0}8 \Theta(\HD_1)$,
$$(|\RR\eta(x)| - \frac{c_0}4 \Theta(\HD_1))_+\leq
(|\RR_{\TT_\Reg}\mu(y)| - \frac{c_0}8 \Theta(\HD_1))_+  + C\PP(Q) + C\QQ_\Reg(Q).$$
Consequently,
\begin{align}\label{eqalgjx2}
I_a & \lesssim \sum_{Q\in\Reg_a} \int_Q \big|(|\RR_{\TT_\Reg}\mu(y)| - 
\frac{c_0}8 \Theta(\HD_1))_+\big|^{p_0}\,d\mu(y) + \!\sum_{Q\in\Reg_a} (\PP(Q)^{p_0} + \QQ_\Reg(Q)^{p_0})\,\mu(Q)\\
& \lesssim \sum_{S\in\End_a}\int_S \big|(|\RR_{\TT_\Reg}\mu| - 
\frac{c_0}8 \Theta(\HD_1))_+\big|^{p_0}\,d\mu + \Sigma_{p_0}^\PP(\Reg),\nonumber
\end{align}
where, in the last inequality, we used the fact that all the cubes from $\Reg_a$
are contained in some cube $S\in\End_a$ and we applied Lemma \ref{lemregpq}.

For each $S\in\End_a$, by the triangle inequality and the fact that $(\;\cdot\;)_+$ is a $1$-Lipschitz function, we get 
\begin{align*}
\int_S \big|(|\RR_{\TT_\Reg}\mu| - 
\frac{c_0}8 \Theta(\HD_1))_+\big|^{p_0}\,d\mu &\lesssim 
\int_S \big|(|\RR_{\wt \TT}\mu| - 
\frac{c_0}8 \Theta(\HD_1))_+\big|^{p_0}\,d\mu \\
&\quad+ \int_S\big|\RR_{\wt\TT}\mu - \RR_{\TT_\Reg}\mu\big|^{p_0}\,d\mu
\end{align*}
By Lemma \ref{lemaprox3}, the last integral does not exceed 
$C\EE(2S)^{\frac{p_0}2} \,\mu(S)^{1-\frac{p_0}2}$, and thus we deduce that
\begin{equation}\label{eqIa1}
I_a\lesssim
\sum_{S\in\End_a}\int_S \big|(|\RR_{\wt\TT}\mu| - 
\frac{c_0}8 \Theta(\HD_1))_+\big|^{p_0}\,d\mu + \sum_{S\in\End_a}\EE(2S)^{\frac{p_0}2} \,\mu(S)^{1-\frac {p_0}2} + \Sigma_{p_0}^\PP(\Reg).
\end{equation}

Next we apply Lemma \ref{lemaprox2}, which ensures that
for any $S\in\End_a\subset\wt\End$ and all $x\in S$,
\begin{equation}\label{eqIa2}
|\RR_{\wt\TT}\mu(x)| \leq |\Delta_{\wt\TT}\RR\mu(x)| + C\left(\frac{\EE(4R)}{\mu(R)}\right)^{1/2} +  C\left(\frac{\EE(2S)}{\mu(S)}\right)^{1/2}.
\end{equation}
In case that $R\not\in\DB$, recalling that
$\Lambda  \ge \max(M^{\frac{8n-1}{8n-2}},CM)\gg M$ by \rf{eq:LambdadepM}, we obtain
\begin{equation}\label{eqEr4}
C\left(\frac{\EE(4R)}{\mu(R)}\right)^{1/2} \leq C\,M\,\Theta(R)\leq \frac{c_0}8 \Theta(\HD_1).
\end{equation}
In case that $R\in\DB$, since $R\in\Ty$, we have 
$$
\EE(4R)\lesssim
\sum_{Q\in\DB:Q\sim\TT} \EE_\infty(9Q)\leq \Lambda^{\frac{-1}{3n}}\,\sigma(\HD_1),
$$
and so we also get
\begin{equation}\label{eqEr5}
C\left(\frac{\EE(4R)}{\mu(R)}\right)^{1/2} \leq C \left(\frac{\Lambda^{\frac{-1}{3n}}\,\sigma(\HD_1)}{\mu(R)}\right)^{1/2}\leq C\, \Lambda^{\frac{-1}{6n}} \,\Theta(\HD_1)
\leq \frac{c_0}8 \Theta(\HD_1).
\end{equation}
Thus, in any case,
$$(|\RR_{\wt\TT}\mu(x)| - \frac{c_0}8 \Theta(\HD_1))_+ \leq 
|\Delta_{\wt\TT}\RR\mu(x)| + C\left(\frac{\EE(2S)}{\mu(S)}\right)^{1/2}.$$
Plugging this estimate into \rf{eqIa1}, we get
\begin{equation}\label{eqIa99}
I_a\lesssim
\sum_{S\in\End_a}\int_S |\Delta_{\wt\TT}\RR\mu|^{p_0}\,d\mu + \sum_{S\in\End_a}\EE(2S)^{\frac {p_0}2} \,\mu(S)^{1-\frac {p_0}2} + \Sigma_{p_0}^\PP(\Reg).
\end{equation}

We deal with each term on the right hand side of the preceding inequality separately. First, by H\"older's inequality and the assumptions in the lemma, we have
$$\int |\Delta_{\wt\TT}\RR\mu|^{p_0}\,d\mu \lesssim \|\Delta_{\wt\TT}\RR\mu\|_{L^2(\mu)}^{p_0}\,
\mu(R)^{1-\frac{p_0}2}\leq \Lambda^{-\frac{p_0}2}\,\sigma(\HD_1)^{\frac{p_0}2}\,\mu(R)^{1-\frac{p_0}2}.$$
Regarding the second term in \rf{eqIa99}, by H\"older's inequality again,
\begin{align*}
\sum_{S\in\End_a}\EE(2S)^{\frac{p_0}2} \,\mu(S)^{1-\frac {p_0}2} 
& \leq \bigg(\sum_{S\in\End_a}\EE(2S)\bigg)^{\frac{p_0}2} \,\bigg(\sum_{S\in\End_a} \mu(S)\bigg)^{1-\frac {p_0}2} \\& \lesssim  \bigg(\sum_{S\in\End_a}\EE(2S)\bigg)^{\frac{p_0}2} \,\mu(R)^{1-\frac {p_0}2}.
\end{align*}
We estimate the first factor on the right hand side using \rf{eqlem*4}, \rf{eqlem*4.5}, and \rf{eqlemneg02}:
\begin{align*}
\bigg(\sum_{S\in\End_a}\EE(2S)\bigg)^{\frac{p_0}2} & \leq \bigg(\sum_{S\in\OP_1\cup\OP_2}\EE_\infty(9S) +
\sum_{S\in
\LD_1\cup \LD_2} \EE_\infty(9S) + \sum_{S\in\sM_\Neg} \EE_\infty(9S)\bigg)^{\frac{p_0}2} \\
& \lesssim \big(B\Lambda^{-1}\!+\Lambda^{\frac{-1}{3n}} + B\,M^2\,\delta_0^{\frac2{n+2}} + B\,\Lambda^7\delta_0\big)^{\frac{p_0}2}\,\sigma(\HD_1)^{\frac{p_0}2}
\leq \Lambda^{\frac{-1}{6n}}\,\sigma(\HD_1)^{\frac{p_0}2},
\end{align*}
by the assumption \rf{eqassu78} on $\delta_0$.
In connection with the last summand on the right hand side of \rf{eqIa99}, by Lemma \ref{lemreg73}
we have
$$\Sigma_{p_0}^\PP(\Reg)\lesssim \Lambda^{\frac{-1}{25n}}\,\sigma_{p_0}(\HD_1).$$
Therefore,
$$I_a\lesssim \Lambda^{\frac{-1}{6n}}\,\sigma(\HD_1)^{\frac{p_0}2}\,\mu(R)^{1-\frac{p_0}2} + \Lambda^{\frac{-1}{25n}}\,\sigma_{p_0}(\HD_1).$$
From \rf{eqmuhd1} we derive that 
$\mu(HD_1) \geq\Lambda^{-3}\,\mu(R)$, and then
\begin{equation}\label{eqhd1*p}
 \Lambda^{\frac{-1}{6n}}\,\sigma(\HD_1)^{\frac{p_0}2}\,\mu(R)^{1-\frac{p_0}2} 
\leq \Lambda^{\frac{-1}{6n}}\,\big(\Lambda^3\big)^{1-\frac {p_0}2}\,\sigma_{p_0}(\HD_1)=
\Lambda^{\frac{-1}{12n}}\,\sigma_{p_0}(\HD_1)\leq \Lambda^{\frac{-1}{25n}}\,\sigma_{p_0}(\HD_1).
\end{equation}
So we get
\begin{equation}\label{eqIa93}
I_a\lesssim \Lambda^{\frac{-1}{25n}}\,\sigma_{p_0}(\HD_1).
\end{equation}
\vv

\noi {\bf Estimate of $I_\Ot$.} 
By the same arguments as in \rf{eqalgjx2}, just replacing $\Reg_a$ by $\Reg_\Ot$, we obtain
$$I_\Ot  \lesssim \sum_{Q\in\Reg_\Ot} \int_Q \big|(|\RR_{\TT_\Reg}\mu(y)| - 
\frac{c_0}8 \Theta(\HD_1))_+\big|^{p_0}\,d\mu(y) + \!\sum_{Q\in\Reg_\Ot} (\PP(Q)^{p_0} + \QQ_\Reg(Q)^{p_0})\,\mu(Q).$$
Thus, using again Lemmas \ref{lemregpq} and \ref{lemreg73}, we get
\begin{align}\label{eqiot78}
I_\Ot  &\lesssim \sum_{Q\in\Reg_\Ot} \int_Q \big|(|\RR_{\TT_\Reg}\mu(y)| - 
\frac{c_0}8 \Theta(\HD_1))_+\big|^{p_0}\,d\mu(y) +\Lambda^{\frac{-1}{25n}}\,\sigma_{p_0}(\HD_1)\\
& =: \wt I_\Ot +
\Lambda^{\frac{-1}{25n}}\,\sigma_{p_0}(\HD_1).\nonumber
\end{align}
To estimate the integral $\wt I_\Ot$ on the right hand side, we split
\begin{align*}
\wt I_\Ot & =\sum_{Q\in\Reg_{\Ot}\setminus \Neg(e'(R))}
\int_Q \big|(|\RR_{\TT_\Reg}\mu(x)| - 
\frac{c_0}8 \Theta(\HD_1))_+\big|^{p_0}\,d\mu(x)\\
&\quad+ \sum_{Q\in\Reg_{\Ot}\cap \Neg(e'(R))}
\int_Q \big|(|\RR_{\TT_\Reg}\mu(x)| - 
\frac{c_0}8 \Theta(\HD_1))_+\big|^{p_0}\,d\mu(x)\\
& = \wt I_{\Ot,1} + \wt I_{\Ot,2}.
\end{align*}
Notice that, by definition we have $\Reg_\Ot\subset\TT$. 
Recall that $\Neg = \Neg(e'(R))\cap \End$ and thus we may have $\Reg_{\Ot}\cap \Neg(e'(R))\neq\varnothing$.
In this case, we have $\Reg_{\Ot}\cap \Neg(e'(R))\subset \TT_\sss(e'(R))$ (because $\Neg(e'(R))\subset \TT_\sss(e'(R))$ by construction).

The same argument used to show that $I_b=0$ shows
that 
$$\wt I_{\Ot,2} = 0.$$
To estimate $\wt I_{\Ot,1}$,
denote by $\sM_\Ot$ the family of maximal $\PP$-doubling cubes which are contained in some cube from $\Reg_{\Ot}\setminus \Neg(e'(R))$
and let
$$N_\Ot = \bigcup_{Q\in \Reg_{\Ot}\setminus \Neg(e'(R))} Q\setminus \bigcup_{P\in \sM_{\Ot}} P.$$
We claim that 
\begin{equation}\label{eqinclu827}
\sM_\Ot\subset \TT\quad \text{ and }\quad N_\Ot\subset Z.
\end{equation}
To check this, for a given $P\in\sM_\Ot$ with $P\subset Q\in\Reg_{\Ot}\setminus \Neg(e'(R))$, suppose there exists $S\in\End$ such that $S\supset P$. As $P$ is a contained in some $Q\in\Reg_{\Ot}\setminus \Neg(e'(R))$, we  have
$S\not\in \Neg$. Further, $S\subsetneq Q$ because $Q\in\Reg_{\Ot}$ implies that $Q\not\subset S$. Since $S$ is $\PP$-doubling, we deduce that $P=S$, by the maximality of $P$ as $\PP$-doubling cube contained in $Q$. An analogous argument shows that $N_\Ot\subset Z$.

By H\"older's inequality and \rf{eqlimot62}, for $\ell_0$ small enough we have
\begin{align*}
\wt I_\Ot &\leq \bigg(\sum_{Q\in\Reg_{\Ot}} \int_Q \big|(|\RR_{\TT_\Reg}\mu(x)| - 
\frac{c_0}8 \Theta(\HD_1))_+\big|^{2}\,d\mu(x)\big)^{\frac {p_0}2} \bigg(\sum_{Q\in\Reg_{\Ot}}\mu(Q)\bigg)^{1-\frac {p_0}2}\\
& \leq
\bigg(\sum_{P\in\sM_{\Ot}} \int_P \big|\RR_{\TT_\Reg}\mu\big|^{2}\,d\mu + \int_{N_\Ot} \big|\RR_{\TT_\Reg}\mu\big|^{2}\,d\mu
\bigg)^{\frac {p_0}2} \bigg(\mu(Z) + o(\ell_0)\bigg)^{1-\frac {p_0}2},
\end{align*}
with $o(\ell_0)\to 0$ as $\ell_0\to0$. 

Denote
$$\RR_{\sM_\Ot}\mu(x) = \sum_{P\in\sM_\Ot} \chi_P(x)\,\RR(\chi_{2R\setminus 2P}\mu)(x)$$
and
$$\Delta_{\sM_\Ot}\RR\mu(x) =
\sum_{P\in\sM_\Ot} \chi_P(x)\,\big(m_{\mu,P}(\RR\mu) - m_{\mu,2R}(\RR\mu)\big)
+ \chi_{Z}(x) \big(\RR\mu(x) -  m_{\mu,2R}(\RR\mu)\big)
.$$
Notice that, for $x\in P\in\sM_\Ot$ and $Q\in \Reg_\Ot\setminus \Neg(e'(R))$ such that $Q\supset P$, since there are no $\PP$-doubling cubes $P'$ such that $P\subsetneq P'\subset Q$,
$$\big|\RR_{\sM_\Ot}\mu(x) - \RR_{\TT_\Reg}\mu(x)\big| = |
\RR(\chi_{2 Q\setminus 2P}\mu)(x)|\lesssim \sum_{P: P\subset P'\subset Q} \Theta(P') \lesssim \PP(Q)\lesssim
\Lambda\,\Theta(\HD_1).$$
Almost the same argument shows also that, for $x\in N_\Ot$,
$$\big|\RR(\chi_{2R}\mu)(x) - \RR_{\TT_\Reg}\mu(x)\big| \lesssim
\Lambda\,\Theta(\HD_1).$$
Remark also that
$$\int_{N_\Ot} \big|\RR(\chi_{2R}\mu)\big|^{2}\,d\mu \leq 
\int_Z \big|\RR(\chi_{2R}\mu)\big|^{2}\,d\mu.$$
So we deduce that, for $\ell_0$ small enough,
$$\wt I_\Ot \lesssim \bigg(\sum_{P\in\sM_{\Ot}} \int_P \big|\RR_{\sM_\Ot}\mu\big|^{2}\,d\mu + \int_Z \big|\RR(\chi_{2R}\mu)\big|^{2}\,d\mu + \Lambda^2\,\Theta(\HD_1)^2\,\mu(R)
\bigg)^{\frac {p_0}2} \big(\ve_Z\,\mu(R)\big)^{1-\frac {p_0}2}.$$

Almost the same arguments as in Lemma \ref{lemaprox2} show that for $x\in P\in\sM_\Ot$,
$$\big|\RR_{\sM_\Ot}\mu(x) - \Delta_{\sM_\Ot}\RR\mu(x)\big| \lesssim \PP(R) + \left(\frac{\EE(4R)}{\mu(R)}\right)^{1/2} + \PP(P) +  \left(\frac{\EE(2P)}{\mu(P)}\right)^{1/2}$$
and that, for $x\in Z$,
$$\big|\RR(\chi_{2R}\mu)(x) - \Delta_{\sM_\Ot}\RR\mu(x)\big| \lesssim \PP(R) + \left(\frac{\EE(4R)}{\mu(R)}\right)^{1/2}.$$
Therefore, by \rf{eqEr4} and \rf{eqEr5} and taking into account that $\PP(P)\lesssim\Lambda\,\Theta(\HD_1)$ for 
$P\in\sM_\Ot$, we deduce
$$\wt I_\Ot \lesssim \bigg(\int |\Delta_{\sM_\Ot}\RR\mu|^2\,d\mu +
\sum_{P\in\sM_{\Ot}} \EE(2P) + \Lambda^2\,\Theta(\HD_1)^2\,\mu(R)
\bigg)^{\frac {p_0}2} \big(\ve_Z\,\mu(R)\big)^{1-\frac {p_0}2}.$$ 

By the orthogonality of the functions $\Delta_Q\RR\mu$, $Q\in\DD_\mu$, the assumptions in the lemma, 
and \rf{eqinclu827},
it is clear that
$$\int |\Delta_{\sM_\Ot}\RR\mu|^2\,d\mu\leq \|\Delta_{\wt\TT}\RR\mu\|_{L^2(\mu)}^2\leq \Lambda^{-1}\,\sigma(\HD_1).$$
On the other hand, since the tree $\TT$ is typical,
\begin{align*}
\sum_{P\in\sM_{\Ot}} \EE(2P) & \leq \sum_{P\in\sM_{\Ot}\setminus\DB} \EE(2P) + \sum_{P\in \TT\cap\DB} \EE(2P)
\leq M^2\,\sigma(\sM_{\Ot}) + \Lambda^{\frac{-1}{3n}}\,\sigma(\HD_1)\\
&
\leq 
 M^2\Lambda^2\,\Theta(\HD_1)^2\,\mu(R)+ \Lambda^{\frac{-1}{3n}}\,\sigma(\HD_1)\lesssim \Lambda^4\,\Theta(\HD_1)^2\,\mu(R).
\end{align*}
Thus, using also \rf{eqmuhd1}, 
\begin{align*}
\wt I_\Ot & \lesssim \Lambda^{\frac{-p_0}2}\,\sigma(\HD_1)^{\frac{p_0}2}\, \mu(R)^{1-\frac {p_0}2}+ \ve_Z^{1-\frac {p_0}2}\,\Lambda^{2p_0}\,\Theta(\HD_1)^{p_0}\,\mu(R)\\
&
\lesssim 
 (B\Lambda^2)^{^{1-\frac {p_0}2}}\Lambda^{\frac{-p_0}2}\,\sigma_{p_0}(\HD_1)
+ \ve_Z^{1-\frac {p_0}2}\,B\,\Lambda^{2+{2p_0}}\,\sigma_{p_0}(\HD_1).
\end{align*}
By the choice of $p_0$, $\Lambda$, and $B$ and the assumption $\ve_Z\leq \Lambda^{-300n},$ we have
$$\wt I_\Ot\lesssim \Lambda^{\frac{-1}2}\,\sigma_{p_0}(\HD_1) + \Lambda^{\frac{-1}{25n}}\,\sigma_{p_0}(\HD_1)
\lesssim \Lambda^{\frac{-1}{25n}}\,\sigma_{p_0}(\HD_1)
.$$
Together with \rf{eqiot78}, this yields
$$I_\Ot  \lesssim \Lambda^{\frac{-1}{25n}}\,\sigma_{p_0}(\HD_1).
$$
Gathering the estimates obtained for $I_a$ and $I_\Ot$, the lemma follows.
\end{proof}
\vv

%We denote by $\Reg_\NZ$ the subfamily of the $Q\in\Reg$ such that $\frac12B(Q)$ intersects $V_4$ and
%$(|\RR\eta(x)| - \frac{c_0}2\,\Theta(\HD_1))_+$ is not identically zero on $\frac12B(Q)$ (here $\NZ$ stands for ``non-zero''). We also set
%$$N\!Z=\bigcup_{Q\in\Reg_\NZ} Q.$$

\vv
\begin{lemma}\label{lemNZ}
Suppose that $R\in\Trc\cap\Ty$ and also that $\mu(Z)\leq \ve_Z\,\mu(R)$ and
$$\|\Delta_{\wt\TT} \RR\mu\|_{L^2(\mu)}^2\leq \Lambda^{-1}\,\sigma(\HD_1).$$
 Assume also that
$\ve_Z\leq \Lambda^{-300n}$ and $\ell_0$ is small enough.
Then
$$\eta\big(\big\{x\in \wt V_4: |\RR\eta(x)|> \frac{c_0}2 \Theta(\HD_1)\big\}\big)\lesssim \Lambda^{\frac{-1}{25n}}\,\mu(HD_1).$$
Also, for all $p\in (1,p_0)$,
$$\int_{F_\eta\cap\wt V_4}\big|(|\RR\eta| - \frac{c_0}2 \Theta(\HD_1))_+\big|^{p}
\,d\eta\lesssim \Lambda^{\frac{-1}{25n}}\,\sigma_{p}(\HD_1).$$
\end{lemma}

\begin{proof}
Denote 
$$A = \big\{x\in \wt V_4: |\RR\eta(x)|> \frac{c_0}2 \Theta(\HD_1)\big\}.$$
By the definition of $F_\eta$ we can split
$$\eta(A)\leq \eta(A\cap F_\eta) + \eta\bigg(\bigcup_{Q\in\HD_2} \frac12B(Q)\bigg).$$
Notice first that, since $R\in\Trc$,
\begin{equation}\label{eqhd62}
\eta\bigg(\bigcup_{Q\in\HD_2} \frac12B(Q)\bigg) = \frac1{\Theta(\HD_2)^2}\,\sigma(\HD_2)
\leq \frac{B}{\Lambda^2\,\Theta(\HD_1)^2}\,\sigma(\HD_1) \leq \Lambda^{-1}\,\mu(HD_1).
\end{equation}
On the other hand, for $x\in A$ we have 
$$(|\RR\eta(x)| - \frac{c_0}4 \Theta(\HD_1))_+ \geq \frac{c_0}4 \Theta(\HD_1).$$
So, by Chebyshev and Lemma \ref{leminteta*},
\begin{align*}
\eta(A\cap F_\eta) &\lesssim \frac1{\Theta(\HD_1)^{p_0}} 
\int_{F_\eta\cap\wt V_4}\big|(|\RR\eta| - \frac{c_0}4 \Theta(\HD_1))_+\big|^{p_0}
\,d\eta\\
& \lesssim \frac{\Lambda^{\frac{-1}{25n}}\,\sigma_{p_0}(\HD_1)}{\Theta(\HD_1)^{p_0}}  = 
\Lambda^{\frac{-1}{25n}}\,\mu(HD_1),
\end{align*}
which, together with \rf{eqhd62}, proves the first assertion of the lemma.

For the second statement in the lemma we use H\"older and Lemma \ref{leminteta*} again:
\begin{align*}
\int_{F_\eta\cap\wt V_4}\big|(|\RR\eta| - \frac{c_0}2 \Theta(\HD_1))_+\big|^{p}
\,d\eta& = \int_{A\cap F_\eta}\big|(|\RR\eta| - \frac{c_0}2 \Theta(\HD_1))_+\big|^{p}
\,d\eta\\
& \leq \left(\int_{A\cap F_\eta}\!\big|(|\RR\eta| - \frac{c_0}4 \Theta(\HD_1))_+\big|^{p_0}
d\eta\right)^{\frac p{p_0}}\!\eta(A\cap F_\eta)^{1-\frac p{p_0}}\\
& \lesssim \big(\Lambda^{\frac{-1}{25n}}\,\sigma_{p_0}(\HD_1)\big)^{\frac p{p_0}}\,
\big(\Lambda^{\frac{-1}{25n}}\,\mu(HD_1)\big)^{1-\frac p{p_0}}\\
&= \Lambda^{\frac{-1}{25n}}\,\sigma_{p}(\HD_1).
\end{align*}
\end{proof}

\vv

Observe that, given $R\in\MDW\cap\Trc\cap\Ty$, from Lemmas \ref{lemrieszeta} and \ref{lemNZ}, under
the assumptions in those lemmas, we derive that 

\begin{align}\label{eqh2893} 
\int_{V_4\cap HD_2} \big|(|\RR\eta(x)| - &\frac{c_0}2\,\Theta(\HD_1))_+\big|^p\,d\eta(x)\\
 &\geq c\Lambda^{-(p'+1)\ve_n}\,\sigma_p(\HD_1) - C\Lambda^{\frac{-1}{25n}}\,\sigma_p(\HD_1)\approx
 \Lambda^{-(p'+1)\ve_n}\,\sigma_p(\HD_1),\nonumber
\end{align}
for $p\in (1,p_0]$, assuming that $\ve_n$ and $p$ are chosen so that
$(p'+1)\ve_n \ll \frac{1}{25n}.$
This is the main ingredient for the proof of the next lemma, which is the main result
of this section.
\vv

\begin{lemma}\label{lemalter*}
Let $R\in\MDW\cap\Trc\cap\Ty$. Let $\Lambda>0$ be big enough and suppose that
$\ve_Z\leq \Lambda^{-300n}$.
Then one of the following alternatives holds:
\begin{itemize}
\item[(a)] $\mu(Z) > \ve_Z\,\mu(R)$, or\vv

\item[(b)] $\|\Delta_{\wt \TT} \RR\mu\|_{L^2(\mu)}^2> \Lambda^{-1}\,\sigma(\HD_1).$
\end{itemize}
\end{lemma}

\begin{proof}
Suppose that none of the alternatives holds. Then, by Lemmas  \ref{lemrieszeta} and \ref{lemNZ},
as in \rf{eqh2893}, we have
\begin{equation} \label{eqtrans736}
I_{\HD_2}:=\int_{V_4\cap HD_2} \big|(|\RR\eta(x)| - \frac{c_0}2\,\Theta(\HD_1))_+\big|^p\,d\eta(x)
\gtrsim
 \Lambda^{-(p'+1)\ve_n}\,\sigma_p(\HD_1),
\end{equation}
for all $p\in (1,p_0]$,
assuming that $\ell_0$ is chosen small enough and that 
\begin{equation}\label{eqass934}
{(p'+1)\ve_n} \leq \frac{1}{50n}.
\end{equation}
The appropriate values of $\ve_n$ and $p$ will be chosen at the end of the proof.

By arguments analogous to the ones we used to estimate the integral $I_a$ in the proof of Lemma
\ref{leminteta*} we will ``transfer'' the estimate for $\RR\eta$ in \rf{eqtrans736} to $\Delta_{\wt\TT}\RR\mu$, so that we will obtain a lower estimate for $\|\Delta_{\wt \TT} \RR\mu\|_{L^2(\mu)}$ which will contradict the assumption that (b) does not hold. 
Although some of the estimates below are very similar to the ones to obtain the inequality 
\rf{eqIa99} in the proof of Lemma
\ref{leminteta*}, we will include the full details here for the reader's convenience. 
On the other hand, an important difference between the proof of that lemma and the current proof is that in Lemma 
\ref{leminteta*} we took advantage of the fact that $p_0$ is close to $2$, and the estimates there would not work for the family $\Reg_{\HD_2}$, while in the arguments below it is  essential that we take $p$ close to $1$, and the estimates work fine for the family $\Reg_{\HD_2}$, 
while they would fail for the family $\Reg_a$.

By Lemma \ref{lemaprox1}, for all $Q\in\Reg_{\HD_2}$ such that $\frac12B(Q)\subset\wt V_4$, all $x\in \frac12B(Q)$, and all $y\in Q$,
$$|\RR\eta(x)| \leq |\RR_{\TT_\Reg}\mu(y)| + C\Theta(R) + C\PP(Q) + C\QQ_\Reg(Q).$$
Thus, for $\Lambda$ big enough, since $\Theta(R)=\Lambda^{-1}\Theta(\HD_1)<\frac{c_0}4 \Theta(\HD_1)$,
$$(|\RR\eta(x)| - \frac{c_0}2 \,\Theta(\HD_1))_+\leq
(|\RR_{\TT_\Reg}\mu(y)| - \frac{c_0}4\, \Theta(\HD_1))_+  + C\PP(Q) + C\QQ_\Reg(Q).$$
Therefore,
$$
I_{\HD_2}  \lesssim \sum_{Q\in\Reg_{\HD_2}} \!\int_Q \big|(|\RR_{\TT_\Reg}\mu(y)| - 
\frac{c_0}4 \Theta(\HD_1))_+\big|^{p}\,d\mu(y) + \!\!\sum_{Q\in\Reg_{\HD_2}} \!\!\!\!(\PP(Q)^{p} + \QQ_\Reg(Q)^{p})\,\mu(Q).$$
By \rf{eqlem*6} and Lemma \ref{lemregpq}, we have
\begin{align*}
\sum_{Q\in\Reg_{\HD_2}} (\PP(Q)^{p} + \QQ_\Reg(Q)^{p})\,\mu(Q) & = 
\Sigma_p^\PP(\Reg_{\HD_2}) +\Sigma_p^\QQ(\Reg_{\HD_2})\\ 
&\lesssim
\sigma_p(\HD_2) + \Sigma^\QQ(\Reg)^{\frac p2}\,\mu(HD_2)^{1-\frac p2} \\ 
& \lesssim B\,\Lambda^{p-2}\,\sigma_p(\HD_1) + \Sigma^\PP(\Reg)^{\frac p2}\,\mu(HD_2)^{1-\frac p2}.
\end{align*}
Also, recalling that $\sigma(\HD_2)\leq B\,\sigma(\HD_1)$, we get
$\mu(HD_2)\leq B\,\Lambda^{-2}\,\mu(HD_1)$. Then, by Lemma \ref{lemreg73}, we obtain
$$\Sigma^\PP(\Reg)^{\frac p2}\,\mu(HD_2)^{1-\frac p2} \leq \big(B\,\sigma(\HD_1)\big)^{\frac p2}\,
\big(B\,\Lambda^{-2} \,\mu(HD_1)\big)^{1-\frac p2} = B\,\Lambda^{p-2}\,\sigma_p(\HD_1).
$$
So, since any cube from $\Reg_{\HD_2}$
is contained in some cube $S\in\HD_2$:
$$I_{\HD_2}  \lesssim \sum_{S\in\HD_2} \!\int_S \big|(|\RR_{\TT_\Reg}\mu(y)| - 
\frac{c_0}4 \Theta(\HD_1))_+\big|^{p}\,d\mu(y) + B\,\Lambda^{p-2}\,\sigma_p(\HD_1).$$

For each $S\in\HD_2$, by the triangle inequality and the fact that $(\;\cdot\;)_+$ is a $1$-Lipschitz function, we obtain
\begin{multline*}
\int_S \big|(|\RR_{\TT_\Reg}\mu| - 
\frac{c_0}4 \Theta(\HD_1))_+\big|^{p}\,d\mu \lesssim 
\int_S \big|(|\RR_{\wt \TT}\mu| - 
\frac{c_0}4 \Theta(\HD_1))_+\big|^{p}\,d\mu \\
+ \int_S\big|\RR_{\wt\TT}\mu - \RR_{\TT_\Reg}\mu\big|^{p}\,d\mu
\end{multline*}
By Lemma \ref{lemaprox3}, the last integral does not exceed 
$C\EE(2S)^{\frac{p}2} \,\mu(S)^{1-\frac{p}2}$, and thus
\begin{equation}\label{eqIa1*}
I_{\HD_2}\lesssim\!
\sum_{S\in\HD_2}\int_S \big|(|\RR_{\wt\TT}\mu| - 
\frac{c_0}4 \Theta(\HD_1))_+\big|^{p}\,d\mu +\!\! \sum_{S\in\HD_2}\!\!\EE(2S)^{\frac{p}2} \mu(S)^{1-\frac {p}2} + B\,\Lambda^{p-2}\,\sigma_p(\HD_1).
\end{equation}

Next we apply Lemma \ref{lemaprox2}, which implies that
for any $S\in\HD_2$ and all $x\in S$,
\begin{equation}\label{eqIa2*}
|\RR_{\wt\TT}\mu(x)| \leq |\Delta_{\wt\TT}\RR\mu(x)| + C\left(\frac{\EE(4R)}{\mu(R)}\right)^{1/2} +  C\left(\frac{\EE(2S)}{\mu(S)}\right)^{1/2}.
\end{equation}
In case that $R\not\in\DB$, recalling that
$\Lambda  \gg M$ by \eqref{eq:LambdadepM}, for  $\Lambda$ big enough we obtain
$$C\left(\frac{\EE(4R)}{\mu(R)}\right)^{1/2} \leq C\,M\,\Theta(R)\leq \frac{c_0}4 \Theta(\HD_1),$$
If $R\in\DB$, then we use the fact that $R\in\Ty$, which ensures that
$$\EE(4R)\lesssim \sum_{P\sim \TT:P\in\DB}\EE_\infty(9P) \leq \Lambda^{\frac{-1}{3n}}\,\sigma(\HD_1),$$
and so we also get
$$C\left(\frac{\EE(4R)}{\mu(R)}\right)^{1/2} \leq C \left(\frac{\Lambda^{\frac{-1}{3n}}\,\sigma(\HD_1)}{\mu(R)}\right)^{1/2}\leq C\, \Lambda^{\frac{-1}{6n}} \,\Theta(\HD_1)
\leq \frac{c_0}4 \Theta(\HD_1).$$
Hence, in any case we have
$$(|\RR_{\wt\TT}\mu(x)| - \frac{c_0}4 \Theta(\HD_1))_+ \leq 
|\Delta_{\wt\TT}\RR\mu(x)| + C\left(\frac{\EE(2S)}{\mu(S)}\right)^{1/2}.$$
Plugging this estimate into \rf{eqIa1*}, we get
\begin{equation}\label{eqIa99*}
I_{\HD_2}\lesssim
\int_{HD_2} |\Delta_{\wt\TT}\RR\mu|^{p}\,d\mu + \sum_{S\in\HD_2}\EE(2S)^{\frac{p}2} \,\mu(S)^{1-\frac {p}2} 
+  B\,\Lambda^{p-2}\,\sigma_p(\HD_1).
\end{equation}

Next we will estimate each term on the right hand side. First, by H\"older's inequality, we have
\begin{align*}
\int_{HD_2} |\Delta_{\wt\TT}\RR\mu|^{p}\,d\mu & \lesssim \|\Delta_{\wt\TT}\RR\mu\|_{L^2(\mu)}^{p}\,
\mu(HD_2)^{1-\frac{p}2}\\
&\leq \|\Delta_{\wt\TT}\RR\mu\|_{L^2(\mu)}^{p}\,(B\Lambda^{-2}\mu(HD_1))^{1-\frac{p}2}
\leq \|\Delta_{\wt\TT}\RR\mu\|_{L^2(\mu)}^{p}(\Lambda^{-1}\mu(HD_1))^{1-\frac{p}2}.
\end{align*}
Regarding the second term in \rf{eqIa99*}, by H\"older's inequality again,
\begin{align}\label{eqplug731}
\sum_{S\in\HD_2}\EE(2S)^{\frac{p}2} \,\mu(S)^{1-\frac {p}2} 
& \leq  \bigg(\sum_{S\in\HD_2}\EE(2S)\bigg)^{\frac{p}2} \,\mu(HD_2)^{1-\frac {p}2}.
\end{align}
We estimate now the first factor on the right hand side:
\begin{align*}
\sum_{S\in\HD_2}\EE(2S) & \leq \sum_{S\in\HD_2\setminus\DB}\EE(2S) + \sum_{S\sim\TT:S\in\DB}\EE(2S)\\
&\lesssim M^2\,\sum_{S\in\HD_2\setminus\DB}\sigma(S) + \Lambda^{\frac{-1}{3n}}\,\sigma(\HD_1)\\
&\leq M^2\,\sigma(\HD_2) + \Lambda^{\frac{-1}{3n}}\,\sigma(\HD_1)\\
& \leq B\,M^2\,\sigma(\HD_1) + \Lambda^{\frac{-1}{3n}}\,\sigma(\HD_1) \lesssim B\,M^2\,\sigma(\HD_1).
\end{align*}
Hence, plugging this estimate into \rf{eqplug731} and using that $\mu(HD_2)\leq B\,\Lambda^{-2}\mu(HD_1)$,
$$\sum_{S\in\HD_2}\EE(2S)^{\frac{p}2} \,\mu(S)^{1-\frac {p}2}\lesssim 
\big(B\,M^2\,\sigma(\HD_1)\big)^{\frac p2} \,\big(B\,\Lambda^{-2}\mu(HD_1)\big)^{1-\frac p2}
=M^p\, B\,\Lambda^{p-2}\,\sigma_p(\HD_1).$$

Altogether, we deduce that
$$I_{\HD_2}\lesssim \|\Delta_{\wt\TT}\RR\mu\|_{L^2(\mu)}^{p}\,(\Lambda^{-1}\mu(HD_1))^{1-\frac{p}2} + M^p\, B\,\Lambda^{p-2}\,\sigma_p(\HD_1).$$

Recalling the lower estimate for $I_{\HD_2}$ in \rf{eqtrans736}, using that $p'\geq2$ we obtain 
\begin{equation}\label{eqdelt634}
\|\Delta_{\wt\TT}\RR\mu\|_{L^2(\mu)}^{p}\,(\Lambda^{-1}\mu(HD_1))^{1-\frac{p}2} \geq
c\, \Lambda^{-2p'\ve_n}\,\sigma_p(\HD_1) - C\,M^p\, B\,\Lambda^{p-2}\,\sigma_p(\HD_1).
\end{equation}
Recall now that by \eqref{eq:LambdadepM}
$$M\le \Lambda^{1-\frac1{8n-1}}\ll\Lambda.$$ 
Notice that for $p$ close enough to $1$, we have $M^p\, B\,\Lambda^{p-2}\ll1$, so that the last term on the right hand side of \rf{eqdelt634} is much smaller than the first one, assuming $\ve_n$ close enough to $0$. To be more precise, let us take 
$$p=1+ \frac1{4(8n-1)}.$$
A straightforward calculation gives $M^p\, \Lambda^{p-2} \leq \Lambda^{-\frac1{2(8n-1)} -\frac1{4(8n-1)^2}}$, so that
$$B\,M^p\, \Lambda^{p-2} \leq \Lambda^{\frac1{100n}}\,\Lambda^{-\frac1{2(8n-1)}} \leq 
\Lambda^{-\frac1{4(8n-1)}}.$$
Then we choose $\ve_n$ so that, besides \rf{eqass934}, it satisfies
\begin{equation}\label{eqchooseen}
\ve_n \leq \frac1{16(8n-1)}\,(p-1) = \frac1{64(8n-1)^2},
\end{equation}
and we derive
$$\|\Delta_{\wt\TT}\RR\mu\|_{L^2(\mu)}^{p}\,(\Lambda^{-1}\mu(HD_1))^{1-\frac{p}2} \gtrsim
\Lambda^{-2p'\ve_n}\,\sigma_p(\HD_1),$$
which is equivalent to
\begin{align*}
\|\Delta_{\wt\TT}\RR\mu\|_{L^2(\mu)}^{2} &\gtrsim
\big( \Lambda^{-2p'\ve_n}\,
\,\Lambda^{1-\frac p2} 
\,\sigma_p(\HD_1)\,\mu(HD_1)^{\frac{p}2-1}\big)^{\frac 2p}\\
& = \Lambda^{-\frac{4\ve_n}{p-1} +\frac 2p-1}\,\sigma(\HD_1)\gg \Lambda^{-1}\,\sigma(\HD_1).
\end{align*}
This contradicts the assumption that the alternative (b) in the lemma does not hold.
\end{proof}

\vv

\bigskip%\filbreak 
\begin{center} 
	\Large Part \refstepcounter{parte}\theparte\label{part-3}: The proof of the First Main Proposition
\end{center}
\smallskip
\addcontentsline{toc}{section}{\bf Part 3: The proof of the First Main Proposition}

%\part{The proof of the First Main Proposition}
In this part, corresponding to Sections \ref{sec3.3} - \ref{sec9} we choose $\OP(R) = \varnothing$ for any $R\in\MDW$.
\vvv

\section{The corona decomposition and the Main Lemma}\label{sec3.3}

In order to prove the Main Proposition \ref{propomain} we have to use a suitable corona decomposition which splits
the lattice $\DD_\mu$ into appropriate trees. We need first need to introduce some variant of the family $\HD(R)$, for $R\in\DD_\mu^\PP$.
For $N=500n$, we denote
$$\Lambda_* = \Lambda^{\frac{N}{N-1}} = A_0^{\frac{N}{N-1}k_\Lambda n},$$ so that
$\Lambda = \Lambda_*^{1-\frac1N}.$ We assume that $k_\Lambda$ is a multiple of $N-1$, so that $k_{\Lambda_*}:=\frac{N}{N-1}k_\Lambda$
is an integer.
Then we set
$$\HD_*(R) = \hd^{k_{\Lambda_*}}(R).$$

Further, we assume that the constant $\delta_0$ in the definition of the family $\LD(R)$ in Section \ref{subsec:enlar} is of the form
$$\delta_0 = \Lambda_*^{-N_0 - \frac1{2N}},$$
with $N_0$ large enough so that moreover $\delta_0\leq \Lambda^{-4n^2}$, as required in Section \ref{subsec:enlar}.

We denote by $\sss_*(R)$ the family of maximal cubes from $\HD_*(R)\cup\LD(R)$ which are contained in $R$.
Also, we let $\End_*(R)$ be the family of maximal $\PP$-doubling cubes which are contained in some cube
from $\sss_*(R)$. Notice that, by Lemma \ref{lempdoubling}, the cubes from $\HD_*(R)\cap \DD_\mu(4R)$ are $\PP$-doubling, and thus
any cube from $\sss_*(R)\cap\HD_*(R)$ belongs to $\End_*(R)$. Finally, we let $\tree(R)$ denote the subfamily of the cubes from $\DD_\mu(R)$ which are not strictly contained in any cube
from $\End_*(R)$, and we say that $R$ is the root of the tree.

Next we define the family $\ttt$ inductively. 
We assume that $\supp\mu$ coincides with a cube $S_0$, and then
we set $\ttt_0=\{S_0\}$. Assuming $\ttt_k$ to be defined, we let
$$\ttt_{k+1} = \bigcup_{R\in\ttt_k} \End_*(R).$$
Then we let
$$\ttt =\bigcup_{k\geq0} \ttt_k.$$
Notice that we have 
$$\DD_\mu=\bigcup_{R\in\ttt} \tree(R).$$
Two trees $\tree(R)$, $\tree(R')$, with $R,R'\in\ttt$, $R\neq R'$ can only intersect if one
of the roots is and ending cube of the other, i.e., $R'\in\End_*(R)$ or $R\in\End_*(R')$.

\vv

The main step for the proof of Main Proposition \ref{propomain} consists of proving the following.

\begin{mlemma}\label{mainlemma}
	Let $\mu$ be a Radon measure in $\R^{n+1}$ such that
	$$
	\mu(B(x,r))\leq \theta_0\,r^n\quad \mbox{ for all $x\in\supp\mu$ and all $r>0$}.
	$$
	Then, for any choice of $M>1$,
	\begin{equation}\label{eqmainlemma*}
		\sum_{R\in\ttt} \Theta(R)^2\,\mu(R)\leq C\,\big(\|\RR\mu\|_{L^2(\mu)}^2 + \theta_0^2\,\|\mu\|
		+ \sum_{Q\in\DB(M)}\EE_\infty(9Q)\big),
	\end{equation}
	with $C$ depending on $M$.
\end{mlemma}

The rest of the current section, together with Sections \ref{sec-layers} and \ref{sec8} are devoted to the proof of this lemma. Later, in Section \ref{sec9} we will complete the 
proof of Main Proposition \ref{propomain}.

\vv
Recall that a cube $R\in\DD_{\mu}$ belongs to
$\MDW$ if $R$ is $\PP$-doubling and
$
\sigma(\HD(R)\cap\sss(R))\geq B^{-1}\,\sigma(R),
$
where
$B= \Lambda^{\frac1{100n}}.$

\vv
\begin{lemma}
For $R\in\ttt\setminus \MDW$, we have
\begin{equation}\label{eqmdw81}
\sigma(\End_*(R))\leq 2\Lambda_*^{2/N}B^{-1}\,\sigma(R).
\end{equation}
\end{lemma}

\begin{proof}
We have
$$\sigma(\End_*(R)) = \sigma(\HD_*(R)\cap\sss_*(R)) + \sum_{Q\in\LD(R)\cap\sss_*(R)}\sum_{P\in\End_*(R):P\subset Q}
\sigma(Q).$$
Clearly, since any cube from $\HD_*(R)$ is contained in some cube from $\HD(R)$, we infer that
\begin{align*}
\sigma(\HD_*(R)\cap\End_*(R)) & =\Lambda_*^2\,\Theta(R)^2\!\sum_{Q\in \HD_*(R)\cap\End_*(R)}\!\!\mu(Q)
\leq \Lambda_*^2\,\Theta(R)^2\!\sum_{Q\in \HD(R)\cap\sss(R)}\!\!\mu(Q) \\
& = \frac{\Lambda_*^2}{\Lambda^2}\,\sigma
(\HD(R)\cap\sss(R))\leq \Lambda_*^{2/N}B^{-1}\,\sigma(R),
\end{align*}
where in the last inequality we used that $R\notin\MDW$.
For $P\in\End_*(R)\setminus\HD_*(R)$, there exists some $Q$ such that $P\subset Q\in\LD(R)\cap\sss_*(R)$.  By \eqref{eqcad35} we have 
$$\Theta(P)\lesssim\PP(Q)\leq \delta_0\,\Theta(R),$$
and thus
$$
\sigma(\End_*(R)) \leq \Lambda_*^{2/N} B^{-1}\,\sigma(R) + C\delta_0^2\,\sigma(R) \leq 2\Lambda_*^{2/N} B^{-1}\,\sigma(R),
$$
since $\delta_0\ll B^{-1}$.
\end{proof}
\vv

Observe that $N=500n$ is big enough so that 
$$2\Lambda_*^{2/N}\,B^{-1}\leq \frac12.$$

\begin{lemma}\label{lemtoptop}
We have
$$\sigma(\ttt)  \lesssim \sigma(\ttt\cap\MDW)+ \theta_0^2\,\|\mu\|.$$
\end{lemma}

\begin{proof}
For each $R\in \{S_0\}\cup (\ttt\cap\MDW)$, we denote
$I_0(R)=\{R\}$, and for $k\geq0$,
$$I_{k+1}(R) = \bigcup_{Q\in I_k(R)} \End_*(Q)\setminus \MDW.$$ 
In this way, we have
\begin{equation}\label{eqtttot}
\ttt = \bigcup_{R\in \{S_0\}\cup (\ttt\cap\MDW)} \;\bigcup_{k\geq0} I_k(R).
\end{equation}
Indeed, for each $Q\in \ttt$, let $R$ be the minimal cube from $\ttt\cap\MDW$ that contains $Q$, and in case this does not exists, let $R=S_0$. Then it follows that
$$Q\in \bigcup_{k\geq0} I_k(R).$$

Given $R\in\{S_0\}\cup (\ttt\cap\MDW)$, for each $k\geq 1$ and $Q\in I_k(R)$, 
since $I_k(R)\subset \ttt\setminus \MDW$, by \rf{eqmdw81} we have 
$$\sigma(\End_*(Q)\setminus \MDW)\leq 2 \Lambda_*^{2/N}B^{-1}\,\sigma(Q)\leq \frac12\,\sigma(Q).$$
Then we deduce that
$$\sigma(I_{k+1}(R)) = \sum_{Q\in I_k(R)} \sigma(\End_*(Q)\setminus \MDW) \leq \frac12 \sum_{Q\in I_k(R)}\sigma(Q)
= \frac12\, \sigma(I_k(R)).$$
So 
$$\sigma(I_{k}(R))\leq \frac1{2^k}\,\sigma(R)\quad \mbox{ for each $k\geq0$.}$$
Then, by \rf{eqtttot}, 
\begin{align*}
\sigma(\ttt) & = \sum_{R\in \{S_0\}\cup (\ttt\cap\MDW)} \;\sum_{k\geq0} \sigma(I_k(R))
\leq \sum_{R\in \{S_0\}\cup (\ttt\cap\MDW)}\;\sum_{k\geq0} \frac1{2^k}\,\sigma(R) \\
& \approx \sigma(S_0) + \sigma(\ttt\cap\MDW)
\lesssim \theta_0^2\,\|\mu\|+\sigma(\ttt\cap\MDW).
\end{align*}
\end{proof}

\vv

\vv

\section{The layers \texorpdfstring{$\sF_j^h$ and $\sL_j^h$}{Fjh and Ljh}}\label{sec-layers}

Consider an arbtrary subfamily $\sF\subset \MDW$. Quite soon we will choose $\sF = \ttt\cap \MDW$, but in another application of this construction later on, we will choose $\sF$ to be a different subfamily.
For $j\in\Z$,
we denote
$$\sF_j= \big\{R\in\sF: \,\Theta(R)=A_0^{nj}\big\},$$
so that 
$$\sF= \bigcup_{j\in\Z} \sF_j.$$
Next we split $\sF_j$ into layers $\sF_j^h$, $h\geq1$, which are defined as follows:
$\sF_j^1$ is the family of maximal cubes from $\sF_j$, and by induction
$\sF_j^h$ is the family of maximal cubes from $\sF_j\setminus \bigcup_{k=1}^{h-1} \sF_j^{h-1}$.
So we have the splitting
$$\sF= \bigcup_{j\in\Z}\,\bigcup_{h\geq1} \sF_j^h.$$

Our next objective is to choose a suitable subfamily $\sL_j^h\subset \sF_j^h$, for each $j,h$.
By the covering Theorem 9.31 from \cite{Tolsa-llibre}, there
is a family $J_0\subset \sF_j^h$ such that\footnote{Actually the property 1) is not stated in that theorem, however this can be obtained by preselecting a subfamily of maximal balls from $\sF_j^h$
with respect to inclusion and then applying the theorem to the maximal subfamily.}
\begin{itemize}
\item[1)] no ball $B(e^{(4)}(Q))$, with $Q\in J_0$, is contained in any other ball  
$B(e^{(4)}(Q'))$, with $Q'\in \sF_j^h$, $Q'\neq Q$,
\item[2)] the balls $B(e^{(4)}(Q))$, with $Q\in J_0$, have finite superposition, 
and
\item[3)] 
every ball $B(e^{(4)}(Q))$, with $Q\in\sF_j^h$, is contained in some ball 
$(1+8A_0^{-1})\,B(e^{(4)}(R))$, with $R\in J_0$. Consequently,
$$\bigcup_{Q\in\sF_j^h} B(e^{(4)}(Q)) \subset \bigcup_{R\in J_0} (1+8A_0^{-1})\,B(e^{(4)}(R)).$$
\end{itemize}

From the finite superposition property 2), by rather standard arguments which are analogous to the
ones in the proof of Besicovitch's covering theorem in \cite[Theorem 2.7]{Mattila-llibre}, say, 
one deduces that $J_0$ can be split into 
 $m_0$ subfamilies $J_1,\ldots, J_{m_0}$ such that, for each $k$,  the balls $\{B(e^{(4)}(Q)): Q\in J_k\}$  are pairwise disjoint, with $m_0\leq C(A_0)$.

By the condition 3),  
%and Lemma \ref{lem-calcf} applied to $Q$, we obtain
%\begin{equation*}%\label{equni98}
%\bigcup_{Q\in\sF_j^h} Q\subset \bigcup_{Q\in\sF_j^h} B(e^{(4)}(Q) )\subset \bigcup_{R\in J_0} (1+8A_0^{-1})\,B(e^{(4)}(R)) \subset \bigcup_{R\in J_0} B(e^{(10)}(R)).
%\end{equation*}
we get
\begin{align*}
\sum_{Q\in \sF_j^h} \sigma(\HD_1(Q)) & = \Lambda^2 A_0^{2nj}\sum_{Q\in \sF_j^h} 
\sum_{P\in\HD_1(Q)} \mu(P)\\
& \leq \Lambda^2 A_0^{2nj}\sum_{R\in J_0} \sum_{\substack{Q\in \sF_j^h:\\ B(e^{(4)}(Q))
\subset (1+8A_0^{-1})B(e^{(4)}(R))}} \sum_{P\in\HD_1(Q)} \mu(P).
\end{align*}
Observe that, for $R\in J_0$ and $Q\in\sF_j^h$ such that  $B(e^{(4)}(Q))
\subset (1+8A_0^{-1})\,B(e^{(4)}(R))$,
 any cube $P\in\HD_1(Q)$ is contained in some cube from $\HD_1(e^{(10)}(R))$ since
$\supp\mu\cap (1+8A_0^{-1})B(e^{(4)}(R))\subset e^{(10)}(R)$,  by Lemma \ref{lem-calcf}.
 Using also that the cubes from $\sF_j^h$ are disjoint and Lemma \ref{lem:43}, we deduce that
\begin{align*}
\sum_{Q\in \sF_j^h} \sigma(\HD_1(Q)) & \leq \Lambda^2 A_0^{2nj}\sum_{R\in J_0}
\sum_{P\in\HD_1(e^{(10)}(R))} \mu(P) \\ &= \sum_{R\in J_0}\sigma(\HD_1(e^{(10)}(R))) \leq B^{1/4}\sum_{R\in J_0}\sigma(\HD_1(e(R))).
\end{align*}
Next we choose $\sL_j^h=J_k$ to be the family such that
$$\sum_{Q\in J_k}\sigma(\HD_1(e(Q)))$$
is maximal among $J_1,\ldots,J_{m_0}$, so that%, by Lemma \ref{lem5.2} and \rf{equni98},
\begin{align*}
\sum_{Q\in \sF_j^h} \sigma(\HD_1(Q))\leq m_0\,B^{1/4}\sum_{Q\in \sL_j^h}\sigma(\HD_1(e(Q))).\end{align*}
So we have:

\begin{lemma}\label{lemljh}
The family $\sL_j^h$ satisfies:
\begin{itemize}
\item[(i)] no ball $B(e^{(4)}(Q))$, with $Q\in \sL_j^h$, is contained in any other ball  
$B(e^{(4)}(Q'))$, with $Q'\in \sF_j^h$, $Q'\neq Q$,
\item[(ii)] the balls $B(e^{(4)}(Q))$, with $Q\in \sL_j^h$, are pairwise disjoint, 
and
\item[(iii)] 
 $$\sum_{Q\in \sF_j^h} \sigma(\HD_1(Q)) \lesssim B^{1/4} 
\sum_{Q\in \sL_j^h}\sigma(\HD_1(e(Q))).$$
\end{itemize}
\end{lemma}

\vv
We denote 
$$\sL_j= \bigcup_{h\geq 1}\sL_j^h,\qquad \sL= \bigcup_{j\in\Z}\sL_j =
\bigcup_{j\in\Z}\,\bigcup_{h\geq 1}\sL_j^h.$$
By the property (iii) in the lemma, we have
\begin{align}\label{eqover5}
\sum_{R\in \sF}\!\!\sigma(\HD_1(R)) & = \sum_{j\in\Z, \,h\geq0}\,
\sum_{R\in\sF_j^h} \sigma(\HD_1(R)) \\ 
& \lesssim m_0\,B^{1/4} \sum_{j\in\Z, \,h\geq0}\,\sum_{R\in \sL_j^h}\sigma(\HD_1(e(R))) = m_0B^{1/4}\sum_{R\in \sL}\sigma(\HD_1(e(R))).\nonumber
\end{align}

\vv

\begin{lemma}\label{lemsuper**9}
Let $\sF=\MDW\cap \ttt$. Then we have
$$\sigma(\ttt) \lesssim B^{5/4} \sum_{R\in \sL} \sum_{k\geq0} B^{-k/2}\sum_{Q\in\Trc_k(R)}\sigma(\HD_1(e(Q))) + \theta_0^2\,\|\mu\|.$$
\end{lemma}

Recall that the families $\Trc_k(R)$ have been defined in Section \ref{subsec:trc}, just before Lemma \ref{eqtec74}. We assume that $\OP(Q)=\varnothing$ for any $Q\in\MDW$ in this constrcution. 

\begin{proof}
This is an immediate consequence of Lemma \ref{lemtoptop} and our earlier estimates: \todo{the estimate $\sigma(R)\le B\sigma(\HD_1(R))$ was missing, so the constant $B^{1/4}$ had to be changed to $B^{5/4}$}
\begin{align*}
\sigma(\ttt)  & \lesssim \sigma(\ttt\cap\MDW)+ \theta_0^2\,\|\mu\|\\
& \overset{\eqref{eq:MDWdef2}}{\lesssim} B\sum_{R\in \ttt\cap\MDW}\sigma(\HD_1(R)) + \theta_0^2\,\|\mu\|\\
& \overset{\rf{eqover5}}{\lesssim} B^{5/4}\sum_{R\in \sL}\sigma(\HD_1(e(R))) + \theta_0^2\,\|\mu\|\\
& \overset{\eqref{eqiter*44}}{\lesssim} B^{5/4} \sum_{R\in \sL} \sum_{k\geq0} B^{-k/2}\sum_{Q\in\Trc_k(R)}\sigma(\HD_1(e(Q))) + \theta_0^2\,\|\mu\|.
\end{align*}
%Remark that the implicit constants above depend on $m_0$, which in turn depends just on $A_0$ and $n$.
\end{proof}

\vv

To be able to apply later the preceding lemma, we need to get an estimate for $\#\sL(P,k)$, where $P\in\DD_\mu,\, k\ge 0$ and
\begin{equation}\label{eqlpk*}
	\sL(P,k)= \big\{R\in\sL:\exists \,Q\in\Trc_k(R) \mbox{ such that } P\in\TT(e'(Q))\big\}.
\end{equation}
For $j\in\Z$ set also 
\begin{equation}\label{eqlpkj*}
\sL_j(P,k)= \big\{R\in\sL_j:\exists \,Q\in\Trc_k(R) \mbox{ such that } P\in\TT(e'(Q))\big\},
\end{equation}
so that $\sL(P,k) = \bigcup_j \sL_j(P,k)$. The following important technical result is the main achievement in this section.

\begin{lemma}\label{lemimp9}
Let $\sF=\MDW\cap \ttt$. In this case there exists some constant $C_4$ such that, for all $P\in\DD_\mu$ and all $k\geq0$,
$$\#\sL(P,k)\leq C_4\,\log\Lambda.$$
More precisely, for each $P\in\DD_\mu$ and $k\geq0$
\begin{equation}\label{eq:lemimp91}
	\#\{j\in\Z:\sL_j(P,k)\neq\varnothing\}\lesssim \log\Lambda,
\end{equation}
and for each $j\in\Z$, $P\in\DD_\mu$, $k\geq0$,
\begin{equation}\label{eqlj83}
	\#\sL_j(P,k) \leq C_5.
\end{equation}
\end{lemma}

\vv

We prove first \eqref{eq:lemimp91}.
\begin{proof}[Proof of \eqref{eq:lemimp91}]
	Let $\wt P_1$ be the smallest $\PP$-doubling cube containing $P$, and let $\wt P_2$ be be the smallest $\PP$-doubling cube strictly containing $\wt P_1$. Suppose that  $R\in\sL_j(P,k)$. There are two cases to consider.
	\vv
	
	\emph{Case 1.} There exists $Q\in\Trc_k(R)$ such that $P\in\TT(e'(Q))\setminus\Neg(e'(Q))$. We claim that in this case we have $\wt P_i\in \TT_\sss(e'(Q))$ for some $i\in\{1,2\}$. Indeed, if $P\in\TT_\sss(e'(Q))\setminus\Neg(e'(Q))$, then  necessarily $\wt P_1\in \TT_\sss(e'(Q))$, by the definition of
	the family $\Neg(e'(Q))$. If $P\notin\TT_\sss(e'(Q)),$ then either:
	\begin{itemize}
		\item $P=\wt P_1,$ which implies $P\in\End(e'(Q))$, in which case $\wt P_2\in \TT_\sss(e'(Q))$, or
		\item $P\neq \wt P_1$ and we have $\wt P_1\in \TT_\sss(e'(Q))$, again by the definition of $\Neg(e'(Q))$.
	\end{itemize}  
	
	Choosing $i\in\{1,2\}$ such that $\wt P_i\in \TT_\sss(e'(Q))$ we see by the definition of $\TT_\sss(e'(Q))$ that
	$$\delta_0\,\Theta(Q)\lesssim \Theta(\wt P_i)\leq \Lambda^2\,\Theta(Q).$$
	Since $\Theta(Q)=\Lambda^k\Theta(R)$ (by the definition of $\Trc_k(R)$), the above is equivalent to
	\begin{equation*}
		\Lambda^{-2}\Theta(\wt P_i)\le \Lambda^k\Theta(R)\le C\delta_0^{-1}\Theta(\wt P_i)
	\end{equation*}
	We have $\Theta(R)=A_0^{nj}$ because $R\in \sL_j(P,k)$, and so it follows that 
	$$-C\log\Lambda\leq j + c\,k\log\Lambda - c'\log\Theta(\wt P_i)\leq C|\log\delta_0| = C'\log\Lambda.$$
	Recall that $k\ge 0$ is fixed, and $\Theta(\wt P_i)$ is equal to either $\Theta(\wt P_1)$ or $\Theta(\wt P_2)$, where both of these values depend only on $P$, which is fixed. Thus, there are at most $C''\log\Lambda$ integers $j$ such that there exists  $R\in\sL_j(P,k)$ and $Q\in\Trc_k(R)$ for which $P\in\TT(e'(Q))\setminus\Neg(e'(Q))$.
	
	\vv
	
	\emph{Case 2.}
	Suppose now that there exists $Q\in\Trc_k(R)$ such that $P\in\Neg(e'(Q))\subset\TT(e'(Q))$.
	In this case, by Lemma \ref{lemnegs}, $\ell(P) \gtrsim \delta_0^{-2}\,\ell(Q)$. Hence, there
	are at most $C\,|\log\delta_0|\approx \log\Lambda$ cubes $Q$ such that 
	$P\in\TT(e'(Q))\cap\Neg(e'(Q))$. 
	
	For each such cube we have $\Theta(Q)=\Lambda^k\Theta(R) = \Lambda^k A_0^{nj}$. Thus, for each cube $Q$ as above there is exactly one value of $j$ such that there may exist $R\in\sL_j$ with $Q\in\Trc_k(R)$. It follows that there are at most 
	$C'''\log\Lambda$ values of $j$ such that
	there exists $R\in\sL_j(P,k)$ and $Q\in\Trc_k(R)$ for which $P\in\TT(e'(Q))\cap\Neg(e'(Q))$.
	
	\vv
	Putting the estimates from both cases together we get that $\sL_j(P,k)$ is non-empty for at most $(C''+C''')\log\Lambda$ integers $j$.
\end{proof}

The proof of \eqref{eqlj83} is more involved. Its key ingredient is the following auxiliary result.

\begin{lemma}\label{lemtrucguai}
There exists some positive integer
$N_1$ depending on $n$ (with $N_1\leq C\,N_0 N$) such that the following holds.
For a given $\theta>0$, consider the interval
$$I_\theta = \big(\theta\,\Lambda_*^{-\frac1{4N}}\delta_0,\, \theta\,\Lambda_*^{\frac1{4N}}\Lambda\big).$$
Let $R_1,R_2,\ldots,R_{N_1}$ be cubes from $\ttt$ such that $R_{k+1}\in\End_*(R_{k})$ for $k\geq1$. 
Then at least one of the cubes $R_k$, with $1\leq k\le N_1$, satisfies
$$\Theta(R_k)\not \in I_\theta.$$
\end{lemma}

\begin{proof}
Recall that
$$\Lambda= \Lambda_*^{1-\frac1N}\quad\mbox{ and }\quad\delta_0 = \Lambda_*^{-N_0 - \frac1{2N}},$$
so that
$$I_\theta = \big(\theta\,\Lambda_*^{-N_0-\frac3{4N}},\, \theta\,\Lambda_*^{1-\frac3{4N}}\big).$$
Consider a sequence $R_1,R_2,\ldots,R_{N_1}$ of cubes from $\ttt$ such that $R_{k+1}\in\End_*(R_{k})$ for $k\geq1$. By the definition of $\End_*(R_k)$ there are only two possibilities: either $R_{k+1}\in\HD_*(R_k)$, or $R_{k+1}$ is a maximal $\PP$-doubling cube contained in some cube from $\LD(R_k)$. Note that in the former case we have $\Theta(R_{k+1})=\Lambda_*\Theta(R_k)$, and in the latter case we have $\Theta(R_{k+1})\le C\delta_0\Theta(R_k)$, by \eqref{eqcad35} and the definition of $\LD(R_k)$.

The key observation is the following: 
\begin{equation}\label{eqsist1}
\Theta(R_k)\leq \theta\,\Lambda_*^{\frac{-1}{3N}} \quad \Rightarrow \quad \mbox{either \;$\Theta(R_{k+1})\not\in I_\theta$ \;or\;
$\Theta(R_{k+1})=\Lambda_*\,\Theta(R_k)$.}
\end{equation}
This follows from the fact that, in the case $\Theta(R_{k+1})\in I_\theta$, we have 
$R_{k+1}\in\HD_*(R_k)$ because otherwise 
$$\Theta(R_{k+1})\leq C\delta_0\,\Theta(R_k) \leq C\,\theta\,\delta_0\,\Lambda_*^{\frac{-1}{3N}}
\leq \theta\,\delta_0\,\Lambda_*^{\frac{-1}{4N}}
.$$
Analogously, 
\begin{equation}\label{eqsist2}
\Theta(R_k)\geq \theta\,\Lambda_*^{\frac{-3}{4N}} \quad \Rightarrow \quad \mbox{either \;$\Theta(R_{k+1})\not\in I_\theta$ \;or\;
$\Theta(R_{k+1})\leq C\delta_0\Theta(R_k)$,}
\end{equation}
because, in the case $\Theta(R_{k+1})\in I_\theta$, we have 
$R_{k+1}\not\in\HD_*(R_k)$,
since otherwise
$$\Theta(R_{k+1}) = \Lambda_*\,\Theta(R_k) \geq \theta\,\Lambda_*^{1-\frac3{4N}}
.$$

To prove the lemma, suppose that $\Theta(R_1)\in I_\theta$. Otherwise we are done. 
By applying \rf{eqsist1} $N_0+1$ times, we deduce that either
$\Theta(R_{k})\not\in I_\theta$ for some $k\in(1,N_0+2]$, or there exists some $k_1\in[1,N_0+1]$ such that
$$\Theta(R_{k_1})\geq \theta\,\Lambda_*^{\frac{-1}{3N}}.$$
Then, from \rf{eqsist2}, we deduce that either 
$\Theta(R_{k_1 + 1})\not\in I_\theta$, or
\begin{equation}\label{eqk111}
\Theta(R_{k_1+1})\leq C\delta_0\,\Theta(R_{k_1})\leq C\,\delta_0\,\theta\,\Lambda_*^{1-\frac3{4N}} \leq \delta_0\,\theta\,\Lambda_*^{1-\frac1{3N}}.
\end{equation}

Now we have:

\begin{claim}
Let $k\geq1$ and $a\in (0,1)$. Suppose that
$$\Theta(R_k) \in (\delta_0\,\theta\,\Lambda_*^{-\frac1{3N}},\,\delta_0\,\theta\,\Lambda_*^{a}).$$
Then, either there exists some
$k_2\in [k+1,k+N_0+1]$ such that $\Theta(R_{k_2})\not\in I_\theta$, or 
$$\Theta(R_{k+N_0+1}) \in 
(\delta_0\,\theta\,\Lambda_*^{-\frac1{4N}},\,\delta_0\,\theta\,\Lambda_*^{a-\frac1{3N}}).$$
\end{claim}

In  case that $a\leq\frac1{12N}$, one should understand that the second alternative is not possible.

\begin{proof}
Suppose that the first alternative in the claim does not hold.
Then we deduce that
$$\Theta(R_{k+j}) = \Lambda_*^j\,\Theta(R_k)\quad \mbox{ for $j=1,\ldots,N_0$,}$$
because, for all $j=1,\ldots,N_0-1$,
$$\Theta(R_{k+j}) \leq \Lambda_*^j\,\Theta(R_k)
\leq \delta_0\,\theta\,\Lambda_*^{N_0-1+a} = \theta\,\Lambda_*^{\frac{-1}{2N}-1+a}\leq \theta\,\Lambda_*^{\frac{-1}{2N}},
$$
and then \rf{eqsist1} implies that $\Theta(R_{k+j+1})=\Lambda_*\,\Theta(R_{k+j})$.
So we infer that
$$\Theta(R_{k+N_0}) = \Lambda_*^{N_0}\,\Theta(R_k)\geq\delta_0\,\theta\,\Lambda_*^{-\frac1{4N} + N_0}
= \theta\,\Lambda_*^{-\frac3{4N}}.
$$
Then, by \rf{eqsist2} we have
\begin{align*}
\Theta(R_{k+N_0+1})& \leq C\delta_0\,\Theta(R_{k+N_0}) = C\delta_0\Lambda_*^{N_0}\,\Theta(R_k)\\
&
= C\Lambda_*^{\frac{-1}{2N}}\,\Theta(R_k)
\leq C\Lambda_*^{\frac{-1}{2N}}\,\delta_0\,\theta\,\Lambda_*^{a} \leq \delta_0\,\theta\,\Lambda_*^{a-\frac1{3N}}.
\end{align*}
\end{proof}
\vv

To complete the proof of the lemma observe that, by \rf{eqk111} and a repeated application of the preceding claim, we infer that
 there exists some $k\in [k_1+2,k_1+C N_0\,N]$ such that $\Theta(R_{k})\not\in I_\theta$, since after $CN$ iterations the second alternative in the lemma is not possible. This concludes the proof of the lemma.
\end{proof}

%*********************************************************************

\vvv
We proceed with the proof of \rf{eqlj83} from Lemma \ref{lemimp9}, that is, the estimate $\#\sL_j(P,k) \leq C_5$.
\begin{proof}[Proof of \rf{eqlj83}]
%To shorten notation, we will write here  $\sL$, $\sL_j$ and $\sL_j^h$ instead of $\sL(\GDF)$, $\sL_j(\GDF)$.

%Next, to shorten notation, we write, we will write here  $\sL$, $\sL_j$ and $\sL_j^h$ instead of $\sL(\GDF)$, $\sL_j(\GDF)$.
For $h\ge 1$ set
\begin{equation*}
\sL_j^{h}(P,k):= \sL_j(P,k)\cap \sL_j^{h}.
\end{equation*}
Notice that each family $\sL_j^h(P,k)$ consists of a single cube, at most.
Indeed, we have
\begin{equation}\label{eq:ljh-inclusion}
R\in\sL_j(P,k)\quad\Rightarrow\quad P\subset B(e^{(3)}(R))
\end{equation}
because $P\in e'(Q)$ for some $Q\in\Trc_k(R)$ and $Q\subset B(e''(R))$ by Lemma \ref{eqtec74}. Thus, if $R,R'\in\sL_j^h(P,k)$,
then $B(e^{(3)}(R))\cap B(e^{(3)}(R'))\neq\varnothing$, which can only happen if $R=R'$ (by Lemma \ref{lemljh} (ii)).

Let $R_0$ be a cube in $\sL_j(P,k)$ with maximal side length, and let $h_0$ be such that $R_0\in \sL_j^{h_0}(P,k)$. We will show that $\sL_j^{h_1}(P,k)\neq\varnothing$ implies $h_0\le h_1\le h_0+C_5$. Together with the observation $\#\sL_j^h(P,k)\le 1$ this will conclude the proof of \rf{eqlj83}.

\begin{claim}
Let $R_1\in\sL_j(P,k)\setminus\{R_0\}$, and let $h_1$ be such that $R_1\in \sL_j^{h_1}(P,k)$. 
Then $h_1\geq h_0$.
\end{claim}

\begin{proof}
Suppose that $h_1< h_0$.
Let $R_0^{h_1}$ be the cube that contains $R_0$ and belongs to
$\sF_j^{h_1}$. Observe that
\begin{multline*}
P \overset{ \eqref{eq:ljh-inclusion}}{\subset}  B(e^{(3)}(R_0)) \cap B(e^{(3)}(R_1)) \subset B(x_{R_0},\tfrac32\ell(R_0)) \cap
B(x_{R_1},\tfrac32\ell(R_1))\\
 \subset B(x_{R_0^{h_1}},\tfrac12\ell(R^{h_1}_0) + \tfrac32\ell(R_0)) \cap
B(x_{R_1},\tfrac32\ell(R_1)).
\end{multline*}
So the two balls $B(x_{R_0^{h_1}},\tfrac12\ell(R^{h_1}_0) + \tfrac32\ell(R_0))$ and 
$B(x_{R_1},\tfrac32\ell(R_1))$ have non-empty intersection. Since $\ell(R_0^{h_1})\ge A_0\,\ell(R_0)\ge A_0\,\ell(R_1)$ (the last inequality follows by the choice of $R_0$), we deduce that
\begin{equation*}
B(e^{(4)}(R_1))\subset B(x_{R_1},\tfrac32\ell(R_1)) \subset B(x_{R_0^{h_1}},\tfrac12\ell(R^{h_1}_0) + 5\ell(R_0))
\subset B(e^{(4)}(R_0^{h_1})).
\end{equation*}
However, these inclusions contradict the property (i) of the family $\sL_j^{h_1}$ in Lemma 
\ref{lemljh} because $R_1\neq R_0^{h_1}$.
\end{proof}

\begin{claim}
Let $R_1\in\sL_j(P,k)\setminus\{R_0\}$, and let $h_1$ be such that $R_1\in \sL_j^{h_1}(P,k)$. 
Then 
\begin{equation}\label{eqclaimh1}
h_1\leq h_0+C.
\end{equation}
% or
%\begin{equation}\label{eqsec85}
%\ell(R_1)\approx\ell(P)\quad\mbox{ and }\quad \dist(R_1,P)\lesssim \ell(P).
%\end{equation}
\end{claim}

\begin{proof}
Suppose that $h_1> h_0+1$. This implies that there are cubes
$\{R_1^h\}_{h_0+1\leq h \leq h_1-1}$, with $R_1^h\in\sF_j^h$ for each $h$, such that
$$R_1^{h_0+1}\supsetneq R_1^{h_0+2}\supsetneq\ldots \supsetneq R_1^{h_1-1}\supsetneq R_1^{h_1}=R_1.$$
Observe now that $\ell(R_1^{h_0+1})< \ell(R_0)$. Otherwise, there exists some cube
$R_1^{h_0}\in\sF_j^{h_0}$ that contains $R_1^{h_0+1}$ with 
$$\ell(R_1^{h_0})\geq A_0\,\ell(R_1^{h_0+1})\geq A_0\,\ell(R_0).$$
Since $P \subset  B(e^{(3)}(R_0)) \cap B(e^{(3)}(R_1))$, arguing as in the previous claim, we deduce that
$B(e^{(4)}(R_0))\subset B(e^{(4)}(R_1^{h_0}))$, which contradicts again the property (i) of the family $\sL_j^{h_0}$ in Lemma \ref{lemljh}, as above. So we have
$$\ell(R_1^h)\leq \ell(R_1^{h_0+1})< \ell(R_0)\quad \mbox{ for $h\geq h_0+1$.}$$

By the construction of $\Trc_k(R_0)$, there exists a sequence of cubes
$S_0=R_0, S_1, S_2, \ldots, S_k=Q$ such that 
$$ S_{i+1}\in \GH(S_i)\; \mbox{ for $i=0,\ldots,k-1$,}$$
and $P\in\TT(e'(S_{k}))$. In case that $P$ is contained in some $Q'\in\HD_1(e'(Q))=\HD_1(e'(S_k))$, we write $S_{k+1}=Q'$, and we let $\tilde k:=k+1$. Otherwise, we let $\tilde k:=k$. All in all, we have
\begin{equation}\label{eq:bla1}
	S_{i+1}\in\HD_1(e'(S_i))\quad\text{for $i=0,\dots,\tilde k-1$},
\end{equation}
and $S_{\tilde{k}+1}:=P\subset e'(S_{\tilde k})$ is not strictly contained in any cube from $\HD_1(e'(S_{\tilde k}))$.

Obviously we have  $\ell(S_{i+1})<\ell(S_i)$ for all $i$.
So, for each $h$ with $h_0+1\leq h\leq h_1$ there is some $i=i(h)$ such that 
\begin{equation}\label{eq0asd}
\ell(S_i)>\ell(R_1^h)\geq \ell(S_{i+1}),
\end{equation}
with $0\leq i \leq \tilde k$. 
We claim that either $i\lesssim 1$ or $i=\wt k$, with the implicit constant depending on $n$. 
Indeed, in the case $i<\wt k$, let $T\in\DD_\mu$ be such that $T\supset S_{i+1}$ and $\ell(T)=\ell(R_1^h)$.
Notice that, since $2R_1^h\cap 2T\neq \varnothing$ (because both $2R_1^h$ and $2T$ contain $P$) and 
$\ell(R_1^h)=\ell(T)$, we have
\begin{equation}\label{eq1asd}
\PP(T) \approx \PP(R_1^h)\approx \Theta(R_1^h)=\Theta(R_0),
\end{equation}
where in the last equality we used the definition of $\sL_j$.
On the other hand, 
\begin{equation}\label{eq2asd}
\PP(T)\geq \delta_0\,\Theta(S_i)
\end{equation} 
because otherwise 
$T$ is contained in some cube from $\LD(S_i)$, which would imply that $S_{i+1}$ does not belong to $\HD_1(e'(S_i))$.
Thus, from \rf{eq1asd} and \rf{eq2asd}
we derive that
$$\Theta(R_0)\gtrsim \delta_0\,\Theta(S_i) =  \delta_0\,\Lambda_*^i\,\Theta(R_0).$$
Hence  $\Lambda_*^i\lesssim\delta_0^{-1}$, which yields $i\lesssim_n 1$ if $i<\wt k$, as claimed.

The preceding discussion implies that, in order to prove \rf{eqclaimh1}, it suffices to show that, for each fixed 
$i=0,\ldots,\tilde k$, there are at most $C=C(n)$ cubes $R_1^h$ satisfying 
\rf{eq0asd} with this fixed $i$. 

\emph{Case $i<\tilde k$}. Assume first that $i<\tilde k$. Recall that $N_1$ is the constant given by Lemma \ref{lemtrucguai}, and suppose that there exist more than $N_1$ cubes $R_1^h$ satisfying 
\rf{eq0asd}. Since $\{R_1^h\}$ is a nested sequence of cubes, this is equivalent to saying that there exists some $s\in [h_0+1,\, h_1-N_1]$ such that
\begin{equation}\label{eqfam11}
\text{for $h\in[s,s+N_1]$ the cubes $R_1^h$ satisfy \rf{eq0asd}.}
\end{equation}
Taking $\theta=\Theta(S_{i})$, Lemma \ref{lemtrucguai} ensures
that there exists some cube $T\in\ttt$ such that $R_1^s\supset T\supset R_1^{s+N_1}$ which satisfies either
\begin{equation}\label{eqdisju9}
\Theta(T)\leq \Lambda_*^{-\frac1{4N}}\delta_0\,\Theta(S_i) \qquad\mbox{ or }\qquad 
\Theta(T)\geq \Lambda_*^{\frac1{4N}}\Lambda\,\Theta(S_i).
\end{equation}
Now, let $T'\in\DD_\mu$ be such that $S_{i+1}\subset T'\subset e'(S_{i})$ and $\ell(T')=\ell(T)$, where we use the fact that $\ell(S_i)>\ell(R_1^s)\ge \ell(T)\ge\ell(R_1^{s+N_1})\ge\ell(S_{i+1})$ and $S_{i+1}\subset e'(S_i)$. Notice that
\begin{equation}\label{eqdisju99}
\PP(T')\approx\PP(T)\approx \Theta(T),
\end{equation}
because $2T\cap 2T'\neq\varnothing$.

If the first option in \rf{eqdisju9} holds, we deduce that
$$\PP(T')\leq C \Theta(T)\leq C\Lambda_*^{-\frac1{4N}}\delta_0\,\Theta(S_i) < \delta_0\,\Theta(S_i).$$
This implies that 
  $T'$ is contained in some cube from $\LD(S_i)\cap\sss(e'(S_i))$, which ensures that 
  $S_{i+1}\not\in\HD_1(e'(S_i))$ (notice that we are using the fact that $i<\tilde k$), which is a contradiction with \eqref{eq:bla1}. Thus, $T$ must satisfy the second estimate of \eqref{eqdisju9}. But in this case \rf{eqdisju99} yields $\PP(T')> \Lambda_*\Theta(S_i)$, and so $T'$ is strictly contained in some cube from $\HD_1(e'(S_i))$. Hence, $S_{i+1}\not\in\HD_1(e'(S_i))$, which again gives a contradiction.
  In consequence, if $i<\tilde k$, then \rf{eqfam11} does not hold for any $s$.
 \vv

\emph{Case $i=\tilde k$}. Assume again that
\rf{eqfam11} holds for some $s\in [h_0+1,\, h_1-N_1]$, and let $s$ be the smallest possible such that \rf{eqfam11} holds. The same argument as above shows that the cube $T'$ from the preceding paragraph is contained in some cube $T''\in\sss(e'(S_i))$. Since we assumed that $s$ is minimal, $R_1^s\supset T\supset R_1^{s+N_1}$, and $\ell(T'')\ge\ell(T')=\ell(T)$, we get that there are at most $N_1$ cubes $R_1^h$ satisfying \rf{eq0asd} such that 
$\ell(S_i)>\ell(R_1^h)\geq \ell(T'')$.
We claim that there is also a bounded number of cubes $R_1^h$ such that 
\begin{equation}\label{eqlastclaim5}
\ell(T'')\geq \ell(R_1^h)\geq \ell(P).
\end{equation}
Indeed, by the definition of the family $\End(e'(R))$ and Lemma
 \ref{lemdobpp}, if we denote by $T_m$ the $m$-th descendant of $T''$ which contains $P$, for $m'\geq m\geq 0$,
 it follows that 
$$\PP(T_{m'})\leq A_0^{-|m-m'|/2}\,\PP(T_m)\leq  A_0^{-m/2}\,\PP(T'').$$
Suppose then that there are two cubes $R_1^h$, $R_1^{h'}$ such that  $\ell(R_1^{h'})\leq \ell(R_1^h)\leq \ell(T'')$.
Let $T_m$ and $T_{m'}$ be such that $\ell(R_1^h)=\ell(T_m)$ and $\ell(R_1^{h'})= \ell(T_{m'})$.
By arguments analogous to the ones in \rf{eqdisju99}, we derive that 
$$\PP(T_m)\approx \PP(R_1^h) \approx \Theta(R_0)\quad \text{ and }\quad
\PP(T_{m'})\approx \PP(R_1^{h'}) \approx \Theta(R_0),$$
where we also used the fact that $\Theta(R_1^{h})=\Theta(R_1^{h'})=\Theta(R_0)$.
On the other hand, since $|m-m'|\geq |h-h'|$, we have
$$\Theta(R_0)\approx\PP(T_{m'})\leq A_0^{-|h-h'|/2}\,\PP(T_m)\approx A_0^{-|h-h'|/2}\,\Theta(R_0),$$
which implies that $|h-h'|\lesssim1$. From this fact it follows that there is a bounded number of cubes $R_1^h$ satisfying \rf{eqlastclaim5}, as claimed. Putting all together, we get \rf{eqclaimh1}.
\end{proof}
This ends the proof of Lemma \ref{lemimp9}.
\end{proof}

\vv

%%%%%%%%%%%%%%%%%%%%%%%%%%%%%%%%%%%%%%%%%%%%%%%%%%%%%%%%%%%%%%%%%%%%%%%%%%%%%%%%%%%%%%%%

\section{The proof of Main Lemma \ref{mainlemma}}\label{sec8}

We have to show that
$$\sigma(\ttt)\leq C\,\big(\|\RR\mu\|_{L^2(\mu)}^2 + \theta_0^2\,\|\mu\|
+ \sum_{Q\in\DB}\EE_\infty(9Q)\big).$$
Recall that by Lemma \ref{lemsuper**9}, we have
$$\sigma(\ttt) \lesssim B^{5/4} \sum_{R\in \sL} \sum_{k\geq0} B^{-k/2}\sum_{Q\in\Trc_k(R)}\sigma(\HD_1(e(Q))) + \theta_0^2\,\|\mu\|.$$
By Lemma \ref{lemalter*} we know that, for each cube $Q$ in the last sum, one of the following alternatives holds:
\begin{itemize}
\item 
$\TT(e'(Q))$ is not typical, so that
$$\sum_{P\in\DB:P\sim\TT(e'(Q))} \EE_\infty(9P) >\Lambda^{\frac{-1}{3n}} \,\sigma(\HD_1(e(Q))),$$ or
\vv

\item $\mu(Z(Q))> \ve_Z\,\mu(Q)$,\; where $Z(Q)$ is the set $Z$ appearing in \eqref{eqdef*f} (replacing $R$ by $Q$ there), which implies that
$$\sigma(\HD_1(e(Q)))\leq \theta_0^2\,\mu(e(Q)) \lesssim \ve_Z^{-1}\,\theta_0^2\,\mu(Z(Q)),$$
or
\vv

\item
$\|\Delta_{\wt \TT(e'(Q))} \RR\mu\|_{L^2(\mu)}^2\geq \Lambda^{-2}\,\sigma(\HD_1(e(Q))).$
\end{itemize}

So the following holds in any case:
$$\sigma(\HD_1(e(Q)))\lesssim \Lambda^2\,\|\Delta_{\wt \TT(e'(Q))} \RR\mu\|_{L^2(\mu)}^2+ \Lambda^{\frac1{3n}}\!\!\sum_{P\in\DB:P\sim\TT(e'(Q))} \!\!\EE_\infty(9P) + \ve_Z^{-1}\, \theta_0^2\,\mu(Z(Q)).$$
Consequently,
\begin{align}\label{eqsumtot93}
\sigma(\ttt) & \lesssim  B^{5/4}\,\Lambda^2
\sum_{R\in\sL}\,  \sum_{k\geq0} B^{-k/2}\sum_{Q\in\Trc_k(R)}
\|\Delta_{\wt \TT(e'(Q))} \RR\mu\|_{L^2(\mu)}^2\\
&\quad+ B^{5/4}\,
\Lambda^{\frac1{3n}}
\sum_{R\in\sL}\,  \sum_{k\geq0} B^{-k/2}\sum_{Q\in\Trc_k(R)}
\sum_{P\in\DB:P\sim\TT(e'(Q))} \!\!\EE_\infty(9P)\nonumber\\
&\quad+ \,B^{5/4}\,
\ve_Z^{-1}\, \theta_0^2\,
\sum_{R\in\sL}\,  \sum_{k\geq0} B^{-k/2}\sum_{Q\in\Trc_k(R)}
\mu(Z(Q)) + \theta_0^2\,\|\mu\|.\nonumber
\end{align}
\vv

Next lemma deals with each of the summands above.

% ********************************************************************************************

\begin{lemma}\label{lemt1t2t3}
We have
\begin{equation}\label{eq:T1*}
T_1:=\sum_{R\in\sL}\,  \sum_{k\geq0} B^{-k/2}\sum_{Q\in\Trc_k(R)}
\|\Delta_{\wt \TT(e'(Q))} \RR\mu\|_{L^2(\mu)}^2
 \lesssim_{\Lambda} \|\RR\mu\|_{L^2(\mu)}^2,
 \end{equation}
 \begin{equation}\label{eq:T2*}
 T_2:=\sum_{R\in\sL}\,  \sum_{k\geq0} B^{-k/2}\sum_{Q\in\Trc_k(R)}
\sum_{P\in\DB:P\sim\TT(e'(Q))} \!\!\EE_\infty(9P)
  \lesssim_{\Lambda}
\sum_{P\in \DB} \EE_\infty(9P),
\end{equation}
and
\begin{equation}\label{eq:T3*}
T_3:=\sum_{R\in\sL}\,  \sum_{k\geq0} B^{-k/2}\sum_{Q\in\Trc_k(R)}
\mu(Z(Q))\lesssim_{\Lambda} \|\mu\|.
\end{equation}
\end{lemma}

From \rf{eqsumtot93} and the estimates above, it follows that 
$$\sigma(\ttt)\lesssim_{\Lambda} \|\RR\mu\|_{L^2(\mu)}^2 + \sum_{P\in \DB} \EE_\infty(9P)+ \ve_Z^{-1}
\theta_0^2\,\|\mu\|.$$
Recalling that $\varepsilon_Z=\varepsilon_Z(\Lambda)$, and $\Lambda=\Lambda(M)$, this concludes the proof of Main Lemma \ref{mainlemma}.
\vv

In order to prove Lemma \ref{lemt1t2t3}, a basic tool to estimate the terms $T_1,T_2,T_3$ is Lemma \ref{lemimp9}, which asserts that
 for all $P\in\DD_\mu$ and all $k\geq0$,
 \begin{equation}\label{eqtrk121}
\#\big\{R\in\sL:\exists \,Q\in\Trc_k(R) \mbox{ such that } P\in\TT(e'(Q))\big\}\leq C_4\,\log\Lambda.
\end{equation}

\subsection{Estimate of \texorpdfstring{$T_1$}{T1}}  \label{subsec9.1**}
Recall that 
\begin{align*}
\Delta_{\wt\TT(e'(Q))}\RR\mu(x) & = \sum_{P\in\wt\End(e'(Q))} \chi_P(x)\,\big(m_{\mu,P}(\RR\mu) - m_{\mu,2Q}(\RR\mu)\big)\\
&\quad +  \chi_{Z(Q)}(x) \big(\RR\mu(x) -  m_{\mu,2R}(\RR\mu)\big).
\end{align*}
For $Q\in\MDW$, we write $S\prec Q$ if $S\in\DD_\mu$ is a maximal cube that belongs to $\TT(e'(Q))$. Then we denote
$$\wh\Delta_Q\RR\mu = \sum_{S\prec Q} (m_{\mu,S}(\RR_\mu) - m_{\mu,2Q}(\RR\mu)\big)\,\chi_S$$
and 
$$\wt E(Q) = \bigcup_{P\in\wt\End(e'(Q))} P,\qquad\quad \wt G(Q) = \wt E(Q) \cup Z(Q)
.$$
Notice that it may happen that $\wt G(Q)\neq e'(Q)$ because of the presence of negligible cubes.
Then we have
\begin{align*}
\sum_{P\in\wt\End(e'(Q))} \!\!\chi_P\,\big(m_{\mu,P}(\RR\mu) - m_{\mu,2Q}(\RR\mu)\big)  &=\!
\sum_{P\in\wt\End(e'(Q))}\!\! \chi_P\,\bigg(\sum_{S\in \wt\TT(e'(Q)):P\subsetneq S}\!\! \Delta_{S}\RR\mu
+ \wh\Delta_Q\RR\mu\bigg)\\
& = \sum_{S\in \wt\TT(e'(Q))\setminus \wt \End(e'(Q))}\!\Delta_{S}\RR\mu \sum_{P\in\wt\End(e'(Q)):P\subset S} \chi_P \\
&\quad + \chi_{\wt E(Q)}\wh\Delta_Q\RR\mu\\
&= \chi_{\wt E(Q)}\bigg(\sum_{S\in \wt\TT(e'(Q))\setminus \wt \End(e'(Q))}\!\! \Delta_{S}\RR\mu
+ \wh\Delta_Q\RR\mu\bigg).
\end{align*}
A similar argument shows that
$$\chi_{Z(Q)}(x) \big(\RR\mu(x) -  m_{\mu,2Q}(\RR\mu)\big) = \chi_{Z(Q)}\bigg(\sum_{S\in \wt\TT(e'(Q))}\! \Delta_{S}\RR\mu
+ \wh\Delta_Q\RR\mu\bigg).$$
Hence,
$$\Delta_{\wt\TT(e'(Q))}\RR\mu = \chi_{\wt G(Q)}\bigg(\sum_{P\in\wt\TT(e'(Q))\setminus \wt\End(e'(Q))} \Delta_P \RR\mu + \wh\Delta_Q\RR\mu\bigg).$$ 
It is also immediate to check that, for a fixed $Q$ and $P\in\wt\TT(e'(Q))\setminus \wt\End(e'(Q))$, the functions 
$\wh\Delta_Q\RR\mu$ and $\Delta_P \RR\mu$ are mutually orthogonal in $L^2(\mu)$. Then, since all the cubes $P\in\wt\TT(e'(Q))$ satisfy 
$P\sim\TT(e'(Q))$, we get
\begin{align*}
\|\Delta_{\wt \TT(e'(Q))} \RR\mu\|_{L^2(\mu)}^2 & \leq
\sum_{P\in\wt\TT(e'(Q))\setminus \wt\End(e'(Q))}\| \Delta_P \RR\mu\|_{L^2(\mu)}^2  + \|\wh\Delta_Q\RR\mu\|_{L^2(\mu)}^2 \\
& \leq \sum_{P\sim\TT(e'(Q))} \|\Delta_P \RR\mu\|_{L^2(\mu)}^2  + \|\wh\Delta_Q\RR\mu\|_{L^2(\mu)}^2.
\end{align*}
Therefore,
\begin{align*}
T_1 & \lesssim_{\Lambda}\sum_{R\in\sL}\,
\sum_{k\geq0} B^{-k/2} \sum_{Q\in\Trc_k(R)}
\sum_{P\sim\TT(e'(Q))} \|\Delta_P \RR\mu\|_{L^2(\mu)}^2\\
&\quad + \sum_{R\in\sL}\,\sum_{k\geq0} B^{-k/2} \sum_{Q\in\Trc_k(R)}
\|\wh\Delta_Q\RR\mu\|_{L^2(\mu)}^2\\
& =: T_{1,1} + T_{1,2}.
\end{align*}
Regarding the term $T_{1,1}$, by Fubini we have
\begin{align*}
T_{1,1}
 &\leq
\sum_{P\in \DD_\mu} \|\Delta_P \RR\mu\|_{L^2(\mu)}^2
 \sum_{k\geq0} B^{-k/2}\,
\# A(P,k),
\end{align*}
where
\begin{equation}\label{eqApq2}
A(P,k)= 
\big\{R\in \sL:\exists \,Q\in\Trc_k(R) \text{ such that }P\sim\TT(e'(Q))\big\}.
\end{equation}
From the definition  \rf{defsim0} and Lemma \ref{lemimp9}, it follows that
\begin{align*}
\#A(P,k) &\leq \sum_{\substack{
P'\in\DD_\mu: 20P'\cap20P\neq\varnothing\\A_0^{-2}\ell(P)\leq \ell(P')\leq A_0^2\ell(P)
}} \!\!\#
\big\{R\in \sL:\exists \,Q\in\Trc_k(R) \text{ such that }P'\in\TT(e'(Q))\big\}\\
&\lesssim \sum_{\substack{
P'\in\DD_\mu: 20P'\cap20P\neq\varnothing\\A_0^{-2}\ell(P)\leq \ell(P')\leq A_0^2\ell(P)
}}\!\!\!\!\log\Lambda\lesssim \log\Lambda.
\end{align*}
Hence,
$$T_{1,1}
 \lesssim_{\Lambda}
\sum_{P\in \DD_\mu} \|\Delta_P \RR\mu\|_{L^2(\mu)}^2= \|\RR\mu\|_{L^2(\mu)}^2.$$
Concerning $T_{1,2}$, we argue analogously:
\begin{align*}
T_{1,2}
 &\le 
\sum_{Q\in \MDW} \|\wh\Delta_Q\RR\mu\|_{L^2(\mu)}^2
 \sum_{k\geq0} B^{-k/2}\,
\# \wt A(Q,k),
\end{align*}
where
$$\wt A(Q,k)= \big\{R\in \sL:
Q\in\Trc_k(R)\big\}.$$
Since 
$$\#\wt A(Q,k) \leq\#A(Q,k)\lesssim\log\Lambda,$$
we deduce that
$$T_{1,2}
 \lesssim_{\Lambda}
\sum_{Q\in \MDW} \|\wh \Delta_Q \RR\mu\|_{L^2(\mu)}^2.$$
As the next lemma shows, the right hand side above is also bounded by $C\|\RR\mu\|_{L^2(\mu)}^2$.
So we have
$T_1
 \lesssim_{\Lambda} \|\RR\mu\|_{L^2(\mu)}^2.$

\vv

\begin{lemma}\label{lemortog}
For any $f\in L^2(\mu)$, we have
$$\sum_{Q\in \MDW} \|\wh \Delta_Q f\|_{L^2(\mu)}^2\lesssim \|f\|_{L^2(\mu)}^2.$$
\end{lemma}

We defer the proof of this result to Section \ref{sec10.3}.

\vv
% ********************************************************************************************

\subsection{Estimate of \texorpdfstring{$T_2$}{T2}}

By the same argument we used to deal with $T_{1,1}$ above, we get
$$T_2
 \leq
\sum_{P\in \DB} \EE_\infty(9P)
 \sum_{k\geq0} B^{-k/2}\,
\# A(P,k),
$$
where $A(P,k)$ is given by \rf{eqApq2}.
Since $\#A(P,k)\lesssim\log\Lambda$, we obtain
$$T_2
 \lesssim_{\Lambda}
\sum_{P\in \DB} \EE_\infty(9P).$$

\vv
% ********************************************************************************************

\subsection{Estimate of \texorpdfstring{$T_3$}{T3}}

We have
\begin{align*}
T_3 & = \sum_{R\in\sL}\,
\sum_{k\geq0} B^{-k/2} \!\!\!\!\sum_{Q\in\Trc_k(R)}\!
\theta_0^2\,\ve_Z^{-1}\,\mu(Z(Q))\\
& =
\sum_{R\in\sL}\,
\sum_{k\geq0} B^{-k/2} \!\!\!\!\sum_{Q\in\Trc_k(R)}\!
\int_{Z(Q)}\theta_0^2\,\ve_Z^{-1}\,d\mu\\
&= \int\theta_0^2\,\ve_Z^{-1}\, \bigg(\sum_{R\in\sL}\,\sum_{k\geq0} B^{-k/2}\!\!\!
\sum_{Q\in\Trc_k(R)}\!\chi_{Z(Q)}\bigg)\,d\mu.
\end{align*}
By Fubini, we have
$$\sum_{R\in\sL}\,\sum_{k\geq0} B^{-k/2}\!\!\!
\sum_{Q\in\Trc_k(R)}\!\chi_{Z(Q)} \leq \sum_{k\geq0} B^{-k/2} \,\# D(x,k),$$
where
$$D(x,k) = 
\big\{R\in \sL:\exists \,Q\in\Trc_k(R) \text{ such that }x\in Z(Q)\big\}.
$$
Observe now that, given $j\ge1$, if we let
\begin{align*}D_j(x,k) = 
\big\{R\in \sL: \exists \,Q\in&\Trc_k(R) \text{ such that $\TT(e'(Q))$ contains} \\
&\text{
every $P\in\DD_\mu$ such that $x\in P$ and $\ell(P)\leq A_0^{-j}$}\big\},
\end{align*}
then we have
$$D(x,k) = \bigcup_{j\geq 1} D_j(x,k),$$
and moreover $D_j(x,k)\subset D_{j+1}(x,k)$ for all $j$.
From Lemma \ref{lemimp9} we deduce that
$$\#D_j(x,k)\leq C\,\log\Lambda\quad \mbox{ for all $j\geq 1$.}$$
Thus, $\#D(x,k)\leq C\,\log\Lambda$ too. Consequently,
$$\sum_{R\in\sL}\,\sum_{k\geq0} B^{-k/2}\!\!\!
\sum_{Q\in\Trc_k(R)}\!\chi_{Z(Q)} \lesssim_{\Lambda} 1,$$
and so
$T_3\lesssim_{\Lambda} \|\mu\|.$
Together with the estimate we obtained for $T_1$ and $T_2$, this concludes the proof of Lemma \ref{lemt1t2t3}, modulo the proof of Lemma \ref{lemortog}.
\vv

% ********************************************************************************************

\subsection{Proof of Lemma \ref{lemortog}}\label{sec10.3}
For $Q\in \MDW$, denote by $\cA(Q)$ the family of cubes $R\in\DD_\mu$ such that $R\cap2Q\neq \varnothing$
and $\ell(R)=A_0\,\ell(Q)$. 
Also, let $\cF(Q)$ be the family of cubes $P$ which are contained in some cube from $\cA(Q)$ and satisfy
$\ell(Q)\leq \ell(P)\leq A_0\,\ell(Q)$. Notice that, by Lemma \ref{lempois00}, the cubes from $\cA(Q)$ belong to $\DD_\mu^{db}$. So taking into account that $Q\subset 2B_R$ for any $R\in\cA(Q)$, that $R\subset CQ$, and that
$Q$ is $\PP$ doubling,
\begin{equation}\label{eqcompar39}
\mu(Q)\approx \mu(R)\quad\mbox{ for all $R\in\cA(Q)$.}
\end{equation}

Denote
$$\widecheck \Delta_Q f = \sum_{R\in\cA(Q)} \big(m_{\mu,R}(f)- m_{\mu,2Q}(f)\big)\,\chi_{R}.$$
It is immediate to check that
$$\|\wh\Delta_Q f\|_{L^2(\mu)}\lesssim \sum_{P\in\cF(Q)}\|\Delta_Pf\|_{L^2(\mu)} + \|\widecheck\Delta_Q f\|_{L^2(\mu)}.$$
Remark that the main advantage of the operator $\wk\Delta_Q$ over $\wh\Delta_Q$ is that the cubes
$R\in\cA(Q)$ involved in the definition of $\wk\Delta_Q$ are doubling, which may not be the case for the cubes $S\prec Q$ in the definition of $\wh\Delta_Q$.
From the last inequality, we get
$$\sum_{Q\in\MDW}\|\wh\Delta_Q f\|_{L^2(\mu)}^2 \lesssim \sum_{Q\in\MDW}\sum_{P\in\cF(Q)}\|\Delta_Pf\|_{L^2(\mu)}^2
+ \sum_{Q\in\MDW}\|\widecheck\Delta_Q f\|_{L^2(\mu)}^2 
$$
Since
\begin{align*}
\sum_{Q\in\MDW}\sum_{P\in\cF(Q)}\|\Delta_Pf\|_{L^2(\mu)}^2 &\leq \sum_{P\in\DD_\mu}\|\Delta_Pf\|_{L^2(\mu)}^2
\sum_{Q\in\MDW: P\in\cF(Q)}1\\
&\lesssim \sum_{P\in\DD_\mu}\|\Delta_Pf\|_{L^2(\mu)}^2\lesssim\|f\|_{L^2(\mu)}^2,
\end{align*}
the lemma follows from the next result.

\begin{lemma}\label{lemortog2}
For any $f\in L^2(\mu)$, we have
$$\sum_{Q\in \DD_\mu^\PP} \|\wk \Delta_Q f\|_{L^2(\mu)}^2\lesssim \|f\|_{L^2(\mu)}^2.$$
\end{lemma}

\begin{proof}
For $Q\in \DD_\mu^\PP$ and $P\in\DD_\mu$ such that $P\in\cA(Q)$, let
$$\vphi_{Q,P} = \left(\frac1{\mu(P)}\,\chi_P - \frac1{\mu(2Q)}\,\chi_{2Q}\right)\,\mu(P)^{1/2}.$$
Observe that
$$\|\wk\Delta_Q f\|_{L^2(\mu)}^2 = \sum_{P\in\cA(Q)} \big|m_{\mu,P}(f) - m_{\mu,2Q}(f)\big|^2\,\mu(P) 
=\sum_{P\in\cA(Q)}\langle f,\vphi_{Q,P}\rangle^2 .$$
So we have to show that
$$\sum_{Q\in\DD_\mu^\PP} \sum_{P\in\cA(Q)}\langle f,\,\vphi_{Q,P}\rangle^2 \lesssim \|f\|_{L^2(\mu)}^2.$$
To shorten notation, we denote by $\cI$ the set of all pairs $(Q,P)$ with $Q\in\DD_\mu^\PP$ and $P\in\cA(Q)$, so that 
the double sum above can be written as $\sum_{(Q,P)\in \cI}$.

Arguing by duality, we have
\begin{align*}
\bigg(\sum_{(Q,P)\in \cI}\langle f,\vphi_{Q,P}\rangle^2 \bigg)^{1/2}& =
\sup \bigg| \sum_{(Q,P)\in \cI}\langle f,\,\vphi_{Q,P}\rangle \,b_{Q,P}\bigg|=  \sup \bigg| \bigg\langle f,\,\sum_{(Q,P)\in \cI} b_{Q,P}\,\vphi_{Q,P}\bigg\rangle \,\bigg|
,
\end{align*}
where the supremum is taken over all the sequences $b:=\{b_{Q,P}\}_{(Q,P)\in \cI}$
such that $\|b\|_{\ell^2}\leq1$.
Since
$$\bigg(\sum_{(Q,P)\in \cI}\langle f,\vphi_{Q,P}\rangle^2 \bigg)^{1/2} \leq \big\|f\big\|_{L^2(\mu)}\,\sup\bigg\|\sum_{(Q,P)\in \cI} b_{Q,P}\,\vphi_{Q,P}\bigg\|_{L^2(\mu)},$$
to prove the lemma it suffices to show that
$$\bigg\|\sum_{(Q,P)\in \cI} b_{Q,P}\,\vphi_{Q,P}\bigg\|_{L^2(\mu)}\lesssim 1\quad\mbox{ for all
$b=\{b_{Q,P}\}_{(Q,P)\in \cI}$
such that $\|b\|_{\ell^2}\leq1$.}$$
To this end, we write
\begin{align}\label{eqdobb836}
\bigg\|\sum_{(Q,P)\in \cI} b_{Q,P}\,\vphi_{Q,P}\bigg\|_{L^2(\mu)}^2 & = 
\sum_{(Q,P),(R,S)}\big\langle b_{Q,P}\,\vphi_{Q,P},\, b_{R,S}\,\vphi_{R,S}\big\rangle \\
& \le 2
\sum_{\substack{(Q,P),(R,S)\\ \ell(Q)\leq \ell(R)}} |b_{Q,P}\, b_{R,S}| \,\, \big|\big\langle\vphi_{Q,P},\, \vphi_{R,S}\big\rangle \big|.\nonumber
\end{align}

Denote 
$$a(Q)=\bigcup_{P\in\cA(Q)} P.$$
Observe that, for some $C$ depending just on $A_0$,
$$\supp\vphi_{Q,P}\subset \overline{a(Q)}\subset CQ,\qquad \supp\vphi_{R,S}\subset \overline{a(R)}\subset CR,$$
and
$$\big\|\vphi_{Q,P}\big\|_{L^\infty(\mu)}\lesssim \frac1{\mu(Q)^{1/2}},\qquad \big\|\vphi_{R,S}\big\|_{L^\infty(\mu)}\lesssim \frac1{\mu(R)^{1/2}},$$
 taking into account \rf{eqcompar39}.
Thus, for $(Q,P)$, $(R,S)\in\mathcal{I}$ with $\ell(Q)\leq\ell(R)$, we have
$$\big|\big\langle\vphi_{Q,P},\, \vphi_{R,S}\big\rangle \big| = \left|\int_{a(Q)} \vphi_{Q,P}\, \vphi_{R,S}\,d\mu\right|
\lesssim \frac{\mu(a(Q))}{\mu(Q)^{1/2}\,\mu(R)^{1/2}}
 \approx \bigg(\frac{\mu(Q)}{\mu(R)}\bigg)^{1/2}.$$
Further, using that $\vphi_{Q,P}$ has zero mean and that $\vphi_{R,S}$ is constant in $2R\cap S$, in $2R\setminus S$, 
and in $S\setminus 2R$,
it follows that $\big\langle\vphi_{Q,P},\, \vphi_{R,S}\big\rangle =0$ in the following cases:
\begin{itemize}
\item[(i)] if $a(Q)\cap (2R\cup S) = \varnothing$,
\item[(ii)] if $a(Q)\subset 2R\cap S$,
\item[(iii)] if $a(Q)\subset 2R\setminus S$,
\item[(iv)] if $a(Q)\subset S\setminus 2R$.
\end{itemize}
For $d>0$, denote
$$\mathcal N_d(S) = \{x\in\supp\mu\setminus S:\dist(x,S)\leq d\}
\cup \{x\in S:\dist(x,\supp\mu\setminus S)\leq d\}$$
and, analogously,
$$\mathcal N_d(2R) = \{x\in\supp\mu\setminus 2R:\dist(x,2R)\leq d\}
\cup \{x\in 2R:\dist(x,\supp\mu\setminus 2R)\leq d\}.$$
Observe that if none of the conditions (i), (ii), (iii), (iv), holds, then
\begin{equation}\label{eqdd22}
a(Q)\subset \mathcal N_{\diam(a(Q))}(S) \cup \mathcal N_{\diam(a(Q))}(2R).
\end{equation}

From the thin boundary condition \rf{eqfk490} and the fact that $2R$ is a finite number of
 cubes of the same generation as $R$, using also that $R$ is $\PP$-doubling and $S$ is doubling, we deduce that
\begin{equation}\label{eqdd23}
\mu\big(\mathcal N_d(S)\big) + \mu\big(\mathcal N_d(2R)\big) \lesssim \left(\frac d{\ell(R)}\right)^{1/2}\,\mu(R)
\quad \mbox{ for all $d\in (0,C\ell(R)$),}
\end{equation}
with the implicit constant depending on $C$.
Consequently, denoting by ``$(Q,P) \dashv (R,S)$" the situation when $\ell(Q)\leq \ell(R)$ and \rf{eqdd22} holds, by \rf{eqdobb836} we get
\begin{align*}
\bigg\|&\sum_{(Q,P)\in \cI} b_{Q,P}\,\vphi_{Q,P}\bigg\|_{L^2(\mu)}^2  \le 2
\sum_{(Q,P)\dashv(R,S)} |b_{Q,P}\, b_{R,S}| \,\, \bigg(\frac{\mu(Q)}{\mu(R)}\bigg)^{1/2}\\
&\quad\leq 2\bigg(\sum_{(R,S)\in \cI} |b_{R,S}|^2\bigg)^{1/2}\, \Bigg(\sum_{(R,S)\in \cI} 
\Bigg(\sum_{\substack{(Q,P)\in \cI:\\ (Q,P)\dashv(R,S)}}\!\!\! |b_{Q,P}| \, \bigg(\frac{\mu(Q)}{\mu(R)}\bigg)^{1/2}\Bigg)^2\Bigg)^{1/2}\\
&\quad \lesssim \Bigg( \sum_{(R,S)\in\cI} \Bigg(
\sum_{\substack{(Q,P)\in \cI:\\ (Q,P)\dashv(R,S)}}\!\!\! |b_{Q,P}|^2 \,\bigg(\frac{\ell(Q)}{\ell(R)}\bigg)^{1/4}\Bigg)
 \,
\Bigg( \sum_{\substack{(Q,P)\in \cI:\\ (Q,P)\dashv(R,S)}} \bigg(\frac{\ell(R)}{\ell(Q)}\bigg)^{1/4}
\frac{\mu(Q)}{\mu(R)}\Bigg)\Bigg)^{1/2}.
\end{align*}
We consider now the last sum on the right hand side, which equals
\begin{align*}
\sum_{\substack{(Q,P)\in \cI:\\ (Q,P)\dashv(R,S)}} \bigg(\frac{\ell(R)}{\ell(Q)}\bigg)^{1/4}
\frac{\mu(Q)}{\mu(R)} = \sum_{k\geq 0}
\sum_{\substack{(Q,P)\in \cI:\\ (Q,P)\dashv(R,S)\\
\ell(Q)=A_0^{-k}\ell(R)}}\!\! A_0^{k/4}\,\,
\frac{\mu(Q)}{\mu(R)} 
\end{align*}
Notice that, by \rf{eqdd22} and \rf{eqdd23},  we have
$$\sum_{\substack{(Q,P)\in \cI:\\ (Q,P)\dashv(R,S)\\
\ell(Q)=A_0^{-k}\ell(R)}}\mu(Q) \lesssim 
\mu\big(\mathcal N_{CA_0^{-k}\ell(R)}(S)\big) + \mu\big(\mathcal N_{CA_0^{-k}\ell(R)}(2R)\big)\lesssim A_0^{-k/2}\,\mu(R).
$$
Therefore,
$$\sum_{\substack{(Q,P)\in \cI:\\ (Q,P)\dashv(R,S)}} \bigg(\frac{\ell(R)}{\ell(Q)}\bigg)^{1/4}
\frac{\mu(Q)}{\mu(R)}\lesssim
\sum_{k\geq 0} A_0^{k/4}\,\,
\frac{A_0^{-k/2}\,\mu(R)}{\mu(R)} \lesssim 1.$$

We deduce that
\begin{align*}
\bigg\|\sum_{(Q,P)\in \cI} b_{Q,P}\,\vphi_{Q,P}\bigg\|_{L^2(\mu)}^2 
& \lesssim
 \Bigg( \sum_{(R,S)\in\cI} \Bigg(
\sum_{\substack{(Q,P)\in \cI:\\ (Q,P)\dashv(R,S)}}\!\!\! |b_{Q,P}|^2 \,\bigg(\frac{\ell(Q)}{\ell(R)}\bigg)^{1/4}\Bigg)
\Bigg)^{1/2}\\
& =  \Bigg( \sum_{(Q,P)\in\cI} |b_{Q,P}|^2 
\sum_{\substack{(R,S)\in \cI:\\ (Q,P)\dashv(R,S)}}\!\!\bigg(\frac{\ell(Q)}{\ell(R)}\bigg)^{1/4}
\Bigg)^{1/2}.
\end{align*}
Since
$$\sum_{\substack{(R,S)\in \cI:\\ (Q,P)\dashv(R,S)}}\!\!\bigg(\frac{\ell(Q)}{\ell(R)}\bigg)^{1/4}
\lesssim 
\sum_{\substack{R\in\DD_\mu:\\ \ell(R)\geq \ell(Q)\\ a(Q)\cap a(R)\neq\varnothing}}\!\!\bigg(\frac{\ell(Q)}{\ell(R)}\bigg)^{1/4}\lesssim1,$$
we infer that
$$\bigg\|\sum_{(Q,P)\in \cI} b_{Q,P}\,\vphi_{Q,P}\bigg\|_{L^2(\mu)}^2 \lesssim1,$$
as wished.
\end{proof}

\vv

% *******************************************************************************************************************

% ********************************************************************************************

\section{The proof of the First Main Proposition \ref{propomain}} \label{sec9}

Recall that, for a cube $R\in\ttt$, $\tree(R)$ denotes the subfamily of the cubes from $\DD_\mu(R)$ which are not strictly contained in any cube
from $\End_*(R)$.
In this section we will prove the following result. 

\begin{lemma}\label{lemtreebeta}
For each $R\in\ttt$, the following holds:
\begin{align*}
\sum_{Q\in\tree(R)} \!\beta_{\mu,2}(2B_Q)^2\,\Theta(Q)\,\mu(Q) & \lesssim_{\Lambda_*,\delta_0}\!\sum_{Q\in\tree(R)}\!\|\Delta_Q\RR\mu\|_{L^2(\mu)}^2\\&\quad \;+ \!\sum_{Q\in\tree(R)\cap\DB}\! \EE_\infty(9Q) +\Theta(R)^2\,\mu(R).
\end{align*}
\end{lemma}

Observe that summing the above inequality over $R\in\ttt$ we get
\begin{multline*}
\sum_{Q\in\DD_\mu} \!\beta_{\mu,2}(2B_Q)^2\,\Theta(Q)\,\mu(Q) \le \sum_{R\in\ttt} \sum_{Q\in\tree(R)} \!\beta_{\mu,2}(2B_Q)^2\,\Theta(Q)\,\mu(Q)  \\
\lesssim_{\Lambda_*,\delta_0} \sum_{R\in\ttt} \sum_{Q\in\tree(R)}\|\Delta_Q\RR\mu\|_{L^2(\mu)}^2 +\sum_{R\in\ttt} \sum_{Q\in\tree(R)\cap\DB}\! \EE_\infty(9Q) +\sum_{R\in\ttt}\Theta(R)^2\,\mu(R)\\
\lesssim \|\RR\mu\|_{L^2(\mu)}^2 +  \sum_{Q\in\DB}\EE_\infty(9Q) + \sum_{R\in\ttt}\Theta(R)^2\,\mu(R).
\end{multline*}
Together with Main Lemma \ref{mainlemma}, and recalling that $\delta_0 = \delta_0(\Lambda_*)$ and $\Lambda_*=\Lambda_*(M)$, this yields Main
Proposition \ref{propomain}.

% ********************************************************************************************

\vv
\subsection{The approximating measure \texorpdfstring{$\eta$}{eta} on a subtree \texorpdfstring{$\wh\tree_0(R)$}{Tree\_0(R)}}\label{subsec:91}

To prove Lemma \ref{lemtreebeta} for a given cube $R_0\in\ttt$ (in place of $R$), we will consider a corona 
decomposition of $\tree(R_0)$ into subtrees by introducing appropriate new stopping conditions. 
In this section we will deal with the construction of each subtree and an associated AD-regular measure
which approximates $\mu$ in that subtree.
To this end we need some additional notation. 

First, for a cube 
$R\in\tree(R_0)\cap\DD_\mu^\PP$, we write $Q\in \BR(R)$ (which stands for ``big Riesz transform'') if $Q$ is a $\PP$-doubling maximal cube which does not belong to $\HD_*(R_0)\cup\LD(R_0)$ and satisfies
$$|\RR_\mu\chi_{2R\setminus 2Q}(x_Q)|\geq K\,\Theta(R),$$
where $K$ is some big constant to be fixed below, depending on $\Lambda_*$, $\delta_0$, and $M$. Also, for a cube $Q\in\tree(R_0)$,
we denote by $\wh \Ch(Q)$ the family of maximal cubes $P\in\DD_\mu(Q)\setminus\{Q\}$ that satisfy one of the following conditions:
\begin{itemize} 
\item $P\in\DD_\mu^\PP$, i.e.\ $P$ is $\PP$-doubling, or
\item $P\in\LD(R_0)$.
\end{itemize}
From Lemma \ref{lemdobpp}, it is immediate to check that if $Q$ is not contained in any cube from $\LD(R_0)$, then the cubes from $\wh\Ch(Q)$
cover $Q$, and also 
\begin{equation}\label{eqcompa492}
\ell(P)\approx_{\Lambda_*,\delta_0}\ell(Q)\quad \mbox{ for each $P\in\wh\Ch(P)$,} 
\end{equation}

Given $R\in\DD_\mu^\PP\in\tree(R_0)\setminus \End_*(R_0)$, we will construct a tree $\wh \tree_0(R)$ inductively, consisting just of $\PP$-doubling cubes and stopping cubes from $\LD(R_0)$. At the
same time we will construct an approximating AD-regular measure for this tree. We will do this by ``spreading''
the measure of the cubes from $\wh\tree_0(R)\cap \LD(R_0)$ among the other cubes from $\wh\tree_0(R)$. To this end, we will consider some coefficients $s(Q)$, $Q\in\DD_\mu^\PP\cap\wh\tree_0(R)$, that, in a sense, quantify the additional measure $\mu$ spreaded on $Q$
due to the presence of close cubes from $\LD(R_0)$. The algorithm is the following.

First we choose $R$ as the root of $\wh \tree_0(R)$, and we set $s(R)=0$. Next, suppose that $Q\in\wh \tree_0(R)$
(in particular, this implies that $Q\in\tree(R_0)$), and assume that we have not 
decided yet if the cubes from $\wh\Ch(Q)$ belong to $\wh \tree_0(R)$. 
First we decide that $Q\in\wh\sss(R)$ if one of the following conditions hold:
\begin{itemize}
\item[(i)] $Q\in\HD_*(R_0)\cup\LD(R_0)\cup\BR(R)$, or\vspace{1mm}

\item[(ii)] $s(Q)\geq \mu(Q)$ and (i) does not hold, or \vspace{1mm}

\item[(iii)] $\sum_{P\in\wh\Ch(Q)\cap \LD(R_0)}\mu(P) \geq \frac12\,\mu(Q)$ and neither (i) nor (ii) hold.
\end{itemize}
\vspace{1mm}

\noi If $Q\in\wh\sss(R)$, no descendants of $Q$ are allowed to belong to $\wh\tree_0(R)$.
Otherwise, all the cubes from $\wh\Ch(Q)$ are chosen to belong to $\wh\tree_0(R)$, and for each $P\in\wh\Ch(Q)$, we define
$$s(P)=-\mu(P) \quad \mbox{ if $P\in\LD(R_0)$,}$$
and, otherwise, we set 
\begin{equation}\label{eqdqaq14}
t(Q) = \sum_{S\in\wh\Ch(Q)\cap\LD(R_0)}\mu(S)\quad \mbox{ and }\quad
s(P) = \big(s(Q) + t(Q)\big)\, 
\frac{\mu(P)}{\mu(Q)-t(Q)}.
\end{equation}
Observe that
\begin{align*}
\sum_{P\in\wh\Ch(Q)}s(P) & = \sum_{P\in\wh\Ch(Q)\cap \LD(R_0)} s(P) + \sum_{P\in\wh\Ch(Q)\setminus \LD(R_0)}s(P) \\
& = -t(Q) + \big(s(Q) + t(Q)\big) = s(Q).
\end{align*}

By induction, the coefficients $s(\cdot)$ satisfy the following.
If $Q\in\wh \tree_0(R)$ and $\cI$ is some finite family of cubes from $\wh\tree_0(R)\cap\DD_\mu(Q)$
which cover $Q$ and are disjoint, then
\begin{equation}\label{eqaq89}
\sum_{P\in \cI} s(P) = s(Q).
\end{equation}
Further, $s(Q)\geq0$ for all $Q\in\wh\tree_0(R)\setminus \LD(R_0)$.

Now we are ready to define an approximating measure $\eta$ associated with $\wh\tree_0(R)$. 
First, we denote
$$\wh \sG(R) = R\,\setminus \bigcup_{Q\in\wh\sss(R)} Q,$$
and for each $Q\in\DD_\mu$ we let $D_Q$ be an $n$-dimensional disk passing through $x_Q$
with radius $\frac12 \,r(Q)$ (recall that $r(Q)$ is the radius of $B(Q)$).
In  case that $\mu(\wh\sG(R))=0$, we define
$$\eta = \sum_{Q\in\wh \sss(R)} \big(s(Q) + \mu(Q)\big)\,\frac{\HH^n\rest_{D_Q}}{\HH^n(D_Q)}.$$
Observe that $\eta(D_Q)=\mu(Q)+ s(Q)$ for all $Q\in\wh\sss(R)$ and, in particular $\eta(D_Q)=0$
if $Q\in\wh\sss(R)\cap\LD(R_0)$.

In case that $\mu(\wh\sG(R))\neq0$, we have to be a little more careful. For a given
$N\geq1$ we let $\wh\sss_N(R)$ be the family consisting of all the cubes from 
$\wh\sss(R)$ with side length larger that $A_0^{-N}\,\ell(R)$, and we let $\cI_N$ be the family of the cubes from $
\wh\tree_0(R)$ which have side length smaller than $A_0^{-N}\,\ell(R)$ and are maximal.
We denote 
\begin{equation}\label{eqdefetaN}
\eta_N = \sum_{Q\in\wh \sss_N(R)} \big(s(Q) + \mu(Q)\big)\,\frac{\HH^n\rest_{D_Q}}{\HH^n(D_Q)}+
\sum_{Q\in \cI_N(R)} \big(s(Q) + \mu(Q)\big)\,\frac{\mu\rest_Q}{\mu(Q)}
,
\end{equation}
and we let $\eta$ be a weak limit of $\eta_N$ as $N\to\infty$.
%In this case, it is immediate to check that \rf{eqetq5} holds too.

As in Section \ref{sec6.2*}, we use the following notation.
To each $Q\in\wh\tree_0(R)$ we associate another ``cube'' $Q^{(\eta)}$ defined as follows:
$$Q^{(\eta)}= (\sG(R)\cap Q)\cup \bigcup_{P\in\wh \sss(R):P\subset Q} D_P.$$
We let
$$\wh\tree_0^{(\eta)}(R) \equiv \wh\tree_0^{(\eta)}(R^{(\eta)}):= \{Q^{(\eta)}: Q\in\wh\tree_0(R)\}.$$
%Further, we consider a lattice $\DD_\eta$ associated with the measure $\eta$ which is made up of the cubes
%$Q^{(\eta)}$ with $Q\in \TT_{\Reg}(e'(R))$ and other cubes which are descendants of the cubes from $\Reg(e'(R))$.
%We assume that $\DD_\eta$ satisfies the first two properties of Lemma \ref{lemcubs} with the same parameters $A_0$ and $C_0$ as $\DD_\mu$. It is straightforward to check that $\DD_\eta$ can be constructed in this way.
For $S=Q^{(\eta)}\in \wh\tree_0^{(\eta)}(R)$ with $Q\in\wh\tree_0(R)$, we denote $Q=S^{(\mu)}$ and we write
$\ell(S):=\ell(Q)$.

Observe now that, from \rf{eqaq89} and the definition of $\eta$, we have the key property
\begin{equation}\label{eqetq5}
\eta(Q^{(\eta)})= \mu(Q) + s(Q)\quad \mbox{ for all $Q\in\wh\tree_0(R)$}.
\end{equation}
So $s(Q)$ is the measure added to $\mu(Q)$ to obtain $\eta(Q^{(\eta)})$.

\vv

\begin{lemma}\label{lem9.2}
The measure $\Theta(R_0)^{-1}\eta$ is AD-regular (with a constant depending on $\Lambda_*$ and $\delta_0$), and 
$\eta(Q^{(\eta)})=0$ for all $Q\in\LD(R_0)\cap\wh\tree_0(R)$.
\end{lemma}

\begin{proof}
The fact that $\eta(Q^{(\eta)})=0$ for all $Q\in\LD(R_0)\cap\wh\tree_0(R)$ follows by construction and has already been mentioned above.
To prove the AD-regularity of $\eta$, by standard arguments, it is enough to show that
$$\eta(Q^{(\eta)})\approx_{\Lambda_*,\delta_0} \Theta(R_0)\,\ell(Q)^n\quad\mbox{ for all $Q\in\wh\tree_0(R)\setminus
\LD(R_0)$,}$$
taking into account \rf{eqcompa492}.
Given such a cube $Q$, the fact that $Q\not\in\LD(R_0)$ ensures that
$$\eta(Q^{(\eta)})\geq \mu(Q)\gtrsim\delta_0\,\Theta(R_0)\,\ell(Q)^n.
$$
To show the converse estimate we can assume $Q\neq R$. By
the condition (ii) in the definition of $\wh \sss(\wh Q)$, where $\wh Q$ is the first ancestor of $Q$ in $\wh \tree_0(R)$
(i.e., $\wh Q$ is the smallest cube from $\wh\tree_0(R)$ that strictly contains $Q$), we have
$$s(\wh Q)\leq \mu(\wh Q).$$
Also, by (iii) (which does not hold for $\wh Q$), the coefficient $t(\wh Q)$ in \rf{eqdqaq14} satisfies
$$t(\wh Q) = \sum_{P\in\wh\Ch(\wh Q)\cap \LD(R_0)}\mu(P) < \frac12\,\mu(\wh Q).$$
Therefore,
\begin{equation}\label{eqaq745}
s(Q) = \big(s(\wh Q) + t(\wh Q)\big)\, 
\frac{\mu(Q)}{\mu(\wh Q)-t(\wh Q)} \leq 2 \big(\mu(\wh Q) + \frac12\,\mu(\wh Q)\big)\, 
\frac{\mu(Q)}{\mu(\wh Q)} = 3\,\mu(Q),
\end{equation}
and so
$$\eta(Q^{(\eta)}) = s(Q) + \mu(Q)\leq 4\,\mu(Q)\leq 4\,\mu(\wh Q)\lesssim \Lambda_*\, \Theta(R_0)\,\ell(\wh Q)^n
\approx_{\Lambda_*,\delta_0} \Theta(R_0)\,\ell(Q)^n,
$$
taking into account that $\wh{Q}\not \in\HD_*(R_0)$, by (i).
\end{proof}
\vv

\begin{rem}
For the record, notice that from \rf{eqaq745} it follows that, for all $Q\in\wh\tree_0(R)$, either
\begin{equation}\label{eq:aQmuQ}
0\leq s(Q)\leq 3\,\mu(Q),
\end{equation}
or 
$$s(Q)=-\mu(Q).$$
The latter case happens if and only if $Q\in\LD(R_0)$.
\end{rem}
\vv

\begin{rem}
Consider the measure defined by
$$\eta' = \sum_{Q\in\wh \sss(R)\setminus\LD(R_0)} \mu(Q)\,\frac{\HH^n\rest_{D_Q}}{\HH^n(D_Q)} + \mu\rest_{\wh\sG(R)}.$$
This measure is mutually absolutely continuous with $\eta$. Further, since the
coefficients $s(Q)$, with $Q\in\wh\tree_0(R)$, are uniformly bounded (by the previous remark), it turns out that
$$\eta'= \rho\,\eta,$$
for some function $\rho\in L^\infty(\eta)$ satisfying $\rho\approx1$.
Consequently, by Lemma \ref{lem9.2}, $\eta'$ is also AD-regular.
\end{rem}

\vv

For a family of cubes $\cI\subset \tree(R_0)$, we denote 
$$\wh\Ch(\cI)=\bigcup_{Q\in \cI} \wh\Ch(Q).$$
For $Q\in\wh\tree_0(R)$, we write $Q\in (i)_R$, if $Q\in\wh\sss(R)$ and the condition (i) in the definition of 
$\wh\sss(R)$ holds for $Q$, and analogously regarding the notations $Q\in (ii)_R$ and $Q\in (iii)_R$.

\begin{lemma}\label{lem9.5*}
The following holds:
$$\sum_{Q\in\wh\sss(R)\cap(\LD(R_0)\cup\HD_*(R_0)\cup\BR(R))}\mu(Q) 
+\sum_{Q\in\wh\Ch((iii)_R)\cap \LD(R_0)}\mu(Q) + \mu(\wh\sG(R))
\approx \mu(R).$$
\end{lemma}

\begin{proof}
It is clear that the left hand side above is bounded by $\mu(R)$.
For the converse estimate,  we write
\begin{equation}\label{eqsplit63}
\mu(R) = \sum_{Q\in\wh\sss(R)}\mu(Q) + \mu(\wh\sG(R)) = 
\sum_{Q\in (i)_R} \mu(Q) + \sum_{Q\in (ii)_R} \mu(Q) +\sum_{Q\in (iii)_R} \mu(Q)  
 + \mu(\wh\sG(R)).
 \end{equation}
 By construction,
\begin{equation}\label{eqsplit64}
\sum_{Q\in (i)_R} \mu(Q) =\sum_{Q\in \wh\sss(R)\cap(
\HD_*(R_0)\cup\LD(R_0)\cup\BR(R))}\mu(Q).
\end{equation}
Also, if $Q\in (iii)_R$, then
$$\mu(Q)\leq 2\sum_{P\in\wh\Ch(Q)\cap \LD(R_0)}\mu(P),$$
and thus
\begin{equation}\label{eqsplit65}
\sum_{Q\in (iii)_R} \mu(Q)\leq 2 \sum_{P\in\wh\Ch(\wh\sss(R))\cap \LD(R_0)}\mu(P).
\end{equation}

On the other hand, if $Q\in (ii)_R$, then $0\leq \mu(Q)\leq s(Q)$, and so
$$\sum_{Q\in (ii)_R} \mu(Q)\leq \sum_{Q\in \wh\sss(Q):s(Q)\geq 0} s(Q) = \sum_{Q\in \wh\sss(R)\setminus \LD(R_0)} s(Q).$$
For a given $N\geq 1$, consider the families $\wh\sss_N(R)$ and $\cI_N$ defined just above \rf{eqdefetaN}.
Notice that $\cJ_N:= \wh\sss_N(R)\cup\cI_N$ is a finite family of cubes which cover $R$, and thus, 
from the property \rf{eqaq89},
it follows that
$$0 = s(R) = \sum_{Q\in \cJ_N}s(Q) = \sum_{Q\in \cJ_N\cap \LD(R_0)} s(Q) + 
\sum_{Q\in \cJ_N\setminus \LD(R_0)} s(Q). $$
Since $s(Q) = -\mu(Q)$ for all $Q\in \cJ_N\cap \LD(R_0)$, we deduce
$$\sum_{Q\in \cJ_N\setminus \LD(R_0)} s(Q) = \sum_{Q\in \cJ_N\cap \LD(R_0)} \mu(Q).$$ 
Letting $N\to\infty$ and taking into account that $s(Q)\geq 0$ for all
$Q\in \cJ_N\setminus \LD(R_0)$, we get
$$\sum_{Q\in \wh\sss(R)\setminus \LD(R_0)} s(Q)\leq \sum_{Q\in \wh\sss(R)\cap \LD(R_0)} \mu(Q),$$
and thus
\begin{equation}\label{eqsplit66}
\sum_{Q\in (ii)_R} \mu(Q)\leq \sum_{Q\in \wh\sss(R)\cap \LD(R_0)} \mu(Q).
\end{equation}

The lemma follows from the splitting \rf{eqsplit63} and the inequalities \rf{eqsplit64}, \rf{eqsplit65}, \rf{eqsplit66}.
\end{proof}

\vv

\begin{lemma}\label{lemriesz*eta}
The operator $\RR_\eta$ is bounded in $L^2(\eta)$, with
$$\|\RR_\eta\|_{L^2(\eta)\to L^2(\eta)}\lesssim_{\Lambda_*,\delta_0,K}\Theta(R).$$
\end{lemma}

\begin{proof}
To prove this lemma we will use the suppressed kernel $K_\Phi$ introduced in Section \ref{sec6.1},
with the following $1$-Lipchitz function:
$$\Phi(x) =\inf_{Q\in\wh\tree_0(R)} \big(\ell(Q) + \dist(x,Q)\big).$$
We will prove first that $\RR_{\Phi,\mu\rest_{2R}}$ is bounded in $L^2(\mu\rest_{2R})$ by applying Theorem
\ref{teontv}, and later on we will 
show that $\RR_\eta$ is bounded in $L^2(\eta)$ by approximation. 

In order to apply Theorem \ref{teontv}, we will show that
\begin{itemize}
\item[(a)] $\mu(B(x,r)\cap 2R)\leq C\,\Theta(R)\,r^n$ for all $r\geq \Phi(x)$, and
\item[(b)] $\sup_{r>\Phi(x)}\big|\RR_r(\chi_{2R}\mu)(x)\big|\leq C\Theta(R)$,
\end{itemize}
with $C$ possibly depending on $\Lambda_*$, $\delta_0$, and $K$. Once these conditions are proven, then Theorem \ref{teontv} applied to the measure $\Theta(R)^{-1}\,\mu\rest_{2R}$ ensures that 
\begin{equation}\label{eqphi894}
\|\RR_{\Phi,\mu\rest_{2R}}\|_{L^2(\mu\rest_{2R})\to L^2(\mu\rest_{2R})}\lesssim_{\Lambda_*,\delta_0,K}\Theta(R).
\end{equation}

The proof of (a) is quite similar to the proof of Lemma \ref{lem6.77}. However, we repeat
here the arguments for the reader's convenience.
In the case $r>\ell(R)/10$ we just use that
$$\mu(B(x,r)\cap 2R)\leq \mu(2R)\lesssim \Theta(R)\,\ell(R)^n\lesssim \Theta(R)\,r^n.$$
So we may assume that $\Phi(x)<r\leq \ell(R)/10$.
By the definition of $\Phi(x)$, there exists $Q\in\wh\tree_0(R)$ such that
$$\ell(Q) + \dist(x,Q)\leq r.$$
Therefore, $B_Q\subset B(x,4r)$ and so there exists an ancestor $Q'\supset Q$ which belongs to $\wh\tree_0(R)$ such that $B(x,r)\subset  2B_{Q'}$, with $\ell(Q')\approx r$.  
Then,
$$\mu(B(x,r)\cap 2R)\leq \mu(2B_{Q'}) \lesssim \Lambda_*\,\Theta(R_0)\,\ell(Q')^n\approx_{\Lambda_*,\delta_0} \Theta(R)\,r^n,$$
as wished.

Let us turn our attention to the property (b).
In the case $r>\ell(R)/10$ we have
$$\big|\RR_r(\chi_{2R}\mu)(x)\big|\leq \frac{\mu(2R)}{r^n}\lesssim \Theta(R).$$
In the case $\Phi(x)<r\leq \ell(R)/10$ we consider the same cube $Q'\in\wh\tree_0(R)$ as above, 
which satisfies $B(x,r)\subset  2B_{Q'}$ and $\ell(Q')\approx r$. Further, by replacing
$Q'$ by the first ancestor in $\wh\tree_0(R)$ if necessary, we may assume that
$$|\RR_\mu\chi_{2R\setminus 2Q'}(x_{Q'})|\lesssim K\,\Theta(R).$$
Since $|x-x_{Q'}|\lesssim\ell(Q')$, by standard arguments which use the fact that $K_\Phi$ is a Calder\'on-Zygmund kernel (see \rf{eqkafi1}
and \rf{eqkafi2}),
it follows that
$$\big|\RR_\mu\chi_{2R\setminus 2Q'}(x_{Q'}) - \RR_r(\chi_{2R}\mu)(x)\big|\lesssim \PP(Q')\lesssim
\Lambda_*\,\Theta(R_0)\approx_{\Lambda_*,\delta_0}\Theta(R).$$
Thus,
$$\big|\RR_r(\chi_{2R}\mu)(x)\big|\leq 
\big|\RR_\mu\chi_{2R\setminus 2Q'}(x_{Q'})\big| + 
\big|\RR_\mu\chi_{2R\setminus 2Q'}(x_{Q'}) - \RR_r(\chi_{2R}\mu)(x)\big|\lesssim_{\Lambda_*,\delta_0,K}\Theta(R).$$
So both (a) and (b) hold, and then \rf{eqphi894} follows.
\vv

Next we deal with the $L^2(\eta)$ boundedness of $\RR_\eta$. 
First notice that  $\RR_{\mu\rest_{\wh\sG(R)}}$ is bounded in $L^2(\mu\rest_{\wh\sG(R)})$ with norm
at most $C\Theta(R)$ because $\Phi(x)=0$ on $\wh\sG(R)$. Since $\eta\rest_{\wh\sG(R)} = \rho\,\mu\rest_{\wh\sG(R)}$ for some function $\rho\approx1$, $\RR_{\eta\rest_{\wh\sG(R)}}$ is also bounded in $L^2(\eta\rest_{\wh\sG(R)})$ with norm bounded by $C\Theta(R)$.
So it suffices to show that $\RR_{\eta\rest_{\wh\sG(R)^c}}$ is bounded in $L^2(\eta\rest_{\wh\sG(R)^c})$.
This follows from the fact that if $\alpha$ and $\beta$ are Radon measures with polynomial growth of degree $n$ such that $\RR_\alpha$ is bounded in $L^2(\alpha)$ and 
$\RR_\beta$ is bounded in $L^2(\beta)$, then $\RR_{\alpha+\beta}$ is bounded in $L^2(\alpha+\beta)$,
and then choosing $\alpha= \Theta(R)^{-1}\mu\rest_{\wh\sG(R)}$ and $\beta= \Theta(R)^{-1}\mu\rest_{\wh\sG(R)^c}$.
See for example Proposition 3.1 from \cite{NToV2}.
 
It remains to show that $\RR_{\eta\rest_{\wh\sG(R)^c}}$ is bounded in $L^2(\eta\rest_{\wh\sG(R)^c})$
with norm bounded above by $C\,\Theta(R)$.
Notice first that, by \rf{eqfk490}, there exists some constant $b>0$ depending at most on 
$C_0,A_0,n$ such that
$$\mu\bigl(\{x\in Q:\dist(x,\supp\mu\setminus Q)\leq b\,\ell(Q)\}\bigr) \leq\frac12\,\mu(Q).$$
We denote
\begin{equation}\label{eqsepbb*}
Q^{(0)}= \{x\in Q:\dist(x,\supp\mu\setminus Q)> b\,\ell(Q)\},
\end{equation}
so that $\mu(Q^{(0)})\geq \frac12\,\mu(Q)$.

We have to show that
\begin{equation}\label{eqbeta8430}
\|\RR (g\,\eta\rest_{\wh\sG(R)^c})\|_{L^2(\eta\rest_{\wh\sG(R)^c})} \lesssim_{\Lambda_*,\delta_0,K}\Theta(R)\, \|g\|_{L^2(\eta)}
\end{equation}
for any given $g\in L^2(\eta\rest_{\wh\sG(R)^c})$, with $\RR (g\,\eta\rest_{\wh\sG(R)^c})$ 
understood in the principal value sense. To this end, we take
the function $f\in L^2(\mu\rest_{2R})$ defined as follows:
\begin{equation}\label{deff99}
f = \sum_{Q\in\wh\sss(R)} \int_{D_Q}\! g\,d\eta \,\,\frac{\chi_{Q^{(0)}}}{\mu(Q^{(0)})}.
\end{equation}
We also consider the signed measures
$$\alpha = f\,\mu,\qquad \beta = g\,\eta,$$
so that $\alpha(Q) = \alpha(Q^{(0)}) = \beta(D_Q)$ for all $Q\in\wh\sss(R)$.

As a preliminary step to obtain \rf{eqbeta8430}, we will show first 
\begin{equation}\label{eqbeta843}
\|\RR_\Phi \beta\|_{L^2(\eta\rest_{\wh\sG(R)^c})} \lesssim_{\Lambda_*,\delta_0,K}\Theta(R)\, \|g\|_{L^2(\eta)}.
\end{equation}
For that purpose, first we will estimate the term  $|\RR_\Phi \alpha(x) -\RR_\Phi\beta(y)|$,
with $x\in Q^{(0)}$, $y\in D_Q$, for $Q\in\wh\sss(R)$, in terms of the coefficients
$$\PP_\alpha(Q) := \sum_{P\in\DD_\mu:P\supset Q} \frac{\ell(Q)}{\ell(P)^{n+1}} \,|\alpha|(2B_P)
\quad \mbox{ and }\quad
\QQ_\alpha(Q) := \sum_{P\in\wh\sss(R)} \frac{\ell(P)}{D(P,Q)^{n+1}}\,|\alpha|(P),
$$
and
$$\PP_\beta(Q) := \sum_{P\in\DD_\mu:P\supset Q} \frac{\ell(Q)}{\ell(P)^{n+1}} \,|\beta|(2B_P)
\quad \mbox{ and }\quad
\QQ_\beta(Q) := \sum_{P\in\wh\sss(R)} \frac{\ell(P)}{D(P,Q)^{n+1}}\,|\beta|(D_P).
$$
We claim that
\begin{equation}\label{eqclaim7284}
|\RR_\Phi \alpha(x) -\RR_\Phi\beta(y)| \lesssim \PP_\alpha(Q) + \QQ_\alpha(Q)+\PP_\beta(Q) + \QQ_\beta(Q)
\end{equation}
for all $x\in Q^{(0)}$, $y\in D_Q$, with $Q\in\wh\sss(R)$.
The arguments to prove this are quite similar to the ones in Lemma \ref{lemaprox1}, but we 
will show the details for completeness.
By the triangle inequality, for $x$, $y$ and $Q$ as above, we have
\begin{align*}
\big|\RR_\Phi \alpha(x) - \RR_\Phi \beta(y)\big| & \leq 
\big|\RR_\Phi \alpha(x) - \RR_\Phi \alpha(x_Q)\big|\\
&\quad + 
\big|\RR_\Phi \alpha(x_Q) - \RR_\Phi \beta(x_Q)\big| 
+ \big|\RR_\Phi \beta(x_Q) - \RR_\Phi \beta(y)\big| \\
& =: I_1 + I_2 + I_3.
\end{align*}

First we estimate $I_1$ using the properties of the kernel $K_\Phi$ in \rf{eqkafi1}
and \rf{eqkafi2}, and taking into account that for $x\in Q^{(0)}$ (and thus for $x_Q$)
$\Phi(x)\approx_b \ell(Q)$, because of the separation condition in the definition of
$Q^{(0)}$ in \rf{eqsepbb*}. Then we get
\begin{align*}
|I_1| &\leq \int |K_\Phi(x,z) -  K_\Phi(x_Q,z)|\,d|\alpha|(z)\\
& = \bigg(\int_{2B_Q} +
\sum_{P\in\DD_\mu:P\supset Q} \int_{2B_{\wh P}\setminus 2B_P}\bigg) |K_\Phi(x,z) -  K_\Phi(x_Q,z)|\,d|\alpha|(z)\\
&\lesssim \sum_{P\in\DD_\mu:P\supset Q} \frac{\ell(Q)}{\ell(P)^{n+1}} \,|\alpha|(2B_{\wh P}) \lesssim \PP_\alpha(Q),
\end{align*}
where above we denoted by $\wh P$ the parent of $P$.
The same estimate holds for the term $I_3$ (with $\alpha$ replaced by $\beta$), using that $\Phi(x)\approx \ell(Q)$ for all $x\in D_Q$,
since $D_Q\subset\frac12B(Q)$ and $B(Q)\cap\supp\mu\subset Q$. So
$$|I_3| \lesssim \PP_\beta(Q).$$
Finally we deal with the term $I_2$.
Since $\alpha(P^{(0)})=\beta(D_P)$ for all $P\in\wh\sss(R)$, we have
\begin{align*}
I_2 & \leq \sum_{P\in\wh\sss(R)} \left| \int K_\Phi(x_Q-z)\,d\big(\alpha\rest_{P^{(0)}} - 
\beta\rest_{D_P} \big)(z)\right|\\
& \leq \sum_{P\in\wh\sss(R)}  \int |K_\Phi(x_Q-z)-K_\Phi(x_Q-x_P)|\,d\big(|\alpha|\rest_{P^{(0)}} + 
|\beta|\rest_{D_P} \big)(z)
\end{align*}
From the separation condition in \rf{eqsepbb*} and the fact that $D_P\subset\frac12B(P)$, we infer that, for $P,Q\in\wh\sss(R)$ with $P\neq Q$ and $z\in P^{(0)}\cup D_P$,
$$|x_Q-z|\approx |x_Q - x_P|\gtrsim \ell(Q) + \ell(P).$$
Hence, in the case $P\neq Q$,
$$\int |K_\Phi(x_Q-z)-K_\Phi(x_Q-x_P)|\,d\big(|\alpha|\rest_{P^{(0)}} + 
|\beta|\rest_{D_P} \big)(z)\lesssim \frac{\ell(P)}{D(P,Q)^{n+1}}\,\big(|\alpha|(P^{(0)}) + |\beta|(D_P)\big).$$
The same inequality holds in the case $P=Q$ using \rf{eqkafi1} and the fact that $\Phi(x_Q) 
\approx\ell(Q)$.
So we deduce that
$$I_2\lesssim \QQ_\alpha(Q) + \QQ_\beta(Q).$$
Gathering the estimates obtained for $I_1$, $I_2$, $I_3$, the claim \rf{eqclaim7284} follows.

Now we are ready to show \rf{eqbeta843}. By the claim just proven and using that
$\eta(B_Q)\lesssim\mu(Q)$, we obtain
\begin{align}\label{eqdhvn43}
\|\RR_\Phi \beta\|_{L^2(\eta\rest_{\wh\sG(R)^c})}^2 & =
\sum_{Q\in\wh\sss(R)} \int_{B_Q} |\RR_\Phi \beta|^2\,d\eta\\
& \lesssim \sum_{Q\in\wh\sss(R)} \inf_{x\in Q^{(0)}}|\RR_\Phi \alpha(x)|^2\,\mu(Q) + 
\sum_{Q\in\wh\sss(R)}\big(\PP_\alpha(Q)^2 + \QQ_\alpha(Q)\big)^2\,\mu(Q) \nonumber\\
 &\quad+ 
\sum_{Q\in\wh\sss(R)}\big(\PP_\beta(Q)^2 + \QQ_\beta(Q)\big)^2\,\eta(D_Q).\nonumber
\end{align}
Since $\RR_{\Phi,\mu\rest_{2R}}$ is bounded in $L^2(\mu\rest_{2R})$ with norm bounded by 
$C(\Lambda_*,\delta_0,K)\,\Theta(R)$, we infer that
\begin{align*}
\sum_{Q\in\wh\sss(R)} \inf_{x\in Q^{(0)}}|\RR_\Phi \alpha(x)|^2\,\mu(Q)
\leq \int_R |\RR_\Phi \alpha|^2d\mu \leq \|\RR_\Phi (f\mu)\|^2_{L^2(\mu\rest_{2R})}
\lesssim_{\Lambda_*,\delta_0,K} \Theta(R)^2\,\|f\|_{L^2(\mu)}^2.
\end{align*}

To estimate the second sum on the right hand side of \rf{eqdhvn43} we use the fact
that
$$\sum_{Q\in\wh\sss(R)}\QQ_\alpha(Q)^2\,\mu(Q) \lesssim 
\sum_{Q\in\wh\sss(R)}\PP_\alpha(Q)^2\,\mu(Q).$$
This follows by the same argument as in Lemma \ref{lemregpq}, with $p=2$. Indeed, in that lemma
one does not use any specific property of the measure $\mu$ or the family $\Reg$, apart from the
fact the cubes from $\Reg$ are pairwise disjoint. So the lemma applies to $\wh\sss(R)$ too.
Observe also that, for every $x\in Q$, $2B_P\subset B(x,2\ell(P))\subset CB_P$, for some $C>1$, and thus
\begin{align*}
\PP_\alpha(Q) &\leq \sum_{P\in\DD_\mu:P\supset Q} \frac{\ell(Q)}{\ell(P)^{n+1}} \,
\frac{|\alpha|(B(x,2\ell(P)))}{\mu(B(x,2\ell(P)))}\,\mu(CB_P) \\
&\lesssim \PP(Q)\,\cM_\mu f(x)\lesssim_{\Lambda_*,\delta_0}\Theta(R)\,\cM_\eta f(x).
\end{align*}
Consequently,
\begin{align*}
\sum_{Q\in\wh\sss(R)}\big(\PP_\alpha(Q)^2 + \QQ_\alpha(Q)\big)^2\,\mu(Q) &
\lesssim \sum_{Q\in\wh\sss(R)}\PP_\alpha(Q)^2\,\mu(Q) \\
&\lesssim_{\Lambda_*,\delta_0}\Theta(R)^2
\int |\cM_\mu f|^2\,d\mu\lesssim_{\Lambda_*,\delta_0}\Theta(R)^2\,\|f\|_{L^2(\mu)}^2.
\end{align*}

The last sum on the right hand side of \rf{eqdhvn43} is estimated similarly. Indeed, 
by Lemma \ref{lemregpq} we also have
$$\sum_{Q\in\wh\sss(R)}\QQ_\beta(Q)^2\,\eta(D_Q) \lesssim 
\sum_{Q\in\wh\sss(R)}\PP_\beta(Q)^2\,\eta(D_Q),$$
and as above,
$$\PP_\beta(Q) \lesssim_{\Lambda_*,\delta_0}\Theta(R)\,\cM_\eta g(x).
$$
Then it follows that
$$\sum_{Q\in\wh\sss(R)}\big(\PP_\beta(Q)^2 + \QQ_\beta(Q)\big)^2\,\mu(Q) \lesssim_{\Lambda_*,\delta_0}\Theta(R)^2\,\|g\|_{L^2(\eta)}^2.$$
By \rf{eqdhvn43} and the preceding esitmates, to complete the proof of \rf{eqbeta843} it just remains
to notice that, by the definition of $f$ in \rf{deff99} and Cauchy-Schwarz, $\|f\|_{L^2(\mu)}\lesssim
\|g\|_{L^2(\eta)}$.

In order to prove \rf{eqbeta8430}, we denote
$$\RR^{r(Q)/4}\beta(x) = \RR\beta(x) - \RR_{r(Q)/4}\beta(x),$$
and we split
\begin{align}\label{eqalgud77}
\|\RR (g\,\eta\rest_{\wh\sG(R)^c})\|_{L^2(\eta\rest_{\wh\sG(R)^c})}^2 
 & = \sum_{Q\in\wh\sss(R)} \int_{D_Q} |\RR \beta|^2\,d\eta\\
& \lesssim \!\sum_{Q\in\wh\sss(R)}\! \bigg(\int_{D_Q} |\RR^{r(Q)/4}\beta|^2\,d\eta +\int_{D_Q} 
|\RR_{r(Q)/4}\beta - \RR_\Phi \beta|^2\,d\eta\bigg)\nonumber\\
& \quad + \|\RR_\Phi \beta\|_{L^2(\eta\rest_{\wh\sG(R)^c})}^2.\nonumber
\end{align}
Using the fact that $\RR_{\eta\rest_{D_Q}}$ is bounded in $L^2(\eta\rest_{D_Q})$ with norm 
comparable to $\eta(B_Q)/r(Q^n)$ (because $D_Q$
is an $n$-dimensional disk), we deduce that
$$\int_{D_Q} |\RR^{r(Q)/4}\beta|^2\,d\eta = \int_{D_Q}
|\RR^{r(Q)/4}(\chi_{D_Q}g\,\eta)|^2\,d\eta \lesssim_{\Lambda_*,\delta_0} \Theta(R)^2\,\|\chi_{D_Q}g\|_{
L^2(\eta)}^2.$$
Regarding the second integral in \rf{eqalgud77}, 
observe that, by \rf{e.compsup''}, for all $x\in D_Q$ with
$Q\in \wh\sss(R)$,
$$\bigl|\RR_{r(Q)/4}\beta(x) - \RR_{\Phi}\beta(x)\bigr|\lesssim  \sup_{r> \Phi(x)}\frac{|\beta|(B(x,r))}{r^n} \lesssim_{\Lambda_*,\delta_0} \Theta(R)\,\cM_\mu g(x).$$
Then, by the last estimates and \rf{eqbeta843}, we get
\begin{align*}%\label{eqalgud77}
\|\RR_\Phi (g\,\eta\rest_{\wh\sG(R)^c})\|_{L^2(\eta\rest_{\wh\sG(R)^c})}^2 
& \lesssim_{\Lambda_*,\delta_0}\Theta(R)^2 \sum_{Q\in\wh\sss(R)} \int_{D_Q} \big(|g|^2 + |\cM_\eta g|^2\big)\,d\eta + \|\RR_\Phi \beta\|_{L^2(\eta\rest_{\wh\sG(R)^c})}^2\\
&\lesssim_{\Lambda_*,\delta_0,K}\Theta(R)^2\,
\|g\|_{L^2(\eta)}^2,
\end{align*}
which concludes the proof of the lemma.
\end{proof}

\vv

Next, for each $R\in\DD_\mu^\PP\setminus\End_*(R_0)$, we define the family $\wh\End(R)$ and $\wh\tree(R)$, which can be considered as an enlarged version of $\wh\tree_0(R)$.
First we define 
\begin{equation*}
\wh{\Stop_*}(R) = (i)_R\cup (ii)_R\cup \wh\Ch((iii)_R).
\end{equation*}
Let $\wh\End(R)$ be the family of maximal $\PP$-doubling cubes which are contained in the cubes from $\wh\Stop_*(R)$.
%\begin{itemize}
%\item the maximal $\PP$-doubling cubes which are contained in the cubes $Q\in (i)_R\cup (ii)_R$, and
%\item the maximal $\PP$-doubling cubes which are contained in any cube $P\in\wh\Ch(Q)$ for some $Q\in (iii)_R$.
%\end{itemize}
Notice that $R\not\in\wh\End(R)$. Further, the cubes from $\wh\sss(R)\cap(\HD_*(R_0)\cup\BR(R)\cup(ii)_R)$
belong to $\wh\End(R)$ because they are $\PP$-doubling.
Then we let $\wh\tree(R)$ be the family of cubes from $\DD_\mu$ which are contained in $R$ and are not strictly contained in any cube from $\wh\End(R)$. Observe that we do not ask the cubes from 
$\wh\tree(R)$ to be $\PP$-doubling. Similarly, we define $\wh\tree_*(R)$ as the family of cubes from $\DD_\mu$ which are contained in $R$ and are not strictly contained in any cube from $\wh\sss_*(R)$.

\vv
\subsection{Estimating the \texorpdfstring{$\beta$}{beta} numbers on \texorpdfstring{$\wh\tree(R)$}{Tree(R)}}

Our goal is now to prove the following estimate.
\begin{lemma}\label{lembetas99}
	For each $R\in\DD_\mu^\PP\setminus\End_*(R_0)$, we have
	$$\sum_{Q\in\wh \tree(R)}\beta_{\mu,2}(2B_Q)^2\,\mu(Q)\lesssim_{\Lambda_*,\delta_0,K}\Theta(R_0)\,\mu(R).$$
\end{lemma}
We split the proof into several steps. Fix $R\in\DD_\mu^\PP\setminus\End_*(R_0)$. First we deal with cubes in $\wh\tree(R)\setminus\wh\tree_*(R)$.
\begin{lemma}
	We have
	\begin{equation*}
	\sum_{Q\in\wh\tree(R)\setminus\wh\tree_*(R)}\beta_{\mu,2}(2B_Q)^2\mu(Q)\lesssim_{\Lambda_*}\Theta(R_0)\mu(R).
	\end{equation*}
\end{lemma}
\begin{proof}
	We use the trivial bound $\beta_{\mu,2}(2B_Q)^2\lesssim\theta_\mu(2B_Q)$ and Lemma \ref{lemdobpp} to get
	\begin{align*}
	\sum_{Q\in\wh\tree(R)\setminus\wh\tree_*(R)}\beta_{\mu,2}(2B_Q)^2&\mu(Q)
	\lesssim\sum_{Q\in\wh\tree(R)\setminus\wh\tree_*(R)}\theta_{\mu}(2B_Q)\mu(Q)\\
	&=\sum_{P\in\wh\Stop_*(R)}\sum_{Q\in\wh\tree(R):\, Q\subset P }\theta_{\mu}(2B_Q)\mu(Q)\\
	&= \sum_{m\ge 0} \sum_{P\in\wh\Stop_*(R)}\sum_{\substack{Q\in\wh\tree(R)\\ Q\subset P,\, \ell(Q)=A_0^{-m}\ell(P) }}\theta_{\mu}(2B_Q)\mu(Q)\\
	&\le \sum_{m\ge 0} \sum_{P\in\wh\Stop_*(R)}\sum_{\substack{Q\in\wh\tree(R)\\ Q\subset P,\, \ell(Q)=A_0^{-m}\ell(P) }}A_0^{-m/2}\PP(P)\mu(Q)\\
	&\le \sum_{m\ge 0} \sum_{P\in\wh\Stop_*(R)}A_0^{-m/2}\PP(P)\mu(P)\approx  \sum_{P\in\wh\Stop_*(R)}\PP(P)\mu(P).
	\end{align*}
	Recall that for $Q\in\tree(R_0)$ we have $\PP(Q)\lesssim_{\Lambda_*}\Theta(R_0)$, and so
	\begin{equation*}
	\sum_{P\in\wh\Stop_*(R)}\PP(P)\mu(P)\lesssim_{\Lambda_*}\Theta(R)\sum_{P\in\wh\Stop_*(R)}\mu(P) \le \Theta(R_0)\mu(R).
	\end{equation*}
\end{proof}
It remains to prove
\begin{equation*}
\sum_{Q\in\wh \tree_*(R)}\beta_{\mu,2}(2B_Q)^2\,\mu(Q)\lesssim_{\Lambda_*,\delta_0,K}\Theta(R_0)\,\mu(R).
\end{equation*}
Consider the set
\begin{equation*}
\Gamma = \supp\eta = \wh\GG(R) \cup \bigcup_{Q\in\wh\sss(R)\setminus\LD(R_0)} D_Q .
\end{equation*}
Denote $\nu = \HH^n|_{\Gamma}$. We showed in Lemma \ref{lem9.2} that $\Theta(R_0)^{-1}\eta$ is an AD-regular measure, and so it follows by standard arguments (using e.g. \cite[Theorem 6.9]{Mattila-llibre}) that $\Gamma$ is an AD-regular set, and that $\eta = \rho \nu$ for some density $\rho$ satisfying $\rho\approx_{\Lambda_*,\delta_0}\Theta(R_0)$. It is also immediate to check that Lemma \ref{lemriesz*eta} implies that $\RR_\nu$ is bounded in $L^2(\nu)$, with
\begin{equation*}
\|\RR_\nu\|_{L^2(\nu)\to L^2(\nu)}\lesssim_{\Lambda_*,\delta_0,K}1.
\end{equation*}
Hence, by the main result of \cite{NToV1} we know that $\Gamma$ is uniformly $n$-rectifiable. This allows us to use the $\beta$ numbers characterization of uniform rectifiability \cite{DS1} to get
\begin{equation}\label{eq:betasUR}
\int_{B(z,r_0)}\int_0^{r_0}\beta_{\nu,2}(x,r)^2\, \frac{dr}{r}d\nu(x)\lesssim_{\Lambda_*,\delta_0,K}r_0^n,\quad \text{for }z\in\supp\nu, r_0\in(0,\diam(\supp\nu)).
\end{equation}

To transfer these estimates back to the measure $\mu$ we will argue similarly as in Section 7 of \cite{Azzam-Tolsa}. It will be convenient to work with regularized cubes, as we did in Section \ref{sec6.2*}. Consider a function
\begin{equation*}
d_{R,*}(x) = \inf_{Q\in\wh\tree_*(R)} (\dist(x,Q)+\ell(Q)).
\end{equation*}
Note that $d_{R,*}(x)=0$ for $x$ in the closure of  $\wh\GG(R),$ and $d_{R,*}(x)>0$ everywhere else. Moreover, $d_{R,*}$ is 1-Lipschitz. For each $x\in  R\setminus\overline{\wh\GG(R)}$ we define $Q_x$ to be the maximal cube from $\DD_\mu$ that contains $x$ and satisfies
\begin{equation*}
\ell(Q_x)\le \frac{1}{60}\inf_{y\in Q_x} d_{R,*}(y).
\end{equation*}
The family of all the cubes $Q_x,\ x\in R\setminus\overline{\wh\GG(R)},$ will be denoted by $\Reg_*(R)$. We define also the regularized tree $\treg_*(R)$ consisting of the cubes from $\DD_\mu$ that are contained in $R$ and are not strictly contained in any of the $\Reg_*(R)$ cubes. 

It follows easily from the definition of $\Reg_*(R)$ that $\wh\tree_*(R)\subset\treg_*(R)$.
Observe also that $\Reg_*(R)$ consists of pairwise disjoint cubes and satisfies
\begin{equation*}
\mu\bigg(R\setminus \bigg(\bigcup_{P\in\Reg_*(R)}P\cup \wh\GG(R)\bigg)\bigg)=0.
\end{equation*}

The following is an analogue of Lemma \ref{lem74}.
\begin{lemma}\label{lem:reg prop}
	The cubes from $\Reg_*(R)$ satisfy the following properties:
	\begin{itemize}
		\item[(a)] If $P\in\Reg_*(R)$ and $x\in B(x_{P},50\ell(P))$, then $10\,\ell(P)\leq d_{R,*}(x) \leq c\,\ell(P)$,
		where $c$ is some constant depending only on $n$. 
		%In particular, $B(z_{P},50\ell(P))\cap W_0=\varnothing$.		
		\item[(b)] There exists some absolute constant $c>0$ such that if $P,\,P'\in\Reg_*(R)$ satisfy $B(x_{P},50\ell(P))\cap B(x_{P'},50\ell(P'))
		\neq\varnothing$, then
		$$c^{-1}\ell(P)\leq \ell(P')\leq c\,\ell(P).$$
		\item[(c)] For each $P\in \Reg_*(R)$, there are at most $N$ cubes $P'\in\Reg_*(R)$ such that
		$$B(x_{P},50\ell(P))\cap B(x_{P'},50\ell(P'))
		\neq\varnothing,$$
		where $N$ is some absolute constant.		
		%\item[(d)] If $x\not\in B(x_0,\frac1{8} K r_0)$, then $d(x)\approx |x-x_0|$. Thus, if $P\in\reg$ and $B(z_{P},50\ell(P))\not\subset  B(x_0,\frac1{8} K r_0)$, then $\ell(P)\gtrsim  K r_0$.
	\end{itemize}
\end{lemma}
As before, we omit the proof.

\begin{lemma}\label{lem:QtregPtree}
	For all $Q\in\treg_*(R)$ there exists some $P\in\wh\tree_*(R)$ such that $\ell(Q)\approx\ell(P)$ and $2B_Q\subset C B_{P}\subset C' B_{Q}$, where $C$ and $C'$ are some absolute constants. In consequence,
	\begin{equation}\label{eq:treg dens bdd}
	\PP(Q)\lesssim_{\Lambda_*}\Theta(R_0).
	\end{equation}
\end{lemma}
\begin{proof}
	Let $Q\in\treg_*(R)$. If $Q\cap\wh\GG(R)\neq\varnothing,$ then $Q\in\wh\tree_*(R)$ and we can take $P=Q$. If $Q\cap\wh\GG(R)=\varnothing,$ then there exists some $Q_0\in\Reg_*(R)$ such that $Q_0\subset Q$. By the definition of $d_{R,*}$ and Lemma \ref{lem:reg prop} (a), there exists $P_0\in\wh\tree_*(R)$ such that
	\begin{equation*}
	\dist(x_{Q_0},P_0)+\ell(P_0)\le 2 d_{R,*}(x_{Q_0})\approx \ell(Q_0).
	\end{equation*}
	In particular, $\ell(P_0)\lesssim \ell(Q_0)\le\ell(Q)$. If $\ell(P_0)\ge\ell(Q),$ set $P=P_0$, otherwise let $P$ be the ancestor of $P_0$ with $\ell(P)= \ell(Q)$. Clearly, $\ell(P)\approx\ell(Q)$, and moreover
	\begin{equation*}
	\dist(x_Q,x_P)\le \dist(x_{Q_0},P_0)+\ell(Q)+\ell(P)\lesssim\ell(Q_0)+\ell(Q)+\ell(P)\approx \ell(Q)\approx\ell(P),
	\end{equation*}
	which implies $2B_Q\subset C B_{P}\subset C' B_{Q}$ for some absolute $C$ and $C'$.
	
	Finally, to see $\PP(Q)\lesssim_{\Lambda_*}\Theta(R_0)$ recall that $\PP(P)\lesssim \Lambda_*\Theta(R_0)$ for all $P\in\wh\tree_*(R)\subset\tree(R_0)$, and we have $\PP(Q)\lesssim\PP(P)$ because $B_Q\subset C B_P$ and $\ell(Q)\approx \ell(P)$.
\end{proof}

The following lemma states that the uniformly rectifiable set $\Gamma$ lies relatively close to all the cubes from $\Reg_*(R)$. This property will be crucial in our subsequent estimates.
\begin{lemma}\label{lem:RegGamma}
	There exists $C_*=C_*(\Lambda_*,\delta_0)$ such that for all $Q\in\Reg_*(R)$ we have
	\begin{equation}\label{eq:RegGamma}
	\frac{C_*}{2}B_Q\cap\Gamma\neq\varnothing.
	\end{equation}
\end{lemma}
\begin{proof}
	Let $Q\in\Reg_*(R)$ and let $P\in\wh\tree_*(R)$ be the cube from Lemma \ref{lem:QtregPtree}.
	In particular, we have
	\begin{equation}\label{eq:BPCBQ}
	2B_P\subset C B_Q
	\end{equation}
	for some absolute constant $C$.
	
	If $P$ contains some cube $P_1\in\wh\Stop(R)\setminus\LD(R_0)$, then we are done, because in that case $D_{P_1}\subset 2B_P\subset C B_Q$, and $D_{P_1}\subset\Gamma$. Similarly, if $\wh\GG(R)\cap P\neq\varnothing$, then there is nothing to prove.
	
	Now suppose that $P\cap\wh\GG(R)=\varnothing$ and $P$ does not contain any cube from $\wh\Stop(R)\setminus\LD(R_0)$. Since $P\in\wh\tree_*(R)$ and $P\cap\wh\GG(R)=\varnothing$, there exists some $P_1\in\wh\sss_*(R)$ such that $P_1\subset P$. By our assumption $P_1\notin\wh\sss(R)\setminus\LD(R_0)$, and so
	\begin{equation*}
	P_1\in \wh\sss_*(R)\setminus \left(\wh\sss(R)\setminus\LD(R_0)\right) = \wh\Ch((iii)_R) \cup \left(\wh\sss(R)\cap\LD(R_0)\right).
	\end{equation*}
	
	There are two cases to consider. Suppose that $P_1\in\wh\Ch((iii)_R)$. Let $S\in (iii)_R$ be such that $P_1\in\wh\Ch(S)$. Since $S\notin\LD(R_0)$ (otherwise we'd have $S\in (i)_R$), \eqref{eqcompa492} gives $\ell(P_1)\approx_{\Lambda_*,\delta_0}\ell(S)$. Thus, there exists some constant  $C(\Lambda_*,\delta_0)$ such that
	\begin{equation*}
	D_S\subset C(\Lambda_*,\delta_0)B_{P_1}\subset C(\Lambda_*,\delta_0)B_{P}\overset{\eqref{eq:BPCBQ}}{\subset}C\, C(\Lambda_*,\delta_0) B_Q.
	\end{equation*}
	Since $D_S\subset\Gamma$, we get \eqref{eq:RegGamma} as soon as $C_*\ge 2C\, C(\Lambda_*,\delta_0)$.
	
	Finally, suppose that $P_1\in\wh\sss(R)\cap\LD(R_0)$. Let $P_0\in\wh\tree_0(R)\setminus\wh\sss(R)$ be the unique cube such that $P_1\in\wh\Ch(P_0)$. By \eqref{eqcompa492} we have $\ell(P_0)\approx_{\Lambda_*,\delta_0}\ell(P_1)$, and so
	\begin{equation*}
	2B_{P_0}\subset C(\Lambda_*,\delta_0) B_{P_1}\subset C(\Lambda_*,\delta_0) B_{P}\subset C\, C(\Lambda_*,\delta_0) B_{Q}.
	\end{equation*}
	We claim that $2B_{P_0}\cap\Gamma\neq\varnothing$, and so \eqref{eq:RegGamma} is satisfied if we assume $C_*\ge 2C\, C(\Lambda_*,\delta_0)$. First, if $2B_{P_0}\cap\wh\GG(R)\neq\varnothing$, then there is nothing to prove. Assume the contrary. In that case $P_0$ is covered by cubes from $\wh\Stop(R)$. We claim that there exists some $S\in\wh\sss(R)\setminus\LD(R_0)$ such that $S\subset P_0$. Indeed, otherwise $P_0$ would be covered by cubes from $\wh\sss(R)\cap\LD(R_0)$, but then
	\begin{equation*}
	-\mu(P_0)=\sum_{P'\in\wh\sss(R)\cap\LD(R_0)}-\mu(P')=\sum_{P'\in\wh\sss(R)\cap\LD(R_0)}s(P') \overset{\eqref{eqaq89}}{=}s(P_0)\overset{\eqref{eq:aQmuQ}}{\ge}0,
	\end{equation*}
	which is a contradiction. Thus, there exists $S\in\wh\sss(R)\setminus\LD(R_0)$ such that $S\subset P_0$, which implies $D_{S}\subset 2B_{P_0}$. Since $D_{S}\subset\Gamma$, we are done.
\end{proof}

In the following lemma we define functions supported on $\Gamma$ that approximate $\mu$ at the level of $\Reg_*(R)$.

\begin{lemma}\label{lem:gP}
	There exist functions $g_P:\Gamma\to\R,\ P\in\Reg_*(R),$ such that each $g_P$ is supported in $\Gamma\cap \overline{C_* B_P}$,
	\begin{equation}\label{eq:intgP}
	\int_{\Gamma} g_P\ d\nu = \mu(P),
	\end{equation}
	and
	\begin{equation}\label{eq:sumgP}
	\sum_{P\in\Reg_*(R)} g_P \lesssim_{\Lambda_*,\delta_0} \Theta(R).
	\end{equation}
\end{lemma}
\begin{proof}
	Assume first that the family $\Reg_*(R)$ is finite.
	%(which is equivalent to $\wh\GG(R)=\varnothing$)
	 We label the cubes from $\Reg_*(R)$ in the order of increasing sidelength, that is we let $P_1$ be a cube with the minimal sidelength, and then we label all the remaining cubes so that $\ell(P_i)\le\ell(P_{i+1})$.
	
	The functions $g_{i}:=g_{P_i}$ will be of the form $g_i = \alpha_i\chi_{A_i}$ where $\alpha_i\ge 0 $ and $A_i\subset \Gamma\cap C_* B_{P_i}$. We begin by setting $\alpha_1 = \mu(P_1)/\nu(C_* B_{P_1})$ and $A_1 = C_* B_{P_1}\cap\Gamma$. Clearly, \eqref{eq:intgP} holds for $P_1$. Moreover, using the fact that $\nu$ is AD-regular and $\frac{C_*}{2} B_{P_1}\cap\Gamma\neq\varnothing$ we get
	\begin{equation*}
	\|g_1\|_\infty = \alpha_1 = \frac{\mu(P_1)}{\nu(C_* B_{P_1})}\approx_{\Lambda_*,\delta_0}\frac{\mu(P_1)}{\ell(P_1)^n}\overset{\eqref{eq:treg dens bdd}}{\lesssim_{\Lambda_*}}\Theta(R_0).
	\end{equation*}
	
	We define the remaining $g_k,\, k\ge 2$ inductively. Suppose that $g_1,\dots,g_{k-1}$ have already been constructed, and they satisfy
	\begin{equation}\label{eq:sumgi}
	\sum_{i=1}^{k-1} g_i \le C' \Theta(R)
	\end{equation}
	for some constant $C'=C'(\Lambda_*,\delta_0)$ to be fixed below. Let $P_{i_1},\dots,P_{i_m}$ be the subfamily of $P_1,\dots,P_{k-1}$ consisting of cubes such that $C_* B_{P_k}\cap C_* B_{P_{i_j}}\neq\varnothing$. Due to the non-decreasing sizes of $P_i$'s we have $P_{i_j}\subset C_* B_{P_{i_j}}\subset 3C_* B_{P_k}$. Hence, applying \eqref{eq:intgP} to $g_{i_j}$ we get
	\begin{equation*}
	\sum_j \int_{\Gamma} g_{i_j}\ d\nu = \sum_j\mu(P_{i_j})\le\mu(3C_* B_{P_k})\overset{\eqref{eq:treg dens bdd}}{\le}C({\Lambda_*,\delta_0})\Theta(R_0)\ell(P_k)^n\le C''\Theta(R_0)\nu(\Gamma\cap C_* B_{P_k}),
	\end{equation*}
	for some $C''$ depending on $\Lambda_*,\delta_0$. By the Chebyshev's inequality
	\begin{equation*}
	\nu\left(\Gamma\cap\big\{\textstyle{\sum_j}\, g_{i_j}\ge 2C''\Theta(R_0)\big\}\right)\le \frac{1}{2}\nu(\Gamma\cap C_* B_{P_k}).
	\end{equation*}
	Set
	\begin{equation*}
	A_k = \Gamma\cap C_* B_{P_k}\cap \big\{\textstyle{\sum_j}\, g_{i_j}\le 2C''\Theta(R_0)\big\},
	\end{equation*}
	and then by the preceding estimate $\nu(A_k)\ge \nu(\Gamma\cap C_* B_{P_k})/2.$ We define
	\begin{equation*}
	\alpha_k = \frac{\mu(P_k)}{\nu(A_k)},
	\end{equation*}
	so that for $g_k=\alpha_k\chi_{A_k}$ we have $\int g_k\, d\nu = \mu(P_k)$. Moreover, using AD-regularity of $\nu$ and the fact that $\frac{C_*}{2} B_{P_k}\cap\Gamma\neq\varnothing$
	\begin{equation*}
	\alpha_k \le 2 \frac{\mu(P_k)}{\nu(C_* B_{P_k})} \le C({\Lambda_*,\delta_0})\frac{\mu(P_k)}{\ell(P_k)^n}\overset{\eqref{eq:treg dens bdd}}{\le} C'''\Theta(R_0)
	\end{equation*}
	for some $C'''$ depending on $\Lambda_*,\delta_0$. Hence, by the definition of $A_k$
	\begin{equation*}
	g_k(x) + \sum_j g_{i_j}(x) \le C'''\Theta(R_0) + 2C''\Theta(R_0), \quad\text{for $x\in A_k$.}
	\end{equation*}
	For $x\not\in A_k$ we have $g_k=0$, and so it follows from the above and the inductive assumption \eqref{eq:sumgi} that for $C' = C''' + 2C''$ we have
	\begin{equation*}
	\sum_{i=1}^{k} g_i \le C' \Theta(R),
	\end{equation*}
	which closes the induction.
	
	Suppose now that the family $\Reg_*(R)$ is infinite. We can relabel it so that $\Reg_*(R)=\{P^i\}_{i\in\N}$. For each $N$ we consider the family $\{P^i\}_{1\le i\le N}$. We construct functions $g_{P^1}^N,\dots, g_{P^N}^N$ as above, so that they satisfy
	\begin{equation*}
	\int g_{P^1}^N\ d\nu = \mu(P^1),\qquad \sum_{i=1}^{N} g_{P^i}^N \le C' \Theta(R).
	\end{equation*}
	There exists a subsequence $I_1\subset\N$ such that $\{g_{P^1}^k\}_{k\in I_1}$ is convergent in the weak-$*$ topology of $L^\infty(\nu)$ to some function $g_{P^1}\in L^{\infty}(\nu)$. We take another subsequence $I_2\subset I_1$ such that $\{g_{P^2}^k\}_{k\in I_2}$ is convergent in the weak-$*$ topology of $L^\infty(\nu)$ to some $g_{P^2}\in L^{\infty}(\nu)$. Proceeding in this fashion we obtain a family $\{g_{P^i}\}_{i\in\N}$ such that $\supp g_{P^i}\subset \overline{C_* B_{P^i}}$, and the properties \eqref{eq:intgP}, \eqref{eq:sumgP} are preserved (because of the weak-$*$ convergence).
\end{proof}

Recall that by the uniform rectifiability of $\nu$ we have a good estimate on the $\beta_{\nu,2}$ numbers \eqref{eq:betasUR}. We will now use Lemmas \ref{lem:RegGamma} and \ref{lem:gP} to transfer these estimates to the measure $\mu$ and obtain
\begin{equation*}
	\sum_{Q\in\wh \tree_*(R)}\beta_{\mu,2}(2B_Q)^2\,\mu(Q)\lesssim_{\Lambda_*,\delta_0,K}\Theta(R_0)\,\mu(R).
\end{equation*}
In fact, we will show that
\begin{equation*}
	\sum_{Q\in\treg_*(R)}\beta_{\mu,2}(2B_Q)^2\,\mu(Q)\lesssim_{\Lambda_*,\delta_0,K}\Theta(R_0)\,\mu(R),
\end{equation*}
and the former estimate will follow, since $\wh \tree_*(R)\subset\treg_*(R)$.

Let $Q\in\treg_*(R)$, and let $L_Q$ be an $n$-plane minimizing $\beta_{\nu,2}(C_*'B_Q)$, where $C_*'>2$ is some constant depending on $C_*$, to be chosen in Lemma \ref{lem:PsizeQ}. We estimate
\begin{multline}\label{eq:bet1}
\beta_{\mu,2}(2B_Q)^2\mu(Q)\lesssim \frac{\mu(Q)}{\ell(Q)^n}\int_{2B_Q} \left(\frac{\dist(x,L_Q)}{\ell(Q)}\right)^2\,d\mu(x)\\
\overset{\eqref{eq:treg dens bdd}}{\lesssim_{\Lambda_*}}\Theta(R_0)\int_{2B_Q} \left(\frac{\dist(x,L_Q)}{\ell(Q)}\right)^2\,d\mu(x)\\
= \Theta(R_0)\left(\int_{2B_Q\cap R\setminus \wh\GG(R)}\dots\,d\mu(x) + \int_{2B_Q\cap\wh\GG(R)} \dots\,d\mu(x) + \int_{2B_Q\setminus R} \dots\,d\mu(x)\right)\\
=:\Theta(R_0)(I_1+I_2+I_3).
\end{multline}
Concerning $I_3$, we use the trivial estimate
\begin{equation}\label{eq:bet2}
I_3\lesssim\mu(2B_Q\setminus R).
\end{equation}
Estimating $I_2$ is simple because on $\wh\GG(R)$ we have $\mu=\rho'\nu$ for some $\rho'\lesssim_{\Lambda_*} \Theta(R_0)$, and so
\begin{multline}\label{eq:bet3}
I_2 \lesssim_{\Lambda_*} \Theta(R_0)\int_{2B_Q\cap\wh\GG(R)} \left(\frac{\dist(x,L_Q)}{\ell(Q)}\right)^2\,d\nu(x)\le \Theta(R_0)\int_{C_*'B_Q} \left(\frac{\dist(x,L_Q)}{\ell(Q)}\right)^2\,d\nu(x)\\
\approx_{C_*'} \Theta(R_0)\beta_{\nu,2}(C_*'B_Q)^2\ell(Q)^n.
\end{multline}

Bounding $I_1$ requires more work. First, we use the fact that $R\setminus\wh\GG(R)$ is covered $\mu$-a.e. by $\Reg_*(R)$:
\begin{align*}
I_1 &\le \sum_{P\in\Reg_*(R) :\,  P\cap 2B_Q\neq\varnothing} \int_{P} \left(\frac{\dist(x,L_Q)}{\ell(Q)}\right)^2\,d\mu(x) \\
&= \sum_{P\in\Reg_*(R) :\,  P\cap 2B_Q\neq\varnothing} 
\bigg(\int_{\Gamma} \left(\frac{\dist(x,L_Q)}{\ell(Q)}\right)^2 g_P(x)\,d\nu(x)\\
&\quad+ \int \left(\frac{\dist(x,L_Q)}{\ell(Q)}\right)^2 (\chi_{P}(x)d\mu(x) - g_P(x)d\nu(x))\bigg)
\\
&=: \sum_{P\in\Reg_*(R) :\,  P\cap 2B_Q\neq\varnothing} \left(I_{11}(P) + I_{12}(P)\right).
\end{align*}
We need the following auxiliary result.

\begin{lemma}\label{lem:PsizeQ}
	If $P\in\Reg_*(R)$ is such that $P\cap 2B_Q\neq\varnothing$, then $\ell(P)\lesssim\ell(Q)$ and in consequence $\overline{C_*B_P}\subset C_*'B_Q$ for some $C_*'=C_*'(C_*)\ge C_*$.
\end{lemma}
\begin{proof}
	If $\ell(P)\le\ell(Q)$ then there is nothing to prove, so suppose $\ell(P)>\ell(Q)$ (in particular $\ell(P)\ge A_0\ell(Q))$. In that case we have $2B_Q\subset 2B_P$. 
	
	Note that if we had $Q\setminus\wh\GG(R)=\varnothing$, then $d_{R,*}(x_Q)=0$, but by Lemma \ref{lem:reg prop} (a) we know that $d_{R,*}(x_Q)\ge 10\ell(P)$.
	Hence, there exists some $S\in\Reg_*(R)$ such that $S\subset Q$. Together with the fact that $2B_Q\subset 2B_P$ this implies $B_S\cap 2B_P\neq\varnothing$. By Lemma \ref{lem:reg prop} (b) this gives
	\begin{equation*}
	\ell(P)\approx \ell(S)\le\ell(Q).
	\end{equation*}
\end{proof}
By the lemma above and the preceding estimate we get
\begin{equation}\label{eq:bet4}
I_1 \le \sum_{P\in\Reg_*(R) :\,  C_*B_P\subset C_*'B_Q} I_{11}(P) + \sum_{P\in\Reg_*(R) :\,  C_*B_P\subset C_*'B_Q} I_{12}(P).
\end{equation}
We estimate the first sum as follows:
\begin{multline}\label{eq:bet5}
\sum_{P\in\Reg_*(R) :\,  C_*B_P\subset C_*'B_Q} \int_{\Gamma} \left(\frac{\dist(x,L_Q)}{\ell(Q)}\right)^2 g_P(x)\,d\nu(x)\\
=  \int_{\Gamma} \left(\frac{\dist(x,L_Q)}{\ell(Q)}\right)^2 \sum_{P\in\Reg_*(R) :\,  C_*B_P\subset C_*'B_Q}  g_P(x)\,d\nu(x)\\
\overset{{\supp g_P\subset C_* B_P}}{\le} \int_{\Gamma\cap C_*'B_Q} \left(\frac{\dist(x,L_Q)}{\ell(Q)}\right)^2 \sum_{P\in\Reg_*(R)}  g_P(x)\,d\nu(x)\\
\overset{\eqref{eq:sumgP}}{\lesssim}_{\Lambda_*,\delta_0}\Theta(R_0)\int_{\Gamma\cap C_*'B_Q} \left(\frac{\dist(x,L_Q)}{\ell(Q)}\right)^2\,d\nu(x)\approx_{C_*'} \Theta(R_0)\beta_{\nu,2}(C_*'B_Q)^2\ell(Q)^n.
\end{multline}
Concerning $I_{12}(P)$, observe that since $\int g_P\, d\nu=\mu(P)$ by \eqref{eq:intgP}, we have
\begin{equation*}
I_{12}(P) = \int \left(\left(\frac{\dist(x,L_Q)}{\ell(Q)}\right)^2 - \left(\frac{\dist(x_P,L_Q)}{\ell(Q)}\right)^2\right) \left(\chi_{P}(x)d\mu(x) - g_P(x)d\nu(x)\right).
\end{equation*}
For $x\in\supp(\chi_{P}(x)d\mu(x) - g_P(x)d\nu(x))\subset C_*B_P\subset C_*'B_Q$ we have 
\begin{multline*}
\left|\left(\frac{\dist(x,L_Q)}{\ell(Q)}\right)^2 - \left(\frac{\dist(x_P,L_Q)}{\ell(Q)}\right)^2\right| \le \frac{|x-x_P|}{\ell(Q)}\cdot\frac{\dist(x,L_Q)+\dist(x_P,L_Q)}{\ell(Q)}\\
\lesssim\frac{C_*\ell(P)}{\ell(Q)}\cdot\frac{C_*'\ell(Q)}{\ell(Q)} \approx_{C_*,C_*'} \frac{\ell(P)}{\ell(Q)}.
\end{multline*}
Hence,
\begin{equation}\label{eq:bet6}
I_{12}(P) \lesssim_{C_*,C_*'}\frac{\ell(P)}{\ell(Q)}\mu(P).
\end{equation}
Recall that $C_*$ depends on $\Lambda_*,\delta_0$, and $C_*'$ depends on $C_*$. Thus, putting together the estimates \eqref{eq:bet4}, \eqref{eq:bet5}, and \eqref{eq:bet6} yields
\begin{equation*}
I_1\lesssim_{\Lambda_*,\delta_0} \Theta(R_0)\beta_{\nu,2}(C_*'B_Q)^2\ell(Q)^n + \sum_{\substack{P\in\Reg_*(R):\\  C_*B_P\subset C_*'B_Q}}\frac{\ell(P)}{\ell(Q)}\mu(P).
\end{equation*}
Together with \eqref{eq:bet1}, \eqref{eq:bet2}, and \eqref{eq:bet3} this gives
\begin{multline*}	
\beta_{\mu,2}(2B_Q)^2\mu(Q)\lesssim_{\Lambda_*,\delta_0} \Theta(R_0)^2\beta_{\nu,2}(C'_* B_Q)^2\ell(Q)^n \\ +
\Theta(R_0)\sum_{\substack{P\in\Reg_*(R):\\  C_*B_P\subset C_*'B_Q}}\frac{\ell(P)}{\ell(Q)}\mu(P) + \Theta(R_0)\mu(2B_Q\setminus R).
\end{multline*}
Summing over $Q\in\treg_*(R)$ we get
\begin{multline*}	
\sum_{Q\in\treg_*(R)}\beta_{\mu,2}(2B_Q)^2\mu(Q)\lesssim_{\Lambda_*,\delta_0} \Theta(R_0)^2\sum_{Q\in\treg_*(R)}\beta_{\nu,2}(C'_* B_Q)^2\ell(Q)^n \\+
\Theta(R_0)\sum_{Q\in\treg_*(R)}\sum_{\substack{P\in\Reg_*(R):\\  C_*B_P\subset C_*'B_Q}}\frac{\ell(P)}{\ell(Q)}\mu(P) + \Theta(R_0)\sum_{Q\in\treg_*(R)}\mu(2B_Q\setminus R)\\
=:\Theta(R_0)^2 S_1+\Theta(R_0)S_2+\Theta(R_0)S_3.
\end{multline*}

Concerning $S_1$, note that by \eqref{eq:RegGamma} we know that if $Q\in\treg_*(R)$, then $\nu(C_* B_Q\cap\Gamma)\approx_{\Lambda_*,\delta_0}\ell(Q)^n$ and for all $x\in C_* B_Q\cap\Gamma$ we have $C'_* B_Q\subset B(x,2C_*' \ell(Q))$. Thus, $\beta_{\nu,2}(C'_* B_Q)\lesssim \beta_{\nu,2}(x,r)$ for $2C_*' \ell(Q)<r<3C_*' \ell(Q)$. Observe also that the sets $C_* B_Q\cap\Gamma$ corresponding to cubes of the same generation have bounded intersection. It follows easily that
\begin{multline*}
S_1=\sum_{Q\in\treg_*(R)}\beta_{\nu,2}(C'_* B_Q)^2\ell(Q)^n \lesssim_{\Lambda_*,\delta_0} \int_{5C_*' B_{R}}\int_0^{5C_*'\ell(R)}\beta_{\nu,2}(x,r)^2\, \frac{dr}{r}d\nu(x)\\
\overset{\eqref{eq:betasUR}}{\lesssim}_{\Lambda_*,\delta_0,K}\ell(R)^n.
\end{multline*}

To estimate $S_2$ we change the order of summation:
\begin{equation*}
S_2 = \sum_{P\in\Reg_*(R)}\mu(P)\sum_{\substack{Q\in\treg_*(R):\\  C_*B_P\subset C_*'B_Q}}\frac{\ell(P)}{\ell(Q)}.
\end{equation*}
Note that the inner sum is essentially a geometric series, and so
\begin{equation*}
S_2 \lesssim_{C_*,C_*'} \sum_{P\in\Reg_*(R)}\mu(P)\le \mu(R).
\end{equation*}

Finally, we can bound $S_3$ using the small boundaries property of the David-Mattila lattice \eqref{eqsmb2}. To be more precise, note that for $Q\in\treg_*(R)$ if $2B_Q\setminus R\neq\varnothing$ and $\ell(Q)=A_0^{-k}\ell(R)$, then necessarily $Q\subset N_{k-1}(R)$, and even $2B_Q\subset N_{k-1}(R)$. Furthermore, the balls $2B_Q$ for cubes of the same generation have only bounded intersection. Thus,
\begin{equation*}
S_3\le \sum_{k\ge 1}\sum_{\substack{Q\subset N_{k-1}(R),\\ \ell(Q)=A_0^{-k}\ell(R)}}\mu(2B_Q)\lesssim \sum_{k\ge 1}\mu(N_{k-1}(R))\overset{\eqref{eqsmb2}}{\lesssim}\mu(90B(R))\approx\mu(R).
\end{equation*}

Putting the estimates for $S_1,\ S_2$ and $S_3$ together we arrive at
\begin{equation*}
\sum_{Q\in\treg_*(R)}\beta_{\mu,2}(2B_Q)^2\mu(Q)\lesssim_{\Lambda_*,\delta_0,K} \Theta(R_0)^2\ell(R)^n + \Theta(R_0)\mu(R)\lesssim_{\delta_0}\Theta(R_0)\mu(R),
\end{equation*}
where in the last estimate we used the fact that $\Theta(R_0)\lesssim\delta_0^{-1}\Theta(R)$ (note that $R\not\in\LD(R_0)$ because $\DD_\mu^\PP\cap\LD(R_0)\subset\End_*(R_0)$ and we assume $R\in\DD_\mu^\PP\setminus\End_*(R_0)$). This finishes the proof of Lemma \ref{lembetas99}.

% ********************************************************************************************

\vv

\subsection{The corona decomposition and the proof of Lemma \ref{lemtreebeta}}

Now we 
 define $\wh \ttt= \wh\ttt(R_0)$ inductively.
We set $\wh \ttt_0=\{R_0\}$ and, assuming $\wh\ttt_k$ to be defined, we let
$$\wh\ttt_{k+1} = \bigcup_{R\in\wh\ttt_k} (\wh \End(R)\setminus\End_*(R_0)).$$
Then we let
$$\wh\ttt =\bigcup_{k\geq0} \wh\ttt_k.$$
In this way, we have
$$\tree(R_0)=\bigcup_{R\in\wh\ttt} \wh\tree(R).$$

\vv

\begin{lemma}\label{lemtop8}
We have
$$\sigma(\wh \ttt)  \lesssim_{\Lambda_*,\delta_0}\sigma(R_0) +\!\sum_{Q\in\tree(R_0)}\!\|\Delta_Q\RR\mu\|_{L^2(\mu)}^2+ \sum_{Q\in\tree(R_0)\cap\DB}\! \EE_\infty(9Q).$$
\end{lemma}

\begin{proof} 
By Lemma \ref{lem9.5*}, we have
$$\sum_{Q\in\wh\sss(R)\cap(\LD(R_0)\cup\HD_*(R_0)\cup\BR(R))}\mu(Q) 
+\sum_{Q\in\wh\Ch((iii)_R)\cap \LD(R_0)}\mu(Q) + \mu(\wh\sG(R))
\approx \mu(R).$$
By the construction of $\wh\tree(R)$, the cubes from 
$\wh\sss(R)\cup \wh\Ch((iii)_R)$ belong to $\wh\tree(R)$ and the ones from $\wh\sss(R)\cap\BR(R)$ belong to $\wh \End(R)$, and so
$$\mu(R)\approx \sum_{Q\in\wh\tree(R)\cap(\LD(R_0)\cup\HD_*(R_0))}\mu(Q) 
+ \sum_{Q\in\wh\End(R)\cap\BR(R)}\!\!\mu(Q)
+ \mu(\wh\sG(R)).$$
Notice that the families $\wh\tree(R)$, with $R\in\wh\ttt$, are disjoint, with the possible exception of
the roots and ending cubes of the trees $\wh\tree(\cdot)$, which may belong to two different trees. 
Then we deduce that
\begin{align*}
\sum_{R\in\wh\ttt} \mu(R) & \approx \sum_{R\in\wh\ttt} \bigg(
\sum_{Q\in\wh\tree(R)\cap(\LD(R_0)\cup\HD_*(R_0))}\!\!\!\!\mu(Q) + \sum_{Q\in\wh\End(R)\cap\BR(R)}\!\!\!\mu(Q)\bigg)
+\sum_{R\in\wh\ttt}\! \mu(\wh\sG(R))\\
& \lesssim \mu(R_0) + \sum_{R\in\wh\ttt}\,\sum_{Q\in\wh\End(R)\cap\BR(R)}\!\!\mu(Q).
\end{align*}
Since the cubes $Q\in\BR(R)$ do not belong to $\LD(R_0)$, we have $\Theta(Q)\approx_{\Lambda_*,\delta_0} \Theta(R_0)$ for such cubes. The same happens for $R\in\wh\ttt$, and thus
\begin{equation}\label{eqthus883}
\sigma(\wh \ttt)  \lesssim_{\Lambda_*,\delta_0} \sigma(R_0) + \sum_{R\in\wh\ttt} \Theta(R)^2\sum_{Q\in\wh\End(R)\cap\BR(R)}\mu(Q).
\end{equation}

To estimate the last sum above, we claim that for a given $Q\in\BR(R)\cap\wh\End(R)$ we have
$$\big|\RR(\chi_{2R\setminus 2Q}\mu)(x_Q) - 
\big(m_{\mu,Q}(\RR\mu) - m_{\mu,R}(\RR\mu)\big)\big|
\lesssim \PP(R) + \left(\frac{\EE(4R)}{\mu(R)}\right)^{1/2} + \PP(Q) +  \left(\frac{\EE(2Q)}{\mu(Q)}\right)^{1/2}.$$
This is proved exactly in the same way as Lemma \ref{lemaprox2} (see also \rf{eqal842}) and so we omit the arguments. 
In case that both $R,Q\not\in \DB$, then 
$$\left(\frac{\EE(4R)}{\mu(R)}\right)^{1/2}\leq M\,\Theta(R)\quad \text{ and }\quad   \left(\frac{\EE(2Q)}{\mu(Q)}\right)^{1/2}\leq M\,\Theta(Q),$$
and so
we get
$$\big|\RR(\chi_{2R\setminus 2Q}\mu)(x_Q) - 
\big(m_{\mu,Q}(\RR\mu) - m_{\mu,R}(\RR\mu)\big)\big|
\leq C(\Lambda_*,\delta_0,M) \Theta(R).$$
Thus, by the $\BR(R)$ condition,
$$K\,\Theta(R) \leq |\RR(\chi_{2R\setminus 2Q}\mu)(x_Q)| \leq 
\big|m_{\mu,Q}(\RR\mu) - m_{\mu,R}(\RR\mu)\big| + C(\Lambda_*,\delta_0,M) \Theta(R).$$
Hence, for $K\geq 2\,C(\Lambda_*,\delta_0,M)$, we obtain
$$\frac12K\,\Theta(R) \leq 
\big|m_{\mu,Q}(\RR\mu) - m_{\mu,R}(\RR\mu)\big|.$$

In the general case where $Q$ and $R$ may belong to $\DB$, by analogous arguments, we get
$$\frac12K\,\Theta(R) \leq 
\big|m_{\mu,Q}(\RR\mu) - m_{\mu,R}(\RR\mu)\big| + \chi_\DB(R) \left(\frac{\EE(4R)}{\mu(R)}\right)^{1/2}
+ \chi_\DB(Q) \left(\frac{\EE(2Q)}{\mu(Q)}\right)^{1/2},$$
where $\chi_\DB(P)=1$ if $P\in\DB$ and $\chi_\DB(P)=0$ otherwise. Since
$$m_{\mu,Q}(\RR\mu) - m_{\mu,R}(\RR\mu) = \chi_Q \sum_{P\in\wh\tree(R)\setminus\wh\End(R)}\Delta_P(\RR\mu),$$
assuming $K\geq1$, we get
\begin{align}\label{eqfje42}
\Theta(R)^2\sum_{Q\in\wh\End(R)\cap\BR(R)}\mu(Q) &\lesssim 
\sum_{Q\in\wh\End(R)\cap\BR(R)} \int_Q \Big|\sum_{P\in\wh\tree(R)\setminus\wh\End(R)}\Delta_P(\RR\mu)
\Big|^2\,d\mu \\
&\quad + \chi_\DB(R)\!\!\sum_{Q\in\wh\End(R)\cap\BR(R)}  \frac{\EE(4R)}{\mu(R)}\,\mu(Q) +\!
\sum_{Q\in\wh\End(R)\cap\DB} \! \EE(2Q).\nonumber
\end{align}
By orthogonality, the first sum on the right hand is bounded by
$$\int \Big|\sum_{P\in\wh\tree(R)\setminus\wh\End(R)}\Delta_P(\RR\mu)
\Big|^2\,d\mu = \sum_{P\in\wh\tree(R)\setminus\wh\End(R)}\|\Delta_P(\RR\mu)\|_{L^2(\mu)}^2.$$
Also, it is clear that the second sum on the right had side of \rf{eqfje42} does not exceed 
$\chi_\DB(R)\,\EE(4R)$. Therefore,
\begin{multline*}
\Theta(R)^2\!\!\sum_{Q\in\wh\End(R)\cap\BR(R)}\mu(Q) \\
\lesssim
\sum_{P\in\wh\tree(R)\setminus\wh\End(R)}\|\Delta_P(\RR\mu)\|_{L^2(\mu)}^2
+ \chi_\DB(R)\,\EE(4R) + \sum_{Q\in\wh\End(R)\cap\DB} \! \EE(2Q).
\end{multline*}

Plugging the previous estimate into \rf{eqthus883}, we obtain
\begin{align*}
\sigma(\wh \ttt)  & \lesssim_{\Lambda_*,\delta_0} \sigma(R_0) + \sum_{R\in\wh\ttt} 
\sum_{P\in\wh\tree(R)\setminus\wh\End(R)}\|\Delta_P(\RR\mu)\|_{L^2(\mu)}^2\\
&\quad
+ \sum_{R\in\wh\ttt\cap\DB}\EE(4R) +\sum_{R\in\wh\ttt}\,\sum_{Q\in\wh\End(R)\cap\DB} \! \EE(2Q)\\
& \lesssim_{\Lambda_*,\delta_0} \sigma(R_0) + \sum_{P\in\tree(R_0)} 
\|\Delta_P(\RR\mu)\|_{L^2(\mu)}^2 
+ \sum_{Q\in\tree(R_0)\cap\DB}\EE(4Q),
\end{align*}
as wished.
\end{proof}
\vv

\begin{proof}[\bf Proof of Lemma \ref{lemtreebeta}]
Given $R_0\in\ttt$, combining Lemmas \ref{lembetas99} and \ref{lemtop8}, we obtain
\begin{align*}
\sum_{Q\in\tree(R_0)} \!\!\!\beta_{\mu,2}(2B_Q)^2\,\Theta(Q)\,\mu(Q)  & 
\lesssim_{\Lambda_*,\delta_0}  \sum_{R\in \wh\ttt} \Theta(R)\sum_{Q\in\wh\tree(R)} \beta_{\mu,2}(2B_Q)^2\,\mu(Q)\\
& \lesssim_{\Lambda_*,\delta_0}  \sum_{R\in \wh\ttt} \Theta(R)^2 \mu(R)\\
& 
\lesssim_{\Lambda_*,\delta_0}\sigma(R_0) +\!\!\sum_{Q\in\tree(R_0)}\!\|\Delta_Q\RR\mu\|_{L^2(\mu)}^2+ \!\!\sum_{Q\in\tree(R_0)\cap\DB}\! \!\EE_\infty(9Q).
\end{align*}
\end{proof}

\bigskip
\begin{center} 
	\Large Part \refstepcounter{parte}\theparte\label{part-4}: The proof of the Second Main Proposition
\end{center}
\smallskip

\addcontentsline{toc}{section}{\bf Part 4: The proof of the Second Main Proposition}

%\part{The proof of the Second Main Proposition}

In this part of the paper, corresponding to Sections \ref{sec5**} -- \ref{sec100} we choose $\OP(R)= \NDB(R)$ for any $R\in\MDW$ in the
construction of tractable trees. Our goal is to prove Main Proposition \ref{propomain2}.
%\begin{centert}
%\part{Proof of the Second Main Proposition}
%\end{center}
\vvv

% ********************************************************************************************
% ********************************************************************************************

\section{The good dominating family $\GDF$}\label{sec5**}

Let
\begin{equation}\label{eqk00}
M\ge M_0\coloneq A_0^{k_0 n}\quad \mbox{ for some $k_0\geq 4$ big enough.}
\end{equation}
For each $Q\in \DB(M)$ we choose the minimal $k(Q,M)\in\N$ such that \rf{eqDB} holds with $k=k(Q,M)$.
%By taking $M\gg1$, we will have $k(Q,M)\geq1$. 
%Notice that $k(Q,M)$ depends also on $a$. However, to simplify notation we do not keep track of this dependence in the notation.

\begin{lemma}\label{lemkq}
Assume $k_0$ big enough in \rf{eqk00}. For each $Q\in \DB(M)$, we have 
%$$k(Q,M) > \frac12 \bigg(1+ \frac{2n}{2n-a}\bigg)\, k_0$$.
$$k(Q,M) > \frac{8n- 1}{8n-2}\, k_0+4.$$
\end{lemma}

\begin{proof}
This follows from the fact that for $j\geq0$,
\begin{align*}
\sum_{P\in\hd^j(Q)\cap \DD_\mu(9Q)} 
\left(\frac{\ell(P)}{\ell(Q)}\right)^{\!1/2}\Theta(P)^2\mu(P) & \leq  CA_0^{2n j}\,\Theta(Q)^2\, m_j(Q)^{\!1/2}
\sum_{P\in\hd^j(Q)\cap \DD_\mu(9Q)} \mu(P) \\
& \leq C_1A_0^{2n j -\frac j2}\,\Theta(Q)^2\mu(9Q),
\end{align*}
where  we used \rf{eqforakdf33} and we applied \rf{eqmkpet4} to estimate $m_j(Q)$.
Then, for $0\leq j \leq  \frac{8n-1}{8n-2}\, k_0+4$, we have
\begin{align*}
\sum_{P\in\hd^j(Q)\cap \DD_\mu(9Q)} 
\left(\frac{\ell(P)}{\ell(Q)}\right)^{\!1/2}\Theta(P)^2\mu(P) & \leq C_1
A_0^{(2n-\frac12)\left(\frac{8n-1}{8n-2}\,k_0 +4\right)}\,\Theta(Q)^2\mu(9Q)\\
& = C_1
 M_0^{2}\,A_0^{-\frac{k_0}4+8n-2}\,
 \Theta(Q)^2\mu(9Q).
\end{align*}
So, for $k_0$ big enough, the right hand side above is smaller than $M^{2}\,\Theta(Q)^2\mu(9Q),$ 
which ensures that $k(Q,M)> \frac{8n- 1}{8n-2}\, k_0+4$.
\end{proof}
\vv

Let 
$$k_\Lambda = \frac{8n- 1}{8n-2}\, k_0,$$
so that $k(Q,M)>k_\Lambda$ for each $Q\in\DB(M)$, by the preceding lemma. 
Assuming $k_0$ to be a multiple of $8n-2$, it follows that $k_\Lambda$ is natural number.
 Notice also that $k_\Lambda$ is the mean of $k_0$ and $4nk_0/(4n-1)$, so that, if we let
$$\Lambda = A_0^{k_\Lambda n},$$
we have that $\Lambda$ is the geometric mean of $M_0
$ and $M_0^{\frac{4n}{4n-1}}$, that is
\begin{equation}\label{eqlammm}
\Lambda = M_0^{1/2}\,M_0^{\frac{2n}{4n-1}} = M_0^{\frac{8n-1}{8n-2}}>M_0.
\end{equation}
Note that this choice is consistent with \eqref{eq:LambdadepM} for $M=M_0$, assuming that $k_0=k_0(n)$ is big enough. 
\vv

Observe that, for $Q\in \DB(M)$, taking into account that $k(Q,M)-k_\Lambda>4$,
\begin{align}\label{eqprel3}
 M^2\,\Theta(Q)^2\,\mu(9Q) & \leq
\sum_{P\in\hd^{k(Q,M)}(Q)\cap\DD_\mu(9Q)} \left(\frac{\ell(P)}{\ell(Q)}\right)^{\!1/2}\Theta(P)^2\,\mu(P)\\
& = 
\sum_{S\in\hd^{k(Q,M)-k_\Lambda}(Q)\cap\DD_\mu(9Q)} \,\sum_{P\in\hd^{k(Q,M)}(Q): P\subset S}
\left(\frac{\ell(P)}{\ell(Q)}\right)^{\!1/2}\Theta(P)^2\,\mu(P) \nonumber\\
 & \leq \Lambda^2\,
\sum_{S\in\hd^{k(Q,M)-k_\Lambda}(Q)\cap\DD_\mu(9Q)}\!\! \Theta(S)^2\!\!\sum_{P\in\hd^{k(Q,M)}(Q): P\subset S}
\left(\frac{\ell(P)}{\ell(Q)}\right)^{\!1/2}\mu(P).\nonumber
\end{align}

Given $Q\in\DB(M)$ and $S\in \hd^{k(Q,M)-k_\Lambda}(Q)\cap\DD_\mu(9Q)$, we write $S\in G(Q,M)$ if 
\begin{equation}\label{eqgq1}
\mu(S) \leq 2 \Lambda^2
\sum_{P\in\hd^{k(Q,M)}(Q): P\subset S}
\left(\frac{\ell(P)}{\ell(S)}\right)^{\!1/2}\mu(P).
\end{equation}
We also denote $B(Q,M) = \hd^{k(Q,M)-k_\Lambda}(Q)\cap\DD_\mu(9Q)\setminus G(Q,M)$. 
%Here, $G(Q,M)$ and $B(Q,M)$ depend also on $M$ and $a$, but we will not keep track of this dependence.

Given $\lambda>0$, for $Q\in\DB(M)$ and $S\in G(Q,M)$, we denote
$${\rm big}_\lambda(S) = \{P\in\hd^{k(Q,M)}(Q):P\subset S,\,\ell(P)\geq \lambda \,\ell(S)\}.$$

\vv

\vv
\begin{lemma}\label{lem16}
If $\lambda>0$ satisfies 
\begin{equation}\label{eqlambda1}
\lambda\leq \frac{c_3}{\Lambda^{4}} =:\lambda_0(\Lambda)
\end{equation}
for some small absolute constant $c_3\in(0,1)$,
then, for each $Q\in\DB(M)$ we have
$$M^2\Theta(Q)^2\,\mu(Q) \lesssim \Lambda^{-\frac1{2n}}
\sum_{S\in G(Q,M)}\left(\frac{\ell(S)}{\ell(Q)}\right)^{\!1/2}\sum_{P\in {\rm big}_\lambda(S)}
\Theta(P)^2\mu(P).$$
Also, each $S\in G(Q,M)$ satisfies
$$\Theta(S)^2\,\mu(S) \leq 4 \Lambda^{-\frac1{2n}}\sum_{P\in{\rm big}_\lambda(S)}
\Theta_\mu(P)^2\,\mu(P).$$
\end{lemma}

\begin{proof} Arguing as in \rf{eqprel3}, by the definition of $B(Q,M)$ we get
\begin{align*}
M^2\,\Theta(Q)^2\,\mu(9Q) & \leq 
\sum_{S\in G(Q,M)} \sum_{P\in\hd^{k(Q,M)}(Q): P\subset S}
\left(\frac{\ell(P)}{\ell(Q)}\right)^{\!1/2} \Theta(P)^2\,\mu(P)\\
&\quad + 
\Lambda^2\,
\sum_{S\in B(Q,M)}\Theta(S)^2\!\!\sum_{P\in\hd^{k(Q,M)}(Q): P\subset S}
\left(\frac{\ell(P)}{\ell(Q)}\right)^{\!1/2}\mu(P)\\
& \leq \sum_{S\in G(Q,M)} \sum_{P\in\hd^{k(Q,M)}(Q): P\subset S}
\left(\frac{\ell(P)}{\ell(Q)}\right)^{\!1/2} \Theta(P)^2\,\mu(P)\\
&\quad + \frac12 
\sum_{S\in B(Q,M)} \Theta(S)^2
\left(\frac{\ell(S)}{\ell(Q)}\right)^{\!1/2}\mu(S).
\end{align*}
Using that $B(Q,M)\subset \hd^{k(Q,M)-k_\Lambda}(Q)\cap\DD_\mu(9Q)$ and that, by the definition of $k(Q,M)$, 
$$
\sum_{S\in \hd^{k(Q,M)-k_\Lambda}(Q)\cap\DD_\mu(9Q)} 
\left(\frac{\ell(S)}{\ell(Q)}\right)^{\!1/2}\Theta(S)^2\mu(S) \leq M^2\,\Theta(Q)^2\,\mu(9Q),$$
we get
$$M^2\,\Theta(Q)^2\,\mu(9Q) \leq 2
\sum_{S\in G(Q,M)} \sum_{P\in\hd^{k(Q,M)}(Q): P\subset S}
\left(\frac{\ell(P)}{\ell(Q)}\right)^{\!1/2}\Theta(P)^2\,\mu(P).$$

Next, for $S\in G(Q,M)$, we split
$$
\sum_{P\in\hd^{k(Q,M)}(Q): P\subset S}
\left(\frac{\ell(P)}{\ell(Q)}\right)^{\!1/2}\Theta(P)^2\,\mu(P) =
\sum_{P\in {\rm big}_\lambda(S)} \cdots \; + 
\sum_{P\in\hd^{k(Q,M)}(Q)\setminus {\rm big}_\lambda(S): P\subset S}\!\!\! \cdots.$$
We estimate the last sum:
\begin{align*}
\sum_{P\in\hd^{k(Q,M)}(Q)\setminus {\rm big}_\lambda(S): P\subset S}
\left(\frac{\ell(P)}{\ell(Q)}\right)^{\!1/2}\Theta(P)^2\,\mu(P) & \leq \lambda^{1/2}\,\left(\frac{\ell(S)}{\ell(Q)}\right)^{\!1/2}\!\!\!
\sum_{P\in\hd^{k(Q,M)}(Q): P\subset S}\!\!\!
\Theta(P)^2\,\mu(P)\\
&\leq \lambda^{1/2}\,\left(\frac{\ell(S)}{\ell(Q)}\right)^{\!1/2}\,\Lambda^2\,\Theta(S)^2\,\mu(S)\\
& \leq  c_3^{1/2}\,\left(\frac{\ell(S)}{\ell(Q)}\right)^{\!1/2}\,\Theta(S)^2\,\mu(S),
\end{align*}
taking into account the choice of $\lambda$ for the last estimate.
By \rf{eqgq1}, since $S\in G(Q,M)$, we have
\begin{align}\label{eqgh620}
\left(\frac{\ell(S)}{\ell(Q)}\right)^{\!1/2}\,\Theta(S)^2\,\mu(S) &\leq  2 \Lambda^2
\left(\frac{\ell(S)}{\ell(Q)}\right)^{\!1/2}\,\Theta(S)^2\,\sum_{P\in\hd^{k(Q,M)}(Q): P\subset S}
\left(\frac{\ell(P)}{\ell(S)}\right)^{\!1/2}\mu(P)\\
&\leq2
\sum_{P\in\hd^{k(Q,M)}(Q): P\subset S}
\left(\frac{\ell(P)}{\ell(Q)}\right)^{\!1/2}\Theta(P)^2\,\mu(P)\nonumber
\end{align}
Hence, for $c_3$ small enough,
$$\sum_{P\in\hd^{k(Q,M)}(Q)\setminus {\rm big}_\lambda(S): P\subset S}
\left(\frac{\ell(P)}{\ell(Q)}\right)^{\!1/2}\Theta(P)^2\,\mu(P)  \leq \frac12
\sum_{P\in\hd^{k(Q,M)}(Q): P\subset S}
\left(\frac{\ell(P)}{\ell(Q)}\right)^{\!1/2}\Theta(P)^2\,\mu(P).$$
Consequently, for every $S\in G(Q,M)$,
\begin{equation}\label{eqgh621}
\sum_{P\in\hd^{k(Q,M)}(Q): P\subset S}
\left(\frac{\ell(P)}{\ell(Q)}\right)^{\!1/2}\Theta(P)^2\,\mu(P)\leq 2
\sum_{P\in{\rm big}_\lambda(S): P\subset S}
\left(\frac{\ell(P)}{\ell(Q)}\right)^{\!1/2}\Theta(P)^2\,\mu(P).
\end{equation}

From the conditions $P\subset S$ and $\Theta(P)= \Lambda\,\Theta(S)$ we also get
$$\Theta(P)\leq \Theta(S)\,\frac{\ell(S)^n}{\ell(P)^n} 
= \Lambda^{-1}\,\Theta(P)\,\frac{\ell(S)^n}{\ell(P)^n}.$$
Thus,
\begin{equation}\label{eq**2}
\ell(P)\leq \Lambda^{-1/n}\ell(S).
\end{equation}
Then we derive
\begin{align*}
M^2\,\Theta(Q)^2\,\mu(Q) & \leq 4 \Lambda^{-\frac 1{2n}} 
\sum_{S\in G(Q,M)}  \left(\frac{\ell(S)}{\ell(Q)}\right)^{\!1/2}\sum_{P\in{\rm big}_\lambda(S)}
\Theta(P)^2\,\mu(P),
\end{align*}
which proves the first statement of the lemma.

Concerning the second statement, notice that by \rf{eqgh620}, \rf{eqgh621}, and \rf{eq**2}, we have
\begin{align*}
\Theta(S)^2\,\mu(S) & \leq 2
\sum_{P\in\hd^{k(Q,M)}(Q): P\subset S}
\left(\frac{\ell(P)}{\ell(S)}\right)^{\!1/2}\Theta(P)^2\,\mu(P)\\
& \leq 4 \sum_{P\in{\rm big}_\lambda(S)}
\left(\frac{\ell(P)}{\ell(S)}\right)^{\!1/2}\Theta(P)^2\,\mu(P)\\
& \leq 4 \Lambda^{-\frac1{2n}}\sum_{P\in{\rm big}_\lambda(S)}
\Theta(P)^2\,\mu(P).
\end{align*}
\end{proof}

\vv

We denote 
$$\DB := \DB(M_0) =
\bigcup_{M\geq M_0} \big(\DB(M) \setminus \DB(2M)\big).$$
Remark that the last identity holds because of the polynomial growth of $\mu$.
For each $Q\in \DB$, choose $M(Q)$ such that $Q\in\DB(M(Q))\setminus \DB(2M(Q))$. 
We denote by $\GDF$ (which stands for ``good dominating family") the family of the cubes $S\in\DD_\mu^\PP$ belonging to $G(Q,M(Q))$ for some $Q\in\DB$.
In particular, by the preceding lemma, the cubes $S\in\GDF$ 
satisfy the property that there exists a family $\cI_S\subset \DD_\mu^\PP(S)$
such that
\begin{equation}\label{cond21}
\Theta(P)= \Lambda\,\Theta(S)\quad \text{ and } \quad\ell(P)\geq\lambda_0\,\ell(S)\quad
\mbox{ for all $P\in \cI_S$}
\end{equation}
(with $\lambda_0=\lambda_0(\Lambda)$ as in \rf{eqlambda1}),
and
\begin{equation}\label{cond22}
\mu(S) \leq 4 \Lambda^{2- \frac 1{2n}}
\sum_{P\in \cI_S}\mu(P).
\end{equation}
%For a family $\cI\subset\DD_\mu$, we denote
%$$\sigma(\cI) = \sum_{P\in \cI}\Theta(P)^2\,\mu(P),$$
%and for $p\geq1$,
%$$\sigma_p(\cI) = \sum_{P\in \cI}\Theta(P)^p\,\mu(P),$$
%so that $\sigma(\cI)=\sigma_2(\cI)$. 
Observe that, for $S\in\GDF$,
\begin{equation}\label{eqiS}
\sigma(\cI_S) = \Lambda^2\Theta(S)^2\,\sum_{P\in \cI_S}\mu(P) \geq \frac14\, \Lambda^{\frac 1{2n}}
\sigma(S),
\end{equation}
by \rf{cond22}.
\vv

\begin{lemma}\label{lemdbnodb}
We have
$$\sum_{Q\in \DB} \EE_\infty(9Q) \lesssim \Lambda^{-\frac1{2n}} \sum_{S\in \GDF}\sigma(\cI_S).$$
\end{lemma}

\begin{proof}
For each $Q\in \DB$, choose $M=M(Q)$ such that $Q\in\DB(M)\setminus \DB(2M)$.
By the definition of $\EE_\infty(9Q)$, we have
$$\EE_\infty(9Q) \lesssim M(Q)^2\, \Theta(Q)^2\,\mu(Q).$$
Then, by Lemma \ref{lem16} we get
\begin{align*}
\sum_{Q\in \DB} \EE_\infty(9Q) & \lesssim  \Lambda^{-\frac1{2n}}
\sum_{Q\in \DB}
\sum_{S\in G(Q,M(Q))}\left(\frac{\ell(S)}{\ell(Q)}\right)^{\!1/2}\sum_{P\in {\rm big}_\lambda(S)}
\Theta(P)^2\mu(P)\\
& \lesssim \Lambda^{-\frac1{2n}}
\sum_{Q\in \DD_{\mu}^{\PP}} \sum_{S\in \GDF:S\subset 9Q}\left(\frac{\ell(S)}{\ell(Q)}\right)^{\!1/2}\,\sigma(\cI_S)\\
&= \Lambda^{-\frac1{2n}} \sum_{S\in \GDF}\sigma(\cI_S)
\sum_{Q\in \DD_{\mu}^{\PP}:9Q\supset S}
\left(\frac{\ell(S)}{\ell(Q)}\right)^{\!1/2}\\
&\lesssim \Lambda^{-\frac1{2n}} \sum_{S\in \GDF}\sigma(\cI_S).
\end{align*}
\end{proof}

\vv

%\begin{rem}\label{remLam}
%Notice that $M^2\Lambda^{\frac1{2n}}>\Lambda^2$. Indeed,\footnote{I think this remark is not necessary.} by \rf{eqlammm},
%$$M^2\Lambda^{\frac 1{2n}} \geq M_0^2\Lambda^{\frac 1{2n}} =\Lambda^{2\frac{8n-2}{8n-1}} \Lambda^{\frac 1{2n} }
%%%%= \Lambda^{2+a\left(\frac1n -  \frac2{4n-a}\right)} 
%= \Lambda^{2+\frac{4n-1}{n(8n-1)}}.$$
%\end{rem}
%\vv

\begin{lemma}
Let $\delta_0\in (0,1)$.
Let $S,P,P'\in\DD_{\mu}$ be such that $P\subset P'\subset S$. Suppose that 
$$\Theta(P)\geq \Lambda\,\Theta(S) \quad\mbox{ and }\quad\Theta(P')\leq \delta_0\,\Theta(S).$$
Then we have
$$c_4\ell(P) \leq (\delta_0\,\Lambda^{-1})^{1/n} \ell(P')\leq (\delta_0\,\Lambda^{-1})^{1/n} \ell(S) .$$
\end{lemma}

\begin{proof}
This is an immediate consequence of the following:
$$\Theta(P') \,\frac{\ell(P')^n}{\ell(P)^n}\geq c\,\Theta(P)  \geq c\Lambda\,\Theta(S)\geq c\Lambda\delta_0^{-1}\Theta(P').$$
\end{proof}
\vv

\begin{rem}\label{rem*1}
Let 
\begin{equation}\label{eqlambda0}
\lambda=\lambda_0(\Lambda)= \frac{c_3}{\Lambda^{4}}.
\end{equation}
By the preceding lemma, if 
$(\delta_0 \Lambda^{-1})^{1/n} < c_4\lambda,$
or equivalently,
\begin{equation}\label{eqdeltaM}
\delta_0< c' \Lambda^{1-4n},
\end{equation}
then, for any $S\in\GDF$ and $P\in \cI_S$, 
 there does not exist any cube $P'\in\DD_\mu$ satisfying 
$$P\subset P'\subset S
\quad\mbox{ and }\quad \Theta(P')\leq \delta_0\,\Theta(S).$$
\end{rem}
\vv

Another easy (but important) property of the family $\GDF$ is stated in the next lemma.

\begin{lemma}\label{lementrecubs}
Let $S_1,S_2\in\GDF$ be such that $S_2\subsetneq S_1$ and $\Theta(S_1)=\Theta(S_2)$. Then
there exist $Q\in\DB$ and $Q'\in \DD_\mu$ such that $Q'\subset 9Q$, $\ell(Q)=\ell(Q')$ and
$S_1\supset Q'\supset S_2$, with $S_2\in G(Q,M(Q))$ for some $M(Q)\geq M_0$.
\end{lemma}

\begin{proof}
This is due to the fact that, by definition, there exists $Q\in\DB$ such that  $S_2\in G(Q,M)$ for 
some $M=M(Q)\geq M_0$. So $S_2\in\hd^{k(Q,M)-k_\Lambda}(Q)$ and $S_2\subset 9Q$. Then 
$\ell(Q)\leq\ell(S_1)$, because otherwise $S_2\not\in\hd^{k(Q,M)-k_\Lambda}(Q)$ since $S_2$ would not be  
maximal among the cubes $S$ contained in $9Q$ such that $\Theta(S)\geq A_0^{(k(Q,M)-k_\Lambda)n}\,\Theta(Q)$ (as $S_1$ is also contained in $9Q$ and $\Theta(S_1)= A_0^{(k(Q,M)-k_\Lambda)n}\,\Theta(Q)$). The fact that $\ell(S_1)\geq\ell(Q)\geq\ell(S_2)$ implies the existence of a cube $Q'$ such as the one in the lemma.
\end{proof}
\vv

\begin{lemma}\label{lemred*}
For any cube $R\in\GDF$, we have $\cI_R\subset\big(\HD(R)\setminus \NDB(R)\big)\cap\sss(R)$, and thus
$$\sigma(\HD(R)\cap\sss(R)\setminus \NDB(R))\geq \frac14\,\Lambda^{\frac1{2n}}\,\sigma(R)
.$$
\end{lemma}

\begin{proof}
Recall that, by \rf{eqiS},
$$\sigma(\cI_R)\geq \frac14\,\Lambda^{\frac1{2n}}\,\sigma(R)
.$$
By the choice of $\delta_0$ and Remark \ref{rem*1}, there does not exist any $Q\in\LD(R)$ which 
contains any cube from $\cI_R$. Also, the cubes from $Q\in \cI_R$ satisfy $\ell(Q)\geq \lambda\ell(R)$ and so there does exist any cube $Q'\in\NDB(R)$ such that $Q\subset Q'\subset R$. So $\cI_R\subset\HD(R)\cap\sss(R)\setminus \NDB(R)$. The last statement in the lemma follows from \rf{eqiS}.
\end{proof}

Note that the lemma above implies that $\GDF\subset\MDW$.

\vv

% ********************************************************************************************
% ********************************************************************************************

\section{The layers $\sF_j^h$ and $\sL_j^h$ and the typical tractable trees}\label{sec7}

In this section we set $\sF=\GDF$ and we consider the associated subfamilies $\sF_j$, $\sF_j^h$ and $\sL$, $\sL_j$, $\sL_j^h$ defined
in Section \ref{sec-layers}, so that we have the splitting
$$\GDF= \bigcup_{j\in\Z} \sF_j = \bigcup_{j\in\Z}\,\bigcup_{h\geq1} \sF_j^h.$$
Recall the properties shown in Lemma \ref{lemljh} for the families $\sF_j^h$ and $\sL_j^h$.
Notice that, by \rf{eqover5},
$$\sum_{R\in \GDF}\sigma(\HD_1(R)) \leq m_0\,B^{1/4} \!\!\sum_{R\in \sL(\GDF)}\sigma(\HD_1(e(R))),\nonumber
$$
where we wrote $\sL=\sL(\GDF)$ to emphasize the dependence of $\sL$ on $\GDF$.

%In case that $\sF=\GDF$ we write $\sF_j=\GDF_j$ and $\sF_j^h=\GDF_j^h$,
%and in case that $\sF=\MDW$ we write $\sF_j=\MDW_j$ and $\sF_j^h=\MDW_j^h$. Also, we write
%$\sL(\GDF)$, $\sL_j(\GDF)$, and $\sL_j^h(\GDF)$ instead of $\sL$, $\sL_j$ and $\sL_j^h$ when $\sF=\GDF$, and analogously when $\sF=\MDW$.
\vv

Observe that since $\GDF\subset\MDW$ by Lemma \ref{lemred*}, the families $\Trc_k(R)$ introduced in Section \ref{subsec:trc} are well-defined for $R\in\GDF$. Our next objective consists of proving the following lemma, which is the main technical achievement in this section.

\begin{lemma}\label{lemimp9*}
Choose $\sF=\GDF$.
There exists some constant $C_6$ such that, for all $P\in\DD_\mu$ and all $k\geq0$,
$$\#\big\{R\in\sL(\GDF):\exists \,Q\in\Trc_k(R) \mbox{ such that } P\in\TT(e'(Q))\big\}\leq C_6\,(\log\Lambda)^2.$$
\end{lemma}

Although the statement above looks similar to Lemma \ref{lemimp9}, the proof is very different. 
The cubes from the $\NDB(\,\cdot\,)$ play an important role in the arguments. In fact, the main reason for
the introduction of the stopping condition involving the family $\NDB(\,\cdot\,)$ in Section \ref{sec6} is that it allows to prove
this lemma.

\begin{proof}
%To shorten notation, we will write here  $\sL$, $\sL_j$ and $\sL_j^h$ instead of $\sL(\GDF)$, $\sL_j(\GDF)$.
First notice that if $R\in\sL(\GDF)$ and $Q\in\Trc_k(R)$ are such that $P\in\TT(e'(Q))\setminus\Neg(e'(Q))$, then there exists some $\PP$-doubling cube that contains $P$ and belongs to $\TT_\sss(e'(Q))$, by the definition of
the family $\Neg(e'(Q))$. We denote by $\wt P$ the smallest such $\PP$-doubling cube. This satisfies
$$\delta_0\,\Theta(Q)\lesssim \Theta(\wt P)\leq \Lambda^2\,\Theta(Q),$$
or equivalently, $\Lambda^k\Theta(R)\in [\Lambda^{-2}\Theta(\wt P),C\delta_0^{-1}\Theta(\wt P)]$.
Hence, if $\Theta(R)=A_0^{nj}$, it follows that 
$$-C\log\Lambda\leq j + c\,k\log\Lambda - c'\log\Theta(\wt P)\leq C|\log\delta_0| = C'\log\Lambda.$$
Thus, $R$ belongs at most to $C''\log\Lambda$ families  $\sL_j$ such that
there exists $Q\in\Trc_k(R)$ such that $P\in\TT(e'(Q))\setminus\Neg(e'(Q))$

Suppose now that there exists $Q\in\Trc_k(R)$ such that $P\in\Neg(e'(Q))\subset\TT(e'(Q))$.
In this case, by Lemma \ref{lemnegs}, $\ell(P) \gtrsim \delta_0^{-2}\,\ell(Q)$. Hence, there
are at most $C\,|\log\delta_0|\approx \log\Lambda$ cubes $Q$ such that 
$P\in\TT(e'(Q))\cap\Neg(e'(Q))$, which in turn implies that again there are at most 
$C'''\log\Lambda$ families  $\sL_j$ such that
there exists $Q\in\Trc_k(R)$ satisfying $P\in\TT(e'(Q))\cap\Neg(e'(Q))$.

By the previous discussion, to prove
the lemma, it is enough to show that, for each $j\in\Z$, $P\in\DD_\mu$, $k\geq0$,
\begin{equation}\label{eqlj83**}
\#\sL_j(P,k) \leq C_7\log\Lambda,
\end{equation}
where
$$\sL_j(P,k)= \big\{R\in\sL_j:\exists \,Q\in\Trc_k(R) \mbox{ such that } P\in\TT(e'(Q))\big\}.$$

%Next, to shorten notation, we write, we will write here  $\sL$, $\sL_j$ and $\sL_j^h$ instead of $\sL$, $\sL_j$.
To prove \rf{eqlj83**}, let $R_0$ be a cube in $\sL_j(P,k)$ with maximal side length, and let $h_0$ be such that $R_0\in \sL_j^{h_0}(P,k)\equiv \sL_j(P,k)\cap \sL_j^{h_0}$. 

\begin{claim}
Let $R_1$ be another cube from $\sL_j(P,k)$, and let $h_1$ be such that $R_1\in \sL_j^{h_1}(P,k)$. 
Then $h_1\geq h_0$.
\end{claim}

\begin{proof}
Suppose that $h_1< h_0$.
Let $R_0^{h_1}$ be the cube that contains $R_0$ and belongs to
$\sF_j^{h_1}$. Observe that, by Lemma \ref{eqtec74},
\begin{align*}
P& \subset  B(e''(R_0)) \cap B(e''(R_1)) \subset B(x_{R_0},\tfrac32\ell(R_0)) \cap
B(x_{R_1},\tfrac32\ell(R_1)).
\end{align*}
Since $\ell(R_0)\geq\ell(R_1)$, we infer that
$$B(x_{R_1},\tfrac32\ell(R_1)) \subset B(x_{R_0},\tfrac92\ell(R_0)).$$
As $x_{R_0}\in B(x_{R_0}^{h_1}, \frac12\ell(R_0^{h_1}))$ and
$\ell(R_0)\leq A_0^{-1}\ell(R_0^{h_1})$, we deduce that
\begin{align*}
 B(x_{R_0},\tfrac92\ell(R_0))  & \subset B(x_{R_0}^{h_1}, \tfrac12\ell(R_0^{h_1}) +
\tfrac92\ell(R_0)) \subset B(x_{R_0}^{h_1}, \tfrac12\ell(R_0^{h_1})\! +
\tfrac92 A_0^{-1}\ell(R_0^{h_1})) \\
&\subset B(x_{R_0}^{h_1}, (\tfrac12 + 8 A_0^{-1})\ell(R_0^{h_1})) \subset B(e^{(4)}(R_0^{h_1})),
\end{align*}
where the lat inclusion follows from the definition of $B(e^{(4)}(R_0^{h_1}))$.
Then we deduce that
$$ B(e^{(4)}(R_1))\subset B(x_{R_1},\tfrac32\ell(R_1)) \subset
 B(e^{(4)}(R_0^{h_1})),$$
which contradicts the property (i) of the family $\sL_j^{h_1}$ in Lemma \ref{lemljh}, because $R_1\neq R_0^{h_1}$.
\end{proof}

\begin{claim}
Let $R_1$ be another cube from $\sL_j(P,k)$, and let $h_1$ be such that $R_1\in \sL_j^{h_1}(P,k)$. 
Then 
\begin{equation}\label{eqclaimh1**}
h_1\leq h_0+C\,\log\Lambda.
\end{equation}
% or
%\begin{equation}\label{eqsec85}
%\ell(R_1)\approx\ell(P)\quad\mbox{ and }\quad \dist(R_1,P)\lesssim \ell(P).
%\end{equation}
\end{claim}

\begin{proof}
Suppose that $h_1> h_0+1$. This implies that there are cubes
$\{R_1^h\}_{h_0+1\leq h \leq h_1-1}$ such that $R_1^h\in\sF_j^h$, with
$$R_1^{h_0+1}\supset R_1^{h_0+2}\supset\ldots \supset R_1^{h_1-1}\supsetneq R_1^{h_1}=R_1.$$
Observe now that $\ell(R_1^{h_0+1})\leq \ell(R_0)$. Otherwise, there exists some cube
$R_1^{h_0}\in\sF_j^{h_0}$ that contains $R_1^{h_0+1}$ with 
$$\ell(R_1^{h_0})\geq A_0\,\ell(R_1^{h_0+1})\geq A_0\,\ell(R_0),$$
Since $P \subset  B(e''(R_0)) \cap B(e''(R_1))$, arguing as in the previous claim, we deduce that
$B(e^{(4)}(R_0))\subset B(e^{(4)}(R_1^{h_0}))$, which contradicts again the property (i) of the family $\sL_j^{h_0}$, as above. So we have
$$\ell(R_1^h)\leq \ell(R_1^{h_0+1})\leq \ell(R_0)\quad \mbox{ for $h\geq h_0+1$.}$$

By the construction of $\Trc_k(R_0)$ in Section \ref{subsec:trc}, there exists a sequence of cubes
$S_0=R_0, S_1, S_2, \ldots, S_k=Q$ such that 
$$ S_{i+1}\in \GH(S_i)\; \mbox{ for $i=0,\ldots,k-1$,}$$
and $P\in\TT(e'(S_{k}))$. In case that $P$ is contained in some $Q'\in\HD_1(e'(Q))=\HD_1(e'(S_k))$, we write $S_{k+1}=Q'$, and we let $\tilde k:=k+1$. Otherwise, we let $\tilde k:=k$.
In any case, obviously we have  $\ell(S_{i+1})<\ell(S_i)$ for all $i$.
So, for each $h$ with $h_0+1\leq h\leq h_1$ there is some $i=i(h)$ such that 
\begin{equation}\label{eq0asd**}
\ell(S_i)>\ell(R_1^h)\geq \ell(S_{i+1}),
\end{equation}
with $0\leq i \leq \tilde k$, where we understand that $S_{\tilde k +1}= P$. 
We claim that either $i\lesssim 1$ or $i=\wt k$, with the implicit constant depending on $n$. 
Indeed, in the case $i<\wt k$, let $T\in\DD_\mu$ be such that $T\supset S_{i+1}$ and $\ell(T)=\ell(R_1^h)$.
Notice that, since $2R_1^h\cap 2T\neq \varnothing$ (because both $2R_1^h$ and $2T$ contain $P$) and 
$\ell(R_1^h)=\ell(T)$, we have
\begin{equation}\label{eq1asd**}
\PP(T) \approx \PP(R_1^h)\approx \Theta(R_1^h)=\Theta(R_0).
\end{equation}
On the other hand, 
\begin{equation}\label{eq2asd**}
\PP(T)\geq \delta_0\,\Theta(S_i)
\end{equation} 
because otherwise 
either $T\in\LD(S_i)$ or it is contained in some cube from $\LD(S_i)\cup\NDB(S_i)$. In any case, this would imply that $S_{i+1}$ does not belong to $\HD_1(e'(S_i))$.
Thus, from \rf{eq1asd**} and \rf{eq2asd**}
we derive that
$$\Theta(R_0)\gtrsim \delta_0\,\Theta(S_i) =  \delta_0\,\Lambda^i\,\Theta(R_0).$$
Hence  $\Lambda^i\lesssim\delta_0^{-1}$, which yields $i\lesssim_n 1$ if $i<\wt k$, as claimed.

The preceding discussion implies that, in order to prove \rf{eqclaimh1**}, it suffices to show that, for each fixed 
$i=0,\ldots,\tilde k$, there are at most $C\log\Lambda$ cubes $R_1^h$ satisfying 
\rf{eq0asd**}. To this end, suppose first that $i<\tilde k$. It is easy to check that there is at most one 
cube  $R_1^h$ satisfying 
\rf{eq0asd**} such that
\begin{equation}\label{eqr1hh}
\ell(R_1^h)\leq \lambda\,\ell(S_i).
\end{equation}
Indeed, if otherwise $R_1^h$ and $R_1^{h'}$ are such that
\begin{equation}\label{eqigfj4}
\lambda\,\ell(S_i)\geq \ell(R_1^h) >\ell(R_1^{h'})> \ell(S_{i+1}),
\end{equation}
then, by Lemma \ref{lementrecubs},
there exist $T_a\in\DB$ and $T_b\in \DD_\mu$ such that $T_b\subset 9T_a$, $\ell(T_a)=\ell(T_b)$ and
$R_1^h\supsetneq T_b\supsetneq R_1^{h'}$.
Now, let $T_c\in\DD_\mu$ be such that $T_c\supset S_{i+1}$ and $\ell(T_c)=\ell(T_a)$.
 Since $T_a\in\DB$, we infer that $T_c\in\NDB(S_i)$, taking into account that 
 $2T_b\cap 2T_c\neq \varnothing$, $T_b\subset 9T_a$, and $\ell(T_c)<\lambda\,\ell(S_i)$.
 This is a contradiction, because this would imply that either $T_c\in\sss(e'(S_i))$ or 
  $T_c$ is contained in some cube from $\sss(e'(S_i))$, which ensures that
  $S_{i+1}\not\in\GH(S_i)$ (notice that we are using the fact that $i<\tilde k$).
So \rf{eqr1hh} holds. Clearly this implies that
there are at most $C|\log\lambda|\approx \log\Lambda$ cubes $R_1^h$ satisfying 
\rf{eq0asd**}. 

In the case $i=\tilde k$, the same argument as above shows that if
$R_1^h$ and $R_1^{h'}$ satisfy \rf{eqigfj4},
then the cube $T_c$ in the preceding paragraph belongs to $\NDB(S_{\tilde k})$ again.
So again either $T_c\in\sss(e'(S_{\tilde k}))$ or $T_c$ is contained in some cube from 
$\sss(e'(S_{\tilde k}))$. As a consequence, by the definition of the family $\End(e'(R))$ and Lemma
 \ref{lemdobpp}, if we denote by $T_m$ the $m$-th descendant of $T_c$ which contains $P$, it follows that 
$$\Theta(T_m)\lesssim A_0^{-m/2}\,\PP(T_c)\approx A_0^{-m/2}\,\PP(T_b) \leq A_0^{-m/2}\,\PP(R_1^{h'})
\quad \mbox{for all $m\geq1$.}$$
The last inequality follows from the fact that we can assume that $R_1^{h'}\in G(T_a,M)$, for some $M\geq M_0$, by Lemma \ref{lementrecubs}. Since any cube $R_1^{h''}$ with $h''>h'$ is contained in a cube $2T_m$ with $\ell(T_m)\approx\ell(R_1^{h''})$ for some $m\geq h'-h''-1$, we deduce that
\begin{align*}
\Theta(R_1^{h''}) & \lesssim \Theta(2T_m)\lesssim \Theta(T_{m-1})\lesssim A_0^{-(m-1)/2}\,\PP(R_1^{h'})\\
&\lesssim A_0^{-(h'-h'')/2}\,\PP(R_1^{h'})\approx A_0^{-(h'-h'')/2}\,\Theta(R_1^{h'}).
\end{align*}
Since $\Theta(R_1^{h'}) = \Theta(R_1^{h''})$, we get $|h'-h''|\lesssim1$. Consequently, in the case $i=\tilde k$ there are again at most $C|\log\lambda|\approx \log\Lambda$ cubes $R_1^h$ satisfying 
\rf{eq0asd**}. 
\end{proof}

To prove the lemma, notice that each family $\sL_j^h(P,k)$ consists of a single cube, at most.
Indeed, if $R\in\sL_j^h(P,k)$, then $P\subset B(e''(R))$. Thus if $R,R'\in\sL_j^h(P,k)$,
then $B(e''(R))\cap B(e''(R'))\neq\varnothing$, which cannot happen if $R\neq R'$.
From this fact and the preceding claims, we infer that $\#\sL_j(P,k)\leq C\log\Lambda$, so that \rf{eqlj83**} holds.
\end{proof}

\vv
Recall that, for $R\in\sL(\GDF)$, $Q\in\Trc_k(R)$, we write $P\sim \TT(e'(Q))$ if there exists some
$P'\in\TT(e'(Q))$ such that 
$$
A_0^{-2}\ell(P)\leq \ell(P')\leq A_0^2\,\ell(P)\quad \text{ and }\quad 20P'\cap20P\neq\varnothing.
$$
Recall also that $\TT(e'(Q))$ is called a typical tractable tree, and we write $Q\in \Ty$, if
$$
\sum_{P\in\DB:P\sim\TT(e'(Q))} \EE_\infty(9P)\leq \Lambda^{\frac{-1}{3n}}\,\sigma(\HD_1(e(Q))).
$$
%If $Q\in\DD_\mu^\PP $ is such that there exist some 
 %some $R\in\sL(\GDF)$ and some $k$ such that $Q\in\Trc_k(R)\cap \Ty$, we write $Q\in\Ty$ and we say that
 %$\TT(e'(Q))$ is typical.

\vv

\begin{lemma}\label{lemsuper**}
We have
$$\sum_{Q\in\DB} \EE_\infty(9Q) \lesssim \Lambda^{\frac{-1}{2n}}(\log\Lambda)^2 
\sum_{R\in\sL(\GDF)}\,
\sum_{k\geq0} B^{-k/2} \sum_{Q\in\Trc_k(R)\cap\Ty}\sigma(\HD_1(e(Q))).$$
\end{lemma}

\begin{proof}
By Lemmas \ref{lemdbnodb} and \ref{lemred*}, we have
$$\sum_{Q\in \DB} \EE_\infty(9Q) \lesssim \Lambda^{-\frac1{2n}} \sum_{R\in \GDF} \sigma(\HD_1(R)).$$
Also, by \rf{eqover5} and Lemma~\ref{eqtec74},
\begin{align*}
\sum_{R\in \GDF}\sigma(\HD_1(R)) & \leq m_0\,B^{1/4}\sum_{R\in \sL(\GDF)}\sigma(\HD_1(e(R)))\\
& \leq m_0\,B^{1/4}\sum_{R\in \sL(\GDF)} \sum_{k\geq0} B^{-k/2}\sum_{Q\in\Trc_k(R)}\sigma(\HD_1(e(Q))).
\end{align*}
Therefore,
\begin{align}\label{eqqp45}
\sum_{Q\in \DB} \EE_\infty(9Q) & \lesssim m_0\,B^{1/4}\,\Lambda^{-\frac1{2n}} 
\sum_{R\in \sL(\GDF)} \sum_{k\geq0} B^{-k/2}\sum_{Q\in\Trc_k(R)}\sigma(\HD_1(e(Q)))\\
& = m_0\,B^{1/4}\,\Lambda^{-\frac1{2n}} 
\sum_{R\in \sL(\GDF)} \sum_{k\geq0} B^{-k/2}\sum_{Q\in\Ty}\sigma(\HD_1(e(Q)))\nonumber
\\
&\quad + m_0\,B^{1/4}\,\Lambda^{-\frac1{2n}} 
\sum_{R\in \sL(\GDF)} \sum_{k\geq0} B^{-k/2}\sum_{Q\in\Trc_k(R)\setminus \Ty}\sigma(\HD_1(e(Q)))\nonumber\\
&=: S_1 + S_2.\nonumber
\end{align}
By definition, for $Q\in \Trc_k(R)\setminus \Ty$, we have
$$\sigma(\HD_1(e(Q)))\leq \Lambda^{\frac{1}{3n}}\sum_{P\in\DB:P\sim\TT(e'(R))} \EE_\infty(9P).$$
Hence, the term $S_2$ in \rf{eqqp45} does not exceed
\begin{multline*}
m_0\,B^{1/4} \,\Lambda^{-\frac1{2n}} \Lambda^{\frac{1}{3n}}
\sum_{R\in \sL(\GDF)} \sum_{k\geq0} B^{-k/2}\!\!\!\sum_{Q\in\Trc_k(R)\setminus \Ty}\,
\sum_{P\in\DB:P\sim\TT(e'(Q))} \EE_\infty(9P)\\
 \lesssim
m_0\,B^{1/4} \,\Lambda^{-\frac1{6n}} 
\sum_{P\in\DB} \EE_\infty(9P)
 \sum_{k\geq0} B^{-k/2}\,
\# A(P,k),
\end{multline*}
where
$$A(P,k)= 
\big\{R\in \sL(\GDF):P\sim\TT(e'(Q))\text{ for some } Q\in\Trc_k(R)\big\}.$$
From the definition  \rf{defsim0} and Lemma \ref{lemimp9*}, it follows that
\begin{align}\label{eqremaa95}
\#A(P,k) &\leq \sum_{\substack{
P'\in\DD_\mu: 20P'\cap20P\neq\varnothing\\A_0^{-2}\ell(P)\leq \ell(P')\leq A_0^2\ell(P)
}} \!\!\#
\big\{R\in \sL(\GDF):\exists \,Q\in\Trc_k(R) \text{ such that }P'\in\TT(e'(Q))\big\}\\
&\lesssim \sum_{\substack{
P'\in\DD_\mu: 20P'\cap20P\neq\varnothing\\A_0^{-2}\ell(P)\leq \ell(P')\leq A_0^2\ell(P)
}}\!\!\!(\log\Lambda)^2\lesssim (\log\Lambda)^2.\nonumber
\end{align}
Therefore, the term $S_2$ in \rf{eqqp45} satisfies
\begin{align*}
S_2& \lesssim m_0\,B^{1/4} \,\Lambda^{-\frac1{6n}} \,(\log \Lambda)^2
\sum_{P\in\DB} \EE_\infty(9P) \sum_{k\geq0} B^{-k/2} \lesssim m_0\,B^{1/4} \,\Lambda^{-\frac1{6n}} \,(\log \Lambda)^2
\sum_{P\in\DB} \EE_\infty(9P).
\end{align*}
Since $B = \Lambda^{\frac1{100n}}$,  we deduce that
$$S_2\leq \frac12 \sum_{P\in\DB} \EE_\infty(9P)$$
if $\Lambda$ is big enough.
So we get\footnote{Here we are assuming that
	$\sum_{P\in\DB} \EE_\infty(9P) <\infty$. To ensure that this holds, if necessary we may replace the measure $\mu$ by
	another approximating measure of the form $\mu_\ell =\vphi_\ell * \mu$, where $\vphi_\ell$ is a $C^\infty$ bump function supported on $B(0,\ell)$, with $\|\vphi_\ell\|_1=1$. Then we prove the Main Proposition \ref{propomain} for $\mu_\ell$ with bounds independent of $\ell$, and finally we let $\ell\to\infty$.
}
$$\sum_{Q\in\DB} \EE_\infty(9Q)\leq 2S_1,$$
which proves the lemma.
\end{proof}

\vv

% ********************************************************************************************

\section{The proof of the Second Main Proposition \ref{propomain2}} \label{sec100}

We have to show that
$$\sum_{Q\in\DB} \EE_\infty(9Q) \leq C\big( \|\RR\mu\|_{L^2(\mu)}^2 + 
\theta_0^2\,\|\mu\|\big),$$
with $C$ possibly depending on $\Lambda$ and other parameters.
Recall that by Lemma \ref{lemsuper**}, we have
$$\sum_{Q\in\DB} \EE_\infty(9Q) \lesssim \Lambda^{\frac{-1}{2n}}(\log\Lambda)^2 
\sum_{R\in\sL(\GDF)}\,
\sum_{k\geq0} B^{-k/2} \sum_{Q\in\Trc_k(R)\cap\Ty}\sigma(\HD_1(e(Q))).$$
Also, by Lemma \ref{lemimp9*}, it turns out that, for all $P\in\DD_\mu$ and all $k\geq0$,
\begin{equation}\label{eqfaxc25}
\#\big\{R\in\sL(\GDF):\exists \,Q\in\Trc_k(R) \mbox{ such that } P\in\TT(e'(Q))\big\}\leq C_6\,(\log\Lambda)^2.
\end{equation}
Observe now that, by Lemma \ref{lemalter*}, for each $Q\in\Trc_k(R)\cap\Ty$,
either
$$\sigma(\HD_1(e(Q)))\lesssim \theta_0^2\,\mu(Q)\leq \theta_0^2\,\ve_Z^{-1}\,\mu(Z(Q)),$$
where $Z(Q)$ is the set $Z$ appearing in \eqref{eqdef*f} (replacing $R$ by $Q$ there), or
$$\sigma(\HD_1(e(Q))) \leq \Lambda\,\|\Delta_{\wt \TT(e'(Q))} \RR\mu\|_{L^2(\mu)}^2.$$
Therefore,
\begin{align*}
\sum_{Q\in\DB} \EE_\infty(9Q) & \lesssim_\Lambda 
\sum_{R\in\sL(\GDF)}\,
\sum_{k\geq0} B^{-k/2} \sum_{Q\in\Trc_k(R)}
\|\Delta_{\wt \TT(e'(Q))} \RR\mu\|_{L^2(\mu)}^2\\
&\quad+ \theta_0^2\,\ve_Z^{-1}\sum_{R\in\sL(\GDF)}\,
\sum_{k\geq0} B^{-k/2} \!\!\!\!\sum_{Q\in\Trc_k(R)}
\mu(Z(Q))\\
&=: S_1 + S_2.
\end{align*}

% ********************************************************************************************
\vv

By the same arguments used to bound the term $T_1$ in Lemma \ref{lemt1t2t3}, using \rf{eqfaxc25} instead of \rf{eqtrk121},
we derive
$$S_1
 \lesssim_\Lambda \|\RR\mu\|_{L^2(\mu)}^2.$$
 Also, by estimates similar to the ones for the term $T_3$ in Lemma \ref{lemt1t2t3}, we get
$$S_2\lesssim \theta_0^2\,\ve_Z^{-1}\|\mu\|.
$$
So we obtain
$$\sum_{Q\in\DB} \EE_\infty(9Q)\lesssim_\Lambda \|\RR\mu\|_{L^2(\mu)}^2 + \ve_Z^{-1}
\theta_0^2\,\|\mu\|,$$
which concludes the proof of Second Main Proposition \ref{propomain}.

\vv

% ********************************************************************************************
% ********************************************************************************************
% ********************************************************************************************

\appendix

\section{Parameters and families}\label{app:param}
\section*{A list of parameters}
Here is the list of the most important constants and parameters that appear in the paper, in the order of appearance. We also point out the dependence of different constants on each other (usually we do not track dependence on dimension).
\begin{itemize}
	\item $C_0,\, A_0$ are the constants from the David-Mattila lattice, and they depend only on dimension, see Remark~\ref{rema00}. Moreover, $A_0$ is assumed big enough that Lemma~\ref{lem:43} and Lemma~\ref{lem:66} hold. Starting from Section \ref{sec:Pdoubling}, $A_0$ and $C_0$ are considered to be fixed constants, and all the subsequent estimates and parameters may depend on them. We do not track this dependence.
	\item $C_d$ is the constant from the definition of $\PP$-doubling cubes in Subsection~\ref{subsec:Pdoubling}, and we have $C_d =4A_0^n$.
	\item $M$ is the parameter associated with the $\DB$ family in Proposition~\ref{propomain} (the First Main Proposition).
	\item $M_0$ is a big enough dimensional constant associated with the $\DB$ family chosen in Proposition~\ref{propomain2} (the Second Main Proposition). It is of the form $M_0=A_0^{k_0n}$ for some $k_0$ depending on dimension.
	\item $\Lambda$ is the $\HD$ constant introduced in Subsection \ref{secMDW}, and it is of the form $\Lambda = A_0^{k_{\Lambda}n}$ for some large integer $k_{\Lambda}$. It depends on $M$, see \eqref{eq:LambdadepM}. In the proof of Proposition~\ref{propomain2} a more concrete value $\Lambda=M_0^{\frac{8n-1}{8n-2}}$ is chosen, see Section \ref{sec5**}.
	
	\item $\delta_0\in (0,\Lambda^{-4n^2})$ is the $\LD$ constant introduced in Subection \ref{secMDW}. It is fixed at the beginning of Section \ref{sec3.3}, see also Remark~\ref{rem9.12}.
	\item $B$ is the $\MDW$ constant introduced below \eqref{eq:MDWdef}, and it is of the form $B=\Lambda^{\frac{1}{100n}}.$
	\item $\ell_0>0$ is the parameter used in the definition of function $d_{R,\ell_0}$ at the beginning of Section~\ref{sec6.2*}. It is assumed to be small enough in Lemmas~\ref{leminteta*}, \ref{lemNZ}, \ref{lemalter*}, and Remark \ref{rem9.12}. We have no control over how small $\ell_0$ is (it comes from a ``qualitative argument''), but none of the other constants depend on it.
	\item $\varepsilon_n>0$ is a small auxiliary parameter introduced above Lemma~\ref{lem:66}. It is chosen in \eqref{eqchooseen}.
	\item $\varepsilon_Z>0$ is a small constant introduced in Lemma~\ref{lemregot}. We assume that $\ve_Z\in (0, \Lambda^{-300n})$.
	\item $\Lambda_*$ is the $\HD_*$ constant introduced in Section \ref{sec3.3}, and it is of the form $\Lambda_* = \Lambda^{\frac{N}{N-1}}$, where $N$ is a large dimensional constant fixed above Lemma~\ref{lemtoptop} (for example, $N=500n$ works).
	\item $K>0$ is the $\BR$ constant. It is chosen in the proof of Lemma~\ref{lemtop8}, and it is big enough depending on $\Lambda_*, \delta_0$ and $M$.
	\item $C_*>0$ is a constant from Lemma~\ref{lem:RegGamma}. It depends on $\Lambda_*$ and $\delta_0$.
\end{itemize}

\section*{A list of cube families and related objects}\label{app:fam}
Below we list the most important families of cubes that appear in the paper, along with a link to the section where they were defined. We also list some other notions related to cubes (e.g., the approximating measures, or enlarged cubes).
\begin{itemize}
	\item $\DD_\mu$ is the family of David-Mattila cubes from Section~\ref{sec:DMlatt}.
	\item $\DD^\PP_\mu$ is the family of $\PP$-doubling cubes from Subsection~\ref{subsec:Pdoubling}.
	\item $\DB$ are the cubes dominated from below defined in Subsection \ref{subsec:DB}.
	\item $\HD(R),\ \LD(R),\, \OP(R),\, \NDB(R),\ \bad(R),\, \Stop(R)$ are defined in Subsection~\ref{subsec:enlar}.
	\item $\MDW$ is defined in Subsection~\ref{secMDW}.
	\item $e_j(R),\ e(R),\ e'(R),\ e''(R),\ e^{(k)}(R)$ are various enlarged cubes, defined in Subsection~\ref{subsec:enlar}.
	\item $\sss(e_j(R))$ is defined in Subsection~\ref{subsec:enlar}.
	\item $\sss(e(R)),\ \sss(e'(R)),\ \HD_1(R),\ \HD_1(e(R)),\ \HD_1(e'(R)),\ \HD_2(e'(R)),$\\ $\sss_2(e'(R)),\ \TT_\sss(e'(R)),\ \Neg(e'(R)),\ \End(e'(R)),\ \TT(e'(R))$ are defined in Subsection~\ref{subsec:generalized}.
	\item $\Trc\subset\MDW$ is the family of tractable cubes (which are the roots of tractable trees). It is defined in Subsection~\ref{subsec:trc}.
	\item $\GH(R)$ is the family of good high density cubes, and it is defined in Lemma~\ref{lemalg1}.
	\item $\Gen(R)$ are the cubes generated by $R$, and $\Trc(R)= \Gen(R)\cap\Trc$. These families are constructed below Lemma~\ref{lemalg1}.
	\item $\Reg(e'(R))$ is the family of regularized stopping cubes, and $\TT_{\Reg}$ is the corresponding tree. They are defined in Subsection~\ref{sec6.2*}.
	\item $\eta$ is used to denote two different approximating measures; first, in Sections \ref{sec6}-\ref{sectrans} it is a measure approximating $\mu$ at the level of $\Reg(e'(R))$. It is defined below Lemma~\ref{lem74}. Later, in Section \ref{sec9}, it is an AD-regular measure approximating $\mu$ at the level of $\wh\sss(R)$. It is defined above Lemma~\ref{lem9.2}.
	\item $\sH_j(e'(R))$, $\sH$ and $\sH'$ are auxiliary families defined in Subsection~\ref{subsec:H}.
	\item $\nu = \vphi_R\,\eta$ is a smoother version of the approximating measure $\eta$, and it is defined in Subsection~\ref{subsec:H}.
	\item $\TT_{\sH'}$ is a tree of cubes associated to $\sH'$ and it is defined above \eqref{eqdefPsi1}.
	\item The notation $P\sim \TT(e'(R))$ for 	
	$R\in\MDW$ and $P\in\DD_{\mu}$ is defined in \rf{defsim0}.
	\item $\Ty$ is the family if cubes associated to typical tractable trees, and it is defined in \eqref{defty}.
	\item $\Reg_{\Neg},\ \sD_\Neg,\ \sM_\Neg,\ \wt\End,\ \wt\TT,\ Z,\ \wt Z$ are defined in Subsection~\ref{subsec:opera}.
		\item $\HD_*(R),\ \Stop_*(R),\ \End_*(R),\ \tree(R),\ \ttt$ are defined in Section~\ref{sec3.3}.
	\item $\sF_j,\ \sF_j^h,\ \sL_j,\ \sL_j^h,\ \sL$ are defined in Subsection~\ref{sec-layers}.
	\item $\BR(R)$ is the family of cubes with big Riesz transform. It is introduced in Subsection~\ref{subsec:91}.
	\item $\wh\Ch(Q),\ \wh\Stop(R),\ \wh\tree_0(R)$ are defined at the beginning of Subsection~\ref{subsec:91}.
	\item $(i)_R,\ (ii)_R,\ (iii)_R$ are defined above Lemma~\ref{lem9.5*}.
	\item $\wh\End(R),\ \wh\tree(R),\ \wh\Stop_*(R),\ \wh\tree_*(R)$ are defined at the end of Subsection~\ref{subsec:91}.
	\item $\Reg_*(R)$ and $\treg_*(R)$ are constructed above Lemma~\ref{lem:reg prop}. 
	\item $\wh\ttt$ is defined above Lemma~\ref{lemtop8}.
	\item $\GDF$ is the good dominated family defined in Section \ref{sec5**}.
\end{itemize}

\vvv


\begin{thebibliography}{NTWV}

\bibitem[AH]{AH} 
H. Aikawa and K. Hirata. {\em Doubling conditions for harmonic measure in John domains.} Ann. Inst. Fourier (Grenoble) 58 (2008), no. 2, 429--445. 


\bibitem[AHM+]{AHM3TV}
J.~Azzam, S.~Hofmann, J.M. Martell, S.~Mayboroda, M.~Mourgoglou, X.~Tolsa, and
  A.~Volberg. \emph{Rectifiability of harmonic measure}. Geom. Funct. Anal. (GAFA), 26(3) (2016), 703--728. 
  
\bibitem[AMT]{AMT} J. Azzam, M. Mourgoglou and X. Tolsa. {\em Mutual absolute continuity of
interior and exterior harmonic measure implies rectifiability.}  Comm. Pure Appl. Math. Vol. LXX (2017), 2121--2163. 

\bibitem[AMTV]{AMTV} J. Azzam, M. Mourgoglou, X. Tolsa and A. Volberg. {\em On a two-phase problem for harmonic measure in general domains.} 
Amer. J. Math. 141(5) (2019), 1259--1279.

\bibitem[AT]{Azzam-Tolsa} J. Azzam and X. Tolsa. {\em Characterization of n-rectifiability in terms of Jones' square function: Part II.} Geom. Funct. Anal. 25 (2015), no. 5, 1371--1412.

\bibitem[DM]{David-Mattila} G. David and P. Mattila. {\em Removable sets for Lipschitz
harmonic functions in the plane.} Rev. Mat. Iberoamericana 16(1) (2000),
137--215.

\bibitem[DS1]{DS1} G. David and S. Semmes. {\em Singular integrals and
rectifiable sets in $\R^n$: Beyond Lipschitz graphs,} Ast\'{e}risque
No. 193 (1991).

\bibitem[DS2]{DS2} G. David and S. Semmes. \emph{Analysis of and on uniformly
rectifiable sets}, Mathematical Surveys and Monographs, 38. American
Mathematical Society, Providence, RI, (1993).

\bibitem[DT]{DT} D. D\k{a}browski and X. Tolsa. {\em 
The measures with $L^2$-bounded Riesz transform satisfying a subcritical Wolff-type energy condition}. Preprint (2021).


\bibitem[ENV]{ENVo} V. Eiderman, F. Nazarov and A. Volberg. {\em The $s$-Riesz transform of an 
$s$-dimensional measure in $\R^2$ is unbounded for $1<s<2$.} 
J. Anal. Math. 122 (2014), 1--23. 

\bibitem[Gi]{Girela} D. Girela-Sarri\'on. {\em Geometric conditions for the $L^2$-boundedness of singular integral operators with odd kernels with respect to measures with polynomial growth in $R^d$.}  J. Anal. Math. 137 (2019), no. 1, 339--372. 

\bibitem[JN1]{JN1} B. Jaye and F. Nazarov. {\em Reflectionless measures for Calder\'on-Zygmund operators I: general theory.} J. Anal. Math. 135 (2018), no. 2, 599--638. 

\bibitem[JN2]{JN2} B. Jaye and F. Nazarov. {\em Reflectionless measures for Calder\'on-Zygmund operators II: Wolff potentials and rectifiability.} J. Eur. Math. Soc. (JEMS) 21 (2019), no. 2, 549–583.

\bibitem[JNRT]{JNRT} B. Jaye, F. Nazarov, and M.C. Reguera, and X. Tolsa. {\em The Riesz transform of codimension smaller than one and the Wolff energy}.  Mem.\ Amer.\ Math.\ Soc.  266 (2020), no.\ 1293.

\bibitem[JNT]{JNT} B. Jaye, F. Nazarov, and X. Tolsa. {\em  The measures with an associated square function operator bounded in $L^2$.}  Adv. Math. 339 (2018), 60--112.

\bibitem[Jo]{Jones} P.W.\ Jones. {\em Rectifiable sets and the travelling
salesman problem.} Invent. Math.\! 102 (1990), 1--15.

\bibitem[L\'e]{Leger} J.C. L\'eger. {\em Menger curvature and rectifiability.} Ann. of Math. 149 (1999), 831--869.

\bibitem[Ma]{Mattila-llibre} P. Mattila. {\em Geometry of sets and measures in
Euclidean spaces,} Cambridge Stud. Adv. Math. 44, Cambridge Univ.
Press, Cambridge, 1995.

\bibitem[MaV]{Mattila-Verdera} P. Mattila and J. Verdera. {\em Convergence of singular integrals with general measures.}
J. Eur. Math. Soc. (JEMS) 11 (2009), no. 2, 257--271.

\bibitem[Me]{Melnikov} M.S. Melnikov. {\em Analytic capacity: discrete
approach and curvature of a measure.} Sbornik: Mathematics
{186}(6) (1995), 827--846.

\bibitem[MP]{Mattila-Paramonov} P. Mattila and P. V. Paramonov. {\em On geometric properties of harmonic Lip$_1$ capacity.} Pacific J. Math. 171(2) (1995), 469--491.

\bibitem[MV]{MV} M.S. Melnikov and J. Verdera. {\em A geometric proof of the
$L^2$ boundedness of the Cauchy integral on Lipschitz graphs.}
Internat. Math. Res. Notices (1995), 325--331.

\bibitem[NToV1]{NToV1} F. Nazarov, X. Tolsa and A. Volberg. \emph{On the uniform
  rectifiability of {AD}-regular measures with bounded {R}iesz transform
  operator: the case of codimension 1}. Acta Math. {213} (2014), no.~2,
  237--321. 

\bibitem[NToV2]{NToV2} F. Nazarov, X. Tolsa and A. Volberg. {\em The Riesz transform,
rectifiability, and removability for
 Lipschitz harmonic functions.}
Publ. Mat. 58 (2014), 517--532.

\bibitem[NTrV1]{NTrV1} F. Nazarov, S. Treil, and A. Volberg. {\em Cauchy integral and Calder\'on–Zygmund operators on nonhomogeneous spaces.} Internat. Math. Res. Notices 15 (1997), 703--726.

\bibitem[NTrV2]{NTrV2} F. Nazarov, S. Treil, and A. Volberg. {\em The $Tb$-theorem on non-homogeneous spaces.} Acta Math. 190 (2) (2003).

\bibitem[Pa]{Paramonov} P.V. Paramonov. {\em Harmonic approximations in the {$C^1$}-norm}. Mat. Sb. 181 (1990), no. 10, 1341--1365.
 
\bibitem[RT]{Reguera-Tolsa} M.C. Reguera and X. Tolsa. {\em Riesz transforms of non-integer homogeneity on uniformly disconnected sets}. Trans. Amer. Math. Soc. 368 (2016), no. 10, 7045--7095. 

\bibitem[To1]{Tolsa-duke} X. Tolsa. {\em $L^2$-boundedness of the Cauchy
integral operator for continuous measures.} Duke Math. J. { 98}(2) (1999),
269-304.

\bibitem[To2]{Tolsa-sem} X. Tolsa. {\em Painlev\'{e}'s problem and the semiadditivity of analytic capacity}, Acta Math. 190:1 (2003), 105--149.

\bibitem[To3]{Tolsa-bilip} X. Tolsa. {\em Bilipschitz maps, analytic
capacity, and the Cauchy integral}. Ann. of Math. 162:3 (2005), 1241--1302.
 
\bibitem[To4]{Tolsa-memo} X. Tolsa. {\em Rectifiable measures, square functions involving densities, and the Cauchy transform.}
Mem. Amer. Math. Soc. {245} (2016), no.~1158, 1--130.

\bibitem[To5]{Tolsa-llibre} X. Tolsa. {\em Analytic capacity, the Cauchy transform, and non-homogeneous Calder\'on-Zygmund theory.} Progress in Mathematics, vol. 307, Birkh\"auser/Springer, Cham, 2014.

\bibitem[To6]{Tolsa-riesz} X. Tolsa. {\em The measures with $L^2$ bounded Riesz transform and the Painlev\'e problem for Lipschitz harmonic functions.} Preprint (2021).

\bibitem[Vo]{Volberg} A.\ Volberg, {\em Calder\'on--Zygmund capacities and operators on nonhomogeneous spaces.}
CBMS Regional Conf. Ser. in Math. 100, Amer. Math. Soc., Providence, 2003.
\end{thebibliography}
\end{document}